\documentclass[11pt, reqno]{amsart}
\usepackage[dvipsnames]{xcolor}

\usepackage{amscd, amsmath, amssymb, amsthm, yfonts, mathrsfs, hhline, tikz}
\usetikzlibrary{arrows, decorations.markings, patterns, backgrounds}
\tikzset{>=latex}
\usepackage[centering,margin=1.2in,footskip=1cm]{geometry}
\usepackage[matrix,arrow,curve]{xy}

\usepackage{parskip}
\makeatletter 
\def\thm@space@setup{%
 \thm@preskip=\parskip \thm@postskip=0pt
}
\def\th@remark{%
  \thm@headfont{\itshape}%
  \normalfont 
  \thm@preskip\parskip \thm@postskip=0pt
}
\makeatother

\usepackage{tikz-cd}

\usepackage{multirow}

\usepackage{array}

\usepackage{adjustbox}

\usetikzlibrary{braids}

\usepackage{graphicx} 
\usepackage{subcaption}

\usepackage{mystix2}

\usepackage{scalerel}[2016/12/29]

\usepackage{hyperref}

\global\binoppenalty=10000

\newtheorem{theorem}{Theorem}[section]
\newtheorem{lemma}[theorem]{Lemma}
\newtheorem{prop}[theorem]{Proposition}
\newtheorem{cor}[theorem]{Corollary}

\theoremstyle{definition}
\newtheorem{defn}[theorem]{Definition}
\newtheorem{remark}[theorem]{Remark}
\newtheorem{example}[theorem]{Example}

\numberwithin{equation}{section}

\def\beq{\begin{equation}}
\def\eeq{\end{equation}}

\def\longra{\longrightarrow}
\def\longla{\longleftarrow}

\newcommand{\hr}[1]{\left(#1\right)}                                                    
\newcommand{\hm}[1]{\left|#1\right|}                                                    
\newcommand{\hs}[1]{\left[#1\right]}                                                    
\newcommand{\hc}[1]{\left\{#1\right\}}                                                  

\def\le{\leqslant}
\def\ge{\geqslant}

\def\Ad{\operatorname{Ad}}

\def\C{\mathbb C}

\def\msD{\mathscr D}

\def\Fc{\mathcal F'}

\def\Gc{\mathcal G}

\def\Heis{\mathrm{Heis}}

\def\Hc{\mathcal H}

\def\Oc{\mathcal O}

\def\pt{\mathrm{pt}}

\def\Qc{\mathcal Q}

\def\R{\mathbb R}
\def\Rc{\mathcal R}

\def\Res{\mathrm{Res}}

\def\sgn{\operatorname{sgn}}

\def\sl{\mathfrak{sl}}
\def\Sc{\mathcal S}

\def\Z{\mathbb Z}

\newcommand{\gfr}[2]{\genfrac{\{}{\}}{0pt}{1}{#1}{#2}}

\newcommand{\Hsp}{\mathcal{H}}
\newcommand{\Bc}{\mathcal{B}}

\newcommand{\vb}{\mathbf{v}}

\newcommand{\eb}{\mathbf{E}}
\newcommand{\Gauss}{G}

\newcommand{\Qb}{\mathbf{Q}}
\newcommand{\Pb}{\mathbf{P}}
\newcommand{\Xb}{\mathbf{X}}

\newcommand{\bbH}{\mathbb{H}}

\newcommand{\bbU}{\mathbb{U}}
\newcommand{\bbV}{\mathbb{V}}
\newcommand{\rd}{\mathrm{d}\!\,}
\newcommand{\Kc}{\mathcal{K}}
\newcommand{\Fct}{\mathcal{F}}
\newcommand{\Fctt}{\mathcal{F}''}
\newcommand{\tFct}{\mathcal{G}}
\newcommand{\tFc}{\mathcal{G}'}
\newcommand{\tFctt}{\mathcal{G}''}
\newcommand{\tKc}{\mathcal L}

\newcommand{\Nilp}{U}
\newcommand{\Bor}{B}
\newcommand{\Tor}{T}

\newcommand{\nears}{\scaleobj{0.7}{\nearrow}}
\newcommand{\sears}{\scaleobj{0.7}{\searrow}}
\newcommand{\nwars}{\scaleobj{0.7}{\nwarrow}}
\newcommand{\swars}{\scaleobj{0.7}{\swarrow}}
\newcommand{\nee}{\overset{\nears}{e}}
\newcommand{\seee}{\overset{\sears}{e}}
\newcommand{\nwe}{\overset{\nwars}{e}}
\newcommand{\swe}{\overset{\swars}{e}}
\newcommand{\neC}{\overset{\nears}{C}}
\newcommand{\seC}{\overset{\sears}{C}}
\newcommand{\wneC}{\overset{\nears}{C}\,\!'}
\newcommand{\wseC}{\overset{\sears}{C}\,\!'}
\newcommand{\nevarpi}{\overset{\nears}{\varpi}}
\newcommand{\sevarpi}{\overset{\sears}{\varpi}}
\newcommand{\nwvarpi}{\overset{\nwars}{\varpi}}

\newcommand{\nef}{\overset{\scaleobj{0.7}{\nearrow}}{f}}
\newcommand{\sef}{\overset{\scaleobj{0.7}{\searrow}}{f}}

\newcommand{\id}{\mathrm{id}}

\newcommand{\Eb}{\mathbf{E}}
\newcommand{\Fb}{\mathbf{F}}
\newcommand{\Kb}{\mathbf{K}}
\newcommand{\Lb}{\mathbf{L}}
\newcommand{\dint}{\int\!\!\!\!\int}

\newcommand{\Ab}{\mathbf{A}}
\newcommand{\Bb}{\mathbf{B}}
\newcommand{\Cb}{\mathbf{C}}
\newcommand{\Db}{\mathbf{D}}
\newcommand{\Hb}{\mathbf{H}}

\newcommand{\mfsl}{\mathfrak{sl}}
\newcommand{\mfb}{\mathfrak{b}}

\newcommand{\Imm}{\mathrm{Im}}
\newcommand{\Ree}{\mathrm{Re}}

\newcommand{\btd}[1]{\raisebox{\depth}{$\bigtriangledown$\!}_{#1}}
\newcommand{\wbtd}[1]{\raisebox{\depth}{$\bigtriangledown'$\!\!}_{#1}}
\newcommand{\wwbtd}[1]{\Bor^-_{#1}}
\newcommand{\wwwbtd}[1]{\Bor^+_{#1}}

\newcommand{\Tr}{\mathrm{Tr}}

\newcommand{\wU}{\widetilde{U}}
\newcommand{\Wbb}{\mathbb{W}}
\newcommand{\wWbb}{\mathbb{W}_c}
\newcommand{\Vbb}{\mathbb{V}}
\newcommand{\Ubb}{\mathbb{U}}
\newcommand{\wUbb}{\widetilde{\Ubb}}

\newcommand{\wW}{\widetilde{W}}
\newcommand{\wV}{\widetilde{V}}

\newcommand{\Spec}{\mathrm{Spec}}
\newcommand{\opp}{\mathrm{op}}
\newcommand{\Ker}{\mathrm{Ker}}

\newcommand{\ext}{\mathrm{ext}}

\newcommand{\msF}{\mathscr{F}}
\newcommand{\mfm}{\mathfrak{m}}
\newcommand{\mfn}{\mathfrak{n}}
\newcommand{\Sb}{\mathbf{S}}
\newcommand{\Tb}{\mathbf{T}}
\newcommand{\Ec}{\mathcal{E}}

\newcommand{\Lc}{\mathcal{L}}
\newcommand{\G}{\mathbb{G}}

\newcommand{\Bbbb}{\mathbb{B}}
\newcommand{\Jc}{\mathcal{J}}
\newcommand{\Pbb}{\mathbb{P}}

\newcommand{\nsf}{\mathrm{nsf}}
\newcommand{\Hol}{\mathrm{Hol}}
\newcommand{\Rad}{\mathrm{Rad}}
\newcommand{\redu}{\mathrm{red}}
\newcommand{\msE}{\mathscr{E}}

\newcommand*{\ad}{\mathsf{ad}}

\def\Ebf{\mathbf{E}}
\def\stgE{\mathbf{E}}
\def\stgF{\mathbf{F}}

\newcommand{\boxsum}{
  \mathop{
    \vphantom{\bigoplus} 
    \mathchoice
      {\vcenter{\hbox{\resizebox{\widthof{$\displaystyle\bigoplus$}}{!}{$\boxplus$}}}}
      {\vcenter{\hbox{\resizebox{\widthof{$\bigoplus$}}{!}{$\boxplus$}}}}
      {\vcenter{\hbox{\resizebox{\widthof{$\scriptstyle\oplus$}}{!}{$\boxplus$}}}}
      {\vcenter{\hbox{\resizebox{\widthof{$\scriptscriptstyle\oplus$}}{!}{$\boxplus$}}}}
  }\displaylimits 
}

\newcommand{\Heiss}{\mathrm{H}}
\newcommand{\mbi}{\mathbf{i}}
\newcommand{\bbE}{\mathbb{E}}
\newcommand{\vNTens}{\bar{\otimes}}
\newcommand{\standard}{\mathrm{st}}

\newcommand{\triangspace}{space of type $A_n$}

\newcommand{\msP}{\mathscr{P}}

\title{Quantum $SL^+(N,\R)$ as a locally compact quantum group}

\author{K. De Commer}
\address{Vrije Universiteit Brussel}
\email{kenny.de.commer@vub.be}

\author{G. Schrader}
\address{Northwestern University}
\email{gus.schrader@northwestern.edu}

\author{A. Shapiro}
\address{University of Edinburgh}
\email{alexander.shapiro@ed.ac.uk}

\author{C. Voigt}
\address{University of Glasgow}
\email{christian.voigt@glasgow.ac.uk}

\begin{document}

\begin{abstract}
We construct the first examples of purely continuous, $q$-deformed Lie type locally compact quantum groups in higher rank. They arise from Drinfeld--Jimbo quantization, at unimodular deformation parameter, of the totally positive part of higher rank split real Lie groups in type $A$. Our techniques are based on quantum cluster theory, in particular as developed through the work of Fock and Goncharov.
\end{abstract}

\maketitle

\section*{Introduction}

Positive representations arise in the representation theory of split real quantum groups, and are deeply connected with quantum cluster algebras and higher quantum Teichm\"uller theory. They were first studied by Ponsot--Teschner \cite{PT99, PT01} in the rank $ 1 $ case in connection with Liouville theory, and later on in general type by Ip and Frenkel--Ip \cite{Ip12, FI14, Ip15}. A key feature of these representations is that the standard generators of the quantized universal enveloping algebra act by positive self-adjoint (unbounded) operators. This resembles classical Lie algebra representations arising from unitary representations of split semisimple Lie groups, except that, in this analogy, the positivity requirement corresponds to looking only at the totally positive part of the underlying groups.  

The main result of the present paper is the construction of a quantization $SL^+_\hbar(N,\R)$ of the semigroup $SL^+(N,\R)$ of totally positive matrices in $ SL(N, \mathbb{R}) $ as a locally compact quantum group in the sense of Kustermans--Vaes \cite{KV00, KV03}. Here $ \hbar $ is a nonzero real parameter, related to the usual $q$ in the algebraic setting by $ q = e^{\pi i \hbar} $. In the above analogy, this provides the analytic foundations for a global approach to the study of positive representations of quantum groups, and allows one to bring the powerful machinery of operator algebras into the picture. At the same time, it yields the first examples of Drinfeld--Jimbo-type higher rank locally compact quantum groups which are genuinely outside the realm of algebraic constructions.

Historically, the theory of locally compact quantum groups originated from attempts to generalize Pontryagin duality to the nonabelian setting, see \cite{Kus05} for an overview. To a large part, the structure of a locally compact quantum group is governed by a \emph{multiplicative unitary}, encoding the analogue of the regular representation of a locally compact group. The concept of a multiplicative unitary was studied earlier on by Baaj--Skandalis \cite{BS93}, and developed further by Woronowicz \cite{Wor96}, who introduced the notion of \emph{manageability}, later generalized to \emph{modularity} in \cite{SW01}. The key insight of Kustermans and Vaes was that requiring the existence of Haar weights allows one to build an elegant and extremely versatile theory of locally compact quantum groups from a simple set of axioms, closed under Pontryagin duality and encompassing all known examples at the time. They showed in particular that the \emph{Kac--Takesaki operator} of a locally compact quantum group is a manageable multiplicative unitary, thus linking their approach to the work of Woronowicz.  

However, apart from constructions of an essentially algebraic nature, only a few isolated examples of genuine locally compact quantum groups of $q$-deformation type were hitherto known, each involving subtle analytic considerations. These examples are limited to low rank \cite{VD01,WZ02,Wor03,KK03}, with case by case techniques. 

One of these examples is the quantum `$ax + b$'-group introduced by Woronowicz--Zakrzewski \cite{WZ02}.  For the analysis of this quantum group, Woronowicz studied the quantum exponential function \cite{Wor00}, which is closely related to Faddeev's noncompact quantum dilogarithm \cite{Fad95, Kas01, FG09}. Most of the intricacies of the quantum `$ax + b$'-group are due to the fact that the operators involved are only assumed to be self-adjoint, possibly with negative parts in their spectrum. 

It was first realised by Ip \cite{Ip13} that the work of Woronowicz--Zakrzewski implicitly contains all ingredients needed for a \emph{totally positive} version of the quantum `$ax + b$'-group, see also \cite{BT03}. This is surprising in at least two ways: firstly, the classical counterpart of the totally positive `$ax + b$'-quantum group is only a semigroup, and not a group, and secondly, working with positive self-adjoint operators simplifies the analysis significantly. Ip showed that one obtains a modular multiplicative unitary in this situation, and one also has Haar weights, thus giving a locally compact quantum group. One should view this positive `$ax+b$'-quantum group as corresponding to a Borel in the associated totally positive quantum $SL(2, \R) $-group, obtained via the Drinfeld double construction. 

This is the starting point of our work. 
The main algebraic ingredient we use to go beyond the rank 1 case is \emph{quantum cluster theory.} Cluster algebras were first introduced in the commutative setting by Fomin and Zelevinsky in the context of (dual) canonical bases in Lie theory \cite{Lus93}, and the theory was shown to admit a natural $q$-deformation by Berenstein and Zelevinsky in \cite{BZ05}, as well as by Fock and Goncharov~\cite{FG09} as part of their work on the quantization of higher Teichm\"uller spaces. Starting from the data of a \emph{quiver}, a quantum cluster algebra is defined as the intersection of a collection of mildly non-commutative algebras known as \emph{quantum tori} or skew-Laurent polynomial rings. The precise form of the intersection is governed by the combinatorics of \emph{cluster mutation}, with the intersection being taken over all quivers \emph{mutation-equivalent} to the original. A quantum cluster algebra contains a distinguished class of elements known as \emph{global monomials} (or `cluster monomials'). For a large class of quivers, it is proven in~\cite{DM21} (and~\cite{GHKK18} at $q=1$) that the expansion of a global monomial relative to the quantum tori of an arbitrary quiver is a skew-Laurent polynomial whose coefficients are elements of $\mathbb{N}[q^{\pm1}]$ invariant under the bar involution $q\mapsto q^{-1}$. In the cases relevant for our purposes, we have in fact an explicit combinatorial description of these expansion coefficients in terms of partition functions of paths in a directed graph. 

The setting of quantum tori and $q$-Laurent positivity is particularly attractive from an analytic point of view. On the one hand, there is in this setting a well-developed operator-algebraic deformation theory due to Rieffel~\cite{Rie93}. On the other, Laurent positivity guarantees that each summand in the expansion of a global monomial corresponds to a positive self-adjoint operator. By taking \emph{form sums} of such summands, we are able to make sense of the global monomials corresponding to PBW generators of the quantum group in various clusters as rather explicit positive operators.  This perspective seems to be new, and we hope it will find further applications in the analytic theory of quantum cluster algebras initiated by Fock and Goncharov in~\cite{FG09}. 

In~\cite{FG09}, it is explained that to each cluster mutation one can associate a unitary known as a quantum  mutation operator. It is these quantum mutation operators that we will use for the construction of the multiplicative unitary for $SL^+_\hbar(N,\mathbb{R})$. In a little more detail, the particular cluster algebras we use in this paper are those considered by Fock and Goncharov in their work~\cite{FG06} on moduli spaces of framed $SL(N,\mathbb{C})$ local systems on surfaces. In this dictionary, the cluster algebra through which we describe the quantum Borel corresponds to a `triangle', i.e.\ a disk with three marked points on its boundary.  Fock and Goncharov construct an explicit composite of cluster mutations realizing the `flip' intertwining the cluster charts corresponding to the two ways of gluing a disk with four marked points from a pair of triangles along a single edge, and it is the corresponding quantum mutation operator that will serve as our multiplicative unitary. 

Using the form sum construction, we are able to show that the same operator can be obtained via functional calculus from the general factorized form of the $R$-matrix for the quantum group, interpreted as an element of the tensor square of the Heisenberg double of the quantum Borel. On the other hand, the cluster-theoretic description of our multiplicative unitary is highly amenable to explicit computation, a feature which we exploit heavily in verifying its compatibility with the locally compact quantum group axioms. 

We would like to point out that Ip pioneered the study of positive quantum groups, also in higher rank. In \cite{Ip12,Ip13,Ip17}, he proposed the theory of \emph{multiplier Hopf algebras} due to Van Daele \cite{VD98} to approach this problem. Multiplier Hopf algebras with integrals, or more precisely algebraic quantum groups in the sense of \cite{KVD97}, give rise to a well-behaved class of locally compact quantum groups. These exhibit many features of the general theory, but are still close to the purely algebraic theory of Hopf algebras. However, the analytic side of discussion in \cite{Ip12,Ip13,Ip17} lacks mathematical rigor, and some of the claims in these papers are only valid on a heuristic level. As we will show, the scaling constant of the Borel part of $ SL^+_\hbar(N, \mathbb{R}) $ is different from $ 1 $, which means that it cannot be obtained as an algebraic quantum group \cite{DCVD10}. 

Our construction of $SL^+_\hbar(N,\mathbb{R})$ as a locally compact quantum group  opens several avenues for further exploration, some of which we now outline. 
Since the inception of the theory, it has been very desirable to define a meaningful `category of positive representations' arising from cluster-type realizations of quantum groups.\footnote{An attempt at describing algebraic structures that may underlie such a definition has been made in~\cite{SS17}, a preprint with an overly optimistic title.}
We regard the construction of the locally compact quantum group $SL^+_\hbar(N,\mathbb{R})$ presented here as the first step towards addressing this issue. In future work, we hope to make aspects of the representation categories of $SL^+_\hbar(N,\mathbb{R})$ more concrete, including the Plancherel formula and an explicit description of the reduced unitary dual.  

Another direction for further exploration consists of determining to which extent the representation theory of $SL^+_\hbar(N,\mathbb{R})$ can be `integrated' over more general surfaces \cite{BZBJ18} to define an analog of a topological quantum field theory. This question seems closely related to the infinite dimensional analog of a modular functor arising in the work~\cite{FG09} of Fock and Goncharov on quantized higher Teichm\"uller theory.  

Finally, we would like to draw the reader's attention to the fact that all our constructions work for all $\hbar \in \R^\times$, in particular they remain valid when $q = e^{\pi i\hbar}$ is a root of unity. We note that the constructions are sensitive to the value of $\hbar$ rather than that of $q$, in particular we obtain a locally compact quantum group at $\hbar=2$ which corresponds to $q=1$. On a related note, in the algebraic study of positive representations, the notion of a \emph{modular double} plays a prominent role, going back to work of Faddeev \cite{Fad95}. This encompasses two formally commuting copies of the quantized universal enveloping algebra, at dual deformation parameters $ q = e^{i \pi \hbar} $ and $ \check{q} = e^{\pi i \hbar^{-1}} $. Modular duality is a mathematical reflection of the duality properties of Liouville theory and its generalizations. In the analytic setting, this is already built into the construction of the locally compact quantum group, and presents itself in our context as an isomorphism between $ SL^+_\hbar(n, \mathbb{R}) $ and $ SL^+_{\hbar^{-1}}(n, \mathbb{R})$, identifying the corresponding multiplicative unitaries. 

Let us now come to the contents of this paper. 

In the \emph{first section}, we present some basic facts on vector spaces with a skew-symmetric bilinear form, and study their associated representation theory. 

In the \emph{second section}, we construct the multiplicative unitary that will turn out to be associated to a locally compact quantum group, namely the quantization $\Bbbb_N^+$ of the Borel part of $SL^+(N,\R)$. We present two formulas for this multiplicative unitary: one based directly on the formula for the universal $R$-matrix in quantum group theory, and one constructed from a particular Laurent polynomial chart for (a reduction of) the Heisenberg algebra associated to the quantum group. 

In the \emph{third section}, we prove the fundamental fact that our multiplicative unitary is \emph{modular}, which already unlocks many of the desired features for the associated representation theory. Our strategy is to decompose the multiplicative unitary, viewed as a representation of the quantum group, as a tensor product of `rank one'-representations, for which the associated modularity is already known. 

In the \emph{fourth section}, we determine the von Neumann algebras underlying the above modular multiplicative unitary, which are to be seen as the quantized measureable spaces underlying our quantum group. These von Neumann algebras are easily described as twisted group von Neumann algebras, but to get at this result we need access to the more refined C$^*$-algebraic theory, using results of Rieffel's deformation theory \cite{Rie93}. At the end of this section, we also briefly comment on the phenomenon of modular duality. 

In the \emph{fifth section}, we obtain our main result, showing that our modular multiplicative unitary leads to quantum groups admitting Haar measures. The crucial idea here is based on fundamental insights of Woronowicz and Van Daele \cite{VD01,Wor03}, together with a flexibility to use irreducible representations of the (degenerate) Heisenberg algebra when studying our multiplicative unitary. 

In the \emph{sixth section}, we give a complete description of the modular data associated to our quantum group, such as modular elements, modular operators and modular conjugations. 

In the \emph{seventh section}, we consider the duality theory for our quantum group, and determine a formula for the associated left regular multiplicative unitary, or equivalently, the Kac--Takesaki operator, in terms of the multiplicative unitary that we started with. 

Finally, in the \emph{eighth section} we show how one can apply the Drinfeld double construction for the quantized Borel group, to arrive (via the results of \cite{BV05}) at the desired quantization $ SL^+_
\hbar(N, \R) $ as a locally compact quantum group, of the semi-group $SL^+(N,\R)$ of totally positive special linear matrices. 

The paper is concluded with a long list of appendices, gathering some known results that might be less accessible to readers from one of the different communities addressed by this paper (locally compact quantum groups and quantum cluster algebras). 

In \emph{Appendix A}, we gather some analytic information on the quantum dilogarithm function that is crucial for the development of the whole theory.

In \emph{Appendix B}, we shortly explain the proof of the pentagon equation associated to the quantum dilogarithm. 

In \emph{Appendix C}, we comment on the proof of modularity for the rank $1$ case (`$ax+b$'-group), following \cite{WZ02}. 

In \emph{Appendix D}, we explain some technical concepts relating to C$^*$-algebras (multipliers and affiliated elements).

In \emph{Appendix E} we explain how C$^*$-algebras can be deformed when acted on by a vector space with a skew-bilinear form, following \cite{Rie93}. 

In \emph{Appendix F}, we give a brief overview of the theory of von Neumann algebras and their associated weight theory. 

Finally, in \emph{Appendix G} we present a direct algebraic proof of the pentagon equation for the multiplicative unitary constructed in Section \ref{SecConstrMultUni}.

\textbf{Conventions and notations}: 
For a complex-valued function $f$ with domain $X$, we write
\[
\overline{f}(x) = \overline{f(x)},\qquad x\in X. 
\]
For $X$ a locally compact Hausdorff space, we write $C_b(X)$ for the C$^*$-algebra of bounded continuous complex-valued functions on $X$ with pointwise algebra structure and $*$-operation $f \mapsto \overline{f}$. We write $C_0(X) \subseteq C_b(X)$ for its C$^*$-subalgebra of functions vanishing at infinity.

For $A$ a complex algebra, $M$ a left $A$-module and $V \subseteq M$ a linear subspace, we write 
\[
AV = \left\{\sum_{i=1}^n a_iv_i\mid n \geq 1,a_i \in A,v_i\in V\right\} \subseteq M.
\]

We denote by $[-]$ the norm-closure of the linear span of a set inside a normed space, or $[-]^{\tau}$ if we take the closure with respect to another linear topology $\tau$. 

We take inner products for Hilbert spaces anti-linear in the first factor. We write $\Bc(\Hc)$ for the von Neumann algebra of bounded operators on a Hilbert space $\Hc$, and $\Kc(\Hc)$ for the C$^*$-algebra of compact operators. Then for $\Hc$ a Hilbert space and $\xi,\eta \in \Hc$, we write $\omega_{\xi,\eta}$ for the vector functional 
\begin{equation}\label{EqGenNormFunctVect}
\omega_{\xi,\eta}\colon \Bc(\Hc) \rightarrow \C, \qquad \omega_{\xi,\eta}(x) = \langle \xi,x\eta\rangle,\qquad x\in \Bc(\Hc),
\end{equation} 
and we denote by $\theta_{\xi,\eta}$ the rank $1$-operator
\begin{equation}\label{EqRank1Op}
\theta_{\xi,\eta}\in \Kc(\Hc),\qquad \theta_{\xi,\eta}(\zeta) = \langle \eta,\zeta\rangle \xi,\qquad \zeta\in \Hc.
\end{equation}
We denote by
\[
\Bc(\Hc)_* = \Kc(\Hc)^* = [\omega_{\xi,\eta}\mid \xi,\eta \in \Hc] \subseteq \Bc(\Hsp)^*
\]
the space of normal functionals on $\Bc(\Hc)$. We also adapt these notations to spaces of operators between different Hilbert spaces.

We generically use boldface symbols to indicate unbounded operators on Hilbert spaces. For $\mathbf{T}$ an unbounded operator on a Hilbert space $\Hsp$, we denote by $\msD(\mathbf{T})\subseteq \Hsp$ its domain, which by default we assume dense in $\Hsp$. When $\Sb,\Tb$ are two unbounded operators with $\msD(\Sb)\cap \msD(\Tb)$ dense in $\Hsp$, we denote by $\Sb+\Tb$ the sum operator on $\msD(\Sb)\cap \msD(\Tb)$. When $\Sb,\Tb$ are closed and $\Sb+\Tb$ is closeable, we write $\Sb \dotplus \Tb$ for the closure of $\Sb+ \Tb$. 

A self-adjoint operator on a Hilbert space is by default allowed to be unbounded (with dense domain). By a \emph{positive operator} we will always mean a self-adjoint operator with spectrum contained in $\R_{\geq0}$, and by a \emph{strictly positive operator} an invertible positive operator (i.e.\ with $0$ not in the point spectrum). We frequently use functional calculus applied to self-adjoint operators, see e.g.\  \cite[Chapter 5]{Sch12} for an exposition. 

We denote by $\otimes$ the tensor product of Hilbert spaces, or the minimal tensor product of C$^*$-algebras. We denote by $\vNTens$ the spatial tensor product of von Neumann algebras. We  write $\odot$ for algebraic tensor products of vector spaces when disambiguation is needed. For example, we denote by $\Hb_1\odot \Hb_2$ the algebraic tensor product of self-adjoint operators $\Hb_i$, whose domain is the algebraic tensor product $\msD(\Hb_1)\odot \msD(\Hb_2)$, while we denote by $\Hb_1\otimes \Hb_2$ the closure of $\Hb_1\odot \Hb_2$, which is then again self-adjoint, resp.\ (strictly) positive if both $\Hb_i$ are (\cite[Section VIII.10]{RS81}).

Between tensor products of vector spaces or Hilbert spaces, we use $\Sigma$ to denote the flip map:
\[
\Sigma\colon v\otimes w \mapsto w\otimes v.
\]
We indicate standard leg numbering notation for operators $X$ between Hilbert spaces by $X_{ij...}$, so e.g.\ if $\Hc_1,\Hc_2,\Hc_3$ are Hilbert spaces and $X\in \Bc(\Hc_2\otimes\Hc_3)$, then 
\[
X_{23} = 1\otimes X \in \Bc(\Hc_1\otimes \Hc_2\otimes \Hc_3),\qquad X_{32} = 1\otimes \Sigma X \Sigma \in \Bc(\Hc_1\otimes \Hc_3\otimes \Hc_2),\qquad \textrm{etc.} 
\]

\emph{Acknowledgments}: KDC thanks J.\,Krajczok and W.\,Malfait for discussions and comments on preliminary versions of this paper. He also thanks M. Yakimov for discussions on cluster algebra theory at a (very) early stage of this project. GS has been supported by the NSF Standard Grant DMS-2302624. AS\ has been supported by the European Research Council under the European Union’s Horizon 2020 research and innovation programme under grant agreement No 948885 and by the Royal Society University Research
Fellowship.

\section{Skew-symmetric spaces and their representations}

\subsection{Skew-symmetric spaces}

\begin{defn}
A \emph{skew-symmetric space} $(V,\epsilon)$ is a finite-dimensional real vector space $V$, equipped with a skew-symmetric bilinear form 
\[
\epsilon\colon V\times V \rightarrow \R,\qquad \epsilon(v,w) = -\epsilon(w,v),\qquad v,w\in V.
\] 
We call $V$ \emph{symplectic} if $\epsilon$ is non-degenerate, and \emph{trivial} if $\epsilon=0$.
\end{defn}

The collection of skew-symmetric spaces is closed under taking subspaces, finite direct sums and conjugates 
\[
V = (V,\epsilon) \mapsto \overline{V} = (V,-\epsilon).
\]
Formally, it is convenient to change the notation for the underlying vector space of $\overline{V}$ by writing its elements as $\overline{v}$, with $\overline{v}$ corresponding to $v\in V$.

We drop $\epsilon$ from the notation if it is clear from the context, i.e.\ we write 
\[
(v,w) = \epsilon(v,w),\qquad v,w\in V.
\]

To any skew-symmetric space $V$ we can associate its \emph{radical} 
\[
\Rad(V) = \{v\in V\mid (v,-) = 0\} \subseteq V. 
\]
Then $\Rad(V)$ is a trivial skew-symmetric subspace of $V$, and $\epsilon$ passes to the quotient
\begin{equation}\label{EqReduction}
V_{\redu} = V/\Rad(V), 
\end{equation}
with $V_{\redu}$ symplectic. We call $V_{\redu}$ the \emph{reduction} of $V$.

If $V'\subseteq V$ is any subspace complementary to $\Rad(V)$, then $V'$ is a symplectic subspace of $V$, isomorphic to $V_{\redu}$ through the quotient map. We can in particular decompose $V$ as an internal direct sum of skew-symmetric subspaces
\begin{equation}\label{EqDecompRad}
V = \Rad(V) \oplus V'. 
\end{equation}
A \emph{Lagrangian subspace} of the symplectic space $V'$ is any trivially skew-symmetric subspace $W \subseteq V'$ of maximal dimension. For such a subspace one can select a Lagrangian complementary subspace $\hat{W}$ in $V'$, called a \emph{dual} for $W$, so that in particular the pairing between $W$ and $\hat{W}$ is non-degenerate. Then associated to  a basis $\{e_1,\ldots,e_m\}$ of $W$ is a unique \emph{dual} basis $\{\hat{e}_1,\ldots,\hat{e}_m\}$ of $\hat{W}$ such that 
\[
(\hat{e}_i,e_j) = \delta_{ij},\qquad 1\leq i,j\leq m.  
\]
Writing $Z_i = \mathrm{span}_{\R}\{e_i,\hat{e}_i\}$, the $Z_i \subseteq V'$ are two-dimensional symplectic vector subspaces, and \eqref{EqDecompRad} can be further refined into  an internal direct sum of skew-symmetric subspaces
\begin{equation}\label{EqDecompRadRefined}
V = \Rad(V) \oplus V' = \Rad(V) \oplus \bigoplus_{i=1}^m Z_i. 
\end{equation}

\subsection{Representations of skew-symmetric spaces}

\begin{defn}\label{DefUniProjRep}
Let $V$ a skew-symmetric space, and view $(V,+)$ as a locally compact group.

For $\hbar\in \R^{\times} = \R \setminus \{0\} $, a \emph{unitary $\hbar$-representation of $V$} is a strongly continuous unitary $\Omega$-projective representation $(\Hsp,\pi)$ of $(V,+)$ on a Hilbert space $\Hsp$, where $\Omega$ is the unitary $2$-cocycle
\[
\Omega(v,w) = \Omega_{\hbar}(v,w) =  e^{-\pi i \hbar (v,w)},\qquad v,w\in V.
\] 
\end{defn}
Spelled out, this means that $\pi(v)$ is a unitary operator on $\Hsp$ for each $v\in V$, that for each $\xi\in \Hsp$ the map $v\mapsto \pi(v)\xi$ is norm-continuous, and that
\begin{equation}\label{EqProdProj}
\pi(v+w) = e^{-\pi i \hbar (v,w)}\pi(v)\pi(w),\qquad v,w\in V,
\end{equation}
from which we get in particular that 
\begin{equation}
\pi(v)\pi(w) = e^{2\pi i \hbar (v,w)}\pi(w)\pi(v),\qquad v,w\in V.
\end{equation}

\begin{remark}
Note that in principle one can restrict to the case $\hbar = 1$ by subsuming $\hbar$ into the skew-symmetric bilinear form. However, in certain contexts it is convenient to explicitly have $\hbar$ around as a deformation parameter.
\end{remark}
\begin{remark}
When $V$ is trivial as a skew-symmetric space, Definition \ref{DefUniProjRep} becomes independent of $\hbar$, and just gives back the notion of (strongly continuous) unitary $V$-representation.
\end{remark}

Subrepresentations, direct sums and irreducibility of unitary $\hbar$-representations are defined in the obvious way.

\begin{example}\label{ExaStandardRep}
Endowing $V$ with Lebesgue measure, 
$L^2(V)$ carries the \emph{standard} unitary $\hbar$-representation $\pi = \pi_{\standard}$ of $V$, determined by 
\begin{equation}\label{EqStandardDefDualRep}
(\pi(v)g)(w) = e^{\pi i\hbar (v,w)}g(w-v),\qquad v,w\in V,g\in L^2(V).
\end{equation}
\end{example}

\begin{example}\label{ExaHeisenberg}
Assume $V$ is symplectic, and choose complementary Lagrangian subspaces $W,\hat{W} \subseteq V$. Then $L^2(W)$ carries a unique unitary $\hbar$-representation $\pi= \pi_{\Heis}$ of $V$, the \emph{Heisenberg $\hbar$-representation} (with respect to the polarisation $(W,\hat{W})$), such that  \begin{multline*}
(\pi(w)f)(w') =  f(w'-w),\quad (\pi(\hat{w})f)(w') =  e^{2\pi i \hbar (\hat{w},w')}f(w'),\\ f\in L^2(W), w,w'\in W,\hat{w}\in \hat{W}.
\end{multline*}
\end{example}

\begin{theorem}[Stone-von Neumann theorem] If $V$ is symplectic, all its Heisenberg $\hbar$-representations are irreducible and unitarily equivalent. Any other unitary $\hbar$-representation is a (possibly infinite) direct sum of Heisenberg $\hbar$-representations.
\end{theorem}

If $V$ is a skew-symmetric space and $(\Hsp,\pi)$ a unitary $\hbar$-representation,  the one-parameter group of unitaries 
\[
\pi_v \colon \R\rightarrow\Bc(\Hsp),\qquad t \mapsto \pi(tv),\qquad v\in V
\] 
has an infinitesimal generator 
\[
\vb = \vb_{\hbar} = \vb_{\pi},\qquad v\in V,
\]
so $\msD(\vb)$ consists of those vectors $\xi\in \Hsp$ for which $t \mapsto \pi(tv)\xi$ is norm-differentiable, with
\[
\vb \xi = -i \frac{d}{dt}(\pi(tv)\xi)_{\mid t=0},\qquad \xi\in \msD(\vb).
\]
The operator $\vb$ is self-adjoint, and we write
\begin{equation}
\eb(v) = \eb_{\hbar}(v) = \eb_{\pi}(v)
\end{equation}
for the strictly positive operator which is its exponential (using functional calculus):
\begin{equation}\label{EqExponentialpi}
\eb(v) = e^\vb,\qquad \pi(tv) = \eb(v)^{it} =  e^{it\vb},\qquad v\in V,t\in \R.
\end{equation}
Then we can rewrite \eqref{EqProdProj} as
\begin{equation}\label{EqProdY}
\eb(v+w)^{it} = e^{-\pi i \hbar t^2 (v,w)}\eb(v)^{it}\eb(w)^{it},\qquad t\in\R,v,w\in V.
\end{equation}

In the following, we use the shorthand
\begin{equation}\label{EqProdRule}
\eb(v)\star\eb(w) := \eb(v+w),\qquad v,w\in V.
\end{equation}

We conclude this section with the following two operations for unitary $\hbar$-representations:
\begin{defn}
If $V_1,V_2$ are skew-symmetric spaces with unitary $\hbar$-representations $\pi_1,\pi_2$ on respective Hilbert spaces $\Hsp_1,\Hsp_2$, their \emph{product} is the unitary $\hbar$-representation 
\begin{equation}\label{EqDirSumRepUnit}
\pi_{\otimes}\colon V_1\oplus V_2 \rightarrow \Bc(\Hsp_1 \otimes \Hsp_2) ,\qquad \pi_{\otimes}(v_1\oplus v_2) = \pi_1(v_1) \otimes \pi_2(v_2).
\end{equation}
\end{defn}
In terms of exponentials, this means
\begin{equation}\label{EqDirSumRep}
\eb(v_1\oplus v_2) = \eb(v_1)\otimes \eb(v_2),\qquad v_1\in V_1,v_2\in V_2.
\end{equation}

\begin{defn}
If $V$ is a skew-symmetric space with unitary $\hbar$-representation $\pi$ on a Hilbert space $\Hsp$, its \emph{conjugate} is the unitary $\hbar$-representation
\begin{equation}\label{EqDirSumRepUnit2}
\overline{\pi}\colon \overline{V}\rightarrow \Bc(\overline{\Hsp}),\qquad \overline{\pi}(\overline{v}) = \overline{\pi(-v)},
\end{equation}
where $\overline{\Hsp}$ is the conjugate-linear Hilbert space of $\Hsp$.
\end{defn}
With this convention, we can write 
\begin{equation}\label{EqComplexConjj}
\eb(\overline{v}) = \overline{\eb(v)},\qquad v\in V.
\end{equation}

\subsection{Gaussians, quantum dilogarithms and quantum exponential maps}

Fix $\hbar\in \R^{\times}$. Let $V_1,V_2$ be two skew-symmetric spaces. With respect to a product $\hbar$-representation, we define unitaries 
\begin{equation}\label{EqGausDef}
\Gauss(v_1,v_2) = \Gauss_{\hbar}(v_1,v_2) :=e^{\frac{i \vb_1\otimes \vb_2}{2\pi \hbar}} \in \Bc(\Hsp_1 \otimes \Hsp_2),\qquad v_i \in V_i,
\end{equation}
referred to as \emph{Gaussians}. Then for all $v_i,w_i\in V_i$ one computes by functional calculus that
\begin{equation}\label{EqGausEq}
\Gauss(v_1,v_2) \eb(w_1\oplus w_2) \Gauss(v_1,v_2)^* = \eb((w_1+(v_2,w_2)v_1) \oplus (w_2+(v_1,w_1)v_2)),
\end{equation}
\begin{equation}\label{EqGausEqInv}
\Gauss(v_1,v_2)^* \eb(w_1\oplus w_2) \Gauss(v_1,v_2) = \eb((w_1+(w_2,v_2)v_1) \oplus (w_2+(w_1,v_1)v_2)).
\end{equation}

More generally, if we have $z = \sum_{i=1}^n v_i \otimes w_i \in V_1\otimes V_2$ such that the $\{v_i\}_{i=1}^n$ (resp.\ $\{w_i\}_{i=1}^n$) are mutually $\epsilon$-orthogonal, we put
\begin{equation}\label{EqGenTens}
\Gauss(z) := \prod_{i=1}^n \Gauss(v_i,w_i) \in \Bc(\Hsp_1 \otimes \Hsp_2)
\end{equation}
This is indeed independent of the ordering of the factors, and of the presentation of $z$. 

\begin{defn}
For $\hbar\in \R^{\times}$, consider 
\begin{equation}\label{EqFunctWTheta}
W = W_{1/\hbar}\colon \R\rightarrow \C,\qquad W(z) := -\frac{\pi i}{2} \int_{\Omega} \frac{dy}{y}\, \frac{e^{-2 iyz}}{\sinh(2\pi y)\sinh(2\pi \hbar y)},\qquad z \in \R,
\end{equation}
with $\Omega$ any curve in $\C$ going from $-\infty$ to $+\infty$ on the real line, with a small bump in the upper half plane around the origin, staying below the poles of the denominator with strictly positive imaginary part.

We define the \emph{(Faddeev) quantum dilogarithm function} $\varphi = \varphi_{\hbar}$ and the closely associated \emph{(Woronowicz) quantum exponential function}\footnote{In this terminology, one should think of $F$ as an analogue of the exponential map in the form $t \mapsto e^{it}$, but restricted to the semigroup $(\R_{>0},+)$. See Theorem \ref{TheoExponGen} below for a motivation of this interpretation.} $F = F_{\hbar}$  by
\begin{equation}\label{EqQuantExp}
\varphi\colon \R\rightarrow \C,\qquad \varphi(t) := \exp\left(\frac{i}{2\pi}W(t)\right),\qquad t\in \R,
\end{equation}
\begin{equation}\label{EqQuantExp2}
F\colon \R_{>0} \rightarrow \C,\qquad F(r) = \overline{\varphi}(\ln(r)),\qquad r>0.
\end{equation}
\end{defn}

We provide some more information on the quantum dilogarithm and its properties in Appendix \ref{Apqdilog}. For example, we gather directly from Lemma \ref{LemOtherDescW} that
\begin{itemize}
\item  $W$ is real-valued, hence $\varphi$ and $F$ take values on the unit circle, 
\item $F$ has a continuous extension to $0$ with 
\begin{equation}\label{EqValueF0}
F(0)=1,
\end{equation}
or equivalently,
\item $\varphi$ has a continuous extension to $-\infty$ with 
\begin{equation}
\varphi(-\infty)= 1.
\end{equation}
\end{itemize}

Given a unitary $\hbar$-representation of a skew-symmetric space $V$ on a Hilbert space $\Hc$, we use the following shorthand for the associated unitaries obtained through functional calculus: 
\begin{equation}\label{EqAbbrFuncCalc}
\varphi(v) = \varphi(\vb)\in \Bc(\Hsp),\qquad v\in V.
\end{equation}
\begin{theorem}\label{TheoPentagon}
For $\hbar\in \R^{\times}$, the following identities hold: 
\begin{equation}\label{EqCommComm}
\varphi(v)\varphi(w) = \varphi(w)\varphi(v),\qquad \forall v,w\in V \textrm{ with }(v,w) = 0,
\end{equation}
\begin{equation}\label{EqCommKashaev}
\varphi(v)\varphi(w) = \varphi(w)\varphi(v+w)\varphi(v),\qquad \forall v,w\in V \textrm{ with }(v,w) = 1.
\end{equation}
\end{theorem}
Here the first identity is basic functional calculus, while the second identity is the celebrated \emph{Kashaev pentagon identity} \cite{FK94,Wor00,FKV01}. We refer to Appendix \ref{SecPentQE} for more information. 

\begin{defn}
Two strictly positive operators $\Ab,\Bb$ on a Hilbert space $\Hc$ \emph{skew-commute} if there exists $\hbar \in \R^{\times}$ such that
\begin{equation}\label{EqSkewCommDef}
\Ab^{it}\Bb^{is} = e^{2\pi i \hbar st}\Bb^{is}\Ab^{it},\qquad \forall s,t\in \R.
\end{equation}
We then say more specifically that $\Ab$ and $\Bb$ \emph{$\hbar$-commute}.
\end{defn}
In other words, $\Ab,\Bb$ arise from a unitary $\hbar$-representation on $\Hc$ of 
\[
V = \R a+ \R b,\qquad (a,b)=1,
\]
with $\Ab = e^{\mathbf{a}}$ and $\Bb = e^{\mathbf{b}}$. Note that $\hbar$ is uniquely determined by the pair $(\Ab,\Bb)$. We also recall the notation \eqref{EqProdRule}, by which $\Ab^{\alpha}\star\Bb^{\beta}$ is the unique strictly positive operator such that 
\begin{equation}\label{EqStarProd}
(\Ab^{\alpha}\star\Bb^{\beta})^{it} = e^{-\pi i\hbar \alpha \beta t^2}\Ab^{i\alpha t}\Bb^{i\beta t},\qquad \alpha,\beta,t\in\R. 
\end{equation}

Recall now the notion of \emph{form sum} $\Ab \boxplus \Bb$ of two positive operators $\Ab,\Bb$ with dense intersection of the domains of their square roots: it is the unique positive operator such that 
\[
\msD((\Ab\boxplus \Bb)^{1/2}) = \msD(\Ab^{1/2})\cap \msD(\Bb^{1/2}),
\]
\[
\|(\Ab\boxplus \Bb)^{1/2}\xi\|^2 = \|\Ab^{1/2}\xi\|^2+ \|\Bb^{1/2}\xi\|^2,\qquad \forall \xi\in \msD(\Ab^{1/2})\cap \msD(\Bb^{1/2}).
\]
See Section \ref{SecPropSkewComm} for more information. We note that $\Ab\boxplus \Bb$ is strictly positive if $\Ab,\Bb$ are strictly positive.

\begin{remark}\label{RemFormSumFinite}
Let $V$ be a  skew-symmetric space, and choose on $V$ a Euclidian structure. Let $V_{\C}$ be its complexification, and consider for each $t>0,w\in V_{\C}$ and $\kappa \in \C$ the function 
\[
g_{t,w,\kappa}(v) = e^{-t\|v\|^2 +\langle w,v\rangle + \kappa},\qquad v\in V. 
\]
If $\hbar\in\R^{\times}$ and $(\Hsp,\pi)$ is a unitary $\hbar$-representation of a skew-symmetric space $V$, denote 
\[
 \pi(f)\xi = \int_{V} f(v) \pi(v) \xi \rd v,\qquad \xi \in \Hsp,f\in L^1(V),
\]
and put
\[
\Hsp_0 = \textrm{linear span}\{\pi(g_{t,w,\kappa})\xi\mid \xi \in \Hsp,t>0,w\in V_{\C},\kappa \in \C\}.
\]
Then $\Hsp_0$ is dense in $\Hsp$, and for any $v\in V$ we have 
\[
\Hsp_0  \subseteq \msD(\eb(v)),\qquad \eb(v)\pi(g_{t,w,\kappa})\xi = \pi(g_{t,w+2tiv,\kappa+i\langle w,v\rangle})\xi,\qquad \xi \in \Hsp,t>0,w\in V_{\C},\kappa \in \C. 
\]
In particular, we can take the form sum of any finite collection of operators $\eb(v_1),\ldots,\eb(v_n)$. 
\end{remark} 

\begin{prop}\label{PropSkewComm}
Let $\hbar\in\R^{\times}$. For $\Ab,\Bb$ $\hbar$-commuting strictly positive operators, we have
\begin{equation}\label{EqBoxPlusDef}
\Ab \boxplus \Bb = F(\Ab\star\Bb^{-1})^*\Bb F(\Ab\star\Bb^{-1}).
\end{equation}
\end{prop}

We present the proof of Proposition \ref{PropSkewComm} in Section \ref{SecPropSkewComm}. As far as we know, this result is new in this generality.

\begin{remark}
Following \cite[Appendix B]{Rui05}, one can show that when $|\hbar|<1$, the operator $\Ab \boxplus \Bb$ has domain $\msD(\Ab)\cap \msD(\Bb)$, and coincides there with the algebraic sum of operators $\Ab+\Bb$ (which has $\msD(\Ab)\cap \msD(\Bb)$ as domain by definition). When $|\hbar|>1$ on the other hand, one has that the symmetric operator $\Ab+\Bb$ is \emph{not essentially self-adjoint}, but has $\Ab\boxplus \Bb$ as one of its self-adjoint extensions.
\end{remark}

The following is referred to as the \emph{quantum exponential property} \cite{Wor00}. Using Proposition \ref{PropSkewComm}, it is a direct consequence of the Kashaev pentagon relation through functional calculus. 
\begin{theorem}\label{TheoExponGen}
Let $\hbar\in \R^{\times}$. If $\Ab$ and $\Bb$ are $\hbar$-commuting strictly positive operators, then 
\begin{equation}\label{EqExpoGenProp}
F(\Ab \boxplus \Bb) = F(\Ab)F(\Bb). 
\end{equation}
\end{theorem} 

We also record the following property, which is immediate from \eqref{EqBoxPlusDef}:

\begin{lemma}\label{LemSumSkewComm}
Assume $\Ab,\Bb$ and $\Cb$ are strictly positive operators such that $\Ab$ and $\Bb$ skew-commute, and such that there exists $\hbar\in \R^{\times}$ such that $(\Ab,\Cb)$ and $(\Bb,\Cb)$ $\hbar$-commute. Then also $\Ab\boxplus \Bb$ and $\Cb$ $\hbar$-commute, and 
\begin{equation}
\Cb\star(\Ab\boxplus \Bb) = \Cb\star\Ab\boxplus \Cb\star\Bb.
\end{equation}
\end{lemma}

\section{Pentagonal Fock--Goncharov flip}\label{SecConstrMultUni}

In the remainder of this paper, we fix $n \geq 1$ a natural number, and put $N = n+1$. 

\subsection{Skew-symmetric \texorpdfstring{\triangspace}{of type An}}

\begin{defn}\label{DefFGDiag}
We define 
\begin{equation}\label{EqConeC}
C_N = \{(a,b,c) \in \Z_{\geq0}^3\mid a+b+c=N\}.
\end{equation}
We write $\btd{N}$ for the $\binom{N+2}{2}$-dimensional skew-symmetric space with basis vectors $e_{a,b,c}$ for $(a,b,c) \in C_N$, with the only non-trivial positive pairings given by
\begin{multline}
\label{eq:q-comm}
\hr{e_{a,b,c},e_{a+1,b-1,c}} = \hr{e_{a,b,c},e_{a,b+1,c-1}} = \hr{e_{a,b,c},e_{a-1,b,c+1}} \\ =
\begin{cases}
1/2 & \text{if the common index is }0, \\
1 &\text{otherwise.}
\end{cases}
\end{multline}
We refer to $\btd{N}$ as the \emph{skew-symmetric \triangspace} (of order $N$). 

We denote $C_N'$ for $C_N$ with its corner vertices cut off, and we put $\wbtd{N} \subseteq \btd{N}$ for the vector subspace spanned by basis vectors labelled by $C_N'$. We call $\wbtd{N}$ the `snubbed' skew-symmetric \triangspace, and $C_N'$ the associated snubbed diagram. 
\end{defn}

We picture the skew form diagrammatically by a graph with $C_N$ as its nodes, and with an oriented arrow from $(a,b,c)$ to $(a',b',c')$ if the skew product is $1$, resp.\ a dashed oriented arrow if the skew product is $1/2$.  

\begin{figure}[ht]\label{FigTriang}
\adjustbox{scale=0.7,center}{%
\begin{tikzcd}
\color{red}{400} && 310 && 220 && 130 && \color{red}{040}\\
	&301 && 211 && 121 && 031 & \\
	&& 202 && 112 && 022 &&\\
&&&103 && 013 &&&\\
&&&& \color{red}{004} &&&&
	\arrow[red,dashed, from=1-3, to=1-1]
	\arrow[dashed, from=1-5, to=1-3]
	\arrow[dashed, from=1-7, to=1-5]
	\arrow[red,dashed, from=1-9, to=1-7]
	\arrow[red,dashed, from=1-1, to=2-2]
	\arrow[dashed, from=2-2, to=3-3]
	\arrow[dashed, from=3-3, to=4-4]
	\arrow[red,dashed, from=4-4, to=5-5]
	\arrow[dashed, from=2-4, to=1-5]
	\arrow[red,dashed, from=5-5, to=4-6]
\arrow[dashed, from=4-6, to=3-7]
\arrow[dashed, from=3-7, to=2-8]
\arrow[red,dashed, from=2-8, to=1-9]
	\arrow[from=2-2, to=1-3]
\arrow[from=2-4, to=1-5]
\arrow[from=2-6, to=1-7]
\arrow[from=3-3, to=2-4]
\arrow[from=3-5, to=2-6]
\arrow[from=4-4, to=3-5]
\arrow[from=2-4, to=2-2]
\arrow[from=2-6, to=2-4]
\arrow[from=2-8, to=2-6]
\arrow[from=3-7, to=3-5]
\arrow[from=3-5, to=3-3]
\arrow[from=4-6, to=4-4]
\arrow[from=1-3, to=2-4]
\arrow[from=1-5, to=2-6]
\arrow[from=1-7, to=2-8]
\arrow[from=2-4,to=3-5]
\arrow[from=2-6,to=3-7]
\arrow[from=3-5,to=4-6]
\end{tikzcd}
}
\caption{The diagram  \textcolor{red}{$C_4$}  with the diagram $C_4'$ inside.}
\end{figure}
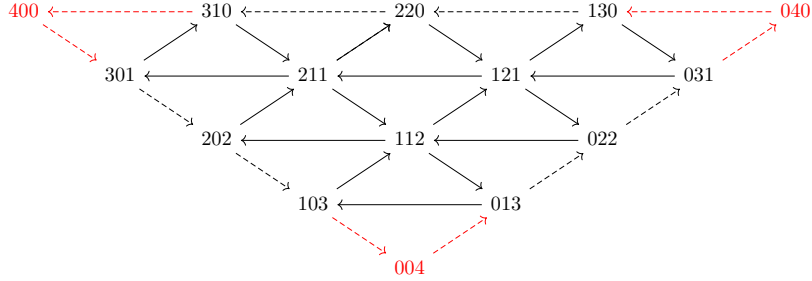

\begin{defn}
For $(a,b,c)\in C_N$ and $k \ge 0$ with $b+k \leq N$ and $0\leq c-k$, we set
$$
\gfr{a,b,c}{a,b+k,c-k} = \sum_{j=0}^k e_{a,b+j,c-j} \in \btd{N},
$$
and similarly for cyclic shifts of the columns/indices.
\end{defn}

For example, we have
$$
\gfr{3,0,2}{3,1,1} \oplus \gfr{3,2,0}{1,2,2}  = (e_{3,0,2}+e_{3,1,1}) \oplus (e_{3,2,0}+e_{2,2,1}+e_{1,2,2}) \in \btd{5}^{\,\oplus2}. 
$$
We also use the shorthand notation 
\begin{equation}\label{EqShorthand}
\nee_{s,k} = \gfr{N-s,0,s}{N-s,k,s-k},\qquad \seee_{s,k} = \gfr{s,N-s,0}{s-k,N-s,k},\qquad 0\leq k \leq s \leq N,
\end{equation}
with the further shorthand.
\begin{equation}\label{EqDiagFull}
\nee_s = \nee_{s,s},\qquad \seee_s = \seee_{s,s},\qquad 0 \leq s \leq N.
\end{equation}
Similarly, we put
\begin{equation}\label{EqShorthandRev}
\nwe_{s,k} = \gfr{0,N-s,s}{k,N-s,s-k},\qquad \swe_{s,k} = \gfr{N-s,s,0}{N-s,s-k,k},\qquad 0\leq k \leq s \leq N
\end{equation}
and $\nwe_s = \nwe_{s,s}, \swe_s = \swe_{s,s}$, so in particular
\begin{equation}\label{EqDiagFullSame}
 \swe_{s}= \nee_s,\qquad \nwe_{s}=\seee_s ,\qquad 0 \leq s \leq N
\end{equation}
and 
\begin{equation}\label{EqDiffOr}
\swe_{s,k} = \nee_s - \nee_{s,s-k-1},\quad \nwe_{s,k} = \seee_s - \seee_{s,s-k-1},\qquad 0 \leq k<s \leq N.
\end{equation}

\begin{figure}[ht]\label{FigTriangVect}
\adjustbox{scale=0.7,center}{%
\begin{tikzcd}
400 && \color{green}{310} && 220 && \color{blue}{130} && 040\\
	& 301 && 211 && \color{red}{121} && \color{blue}{031} & \\
	&& 202 && \color{red}{112} && 022 &&\\
&&& \color{red}{103} && 013 &&&\\
&&&& 004 &&&&
	\arrow[dashed, from=1-3, to=1-1]
	\arrow[dashed, from=1-5, to=1-3]
	\arrow[dashed, from=1-7, to=1-5]
	\arrow[dashed, from=1-9, to=1-7]
	\arrow[dashed, from=1-1, to=2-2]
	\arrow[dashed, from=2-2, to=3-3]
	\arrow[dashed, from=3-3, to=4-4]
	\arrow[dashed, from=4-4, to=5-5]
	\arrow[dashed, from=2-4, to=1-5]
	\arrow[dashed, from=5-5, to=4-6]
\arrow[dashed, from=4-6, to=3-7]
\arrow[dashed, from=3-7, to=2-8]
\arrow[dashed, from=2-8, to=1-9]
	\arrow[from=2-2, to=1-3]
\arrow[from=2-4, to=1-5]
\arrow[from=2-6, to=1-7]
\arrow[from=3-3, to=2-4]
\arrow[red,from=3-5, to=2-6]
\arrow[red,from=4-4, to=3-5]
\arrow[from=2-4, to=2-2]
\arrow[from=2-6, to=2-4]
\arrow[from=2-8, to=2-6]
\arrow[from=3-7, to=3-5]
\arrow[from=3-5, to=3-3]
\arrow[from=4-6, to=4-4]
\arrow[from=1-3, to=2-4]
\arrow[from=1-5, to=2-6]
\arrow[blue,from=1-7, to=2-8]
\arrow[from=2-4,to=3-5]
\arrow[from=2-6,to=3-7]
\arrow[from=3-5,to=4-6]
\end{tikzcd}
}
\caption{The vectors \textcolor{red}{$\protect\nee_{3,2}$}, \textcolor{blue}{$\protect\seee_{1,1}$} and \textcolor{green}{$\protect\swe_{1,0}$} as (undirected) paths in $C_4$.}
\end{figure}

We record here for convenience the following values for pairings between the above vectors. 

\begin{lemma}\label{LemSkewProdFormDiff}
\begin{enumerate}
\item For all $1\leq s,s'\leq n$,  $0\leq k<s$ and $0\leq k'<s'$, 
\begin{equation}\label{EqSameNil}
(\nee_{s,k},\nee_{s',k'}) = (\seee_{s,k},\seee_{s',k'}) =\begin{cases}
1&\text{if}\quad s' = s,k'>k,\\
1/2 &\text{if}\quad s' = s+1,k'\leq k,\\
1/2&\text{if}\quad s'=s-1,k'<k,\\
0& \textrm{if}\quad|s-s'|\geq 2, 
\end{cases}
\end{equation}
with the remaining cases determined by skew-symmetry. 
\item For all $0\leq s\leq N$, $1\leq s'\leq n$ and $0\leq k' < s'$, 
\begin{equation}\label{EqSameCar}
(\nee_s,\nee_{s',k'}) = (\seee_s,\seee_{s',k'}) = \begin{cases}
-1&\text{if}\quad s' = s,\\
1/2 & \text{if} \quad |s'-s|= 1,\\
0&\text{if}\quad |s'-s|\geq 2.
\end{cases}
\end{equation}
\item For all $0\leq s,t\leq N$,\begin{equation}\label{EqVanishing}
(\nee_s,\nee_t) = (\seee_s,\seee_t) = 0.
\end{equation}
\item For all  $1\leq s,s'\leq n$, $0\leq k<s$ and $0\leq k'<s'$,
\begin{equation}\label{EqDiffNil}
(\nee_{s,k},\seee_{s',k'}) = \delta_{k+k',s-1}(\delta_{k+s',n} - \delta_{k+s',N}).
\end{equation}
\item For all $0\leq s\leq N$, $1\leq s'\leq n$ and $0\leq k' < s'$,
\begin{equation}\label{EqDiffCar}
(\nee_s,\seee_{s',k'}) = \begin{cases}
1&\text{if}\quad s' = N-s,\\
-1/2 & \text{if} \quad |s'  -(N-s)|=1,\\
0& \text{if}\quad |s'-(N-s)|\geq 2.
\end{cases}
\end{equation}
\item For all $1\leq s\leq n$, $0\leq k < s$ and $0\leq s'\leq N$, 
\begin{equation}\label{EqDiffCarOpp}
(\nee_{s,k},\seee_{s'}) = 0.
\end{equation}
\item For all $1\leq s,t\leq n$,\begin{equation}\label{EqDiffCart}
(\nee_s,\seee_t) = \begin{cases}
1&\text{if}\quad t = N-s,\\
-1/2 & \text{if} \quad |t-(N-s)|=1,\\
0&  \text{if}\quad |t-(N-s)|\geq 2.
\end{cases}
\end{equation}
\end{enumerate}
\end{lemma}
\begin{proof}
Immediate upon inspection.
\end{proof}

Before we continue, we also introduce the following notations: 

\begin{defn}\label{DefCNOr}
We denote
\[
\neC_N = \{(a,b,c)\in C_N \mid 1\leq a \leq n\}\subseteq C_N', 
\]
so $\neC_N$ is $C_N'$ with its lower right line removed. We put 
\[
\wneC_N = \{(a,b,c)\in C_N \mid 1\leq a \leq n,c>0\}\subseteq C_N', 
\]
so $\wneC_N$ is $C_N'$ with its lower right and upper line removed.

Similarly, we put 
\[
\seC_N = \{(a,b,c)\in C_N \mid 1\leq b \leq n\}\subseteq C_N', 
\]
so $\seC_N$ is $C_N'$ with its lower left line removed, and 
\[
\wseC_N = \{(a,b,c)\in C_N \mid 1\leq b \leq n,a>0\}\subseteq C_N', 
\]
so $\wseC_N$ is $C_N'$ with its lower left and lower right line removed.
\end{defn}

\begin{figure}[ht]\label{FigTriangDoublePrime}
\adjustbox{scale=0.7,center}{%
\begin{tikzcd}
400 && \color{red}{310} && \color{red}{220} && \color{red}{130} && 040\\
	& \color{red}{301} && \color{red}{211} && \color{red}{121} && 031 & \\
	&& \color{red}{202} && \color{red}{112} && 022 &&\\
&&&\color{red}{103} && 013 &&&\\
&&&& 004 &&&&
	\arrow[dashed, from=1-3, to=1-1]
	\arrow[red,dashed, from=1-5, to=1-3]
	\arrow[red,dashed, from=1-7, to=1-5]
	\arrow[dashed, from=1-9, to=1-7]
	\arrow[dashed, from=1-1, to=2-2]
	\arrow[red,dashed, from=2-2, to=3-3]
	\arrow[red,dashed, from=3-3, to=4-4]
	\arrow[dashed, from=4-4, to=5-5]
	\arrow[dashed, from=2-4, to=1-5]
	\arrow[dashed, from=5-5, to=4-6]
\arrow[dashed, from=4-6, to=3-7]
\arrow[dashed, from=3-7, to=2-8]
\arrow[dashed, from=2-8, to=1-9]
	\arrow[red,from=2-2, to=1-3]
\arrow[red,from=2-4, to=1-5]
\arrow[red,from=2-6, to=1-7]
\arrow[red,from=3-3, to=2-4]
\arrow[red,from=3-5, to=2-6]
\arrow[red,from=4-4, to=3-5]
\arrow[red,from=2-4, to=2-2]
\arrow[red,from=2-6, to=2-4]
\arrow[from=2-8, to=2-6]
\arrow[from=3-7, to=3-5]
\arrow[red,from=3-5, to=3-3]
\arrow[from=4-6, to=4-4]
\arrow[red,from=1-3, to=2-4]
\arrow[red,from=1-5, to=2-6]
\arrow[from=1-7, to=2-8]
\arrow[red,from=2-4,to=3-5]
\arrow[from=2-6,to=3-7]
\arrow[from=3-5,to=4-6]
\end{tikzcd}
}
\caption{The diagram \textcolor{red}{$\protect\neC_4$} inside $C_4$.}
\end{figure}

\begin{figure}[ht]\label{FigTriangDoublePrime2}
\adjustbox{scale=0.7,center}{%
\begin{tikzcd}
400 && \color{red}{310} && \color{red}{220} && \color{red}{130} && 040\\
	& 301 && \color{red}{211} && \color{red}{121} && 031 & \\
	&& 202 && \color{red}{112} && 022 &&\\
&&& 103 && 013 &&&\\
&&&& 004 &&&&
	\arrow[dashed, from=1-3, to=1-1]
	\arrow[red,dashed, from=1-5, to=1-3]
	\arrow[red,dashed, from=1-7, to=1-5]
	\arrow[dashed, from=1-9, to=1-7]
	\arrow[dashed, from=1-1, to=2-2]
	\arrow[dashed, from=2-2, to=3-3]
	\arrow[dashed, from=3-3, to=4-4]
	\arrow[dashed, from=4-4, to=5-5]
	\arrow[dashed, from=2-4, to=1-5]
	\arrow[dashed, from=5-5, to=4-6]
\arrow[dashed, from=4-6, to=3-7]
\arrow[dashed, from=3-7, to=2-8]
\arrow[dashed, from=2-8, to=1-9]
	\arrow[from=2-2, to=1-3]
\arrow[red,from=2-4, to=1-5]
\arrow[red,from=2-6, to=1-7]
\arrow[from=3-3, to=2-4]
\arrow[red,from=3-5, to=2-6]
\arrow[from=4-4, to=3-5]
\arrow[from=2-4, to=2-2]
\arrow[red,from=2-6, to=2-4]
\arrow[from=2-8, to=2-6]
\arrow[from=3-7, to=3-5]
\arrow[from=3-5, to=3-3]
\arrow[from=4-6, to=4-4]
\arrow[red,from=1-3, to=2-4]
\arrow[red,from=1-5, to=2-6]
\arrow[from=1-7, to=2-8]
\arrow[red,from=2-4,to=3-5]
\arrow[from=2-6,to=3-7]
\arrow[from=3-5,to=4-6]
\end{tikzcd}
}
\caption{The diagram \textcolor{red}{$\protect\wseC_4$} inside $C_4$.}
\end{figure}

\begin{defn}\label{DefBorel}
We put 
\begin{equation}\label{EqSpaceBNMin}
\Nilp_N^- \subseteq \wwbtd{N} \subseteq \wbtd{N}
\end{equation}
for the vector spaces spanned by respectively the basis vectors labelled by $\wneC_N$ and $\neC_N$. 

Similarly, we denote
\begin{equation}\label{EqSpaceBNPlus}
\Nilp_N^+ \subseteq \Bor_N^+ \subseteq \wbtd{N}
\end{equation}
for the vector spaces spanned by respectively the basis vectors labelled by $\wseC_N$ and $\seC_N$. 

We further put
\begin{equation}
\Tor_N^-  = \mathrm{span}\{\nee_s\mid 1\leq s\leq n \} \subseteq \Bor_N^-,\qquad 
\Tor_N^+  = \mathrm{span}\{\seee_s\mid 1\leq s\leq n \} \subseteq \Bor_N^+,
\end{equation}
so that 
\begin{equation}
\Bor_N^- = \Tor_N^- \oplus \Nilp_N^-,\qquad \Bor_N^+ = \Tor_N^+\oplus \Nilp_N^+.
\end{equation}
\end{defn}

\begin{lemma}\label{LemNonDegPair}
The subspaces $\Bor_N^-$ and $\Bor_N^+$ of $\wbtd{N}$ are non-degenerately paired under the skew pairing.
\end{lemma}
\begin{proof}
Let
\[
\lambda = \sum_{(a,b,c)\in \neC_N} \lambda_{a,b,c}e_{a,b,c}
\]
be a vector in $\Bor_N^-$ which is orthogonal to $\Bor_N^+$. It then follows by induction on $a\colon0 \rightarrow n$ and taking pairings with vectors of the form $e_{a,b,c}$ for $(a,b,c)\in \seC_N$ (with a second induction on  $b$ descending) that $\lambda$ needs to be constant on northeast-oriented diagonals, i.e.\ 
\[
\lambda = \sum_{(a,b,c)\in \neC_N} \lambda_a e_{a,b,c}. 
\]
But then taking pairings with vectors of the form $e_{a,b,0}$ for $(a,b,0) \in \seC_N$, we find that, with $\Lambda$ the column vector with entries $\lambda_a$, we must have  
\[
B \Lambda =0,
\] 
with $B$ the Cartan matrix for type $A_n$ (see \eqref{EqBCartan} below). Hence $\Lambda=0$.
\end{proof}

\begin{cor}\label{CorHeisDeg}
Let 
\begin{equation}\label{EqDegHeisDoub}
D_N := \textrm{\raisebox{\depth}{$\bigtriangledown'$\!\!}}_{N,\redu}
\end{equation}
be the reduction of $\wbtd{N}$ as in~\eqref{EqReduction}. Then through the quotient map from $\wbtd{N}$ we obtain embeddings
\[
\Bor_N^{\pm}\hookrightarrow D_N.
\]
\end{cor}
We still use the notation $e_{a,b,c}$ etc.\ to denote the images of these elements in $D_N$. 

\begin{defn}
Let $B$ be the Cartan matrix of type $A_n$, 
\begin{equation}\label{EqBCartan}
B \in M_n(\Z),\qquad B_{rs} = 2 \delta_{r,s} - \delta_{|r-s|,1}.
\end{equation}
We put 
\begin{equation}\label{EqFundWeightUp}
\nevarpi_s := \sum_{t=1}^n (B^{-1})_{st} \nee_t \in \Tor_N^-,\qquad 1\leq s \leq n,
\end{equation}
and similarly 
\begin{equation}\label{EqFundWeightDown}
\sevarpi_s := \sum_{t=1}^n (B^{-1})_{st} \seee_t \in \Tor_N^+,\qquad 1\leq s \leq n.
\end{equation}
\end{defn}

By \eqref{EqSameCar} and \eqref{EqDiffCar} we get for all $1\leq s,s'\leq n$ and $0\leq k' < s'$ that 
\begin{equation}\label{EqPairFundWeight}
(\nevarpi_s,\nee_{s',k'}) = (\sevarpi_s,\seee_{s',k'}) = -\frac{1}{2}\delta_{s,s'},\qquad (\nevarpi_s,\seee_{s',k'}) = \frac{1}{2}\delta_{s,N-s'},
\end{equation}
while from \eqref{EqVanishing} and \eqref{EqDiffCart} it follows that 
\begin{equation}\label{EqOrthogonalWeights}
(\nevarpi_t,\nevarpi_s)=0 = (\sevarpi_t,\sevarpi_s),\qquad (\nevarpi_t,\seee_s)  = (\nee_t,\sevarpi_s) = \frac{1}{2}\delta_{s,N-t}\qquad 1\leq s,t\leq n.
\end{equation}

\subsection{Colored braid graphs}\label{SecStandGenPrel}
Let $B_m$ be the braid group on $m$ strands with standard generators $\{\sigma_i\}_{i=1}^{m-1}$.

\begin{defn}
By a \emph{positive $m$-strand braid word} we mean an ordered collection of elements of the alphabet $\{\sigma_i\}_{i=1}^{m-1}$, i.e.\ we do not allow inverses of the generators as letters. 

We call two braid words $\mathbf{i},\mathbf{i}'$ \emph{mutation-equivalent} if they can be related by some sequence of the following three types of moves, the first two of which are called  \emph{mutation moves}:
\begin{enumerate}
\item \emph{braid moves}, consisting of a substring replacement $\sigma_i\sigma_{i+1}\sigma_i\leftrightarrow\sigma_{i+1}\sigma_i\sigma_{i+1}$,  and 
\item \emph{$1$- or ${(m-1)}$-Demazure moves}, consisting of substring replacements $\sigma_1\sigma_1 \leftrightarrow \sigma_1$ and $\sigma_{m-1}\sigma_{m-1} \leftrightarrow \sigma_{m-1}$ respectively, and
\item \emph{commutation moves}, consisting of substring replacements $\sigma_i\sigma_j\leftrightarrow \sigma_j\sigma_{i}$ when $|i-j|\geq 2$. 
\end{enumerate}
\end{defn}

Each positive $m$-strand braid word $\mbi$ determines a 1-d manifold with corners $\Gamma_{\mathbf i}$ via the following inductive construction: to a word consisting of a single generator $\sigma_i$ corresponds the subset $\Gamma_{\sigma_i}\in\mathbb{R}^2$ given by the union of $m$ horizontal closed intervals $h_j=[0,1]\times \{-j+1\}$ for $1\leq j \leq m$, together with a single vertical closed interval $\{1/2\}\times[-i+1,-i]$ connecting $h_{i}$ to $h_{i+1}$. We then declare that for any concatenation $\mathbf{i}= \mathbf{i}'*\mathbf{i}''$, the associated manifold $\Gamma_{\mathbf i}$ is obtained by gluing each point $(1,-j)$ on the right boundary of $\Gamma_{\mathbf{i}'}$ to the corresponding point $(0,-j)$ on the left boundary of $\Gamma_{\mathbf {i}''}$, and rescaling the horizontal intervals back to $[0,1]$. The resulting manifold with corners is naturally embedded into the disk $\mathbb{D}^2$, and two such manifolds with corners will be considered the same if they are ambient isotopic.	Note that then $\Gamma_{\mbi}=\Gamma_{\mbi'}$  if and only if $\mbi$ and $\mbi'$ are connected by a sequence of commutation moves. 

We may enhance $\Gamma_{\bf i}$ into a planar directed graph whose vertices are its corners and boundary points, and whose edges are the intervals connecting pairs of corners or pairs of a corner and a boundary point, oriented so that horizontal edges go to the right and vertical ones go down. We further color in blue the vertices which are sources of the vertical edges, and color in red those that are sinks of vertical edges. 
\begin{figure}
\begin{tikzpicture}[>=Stealth, line cap=round, scale=0.8]
  \tikzset{
    wire/.style={line width=0.8pt},
    post/.style={line width=0.9pt,->},
    reddot/.style={circle,draw=red!70!black,fill=red,inner sep=1.2pt},
    bluedot/.style={circle,draw=blue!70!black,fill=blue,inner sep=1.2pt},
    facelabel/.style={font=\small}
  }

  \newcommand{\threewirepanel}[4]{%
    \begin{scope}[shift={(#1,0)}]
      \foreach \yy/\lab in {1/{\(i{-}1\)},0/{\(i\)},-1/{\(i{+}1\)}}{
        \draw[wire,->] (0,\yy) -- (6,\yy);
        \node[left] at (0,\yy) {\lab};
      }
      \foreach \x in {#2}{
        \draw[post] (\x,0) -- (\x,-1);
        \node[reddot] at (\x,-1) {};
        \node[bluedot] at (\x,0)  {};
      }
      \foreach \x in {#3}{
        \draw[post] (\x,1) -- (\x,0);
        \node[bluedot] at (\x,1) {};
        \node[reddot]  at (\x,0) {};
      }
      \foreach \pt/\txt in {#4}{
        \node[facelabel] at \pt {\txt};
      }
    \end{scope}
  }

  \threewirepanel{0}{2,4}{3}{
    (1.5,0.5)/$ $, (3,-0.5)/$*$, (4.5,0.5)/$ $,
    (1,-0.5)/$ $, (5,-0.5)/$ $
  }

  \threewirepanel{8.5}{3}{2,4}{
    (1,0.5)/$ $, (3,0.5)/$ $, (5,0.5)/$ $,
    (2,-0.5)/$ $, (4,-0.5)/$ $
  }
\end{tikzpicture}
\caption{Local graph modification for braid move $ \sigma_{i+1}\sigma_{i}\sigma_{i+1}\mapsto \sigma_{i}\sigma_{i+1}\sigma_{i}$.}
\label{fig:graph-braid}
\end{figure}
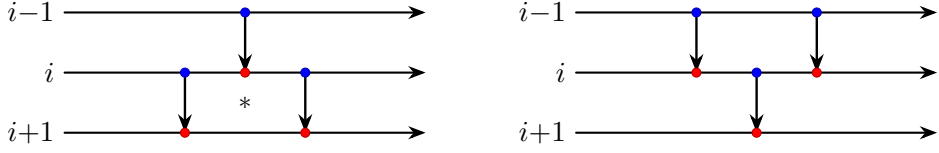

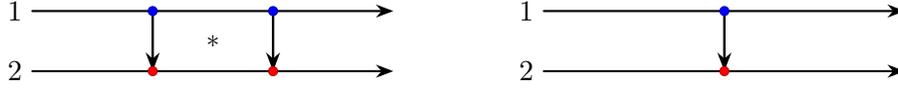
\begin{figure}
\begin{tikzpicture}[>=Stealth, line cap=round, scale=0.8]
  \tikzset{
    wire/.style={line width=0.8pt},
    post/.style={line width=0.9pt,->},
    reddot/.style={circle,draw=red!70!black,fill=red,inner sep=1.2pt},
    bluedot/.style={circle,draw=blue!70!black,fill=blue,inner sep=1.2pt},
    facelabel/.style={font=\small}
  }

  \newcommand{\twowirepanel}[3]{%
    \begin{scope}[shift={(#1,0)}]
      \draw[wire,->] (0,1) -- (6,1);   \node[left] at (0,1) {$1$};
      \draw[wire,->] (0,0) -- (6,0);   \node[left] at (0,0) {$2$};
      \foreach \x in {#2}{
        \draw[post] (\x,1) -- (\x,0);
        \node[reddot]  at (\x,0) {};
        \node[bluedot] at (\x,1) {};
      }
      \foreach \pt/\txt in {#3}{
        \node[facelabel] at \pt {\txt};
      }
    \end{scope}
  }

  \twowirepanel{0}{2,4}{
    (1,0.5)/$  $, (3,0.5)/$*$, (5,0.5)/$ $
  }

  \twowirepanel{8.5}{3}{
    (2,0.5)/$ $, (4,0.5)/$  $
  }
\end{tikzpicture}
\caption{Local graph modification for Demazure move $\sigma_{1}\sigma_{1}\mapsto \sigma_{1}$. }
\label{fig:graph-dem}
\end{figure}

\begin{figure}
\begin{tikzpicture}[>=Stealth, line cap=round, scale=0.8]
  \tikzset{
    wire/.style={line width=0.8pt},
    post/.style={line width=0.9pt,->},
    reddot/.style={circle,draw=red!70!black,fill=red,inner sep=1.2pt},
    bluedot/.style={circle,draw=blue!70!black,fill=blue,inner sep=1.2pt},
    facelabel/.style={font=\small}
  }

  \newcommand{\twowirepanel}[3]{%
    \begin{scope}[shift={(#1,0)}]
      \draw[wire,->] (0,1) -- (6,1);   \node[left] at (0,1) {$m-1$};
      \draw[wire,->] (0,0) -- (6,0);   \node[left] at (0,0) {$m$};
      \foreach \x in {#2}{
        \draw[post] (\x,1) -- (\x,0);
        \node[reddot]  at (\x,0) {};
        \node[bluedot] at (\x,1) {};
      }
      \foreach \pt/\txt in {#3}{
        \node[facelabel] at \pt {\txt};
      }
    \end{scope}
  }

  \twowirepanel{0}{2,4}{
    (1,0.5)/$ $, (3,0.5)/$*$, (5,0.5)/$ $
  }

  \twowirepanel{8.5}{3}{
    (2,0.5)/$ $, (4,0.5)/$ $
  }
\end{tikzpicture}
\caption{Local graph modification for Demazure move $\sigma_{m-1}\sigma_{m-1}\mapsto \sigma_{m-1}$. }
\label{fig:graph-dem2}
\end{figure}
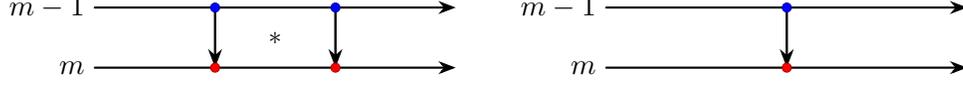

We call the planar graphs $\Gamma =\Gamma_{\mbi}$ that we obtain in this way \emph{braid graphs}. We will simply refer to them as graphs in what follows. Then saying that two graphs $\Gamma,\Gamma'$ are related by a mutation move means that we can find representatives $\Gamma = \Gamma_{\mbi}$ and $\Gamma'= \Gamma_{\mbi'}$ such that $\mbi$ and $\mbi'$ are related by that mutation move. On the level of the graphs such moves can be represented locally: 
\begin{enumerate}
\item If $\mathbf{i},\mathbf{i}'$ are related by a braid move, the corresponding directed graphs $\Gamma_{\bf i},\Gamma_{\bf i'}$ are identical outside the neighborhood illustrated in Figure~\ref{fig:graph-braid}.
\item The case of a Demazure move is similar, and shown in Figures~\ref{fig:graph-dem} and \ref{fig:graph-dem2}.
\end{enumerate}
Clearly, mutation can be defined directly on the level of graphs, and we then say that $\Gamma$ is mutable at a cell if $\Gamma$ can be presented locally in the form \ref{fig:graph-braid},~\ref{fig:graph-dem} or ~\ref{fig:graph-dem2}, with the mutable cell indicated by $*$. 

Now let $V$ be a skew-symmetric space. 
\begin{defn}
If $\Gamma$ is a braid graph, then a \emph{$V$-coloring} $\{v_f\}$ of $\Gamma$ is a labelling of the faces $f$ of $\Gamma$ by linearly independent elements $v_f\in V$ via a map 
\begin{equation}
\ell \colon \mathrm{Faces}(\Gamma)\rightarrow V,
\end{equation} 
such that the following conditions are satisfied:
\begin{enumerate}
\item if two faces $f,g$ are not separated by any edge connecting red and blue vertices, we have $(v_f,v_g)=0$;
\item if $f,g$ are two faces separated by an edge  connecting a red vertex with a blue vertex such that when crossing from $f$ into $g$ the red vertex lies to the right, we have $(v_f,v_g)=1$.
\item if $f,g$ are two faces separated by an edge connecting a red (resp. blue) vertex  to a boundary vertex  such that when crossing from $f$ into $g$ the boundary vertex lies to the left, we have $(v_f,v_g)=1/2$ (resp. $(v_f,v_g)=-1/2$).
\end{enumerate}
\end{defn}

Mutation of graphs can also be lifted to the level of $V$-colored braid graphs: 
\begin{defn}
If $\bf i'$ is obtained from $\bf i$ by a braid move $\sigma_{i+1}\sigma_{i}\sigma_{i+1}\mapsto\sigma_{i}\sigma_{i+1}\sigma_{i}$ and $\ell$ is a $V$-labelling of $\Gamma_{\bf i}$, we define the \emph{mutated} labelling $\mu_v(\ell)$ of $\Gamma_{\bf i'}$ as in Figure~\ref{fig:br-mut1}, with $\mu_v(\ell)$ coinciding with $\ell$ on all faces not appearing in that Figure.  

Similarly, if $\bf i'$ is obtained from $\bf i$ by a braid move $\sigma_{i}\sigma_{i+1}\sigma_{i}\mapsto\sigma_{i+1}\sigma_{i}\sigma_{i+1}$, the mutated labelling is defined in  Figure~\ref{fig:br-mut2}. 

Finally, if $\bf i'$ is obtained from $\bf i$ by a Demazure move, we define the mutated labelling by Figure~\ref{fig:mut-dem} or Figure \ref{fig:mut-dem2}. 
\end{defn}

It is easily checked that $V$-colorings remain $V$-colorings after mutation. As in this case the labelling is by distinct elements of $V$, we may as well refer to a cell by its labelling, and say for example that `a face $v$ is a mutable face'. Note however that performing two consecutive braid moves  $\sigma_{i+1}\sigma_{i}\sigma_{i+1}\mapsto \sigma_{i}\sigma_{i+1}\sigma_{i}\rightarrow\sigma_{i+1}\sigma_{i}\sigma_{i+1}$ returns us to the original graph but \emph{not} the original $V$-labelling.

\begin{figure}
\begin{tikzpicture}[>=Stealth, line cap=round, scale=0.8]
  \tikzset{
    wire/.style={line width=0.8pt},
    post/.style={line width=0.9pt,->},
    reddot/.style={circle,draw=red!70!black,fill=red,inner sep=1.2pt},
    bluedot/.style={circle,draw=blue!70!black,fill=blue,inner sep=1.2pt},
    facelabel/.style={font=\small}
  }

  \newcommand{\threewirepanel}[4]{%
    \begin{scope}[shift={(#1,0)}]
      \foreach \yy/\lab in {1/{\(i{-}1\)},0/{\(i\)},-1/{\(i{+}1\)}}{
        \draw[wire,->] (0,\yy) -- (6,\yy);
        \node[left] at (0,\yy) {\lab};
      }
      \foreach \x in {#2}{
        \draw[post] (\x,0) -- (\x,-1);
        \node[reddot] at (\x,-1) {};
        \node[bluedot] at (\x,0)  {};
      }
      \foreach \x in {#3}{
        \draw[post] (\x,1) -- (\x,0);
        \node[bluedot] at (\x,1) {};
        \node[reddot]  at (\x,0) {};
      }
      \foreach \pt/\txt in {#4}{
        \node[facelabel] at \pt {\txt};
      }
    \end{scope}
  }

  \threewirepanel{0}{2,4}{3}{
    (1.5,0.5)/$v_1$, (3,-0.5)/$v$, (4.5,0.5)/$v_2$,
    (1,-0.5)/$v_4$, (5,-0.5)/$v_3$
  }

  \threewirepanel{8.5}{3}{2,4}{
    (1,0.5)/$v_1+v$, (3,0.5)/$-v$, (5,0.5)/$v_2$,
    (2,-0.5)/$v_4$, (4,-0.5)/$v_3+v$
  }
\end{tikzpicture}
\caption{Local $V$-label modification for braid move $ \sigma_{i+1}\sigma_{i}\sigma_{i+1}\mapsto \sigma_{i}\sigma_{i+1}\sigma_{i}$. }
\label{fig:br-mut1}
\end{figure}
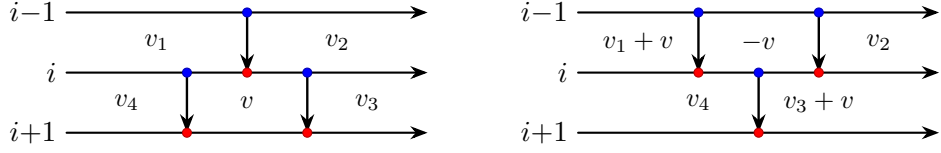

\begin{figure}
\begin{tikzpicture}[>=Stealth, line cap=round, scale=0.8]
  \tikzset{
    wire/.style={line width=0.8pt},
    post/.style={line width=0.9pt,->},
    reddot/.style={circle,draw=red!70!black,fill=red,inner sep=1.2pt},
    bluedot/.style={circle,draw=blue!70!black,fill=blue,inner sep=1.2pt},
    facelabel/.style={font=\small}
  }

  \newcommand{\threewirepanel}[4]{%
    \begin{scope}[shift={(#1,0)}]
      \foreach \yy/\lab in {1/{\(i{-}1\)},0/{\(i\)},-1/{\(i{+}1\)}}{
        \draw[wire,->] (0,\yy) -- (6,\yy);
        \node[left] at (0,\yy) {\lab};
      }
      \foreach \x in {#2}{
        \draw[post] (\x,0) -- (\x,-1);
        \node[reddot] at (\x,-1) {};
        \node[bluedot] at (\x,0)  {};
      }
      \foreach \x in {#3}{
        \draw[post] (\x,1) -- (\x,0);
        \node[bluedot] at (\x,1) {};
        \node[reddot]  at (\x,0) {};
      }
      \foreach \pt/\txt in {#4}{
        \node[facelabel] at \pt {\txt};
      }
    \end{scope}
  }

  \threewirepanel{0}{3}{2,4}{
    (1,0.5)/$v_1$, (3,0.5)/$v$, (5,0.5)/$v_2$,
    (2,-0.5)/$v_4$, (4,-0.5)/$v_3$
}

  \threewirepanel{8.5}{2,4}{3}{
    (1.5,0.5)/$v_1$, (3,-0.5)/$-v$, (4.5,0.5)/$v_2+v$,
    (1,-0.5)/$v_4+v$, (5,-0.5)/$v_3$
  }

\end{tikzpicture}
\caption{Local $V$-label modification for braid move $\sigma_{i}\sigma_{i+1}\sigma_{i} \mapsto \sigma_{i+1}\sigma_{i}\sigma_{i+1}$. }
\label{fig:br-mut2}
\end{figure}
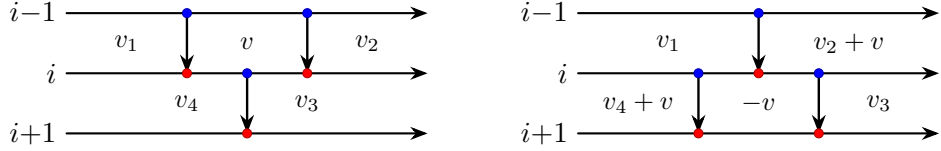

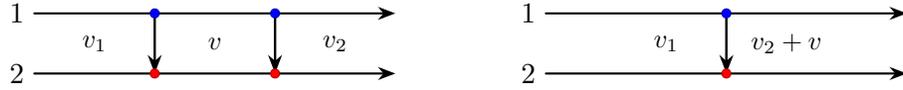
\begin{figure}
\begin{tikzpicture}[>=Stealth, line cap=round, scale=0.8]
  \tikzset{
    wire/.style={line width=0.8pt},
    post/.style={line width=0.9pt,->},
    reddot/.style={circle,draw=red!70!black,fill=red,inner sep=1.2pt},
    bluedot/.style={circle,draw=blue!70!black,fill=blue,inner sep=1.2pt},
    facelabel/.style={font=\small}
  }

  \newcommand{\twowirepanel}[3]{%
    \begin{scope}[shift={(#1,0)}]
      \draw[wire,->] (0,1) -- (6,1);   \node[left] at (0,1) {$1$};
      \draw[wire,->] (0,0) -- (6,0);   \node[left] at (0,0) {$2$};
      \foreach \x in {#2}{
        \draw[post] (\x,1) -- (\x,0);
        \node[reddot]  at (\x,0) {};
        \node[bluedot] at (\x,1) {};
      }
      \foreach \pt/\txt in {#3}{
        \node[facelabel] at \pt {\txt};
      }
    \end{scope}
  }

  \twowirepanel{0}{2,4}{
    (1,0.5)/$v_1$, (3,0.5)/$v$, (5,0.5)/$v_2$
  }

  \twowirepanel{8.5}{3}{
    (2,0.5)/$v_1$, (4,0.5)/$v_2+v$
  }
\end{tikzpicture}
\caption{Local $V$-label modification for $1$-Demazure move.}
\label{fig:mut-dem}
\end{figure}

\begin{figure}
\begin{tikzpicture}[>=Stealth, line cap=round, scale=0.8]
  \tikzset{
    wire/.style={line width=0.8pt},
    post/.style={line width=0.9pt,->},
    reddot/.style={circle,draw=red!70!black,fill=red,inner sep=1.2pt},
    bluedot/.style={circle,draw=blue!70!black,fill=blue,inner sep=1.2pt},
    facelabel/.style={font=\small}
  }

  \newcommand{\twowirepanel}[3]{%
    \begin{scope}[shift={(#1,0)}]
      \draw[wire,->] (0,1) -- (6,1);   \node[left] at (0,1) {$m-1$};
      \draw[wire,->] (0,0) -- (6,0);   \node[left] at (0,0) {$m$};
      \foreach \x in {#2}{
        \draw[post] (\x,1) -- (\x,0);
        \node[reddot]  at (\x,0) {};
        \node[bluedot] at (\x,1) {};
      }
      \foreach \pt/\txt in {#3}{
        \node[facelabel] at \pt {\txt};
      }
    \end{scope}
  }

  \twowirepanel{0}{2,4}{
    (1,0.5)/$v_1$, (3,0.5)/$v$, (5,0.5)/$v_2$
  }

  \twowirepanel{8.5}{3}{
    (2,0.5)/$v_1$, (4,0.5)/$v_2+v$
  }
\end{tikzpicture}
\caption{Local $V$-label modification for $m-1$-Demazure move. }
\label{fig:mut-dem2}
\end{figure}
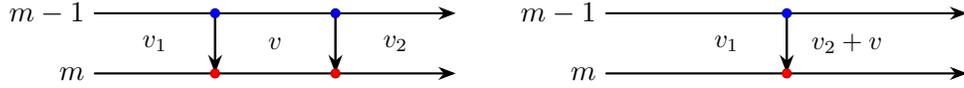

Suppose that $(\Gamma_{\mathbf{i}},\ell)$ is a $V$-labelled graph corresponding to positive braid word $\mathbf{i}$, and $\mathbf{i}_0\subseteq\mathbf{i}$ is some (not necessarily consecutive) substring of $\mathbf{i}$. If we write $\mathbf{i}\setminus\mathbf{i}_0$ for the word obtained from  $\mathbf{i}$ by deleting the subword $\mathbf{i}_0$, the graph $\Gamma_{\mathbf{i}\setminus\mathbf{i}_0}$ is the one built from $\Gamma_{\mathbf{i}}$ by deleting the vertical edges in $\Gamma_{\mathbf{i}}$ associated to $\mathbf{i}_0$. So the faces of $\Gamma_{\mathbf{i}\setminus\mathbf{i}_0}$  are unions of faces of $(\Gamma_{\mathbf{i}},\ell)$.

\begin{defn} 
If $\mbi_0 \subseteq \mbi$ and $\ell$ is a $V$-labelling on $\Gamma_{\mbi}$, we define on $\Gamma_{\mbi \setminus \mbi_0}$ the labelling  
\begin{align}
\label{eq:face-merge}
\overline{l}(\overline{f}) =\sum_{f\subseteq \overline{f}}l(f)
\end{align}
 where the sum is taken over all faces $f$ of $\Gamma_{\mathbf{i}}$ contained in the face $\overline{f}$ of $\Gamma_{\mathbf{i}\setminus\mathbf{i}_0}$. 
\end{defn}

By a direct verification one sees that if $\ell$ is a $V$-coloring, also $\overline{\ell}$ is a $V$-coloring. 

\begin{defn}
Let $p$ be a directed path in a $V$-labelled graph $(\Gamma_{\bf i},\ell)$ starting from the left boundary vertex $a$ and ending at the right boundary vertex $b$. We define its \emph{weight} $w_p$ and \emph{adjusted weight} $u_p$ to be the sums
\begin{equation}
w_p=\sum_{f\subseteq p} v_f\in V,\qquad u_p = w_p - w_{p_{b,b}} \in V,
\end{equation}
where we write $f \subseteq p$ to mean that the face $f$ lies under $p$, and where $p_{b,b}$ is the unique directed path in $\Gamma_{\mathbf{i}}$ connecting the boundary source vertex $b$ with the boundary sink $b$.
\end{defn}
Inspecting Figures~\ref{fig:br-mut1},~\ref{fig:br-mut2},~\ref{fig:mut-dem} and ~\ref{fig:mut-dem2}, we see that the weight of this path is invariant under braid and Demazure moves.

The basic colored graphs that we will make use of are the following: 

\begin{defn}
We define $\Gamma_{\mathbb{E}} = (\Gamma_{\mathbf{w}},\ell)$ to be the $\btd{N}$-colored graph arranged along the decomposition 
\begin{equation}\label{EqLongestWord1}
\mathbf{w}_0 = (\mathbf{w}_n,\mathbf{w}_{n-1},\ldots,\mathbf{w}_2,\mathbf{w}_1),\qquad \mathbf{w}_k = (\sigma_1,\sigma_{2},\ldots,\sigma_{k}),
\end{equation}
with the vector $e_{a,b,c}$ in the $c$'th cell on the $b$'th row.

Similarly, we define $\Gamma_{\mathbb{F}} = (\Gamma_{\overline{\mathbf{w}}_0},\ell)$ to be the $\btd{N}$-colored graph arranged along the decomposition 
\begin{equation}\label{EqLongestWord2}
\overline{\mathbf{w}}_0 = \overline{\mathbf{w}}_n\overline{\mathbf{w}}_{n-1}\ldots \overline{\mathbf{w}}_2\overline{\mathbf{w}}_1,\qquad \overline{\mathbf{w}}_k = (\sigma_n,\sigma_{n-1},\ldots,\sigma_{N-k}),
\end{equation}
with the vector $e_{a,b,c}$ in the $b$'th cell on the $a$'th row.
\end{defn}

See Figures \ref{fig:E-graph} and \ref{fig:F-graph} for illustrations in the case $N=4$.

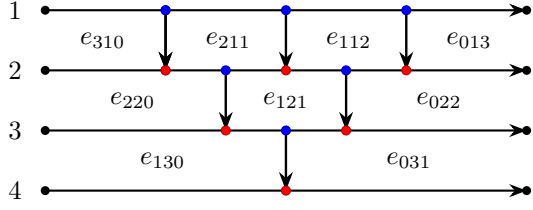
\begin{figure}
\begin{tikzpicture}[>=Stealth, line cap=round, x=8mm, y=8mm]
  \tikzset{
    wire/.style   ={line width=0.8pt},
    post/.style   ={line width=0.9pt,->},
    reddot/.style ={circle,draw=red!70!black,fill=red,inner sep=1.2pt},
    bluedot/.style={circle,draw=blue!70!black,fill=blue,inner sep=1.2pt},
    blackdot/.style={circle,fill=black,inner sep=1.2pt},
    elab/.style   ={font=\small}
  }
  \draw[wire,->] ( -1,0) -- (7,0);
  \draw[wire,->] (  -1,1) -- ( 7,1);
  \draw[wire,->] (  -1,2) -- ( 7,2);
  \draw[wire,->] (  -1,3) -- ( 7,3);

  \foreach \x in {1,3,5,}{
    \draw[post] (\x,3) -- (\x,2);
    \node[reddot] at (\x,2) {};
    \node[bluedot]  at (\x,3) {};
  }
  \foreach \x in {2,4}{
    \draw[post] (\x,2) -- (\x,1);
    \node[reddot] at (\x,1) {};
    \node[bluedot]  at (\x,2) {};
  }
  \foreach \x in {3}{
    \draw[post] (\x,1) -- (\x,0);
    \node[reddot] at (\x,0) {};
    \node[bluedot]  at (\x,1) {};
  }

   \node[blackdot] at (-1,0) {};
   \node[blackdot] at (-1,1) {};
   \node[blackdot] at (-1,2) {};
   \node[blackdot] at (-1,3) {};
   \node[blackdot] at (7,0) {};
   \node[blackdot] at (7,1) {};
   \node[blackdot] at (7,2) {};
   \node[blackdot] at (7,3) {};
   \node[elab] at (-1.5,0) {$4$};
  \node[elab] at (-1.5,1) {$3$};
  \node[elab] at (-1.5,2) {$2$};
  \node[elab] at (-1.5,3) {$1$};

  \node[elab, anchor=east] at (0.5,2.5) {$e_{310}$};
  \node[elab, anchor=west] at (1.5,2.5) {$e_{211}$};
    \node[elab, anchor=west] at (3.5,2.5) {$e_{112}$};
      \node[elab, anchor=west] at (5.5,2.5) {$e_{013}$};

  \node[elab, anchor=east] at (1,1.5) {$e_{220}$};
  \node[elab]              at (3.0,1.5) {$e_{121}$};
  \node[elab, anchor=west] at (5.0,1.5) {$e_{022}$};

  \node[elab, anchor=east] at (1.5,0.5) {$e_{130}$};
  \node[elab, anchor=west] at (4.5,.5) {$e_{031}$};

\end{tikzpicture}
\caption{The $\nabla_N$-labelled graph $\Gamma_{\mathbb{E}}$ for $N=4$. The boundary vertices are drawn in black.}
\label{fig:E-graph}
\end{figure}

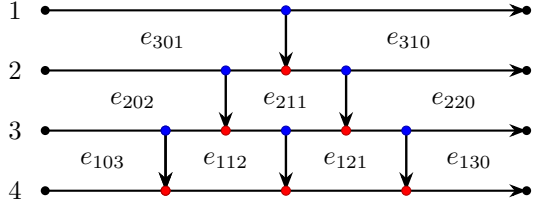
\begin{figure}
\begin{tikzpicture}[>=Stealth, line cap=round, x=8mm, y=8mm]
  \tikzset{
    wire/.style   ={line width=0.8pt},
    post/.style   ={line width=0.9pt,->},
    reddot/.style ={circle,draw=red!70!black,fill=red,inner sep=1.2pt},
    bluedot/.style={circle,draw=blue!70!black,fill=blue,inner sep=1.2pt},
    blackdot/.style={circle,fill=black,inner sep=1.2pt},    
    elab/.style   ={font=\small}
  }
  \draw[wire,->] ( -1,0) -- (7,0);
  \draw[wire,->] (  -1,1) -- ( 7,1);
  \draw[wire,->] (  -1,2) -- ( 7,2);
  \draw[wire,->] (  -1,3) -- ( 7,3);

  \foreach \x in {1,3,5,}{
    \draw[post] (\x,1) -- (\x,0);
    \node[reddot] at (\x,0) {};
    \node[bluedot]  at (\x,1) {};
  }
  \foreach \x in {2,4}{
    \draw[post] (\x,2) -- (\x,1);
    \node[reddot] at (\x,1) {};
    \node[bluedot]  at (\x,2) {};
  }
  \foreach \x in {3}{
    \draw[post] (\x,3) -- (\x,2);
    \node[reddot] at (\x,2) {};
    \node[bluedot]  at (\x,3) {};
  }

   \node[blackdot] at (-1,0) {};
   \node[blackdot] at (-1,1) {};
   \node[blackdot] at (-1,2) {};
   \node[blackdot] at (-1,3) {};
   \node[blackdot] at (7,0) {};
   \node[blackdot] at (7,1) {};
   \node[blackdot] at (7,2) {};
   \node[blackdot] at (7,3) {};

  \node[elab] at (-1.5,0) {$4$};
  \node[elab] at (-1.5,1) {$3$};
  \node[elab] at (-1.5,2) {$2$};
  \node[elab] at (-1.5,3) {$1$};

  \node[elab, anchor=east] at (1.5,2.5) {$e_{301}$};
  \node[elab, anchor=west] at (4.5,2.5) {$e_{310}$};

  \node[elab, anchor=east] at (1,1.5) {$e_{202}$};
  \node[elab]              at (3.0,1.5) {$e_{211}$};
  \node[elab, anchor=west] at (5.25,1.5) {$e_{220}$};

  \node[elab, anchor=east] at (.5,0.5) {$e_{103}$};
  \node[elab]              at (2.0,0.5) {$e_{112}$};
  \node[elab]              at (4.0,0.5) {$e_{121}$};
  \node[elab, anchor=west] at (5.5,.5) {$e_{130}$};

\end{tikzpicture}
\caption{The $\nabla_N$-labelled graph $\Gamma_{\mathbb{F}}$ for $N=4$. The boundary vertices are drawn in black.}
\label{fig:F-graph}
\end{figure}

For later use, we include here the following lemma. 

\begin{lemma}\label{LemLongestDouble}
Let $\Gamma_{\mathbf{w}_0}'$ be the $\btd{N}$-colored subgraph of $\Gamma_{\mathbf{w}_0}$ forgetting all utmost right vertical line on every row. Consider any inclusion $\btd{N}\subseteq V$ and any extension of $\Gamma_{\mathbb{E}}'$ to a $V$-colored graph structure $(\Gamma_{\mathbf{w}_0\mathbf{w}_0},\widetilde{\ell})$.

Then there exist a sequence of mutations reducing $(\widetilde{\Gamma}_{\mathbf{w}_0\mathbf{w}_0})$ to a $V$-labeled graph $(\Gamma_{\mathbf{w}_0},\ell')$ such that $\ell_{\mid \Gamma_{\mathbf{w}_0}'} = \ell_{\mid \Gamma_{\mathbf{w}_0}}$.  
\end{lemma} 

So, in case $N=4$, the lemma says that the Figure \ref{fig:E-graphDoubledAlt} can be reduced to the Figure \ref{fig:E-graphNilpAlt}:

\begin{figure}[ht]
\begin{tikzpicture}[>=Stealth, line cap=round, x=8mm, y=8mm]
  \tikzset{
    wire/.style   ={line width=0.8pt},
    post/.style   ={line width=0.9pt,->},
    reddot/.style ={circle,draw=red!70!black,fill=red,inner sep=1.2pt},
    bluedot/.style={circle,draw=blue!70!black,fill=blue,inner sep=1.2pt},
    elab/.style   ={font=\small}, 
    blackdot/.style={circle,fill=black,inner sep=1.2pt},
    elab/.style   ={font=\small}
  }

  \draw[wire,->] ( -1,3) -- (14,3);
  \draw[wire,->] (  -1,2) -- ( 14,2);
  \draw[wire,->] (  -1,1) -- ( 14,1);
  \draw[wire,->] (  -1,0) -- ( 14,0);

  \foreach \x in {1,3,5,7,9,11,}{
    \draw[post] (\x,3) -- (\x,2);
    \node[bluedot] at (\x,3) {};
    \node[reddot]  at (\x,2) {};
  }
  \foreach \x in {2,4,8,10}{
    \draw[post] (\x,2) -- (\x,1);
    \node[bluedot] at (\x,2) {};
    \node[reddot]  at (\x,1) {};
  }
  \foreach \x in {3,9}{
    \draw[post] (\x,1) -- (\x,0);
    \node[bluedot] at (\x,1) {};
    \node[reddot]  at (\x,0) {};
  }
   \node[blackdot] at (-1,0) {};
   \node[blackdot] at (-1,1) {};
   \node[blackdot] at (-1,2) {};
   \node[blackdot] at (-1,3) {};
   \node[blackdot] at (14,0) {};
   \node[blackdot] at (14,1) {};
   \node[blackdot] at (14,2) {};
   \node[blackdot] at (14,3) {};
  \node[elab] at (-1.5,3) {$1$};
  \node[elab] at (-1.5,2) {$2$};
  \node[elab] at (-1.5,1) {$3$};
  \node[elab] at (-1.5,0) {$4$};

  \node[elab, anchor=east] at (1.5,0.5) {{\tiny $e_{130}$}};
  \node[elab] at (6,0.5) {{\tiny $*$}};
  \node[elab, anchor=west] at (10.5,0.5) {{\tiny $*$}};

  \node[elab, anchor=east] at (1,1.5) {{\tiny $e_{220}$}};
  \node[elab]              at (3.0,1.5) {{\tiny $e_{121}$}};
  \node[elab] at (6,1.5) {{\tiny $*$}};
  \node[elab]              at (9,1.5) {{\tiny $*$}};
  \node[elab, anchor=west] at (11,1.5) {{\tiny $*$}};

  \node[elab, anchor=east] at (.5,2.5) {{\tiny $e_{310}$}};
  \node[elab]              at (2.0,2.5) {{\tiny $e_{211}$}};
  \node[elab]              at (4.0,2.5) {{\tiny $e_{112}$}};
  \node[elab] at (6.0,2.5) {{\tiny $*$}};
  \node[elab]              at (8.0,2.5) {{\tiny $*$}};
  \node[elab]              at (10.0,2.5) {{\tiny $*$}};
  \node[elab, anchor=west] at (11.5,2.5) {{\tiny $*$}};
\end{tikzpicture}
\caption{The $V$-labelled graph $(\widetilde{\Gamma}_{\mathbf{w}_0\mathbf{w}_0},\widetilde{\ell})$.}
\label{fig:E-graphDoubledAlt}
\end{figure}
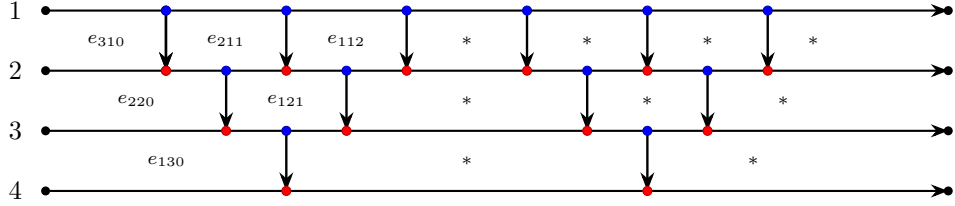

\begin{figure}[ht]
\begin{tikzpicture}[>=Stealth, line cap=round, x=8mm, y=8mm]
  \tikzset{
    wire/.style   ={line width=0.8pt},
    post/.style   ={line width=0.9pt,->},
    reddot/.style ={circle,draw=red!70!black,fill=red,inner sep=1.2pt},
    bluedot/.style={circle,draw=blue!70!black,fill=blue,inner sep=1.2pt},
    elab/.style   ={font=\small},
    blackdot/.style={circle,fill=black,inner sep=1.2pt},
    elab/.style   ={font=\small}
  }

  \draw[wire,->] ( -1,3) -- (8,3);
  \draw[wire,->] (  -1,2) -- ( 8,2);
  \draw[wire,->] (  -1,1) -- ( 8,1);
  \draw[wire,->] (  -1,0) -- ( 8,0);

  \foreach \x in {1,3,5,}{
    \draw[post] (\x,3) -- (\x,2);
    \node[bluedot] at (\x,3) {};
    \node[reddot]  at (\x,2) {};
  }
  \foreach \x in {2,4}{
    \draw[post] (\x,2) -- (\x,1);
    \node[bluedot] at (\x,2) {};
    \node[reddot]  at (\x,1) {};
  }
  \foreach \x in {3}{
    \draw[post] (\x,1) -- (\x,0);
    \node[bluedot] at (\x,1) {};
    \node[reddot]  at (\x,0) {};
  }
  
   \node[blackdot] at (-1,0) {};
   \node[blackdot] at (-1,1) {};
   \node[blackdot] at (-1,2) {};
   \node[blackdot] at (-1,3) {};
   \node[blackdot] at (8,0) {};
   \node[blackdot] at (8,1) {};
   \node[blackdot] at (8,2) {};
   \node[blackdot] at (8,3) {};
  \node[elab] at (-1.5,3) {$1$};
  \node[elab] at (-1.5,2) {$2$};
  \node[elab] at (-1.5,1) {$3$};
  \node[elab] at (-1.5,0) {$4$};

  \node[elab, anchor=east] at (1.5,0.5) {{\tiny $e_{130}$}};
  \node[elab, anchor=west] at (4.5,0.5) {{\tiny $*$}};

  \node[elab, anchor=east] at (1,1.5) {{\tiny $e_{220}$}};
  \node[elab]              at (3.0,1.5) {{\tiny $e_{121}$}};
  \node[elab, anchor=west] at (5.25,1.5) {{\tiny $*$}};

  \node[elab, anchor=east] at (.5,2.5) {{\tiny $e_{310}$}};
  \node[elab]              at (2.0,2.5) {{\tiny $e_{211}$}};
  \node[elab]              at (4.0,2.5) {{\tiny $e_{112}$}};
  \node[elab, anchor=west] at (5.5,2.5) {{\tiny $*$}};

\end{tikzpicture}
\caption{The $V$-labelled graph $(\Gamma_{\mathbf{w}_0},\ell')$.}
\label{fig:E-graphNilpAlt}
\end{figure}
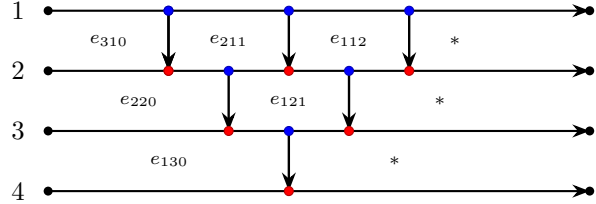

\begin{proof}
We claim in fact that the required sequence of mutations can be found using only the mutations 
\begin{equation}\label{EqBraidMoves}
\sigma_{i+1}\sigma_i \sigma_{i+1} \mapsto \sigma_i \sigma_{i+1} \sigma_i,\qquad \sigma_1\sigma_1 \mapsto \sigma_1. 
\end{equation}
To see this efficiently, we will encode our diagrams slightly differently, in the spirit of the \emph{snake paths,} see~\cite{FG06}.

Consider the set $\msP_n'$ of all possible families $P$ of $n$ non-intersecting  paths in $\Z^2$, with the $i$-th path connecting $(0,i)$ with $(i,0)$, and allowing within each unit square only path pieces of the form $\rightarrow,\downarrow, \sears$ if the path passes through it. We can encode such a family of paths conveniently as a matrix $P \in M_n(\Z)$, with the $(k,l)$-th entry equal to $i$ if the $i$th path hits the vertex $(k,l)$ (and being $0$ if no path passes through it). For convenience, we pad this matrix on the left and bottom with the constant vectors $(1,2,\ldots,n)$ and $(n,n-1,\ldots,1)$.

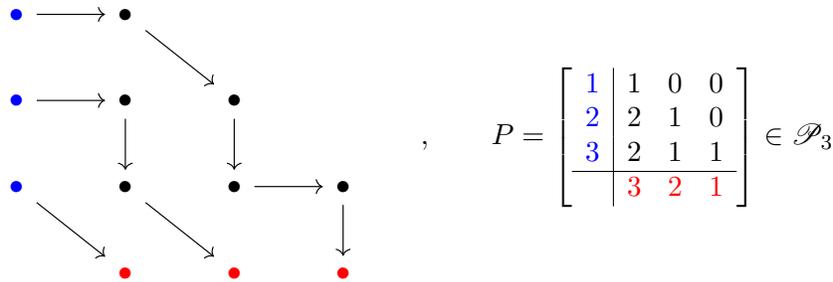
\[\begin{tikzcd}
	\color{blue}{\bullet} & \bullet \\
	\color{blue}{\bullet}& \bullet & \bullet \\
	\color{blue}{\bullet} & \bullet & \bullet & \bullet \\
	& \color{red}{\bullet} & \color{red}{\bullet} & \color{red}{\bullet}
	\arrow[from=1-1, to=1-2]
	\arrow[from=1-2, to=2-3]
	\arrow[from=2-1, to=2-2]
	\arrow[from=2-2, to=3-2]
	\arrow[from=2-3, to=3-3]
	\arrow[from=3-1, to=4-2]
	\arrow[from=3-2, to=4-3]
	\arrow[from=3-3, to=3-4]
	\arrow[from=3-4, to=4-4]
\end{tikzcd}\qquad,\qquad  
P = \left[
\begin{array}{c|ccc}
  \textcolor{blue}{1}  & 1 & 0 &0 \\ \textcolor{blue}{2} & 2 & 1 & 0 \\ \textcolor{blue}{3} & 2 & 1 &1 \\
  \hline  & \textcolor{red}{3} & \textcolor{red}{2} & \textcolor{red}{1} 
\end{array}
\right] \in \msP_3 
\]
(to be consistent with the positioning of the graph, we will count the rows of the matrix from the bottom, starting at $0$, and the columns from the left, starting at $0$).

We then consider within $\msP_n'$ the family $\msP_n$ whose matrices have the fixed borders as above, and such that any $2$-by-$2$-square inside (including ones with border entries) is of one of the following types: 
\begin{multline}\label{EqMatrixGraph}
\begin{pmatrix} i-1 & i-1 \\ i & i-1\end{pmatrix},\quad \begin{pmatrix} i & i-1\\ i & i-1\end{pmatrix},\quad \begin{pmatrix} i- 1 & i-1 \\ i & i \end{pmatrix},\quad \begin{pmatrix}i & i-1 \\i & i \end{pmatrix},\\ \begin{pmatrix} i-1 & i-2 \\ i & i-1\end{pmatrix},\quad \begin{pmatrix} 0 & 0 \\ 0 & 0\end{pmatrix},
\end{multline}
with all entries  between $0$ and $n$. 

To each $P \in \msP_n$, we can associate a braid graph $\Gamma_P$, in the following way: connect a vertex to its right neighbour if they lie on different paths, and else to its northeast neighbour (which is then necessarily on a different path). Necessarily, if a vertex lies on the $i$-th path, it is connected to a vertex on the $i-1$-th path. 

We then construct a map 
\begin{equation}\label{EqCorrPaths}
P \mapsto \Gamma_P, 
\end{equation}
associating to $P$ the braid graph $\Gamma_P$ whose $i$-th row has its cells labelled in order by the vertices of the $i$-th path of $P$, and where the right hand wall of the $k$th cell on the $i$th row has the $l$'th cell of the $i-1$-th row on top of it if and only if the vertices corresponding to $k$ and $l$ are connected in the way described in the previous paragraph.

For example, for \eqref{EqMatrixGraph} we get 
\begin{equation}
\begin{tikzcd}
	\color{blue}{\bullet} & \bullet \\
	\color{blue}{\bullet}& \bullet & \bullet \\
	\color{blue}{\bullet} & \bullet & \bullet & \bullet \\
	& \color{red}{\bullet} & \color{red}{\bullet} & \color{red}{\bullet}
	\arrow[from=1-1, to=1-2]
	\arrow[from=1-2, to=2-3]
	\arrow[color=green, from=2-1, to=1-2]
	\arrow[from=2-1, to=2-2]
	\arrow[color=green, from=2-2, to=2-3]
	\arrow[from=2-2, to=3-2]
	\arrow[from=2-3, to=3-3]
	\arrow[color=green, from=3-1, to=3-2]
	\arrow[from=3-1, to=4-2]
	\arrow[color=green, from=3-2, to=3-3]
	\arrow[from=3-2, to=4-3]
	\arrow[from=3-3, to=3-4]
	\arrow[from=3-4, to=4-4]
\end{tikzcd}, \qquad 
\vcenter{ \hbox{ \begin{tikzpicture}[>=Stealth, line cap=round, x=4mm, y=8mm]
  \tikzset{
    wire/.style   ={line width=0.8pt},
    post/.style   ={line width=0.9pt,->},
    reddot/.style ={circle,draw=red!70!black,fill=red,inner sep=1.2pt},
    bluedot/.style={circle,draw=blue!70!black,fill=blue,inner sep=1.2pt},
    blackdot/.style={circle,fill=black,inner sep=1.2pt},    
    elab/.style   ={font=\small}
  }
  \draw[wire,->] ( -1,0) -- (10,0);
  \draw[wire,->] (  -1,1) -- ( 10,1);
  \draw[wire,->] (  -1,2) -- ( 10,2);
  \draw[wire,->] (  -1,3) -- ( 10,3);

  \foreach \x in {5}{
    \draw[post] (\x,1) -- (\x,0);
    \node[reddot] at (\x,0) {};
    \node[bluedot]  at (\x,1) {};
  }
  \foreach \x in {2,4,6}{
    \draw[post] (\x,2) -- (\x,1);
    \node[reddot] at (\x,1) {};
    \node[bluedot]  at (\x,2) {};
  }
  \foreach \x in {1,3,5,7,8}{
    \draw[post] (\x,3) -- (\x,2);
    \node[reddot] at (\x,2) {};
    \node[bluedot]  at (\x,3) {};
  }

   \node[blackdot] at (-1,0) {};
   \node[blackdot] at (-1,1) {};
   \node[blackdot] at (-1,2) {};
   \node[blackdot] at (-1,3) {};
   \node[blackdot] at (10,0) {};
   \node[blackdot] at (10,1) {};
   \node[blackdot] at (10,2) {};
   \node[blackdot] at (10,3) {};

  \node[elab] at (-1.5,0) {$4$};
  \node[elab] at (-1.5,1) {$3$};
  \node[elab] at (-1.5,2) {$2$};
  \node[elab] at (-1.5,3) {$1$};

\end{tikzpicture}
}}
\end{equation}
  
Let us now say that for $P \in \msP_n$ can be \emph{mutated} at position $(k,l)$ if the matrix $P'$ again lies in $\msP_n$, where $P'$ equals $P$ everywhere except at one entry $(k,l)$ for $1\leq k,l\leq n$, where 
\[
P_{(k,l)}' = P_{(k,l)}-1. 
\]
It is easily checked that such mutations are possible exactly at those vertices $(k,l)$ where a path passes through, say the $i$th one, where the local configuration looks either as 
\begin{equation}\label{EqLocalForm}
 \begin{tikzcd}
	& {i-1} & {i-2} \\
	i & i & {i-1} \\
	 & i
	\arrow[from=1-2, to=2-3]
	\arrow[color={rgb,255:red,92;green,214;blue,92}, from=2-1, to=1-2]
	\arrow[from=2-1, to=2-2]
	\arrow[color={rgb,255:red,92;green,214;blue,92}, from=2-2, to=2-3]
	\arrow[from=2-2, to=3-2]
\end{tikzcd},\quad \textrm{or}\qquad 
\begin{tikzcd}
	& {0} & {0} \\
	1 & 1 & {0} \\
	{2} & 1
	\arrow[from=2-1, to=2-2]
	\arrow[from=2-2, to=3-2]
\arrow[color={rgb,255:red,92;green,214;blue,92}, from=3-1, to=3-2]
\end{tikzcd},
\end{equation}
with $(k,l)$ the center vertex. After mutation, we then get the diagrams
\begin{equation}\label{EqLocalFormAfter}
\begin{tikzcd}
	& {i-1} & {i-2} \\
	i & i-1 & {i-1} \\
	 & i
	\arrow[from=1-2, to=2-2]
	\arrow[color={rgb,255:red,92;green,214;blue,92}, from=2-1, to=2-2]
	\arrow[from=2-2, to=2-3]
    \arrow[color={rgb,255:red,92;green,214;blue,92}, from=2-2, to=1-3]
	\arrow[from=2-1, to=3-2]
\end{tikzcd},\quad \textrm{}\qquad 
\begin{tikzcd}
	& {0} & {0} \\
	1 & 0 & {0} \\
	{2} & 1
	\arrow[from=2-1, to=3-2]
\arrow[color={rgb,255:red,92;green,214;blue,92}, from=3-1, to=3-2]
\end{tikzcd}.
\end{equation}
It is then further immediately checked that these mutations correspond exactly to the braid moves in \eqref{EqBraidMoves} under the correspondence \eqref{EqCorrPaths}. 

Finally, let us now consider $V$ a vector space, and let $\msP_n(V)$ be the set of couples $(P,\ell)$ with $P \in \msP_n$ and $\ell \colon P \rightarrow V$ (identifying $P$ with its set of vertices where paths pass through - other vertices $(k,l)$ can be given the zero vector). So, elements of $\msP_n(V)$ can be represented as matrices with values in $\Z\times V$, e.g.\
\[
\left[
\begin{array}{c|ccc}
  \textcolor{blue}{(1,w_{10})}  & (1,w_{11}) & 0 &0 \\ \textcolor{blue}{(2,w_{20})} & (2,w_{21}) & (1,w_{22}) & 0 \\ \textcolor{blue}{(3,w_{30})} & (2,w_{31}) & (1,w_{32}) &(1,w_{33}) \\
  \hline  & \textcolor{red}{3} & \textcolor{red}{2} & \textcolor{red}{1} 
\end{array}
\right] \in \msP_3(V)
\]
(it will be unnecessary to weight the bottom row). 

Now we also need to specify how such weighted elements in $\msP_n(V)$ mutate: namely, we mutate $P$ into $P'$ as before, and if the local diagram of $P$ is in one of the forms of \eqref{EqLocalForm}, with weighting 
\begin{equation}\label{EqLocalFormWeight}
\begin{tikzcd}
	& w_1 & u \\
	z_1 & z_2 & w_2 \\
	 & z_3
	\arrow[from=1-2, to=2-3]
	\arrow[color={rgb,255:red,92;green,214;blue,92}, from=2-1, to=1-2]
	\arrow[from=2-1, to=2-2]
	\arrow[color={rgb,255:red,92;green,214;blue,92}, from=2-2, to=2-3]
	\arrow[from=2-2, to=3-2]
\end{tikzcd},\quad \textrm{or}\qquad 
\begin{tikzcd}
	& {0} & {0} \\
	z_1 & z_2 & {0} \\
	{v} & z_3
	\arrow[from=2-1, to=2-2]
	\arrow[from=2-2, to=3-2]
\arrow[color={rgb,255:red,92;green,214;blue,92}, from=3-1, to=3-2]
\end{tikzcd},
\end{equation}
then after mutation it is weighted as
\begin{equation}\label{EqLocalFormAfterWeight}
\begin{tikzcd}
	& {w_1 +z_2 -z_1} & {u} \\
	z_1 & w_1 & {w_2} \\
	 & z_3
	\arrow[from=1-2, to=2-2]
	\arrow[color={rgb,255:red,92;green,214;blue,92}, from=2-1, to=2-2]
	\arrow[from=2-2, to=2-3]
    \arrow[color={rgb,255:red,92;green,214;blue,92}, from=2-2, to=1-3]
	\arrow[from=2-1, to=3-2]
\end{tikzcd},\quad \textrm{}\qquad 
\begin{tikzcd}
	& {0} & {0} \\
	z_1 & 0 & {0} \\
	{v} & z_3
	\arrow[from=2-1, to=3-2]
\arrow[color={rgb,255:red,92;green,214;blue,92}, from=3-1, to=3-2]
\end{tikzcd},
\end{equation}
with all other weights remaining the same.

If we now extend \eqref{EqCorrPaths} to 
\begin{equation}\label{EqMapTocomplete}
\msP_n(V) \rightarrow \{V\textrm{-colored graphs}\},
\end{equation}
by sending $P\in \msP_n(V)$ the colored graph $(\Gamma_P,\ell)$, where $\ell$ assigns to the $k$-th cell on the $i$-th row the difference of the weights at the $k$-th and $k-1$-th position of the $i$-th path of $P$, then one again easily checks that the mutation \eqref{EqLocalFormAfterWeight} corresponds to the transformations in \eqref{fig:br-mut1} and \eqref{fig:mut-dem} (note that there is no need of compatibility with a skew-symmetric space structure on $V$ if one simply wants to apply mutations of weighted graphs $\Gamma_{\mbi}$). 

We are now finally ready to deduce that $(\mathbf{w}_0\mathbf{w}_0,\widetilde{\ell})$ can be reduced to $(\mathbf{w}_0,\ell')$. Indeed, on the level of $\msP_n(V)$ this means that we want to find a path of mutations from $P_n(2)$ to $P_n$, where 
\[
P_n(2)_{(k,l)} \left\{\begin{array}{lll}(N-k,\seee_{kl})& \textrm{if}& l<k,\\
(N-l,\sef_{l,n-k})&\textrm{if}&l\geq k.\end{array}\right.,\qquad P_{n,(k,l)} = (\max{0,N-k-l},\seee_{k-l,l}),
\]
where we take $\seee_{i,j}= 0$ if ill-defined (note also again that we do not care about the bottom border row of the matrix), and where the vectors $\sef_{s,k}\in V$ are chosen so that they map to the labelling $\widetilde{\ell}$ under \eqref{EqMapTocomplete}. 

\begin{figure}[ht]\label{FigP2}
\begin{tikzcd}
	\color{blue}{\seee_{30}} & {\seee_{31}} & {\seee_{32}} & {\sef_{30}} \\
	\color{blue}{\seee_{20}} & {\seee_{21}} & {\sef_{20}} & {\sef_{31}} \\
	\color{blue}{\seee_{10}} & {\sef_{10}} & {\sef_{21}} & { \sef_{32}} \\
	& \color{red}{\bullet} & \color{red}{\bullet} & \color{red}{\bullet}
	\arrow[from=1-1, to=1-2]
	\arrow[from=1-2, to=1-3]
	\arrow[from=1-3, to=1-4]
	\arrow[from=1-4, to=2-4]
	\arrow[from=2-1, to=2-2]
	\arrow[from=2-2, to=2-3]
	\arrow[from=2-3, to=3-3]
	\arrow[from=2-4, to=3-4]
	\arrow[from=3-1, to=3-2]
	\arrow[from=3-2, to=4-2]
	\arrow[from=3-3, to=4-3]
	\arrow[from=3-4, to=4-4]
\end{tikzcd}, \qquad \qquad

\begin{tikzcd}
	\color{blue}{\seee_{30}} \\
	\color{blue}{\seee_{20}} & {\seee_{31}} \\
	\color{blue}{\seee_{10}} & {\seee_{21}} & {\seee_{32}} \\
	& \color{red}{\bullet} & \color{red}{\bullet} & \color{red}{\bullet}
	\arrow[from=1-1, to=2-2]
	\arrow[from=2-1, to=3-2]
	\arrow[from=2-2, to=3-3]
	\arrow[from=3-1, to=4-2]
	\arrow[from=3-2, to=4-3]
	\arrow[from=3-3, to=4-4]
\end{tikzcd}
\caption{The cases $P_3(2)$ and $P_3$}
\end{figure}
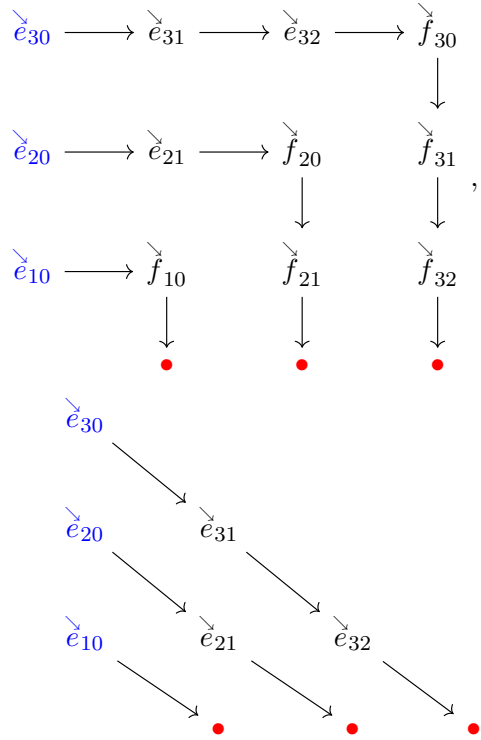

Now, first of all it is clear that we can find such a sequence of mutations on the level of the unweighted paths: this follows simply from the combinatorics of reduction in the positive braid monoid with the Demazure relation, but let us give a direct proof in terms of the above diagrams: indeed, if at some point we have arrived at a situation where the paths $n,n-1, \ldots, i+1$ are all slanted (consisting only of $\sears$-arrows), then the $i$-th path \emph{must} have a corner of the form ${}^{\rightarrow} \!\!\!\downarrow$. If this does not fit into the lower left bottom of the first diagram in \eqref{EqLocalForm}, then the existence of a corner of this form \emph{must} propagate to the next path $i-1$. If this keeps happening till we reach the line $1$, then it means that we have a corner of this form on the first line, and we can hence apply the mutation for the second diagram in \eqref{EqLocalForm}. This means that the only diagram without mutations is the one of $P_n$. 

To see that the reduction of $P_n(2)$ to $P_n$ is also compatible with the weights, we observe that the mutation from \eqref{EqLocalFormWeight} to  \eqref{EqLocalFormAfterWeight} satisfies the following property: if a weight vector $z$ appears before mutation in one of the patterns

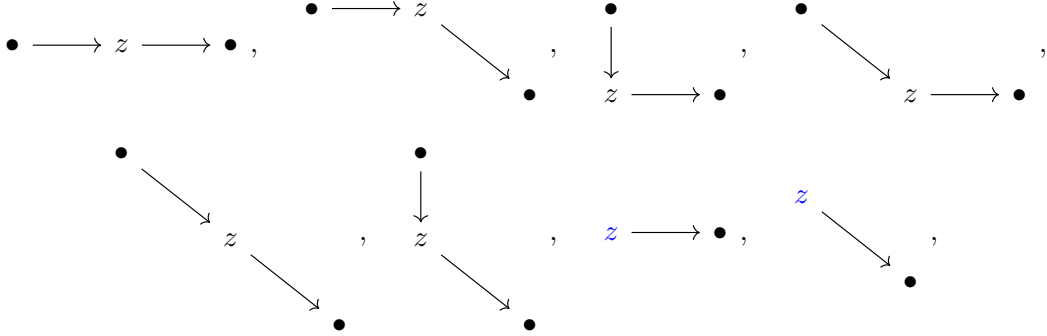
\[
\begin{tikzcd}
	 \bullet & z 	 &	\bullet
    \arrow[from=1-1, to=1-2]
	\arrow[from=1-2, to=1-3]
\end{tikzcd}, \quad
\begin{tikzcd}
	   \bullet &z &\\
	&& \bullet 
	\arrow[from=1-1, to=1-2]
	\arrow[from=1-2, to=2-3]
\end{tikzcd},\quad
\begin{tikzcd}
	 \bullet  &  \\
	  z & \bullet
	\arrow[from=1-1, to=2-1]
	\arrow[from=2-1, to=2-2]
\end{tikzcd},\quad 
\begin{tikzcd}
	 \bullet &&  \\
	& z &\bullet	
    \arrow[from=1-1, to=2-2]
	\arrow[from=2-2, to=2-3]
\end{tikzcd},
\]
\[
\begin{tikzcd}
	 \bullet &&  \\
	& z &	\\
    & &  \bullet
    \arrow[from=1-1, to=2-2]
	\arrow[from=2-2, to=3-3]
\end{tikzcd},\quad 
\begin{tikzcd}
	 \bullet &  \\
	 z& 	\\
     &  \bullet
    \arrow[from=1-1, to=2-1]
	\arrow[from=2-1, to=3-2]
\end{tikzcd}, \quad 
\begin{tikzcd}
	 \color{blue}{z} & \bullet
    \arrow[from=1-1, to=1-2]
\end{tikzcd},\quad 
\begin{tikzcd}
	 \color{blue}{z} \\ & \bullet
    \arrow[from=1-1, to=2-2]
\end{tikzcd},
\]
then the same weight vector also appears somewhere after mutation in again one of these shapes. In particular, the original weight vectors $\seee_{N-i,k}$, for $1 < i< N, 0\leq k <N-i$, always stay present somewhere on the $i$'th line, respecting their original ordering on that line. By cardinality, we see that they \emph{must} be the weights appearing as in the diagram $P_n$. 

\end{proof}

\subsection{Partition functions and standard generators via form sums}\label{SecStandGen}

Let now $\hbar\in \R^{\times}$, and let $(\Hc,\pi)$ be a unitary $\hbar$-representation of $V$ on $\Hsp$. Recall the notation \eqref{EqExponentialpi}.  

\begin{defn}
Let $(\Gamma,\ell)$ be a $V$-colored braid graph.

The \emph{partition function} of $(\Gamma,\ell)$ with boundary conditions $(a,b)$ is defined as the form sum
\begin{align}
\label{eq:Zdef}
Z_{\Gamma,\ell,a,b} := \boxplus_{p:a\rightarrow b} \Eb(u_p), \quad 1\leq a<b\leq m,
\end{align}
where we sum over all paths $p \colon a\rightarrow b$ in $\Gamma$.
\end{defn}

Note that the above form sum makes sense (see Remark \ref{RemFormSumFinite}). If $\mbi$ is a positive braid word in $B_m$, we then further abbreviate 
\begin{equation}\label{EqPartitionAbb}
Z_{\mbi,\ell,a,b}  = Z_{\Gamma_{\mbi},\ell,a,b},
\end{equation}
and by 
\begin{equation}\label{EqPartitionTopBot}
Z_{\mathbf{i},\ell} = Z_{\mathbf{i},\ell,1,m}=\boxplus_{p:1\rightarrow m} \Eb(w_p)
\end{equation}
the partition function of paths starting at the top left boundary vertex 1 and ending at the bottom right boundary vertex $m$.

We use this construction to define some positive operators  analagous to the PBW generators of the subalgebras $U_q(\mfn^{\pm})\subseteq U_q(\mfsl(N,\C))$ corresponding to the nilpotent subalgebras of strictly upper/lower triangular matrices $\mfn^{\pm} \subseteq \mfsl(N,\C)$. 

For $1\leq r<s\leq N$, we define the $\nabla_N$-labelled  \emph{$\mathbb{E}$-graph} $\Gamma_{\mathbb{E}_{rs}}$ to be the subgraph of the graph $\Gamma_{\mathbb{E}}$ illustrated in Figure~\ref{fig:E-graph} bounded by horizontal lines $r$ and $s$. We write $\mathbf{i}_{\mathbb{E}_{rs}}$ for the positive braid word in $B_{r-s+1}$ corresponding to the subgraph $\Gamma_{\mathbb{E}_{rs}}$, with the labelling inherited from $\Gamma_{\mathbb{E}}$. Similarly, we form  the $\nabla_N$-labelled  \emph{$\mathbb{F}$-graphs} as indicated in Figure~\ref{fig:F-graph}.

\begin{defn}
\label{Defstandardgen}
 We define the \emph{standard generators} to be the positive operators
\begin{equation}\label{EqStandardGenE}
\stgE_{rs} = Z_{\mathbf{i}_{\mathbb{E}_{rs}},\ell} = Z_{\Gamma_{\mathbb{E}},\ell,r,s}, \quad 1\leq r<s\leq N
\end{equation}
and
\begin{equation}\label{EqStandardGenF}
\stgF_{rs} = Z_{\mathbf{i}_{\mathbb{F}_{rs}},\ell}=Z_{\Gamma_{\mathbb{F}},\ell,r,s}, \quad 1\leq r<s\leq N.
\end{equation}
\end{defn}

\begin{example}
For $N=3$, we have
$$
\stgE_{1,2} = \Ebf(\seee_{2,0})\boxplus  \Ebf(\seee_{2,1}),\quad \stgE_{2,3} =  \Ebf(\seee_{1,0}),\quad \stgE_{1,3} = \Ebf(\seee_{2,0}+\seee_{1,0})
$$
$$
\stgF_{2,3} =  \Ebf(\nee_{2,0}) \boxplus  \Ebf(\nee_{2,1}),\quad \stgF_{1,2} = \Ebf(\nee_{1,0}) ,\quad \stgF_{1,3} = \Ebf(\nee_{1,0}+\nee_{2,1}).
$$
\end{example}

Recall now the quantum dilogarithm $\varphi = \varphi_{\hbar}$ introduced in \eqref{EqQuantExp}, and the notation for the associated functional calculus introduced in \eqref{EqAbbrFuncCalc}. 
\begin{lemma}
\label{lem:Zmut}
If $v$ is a mutable face of a $V$-colored graph $(\Gamma_{\bf i},\ell)$ and $\bf i'$ the positive braid word obtained by applying the corresponding braid or Demazure move, we have 
$$
\varphi(v)Z_{\mathbf{i},\ell,a,b}\varphi(v)^* = Z_{\mathbf{i}',\mu_v(\ell),a,b}.
$$
\end{lemma} 
\begin{proof}
We treat the case of a braid move $ \sigma_{i+1}\sigma_{i}\sigma_{i+1}\mapsto \sigma_{i}\sigma_{i+1}\sigma_{i}$; the other cases are handled the same way. 

Consider the neighborhoods $N_v,N'_v$ of the face $v$ in $\Gamma_{\bf i}$ and $\Gamma_{\bf i'}$ respectively illustrated in Figure~\ref{fig:br-mut1}.
There is a natural weight-preserving bijection between the set of paths in $\Gamma_{\bf i}$ which do not meet $N_v$ with the set of paths in $\Gamma_{\bf i'}$ which do not meet $N_v'$. Then for any path $p$ in $\Gamma_{\bf i}$ which does not enter $N_v$, we have $(u_p,v)=0$, so that $\varphi(v)\Eb({u}_p)\varphi(v)^*=\Eb({u}_p)$. So the contributions to the partition function $Z_{\bf i,\ell,a,b}$ from paths not meeting $N_v$ is conjugated under the unitary $\varphi(v)$ to the contribution to $Z_{\mathbf{i}',\mu_v(\ell),a,b}$ from paths not meeting $N'_v$. 

The same argument applies to show that the contribution to $Z_{\mathbf{i},\ell,a,b}$ from paths which enter and exit $N_v$ along the same horizontal line $r\in {i-1,i,i+1}$ is conjugated to the corresponding contribution to $Z_{\mathbf{i}',\mu_v(\ell),a,b}$, and similarly for the contributions from paths entering $N_v$ along line $i-1$ and exiting along line $i+1$. Indeed, inspecting Figure~\ref{fig:br-mut1} it is clear that for all such paths we again have $(u_p,v)=0$, and that the natural bijection between the set of such paths in $\Gamma_{\mbi}$ and the corresponding set $\Gamma_{\bf i'}$ is weight preserving. 

More generally, for $r,s\in \{i-1,i,i+1\}$ let $\Pi(\Gamma_{\bf i},v,r,s)$ be the set of paths in $\Gamma_{\bf i}$ which enter $N_v$ through horizontal line $r $ and exit through horizontal line $s$, and define similarly $\Pi(\Gamma_{\bf i'},v,r,s)$. Write
$$
Z_{\Gamma_{\bf i}}(r,s)  = \boxplus_{p\in \Pi(\Gamma_{\bf i},v,r,s)}\Eb(u_p)
$$
for the corresponding contribution to the partition function $Z_{\mathbf{i},\ell}$. We claim that
\begin{align}
\label{eq:conj1}
\varphi(v)Z_{\Gamma_{\bf i}}(i,i+1) \varphi(v)^* = Z_{\Gamma_{\bf i'}}(i,i+1).
\end{align}
To see this, we group the paths in $\Pi(\Gamma,v,i,i+1)$ into pairs $(p',p'')$ where the paths $p'$ and $p''$ are identical outside of the neighborhood $N_v$, and $p''$ reaches line $i+1$ before $p'$. Then we have $u_{p'}=u_{p''}+v$, $(v,u_{p'})=-1$. 
So by Proposition~\ref{PropSkewComm} we have
$$
\varphi(v)(\Eb({u}_{p'})\boxplus \Eb({u}_{p''}))\varphi(v)^* = \Eb({u}_{p''}).
$$
On the other hand, the pair of paths $p',p''$ in $\Gamma_{\bf i}$ corresponds to a single path $p$ in $\Gamma_{\bf i'}$ whose adjusted weight is exactly ${u}_{p''}$, and so we deduce~\eqref{eq:conj1} holds. 

In the same way each path $\pi$ in $\Pi(\Gamma_{\bf i},v,i-1,i)$ corresponds to a pair of paths $\pi',\pi''$ in $\Gamma_{\bf i'}$ identical outside of $N_v$, and this time we have $(v,u_{\pi})=1$. So again by Proposition~\ref{PropSkewComm} we have
$$
\varphi(v)Z_{\Gamma_{\bf i}}(i-1,i) \varphi(v)^* = Z_{\Gamma_{\bf i'}}(i-1,i),
$$ 
and this completes the proof of the Lemma.
\end{proof}

\begin{cor}[Serre relations]
\label{cor-serre}
The pair of positive operators $(\stgE_{i-1,i+1}, \stgE_{i-1,i})$  is $\hbar/2$-commuting in the sense of~\eqref{EqSkewCommDef}, as is $(\stgE_{i,i+1}, \stgE_{i-1,i+1})$.
\end{cor}
\begin{proof}
This follows since $\stgE_{i-1,i}$ and $\stgE_{i-1,i+1}$ can both be computed as appropriate partition functions for the graph $\Gamma_{\mathbb{E}_{i-1,i+1}}$, and the latter can be brought by a sequence of braid and Demazure moves to one corresponding to the braid word  $\sigma_2\sigma_1\sigma_2$. By Lemma~\ref{lem:Zmut}, conjugating by the corresponding unitary brings $\stgE_{i-1,i}, \stgE_{i-1,i+1}$ to positive operators of the form $\Eb(u),\Eb(v)$ with $(u,v)=1/2$, so the Corollary follows.
\end{proof}

Now suppose $\Gamma_1,\Gamma_2$ correspond to $B_m$-braid words with Demazure product admitting prefix $\sigma_1\cdots \sigma_{m-1}$, and are equipped with $V_1$- and $V_2$-colorings $\ell_1,\ell_2$ for skew-symmetric spaces $V_1,V_2$.  Then we can consider a partition function 
$$
Z_{(\Gamma_{1},\ell_1)\times(\Gamma_{2},\ell_2)} =  \boxplus_{(p_1,p_2)}\Eb({w}_{p_1}\oplus {w}_{p_2})
$$ of pairs of paths $(p_1,p_2)$ where $p_1$ is a path in $\Gamma_{1}$ entering along the top horizontal line (which we number 1) and exiting along the bottom horizontal line $m$, and $p_2$ a path in $\Gamma_{2}$ entering at 1 and exiting at $m$. 
To each such path $p_1$ we can associate a sequence of integers $(n_1,\ldots n_{m-1})$ where $n_i$ is the number of faces lying between horizontal lines $i,i+1$ and to the right of $p_1$. The lexicographic ordering on the set of all integer sequences then defines a total order on the set of all paths $p_1$ in $\Gamma_{1}$ from 1 to $m$. Equivalently, $q_1<p_1$ in this total order if the first time that $q_1,p_1$ diverge, it is the path $q_1$ which goes down. We equip the set of all paths from 1 to $m_2$ in $\Gamma_2$ with a total order in the same way, and put the lexical total order on the set of all pairs $(p_1,p_2)$ where we order in the first factor, then break ties with the second.

Recall now the quantum exponential function $F = F_{\hbar}$ from \eqref{EqQuantExp2}. 
\begin{lemma}
\label{lem:dilog-product}
Suppose $\Gamma_1,\Gamma_2$ are graphs corresponding to $B_m$-braid words $\mathbf{i}_1,\mathbf{i}_2$,  equipped with $V_1$- and $V_2$-colorings for skew-symmetric spaces $V_1,V_2$. Then if $\mathbf{i}_1$ and $\mathbf{i}_2$ can be brought by some sequence of braid and Demazure moves to words of the form $\sigma_1\sigma_2\cdots \sigma_{m-1}\mathbf{i}'$ where $\mathbf{i}'$ does not contain $\sigma_{m-1}$, we have
\begin{align}
\label{eq:p-exp}
\overline{F}(Z_{(\Gamma_{1},\ell_1)\times (\Gamma_{2},\ell_2)}) = \overset{\longleftarrow}{\prod_{(p_1,p_2)}}\varphi(\mathrm{w}_{p_1}\oplus \mathrm{w}_{p_2}),
\end{align}
 where the product is taken over all pairs of paths $(p_1,p_2)$, ordered with respect to the lexical total order described above so that the smallest pair is the right-most factor in the product.
\end{lemma} 
\begin{proof}
We argue by induction using the following observation: if $v$ is a mutable face, formula ~\eqref{eq:p-exp} holds for $(\Gamma_1,\Gamma_2)$ if and only if it holds for $(\mu_v(\Gamma_1),\Gamma_2)$. Let us show this in the case that $\mu_v$ corresponds to the braid move~$\sigma_{i}\sigma_{i+1}\sigma_{i}\mapsto\sigma_{i+1}\sigma_{i}\sigma_{i+1}$ from Figure~\ref{fig:br-mut2}. By Lemma~\ref{lem:Zmut}, we have
\begin{align}
\label{eq:FZ}
\overline{F}(Z_{\Gamma_1,\Gamma_2}) = \varphi(v)^*\overline{F}(Z_{\mu_v\Gamma_1,\Gamma_2})\varphi(v).
\end{align}
Now observe that the well-ordered list of paths in $\Gamma_1$ can be obtained from the corresponding list of paths in $\mu_v\Gamma_1$ by the following three-step procedure: first, for each path $\pi$ in $\mu_v\Gamma_1$ entering $\mu_v(N)\subseteq\mu_v\Gamma_1$ along line $i-1$ and exiting along line $i$, replace $\pi$ by a consecutive pair of paths $(\pi',\pi'')$ in $\Gamma_1$ which are identical to $\pi$ outside of $N$, such that the path $\pi'$ reaches line $i$ in $N\subseteq\Gamma_1$ before $\pi''$. Note that $\pi''$ is indeed the successor to $\pi'$ in the lexical ordering for $\Gamma_1$. The second step consists of identifying in the resulting list  each pair $(p',p'')$ of paths in $\mu_v\Gamma_1$ which enter $\mu_v(N)$ along line $i$ and exit along line $i+1$, and are identical outside of $\mu_v(N)$, so that $p''$ is the successor to $p'$ in the $\mu_v\Gamma_1$ lexical order. We then replace the consecutive pair $(p',p'')$ by the single path $p$ in $\Gamma_1$ identical to $p',p''$ outside of $N$ and which enters $\mu_v(N)$ through line $i$ and exits through line $i+1$. Finally, all remaining items in the list which are still paths in $\mu_v\Gamma_1$ are paths which do not intersect $N$, and we replace each of these by the corresponding identical path in $\Gamma_1$. From this description, it follows immediately from formula~\eqref{eq:FZ} and the Kashaev pentagon identity that if~\eqref{eq:p-exp} holds for $\mu_v\Gamma_1,\Gamma_2$ then it also holds for $\Gamma_1,\Gamma_2$. The same argument applies in the case that $\mu_v$ corresponds to the reverse braid move or a Demazure move.

Now let us perform such a sequence of moves to bring $\mathbf{i}_1$ to a word of the form $\sigma_1\sigma_2\cdots \sigma_{m-1}\mathbf{i}'$ where $\mathbf{i}'$ does not contain $\sigma_{m-1}$. Then in the resulting graph $\Gamma_1^\bullet$ there is a unique path starting from the boundary vertex on line $1$ ending at the boundary vertex on line $m$. For such a graph $\Gamma_1^\bullet$, the same argument used above shows that for any mutable face $v$ in $\Gamma_2$, ~\eqref{eq:p-exp} holds for $\Gamma_1^\bullet,\Gamma_2$ if and only if it holds for $\Gamma_1^\bullet,\mu_v\Gamma_2$.
So if we perform a sequence of braid and Demazure moves bringing $\Gamma_2$ to a graph $\Gamma_2^\bullet$ of the same form, there is a single pair of paths in $\Gamma_1^\bullet,\Gamma^\bullet_2$ and~\eqref{eq:p-exp} becomes a tautology. This completes the proof of the Lemma. 
\end{proof}

\begin{remark}
The same proof shows $\overline{F}(Z_{(\Gamma_{1},\ell_1)\times (\Gamma_{2},\ell_2)})$ can also be expanded as a differently ordered product of the same factors where we first order using $\Gamma_2$, then break ties using $\Gamma_1$.
\end{remark}

Since form sums distribute over tensor product, $\stgE_{rs}\otimes\stgF_{rs}$ can be computed as a partition function $Z_{\Gamma_{\mathbb{E}_{rs}}\times\Gamma_{\mathbb{F}_{rs}}}$ of pairs of paths $(p_E,p_F)$ where $p_E$ is a path in $\Gamma_{\mathbb{E}_{rs}}$ entering along the top horizontal line (which we number 1) and exiting along the bottom horizontal line $s-r+1$, and $p_F$ a path in $\Gamma_{\mathbb{F}_{rs}}$ entering at 1 and exiting at $s-r+1$. 

\subsection{\texorpdfstring{$\Rc$-matrix}{R-matrix} and the braided Fock--Goncharov flip}
\label{SecFG}

Consider the unitary
\beq
\label{eq:R-init2}
\Rc = \prod_{a=1}^{\substack{n \\[-2pt] \longrightarrow}} \prod_{b=a+1}^{\substack{n+1 \\[-2pt] \longrightarrow}} \overline{F}_\hbar(\stgF_{a,b} \otimes \stgE_{a,b}).
\eeq

For $1\leq a<b\leq N$ let $\Gamma_{{\mathbb{F}^{[1]}_{a,b}}}$ be the graph obtained from $\Gamma_{{\mathbb{F}_{a,b}}}$ by deleting all vertical edges connecting the top two horizontal lines in $\Gamma_{{\mathbb{F}_{a,b}}}$ except for the left-most such edge, and write $\Gamma_{{\mathbb{F}^{>1}_{a,b}}}$ for the graph obtained from 
$\Gamma_{{\mathbb{F}_{a,b}}}$ by deleting only the left-most such edge.  By the construction~\eqref{eq:face-merge}, the graphs $\Gamma_{{\mathbb{F}^{[1]}_{a,b}}}$ and $\Gamma_{{\mathbb{F}^{>1}_{a,b}}}$  inherit $V$-labellings from that of $\Gamma_{{\mathbb{F}_{a,b}}}$, and we set
$$
\stgF^{[1]}_{a,b} = Z({\Gamma_{{\mathbb{F}^{[1]}_{a,b}}}}),\qquad \stgF^{>1}_{a,b}= Z({\Gamma_{{\mathbb{F}^{>1}_{a,b}}}}),
$$
so that
$$
Z({\Gamma_{{\mathbb{F}_{a,b}}}})= Z({\Gamma_{{\mathbb{F}^{[1]}_{a,b}}}})\boxplus Z({\Gamma_{{\mathbb{F}^{>1}_{a,b}}}}).
$$
Equivalently, we can think of $\stgF^{[1]}_{a,b} $ and $\stgF^{>1}_{a,b} $ as `conditional partition functions' in the original graph $\Gamma_{{\mathbb{F}_{a,b}}}$ given by restricting the sum~\eqref{eq:Zdef} to only paths travsersing (respectively not traversing) the first vertical arrow between horizontal lines $a,a+1$.

Note that when $a=1$ there is a single vertical edge in the top row of $\Gamma_{\mathbb{F}_{a,b}}$ so that for any $b$ we have
$$
Z({\Gamma_{{\mathbb{F}^{>1}_{1,b}}}})=0.
$$
Similarly for $1\leq a<b\leq N$ and $1\leq j\leq N-a$, let $\Gamma_{{\mathbb{E}^{[j]}_{a,b}}}$ be the graph obtained from $\Gamma_{{\mathbb{E}_{a,b}}}$  by deleting all vertical edges connecting the top two horizontal lines in $\Gamma_{{\mathbb{E}_{a,b}}}$ except for the $j$-th such edge encountered starting from the left. By the construction~\eqref{eq:face-merge}, $\Gamma_{{\mathbb{E}^{[j]}_{a,b}}}$ inherits a $V$-labelling from that of $\Gamma_{{\mathbb{E}_{a,b}}}$, and we set
$$
\stgE^{[j]}_{a,b} = Z({\Gamma_{{\mathbb{E}^{[j]}_{a,b}}}}).
$$
Since every path summed over in $Z({\Gamma_{{\mathbb{E}_{a,b}}}})$ traverses exactly one of the vertical edges mentioned above we have
$$
\stgE_{a,b} = Z(\Gamma_{{\mathbb{E}_{a,b}}}) = \boxsum_{j=1}^{N-a}Z(\Gamma_{{\mathbb{E}^{[j]}_{a,b}}}).
$$
In other words, the partition function $Z(\Gamma_{{\mathbb{E}^{[j]}_{a,b}}})$ is the contribution to $Z(\Gamma_{{\mathbb{E}_{a,b}}})$ from paths taking the $j$-th vertical edge between the top pair of vertical lines.

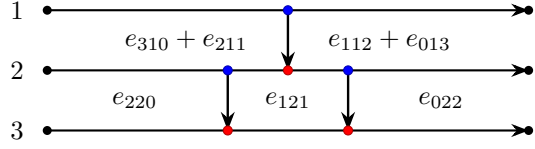
\begin{figure}
\begin{tikzpicture}[>=Stealth, line cap=round, x=8mm, y=8mm]
  \tikzset{
    wire/.style   ={line width=0.8pt},
    post/.style   ={line width=0.9pt,->},
    reddot/.style ={circle,draw=red!70!black,fill=red,inner sep=1.2pt},
    bluedot/.style={circle,draw=blue!70!black,fill=blue,inner sep=1.2pt},
    blackdot/.style={circle,fill=black,inner sep=1.2pt},
    elab/.style   ={font=\small}
  }
  \draw[wire,->] (  -1,1) -- ( 7,1);
  \draw[wire,->] (  -1,2) -- ( 7,2);
  \draw[wire,->] (  -1,3) -- ( 7,3);

  \foreach \x in {3}{
    \draw[post] (\x,3) -- (\x,2);
    \node[reddot] at (\x,2) {};
    \node[bluedot]  at (\x,3) {};
  }
  \foreach \x in {2,4}{
    \draw[post] (\x,2) -- (\x,1);
    \node[reddot] at (\x,1) {};
    \node[bluedot]  at (\x,2) {};
  }

   \node[blackdot] at (-1,1) {};
   \node[blackdot] at (-1,2) {};
   \node[blackdot] at (-1,3) {};
   \node[blackdot] at (7,1) {};
   \node[blackdot] at (7,2) {};
   \node[blackdot] at (7,3) {};
  \node[elab] at (-1.5,1) {$3$};
  \node[elab] at (-1.5,2) {$2$};
  \node[elab] at (-1.5,3) {$1$};

  \node[elab, anchor=east] at (2.5,2.5) {$e_{310}+e_{211}$};
    \node[elab, anchor=west] at (3.5,2.5) {$e_{112}+e_{013}$};

  \node[elab, anchor=east] at (1,1.5) {$e_{220}$};
  \node[elab]              at (3.0,1.5) {$e_{121}$};
  \node[elab, anchor=west] at (5.0,1.5) {$e_{022}$};


\end{tikzpicture}
\caption{The labelled graph $\Gamma_{{\mathbb{E}^{[2]}_{1,3}}}$ for $N=4$.}
\label{fig:conditional-E-graph}
\end{figure}

We illustrate the setup for $N=4, (a,b)=(1,3), j=2$ in Figure~\ref{fig:conditional-E-graph}. Note that for $j_1<j_2$, the paths in the sum defining $Z(\Gamma_{{\mathbb{E}_{a,b}}})$ which contribute to $Z(\Gamma_{{\mathbb{E}^{[j_1]}_{a,b}}})$ are less than those contributing to $Z(\Gamma_{{\mathbb{E}^{[j_2]}_{a,b}}})$ in the lexical order from Lemma~\ref{lem:dilog-product}. Similarly, the paths contributing to $Z(\Gamma_{{\mathbb{F}^{[1]}_{a,b}}})$ are less than those contributing to $Z(\Gamma_{{\mathbb{F}^{>1}_{a,b}}})$. It follows from Lemma~\ref{lem:dilog-product} that we have a factorization
\begin{align*}
\overline{F}_\hbar(\stgF_{a,b}\otimes \stgE_{a,b}) &= \overline{F}_\hbar( \stgF^{>1}_{a,b}\otimes\stgE_{a,b})~\overline{F}_\hbar( \stgF^{[1]}_{a,b}\otimes\stgE_{a,b}),
\end{align*}
so that
\beq
\label{eq:R-init3}
\Rc = \prod_{a=1}^{\substack{n \\[-2pt] \longrightarrow}} \prod_{b=a+1}^{\substack{n+1 \\[-2pt] \longrightarrow}} \overline{F}_\hbar( \stgF^{>1}_{a,b}\otimes \stgE_{a,b})~\overline{F}_\hbar( \stgF^{[1]}_{a,b}\otimes \stgE_{a,b}).
\eeq

\begin{lemma}
We have the following collection of $\hbar/2$-commuting pairs of positive operators:
\begin{align}
\label{eq:pairs1}
(\stgF_{a,b}^{[1]},\stgF_{a,c}^{>1}),\quad (\stgE_{a,c},\stgE_{a,b}), \quad a<b<c,
\end{align}
\begin{align}
\label{eq:pairs2}
(\stgF_{a,b}^{[1]},\stgF_{a,a+1}^{[1]}), \quad (\stgE_{a,a+1}^{[j]},\stgE_{a,b}^{[k]}), \quad b>a+1, \quad j<k.
\end{align}
\begin{align}
\label{eq:pairs3}
(\stgF_{a,a+1}^{[1]},\stgF_{a+1,b}^{>1}), \quad (\stgE_{a,a+1}^{[j+1]},\stgE_{a+1,b}^{>j}), \quad (\stgE_{a+1,b}^{[j]}, \stgE_{a,a+1}^{[k]}), \quad b>a+1, \quad j<k.
\end{align}
The following pairs of positive operators strongly commute:
\begin{align}
\label{eq:comm-pairs}
(\stgF_{a,c}^{[1]},\stgF_{a+1,b}^{>1}),\quad (\stgE_{a,c},\stgE_{a+1,b}), \quad a<b<c.
\end{align}

\end{lemma}
\begin{proof}
We treat case of the pair $(\stgF_{a,b}^{[1]},\stgF_{a,c}^{>1})$, with the others being handled using the same technique. The idea is to use Lemma~\ref{lem:Zmut} to construct a unitary which will conjugate the operators  $(\stgF_{a,b}^{[1]},\stgF_{a,c}^{>1})$ to a standard $\hbar/2$ commuting pair. Consider the piece of the graph $\Gamma_{\mathbb{F}}$ lying between horizontal lines $a$ and $c$. The corresponding braid word takes the form $\beta = \sigma_{c-a}\cdots \sigma_2 \sigma_1\beta'$, and by a sequence of braid and Demazure moves we can transform $\beta'$ to the word $\beta''=(\sigma_1\sigma_2\cdots\sigma_{c-a})\cdots \sigma_1\sigma_2\sigma_1$. The conjugate of $\stgF_{a,c}^{>1}$ under the associated unitary is computed as a sum over a single path in the labelled graph associated to $\tilde\beta=\sigma_{c-a}\cdots \sigma_2 \sigma_1\beta''$. Moreover, observe that deleting all but the first two occurrences of $\sigma_1$ in $\tilde\beta$ to produce a word $\widehat{\beta}$ has no effect on the conditional partition functions computing the unitary conjugates of $(\stgF_{a,b}^{[1]},\stgF_{a,c}^{>1})$. But by a further sequence of braid moves we can bring $\widehat{\beta}$ to the word $\sigma_{c-a}\cdots \sigma_2 \sigma_1 \sigma_1\beta'''$ where $\beta''' = (\sigma_{c-a}\cdots \sigma_3\sigma_2)\cdots \sigma_3\sigma_2\sigma_3$. In this latter graph the conditional partition functions computing the unitary conjugates of both $(\stgF_{a,b}^{[1]},\stgF_{a,c}^{>1})$ both consist of contributions from a single path, and the weights of these paths $\hbar/2$-commute as desired. 
\end{proof}

\begin{cor}
\label{lem:pent2}
For $b>a+1$, the positive operators $(\stgF_{a+1,b}^{>1}\otimes\stgE_{a+1,b}^{>j},\stgF_{a,a+1}^{[1]}\otimes\stgE_{a,a+1}^{[j+1]})$ form an $\hbar$-commuting pair, and we have a pentagon identity
\begin{multline*}
\overline F_\hbar(\stgF^{[1]}_{a,a+1} \otimes \stgE^{[j+1]}_{a,a+1})\overline{F}_\hbar( \stgF^{[1]}_{a,b}\otimes\stgE^{[j+1]}_{a,b})\overline{F}_\hbar( \stgF^{>1}_{a+1,b}\otimes\stgE^{>j}_{a+1,b}) = \\ = \overline{F}_\hbar( \stgF^{>1}_{a+1,b}\otimes\stgE^{>j}_{a+1,b})\overline F_\hbar(\stgF^{[1]}_{a,a+1} \otimes \stgE^{[j+1]}_{a,a+1}).
$$
\end{multline*}
\end{cor}

From~\eqref{eq:pairs1} and~\eqref{eq:comm-pairs} we see that the following pairs of positive operators strongly commute:
$$
 [\stgF^{>1}_{a,c}\otimes\stgE_{a,c},\stgF^{[1]}_{a,b}\otimes\stgE_{a,b}]=0=[\stgF^{>1}_{a+1,b}\otimes\stgE_{a+1,b},\stgF^{[1]}_{a,c}\otimes \stgE_{a,c}], \quad a<b<c.
$$
This observation allows us to re-order the factors in~\eqref{eq:R-init3} as
\beq
\label{eq:R-init4}
\Rc = \prod_{a=1}^{\substack{n \\[-2pt] \longrightarrow}} \hc{ \overline{F}_\hbar( \stgF^{[1]}_{a,a+1} \otimes \stgE_{a,a+1})\prod_{b=a+2}^{\substack{n+1 \\[-2pt] \longrightarrow}}\overline{F}_\hbar( \stgF^{[1]}_{a,b}\otimes\stgE_{a,b}) ~\overline{F}_\hbar( \stgF^{>1}_{a+1,b}\otimes\stgE_{a+1,b})}.
\eeq

\begin{lemma}
\label{lem:1}
For $b>a+1$ we have
\begin{align}
\label{eq:lem-1}
&\overline F_\hbar(\stgF^{[1]}_{a,a+1} \otimes \stgE_{a,a+1}) \overline{F}_\hbar( \stgF^{[1]}_{a,b}\otimes\stgE_{a,b}) ~\overline{F}_\hbar( \stgF^{>1}_{a+1,b}\otimes\stgE_{a+1,b})\\
 &=\overline{F}_\hbar( \stgF^{>1}_{a+1,b}\otimes\stgE_{a+1,b}) \overline F_\hbar(\stgF^{[1]}_{a;a+1} \otimes \stgE_{a,a+1}).
\end{align}
\end{lemma}

\begin{proof}
By Lemma~\ref{lem:dilog-product}, we can split
\begin{align}
\label{eq:leftsplit}
\overline F_\hbar(\stgF^{[1]}_{a;a+1} \otimes \stgE_{a,a+1}) = \prod_{j=1}^{\substack{N-a \\[-2pt] \longleftarrow}} \overline F_\hbar(\stgF^{[1]}_{a,a+1} \otimes \stgE^{[j]}_{a,a+1})
\end{align}
and 
\begin{align}
\label{eq:rightsplit}
\overline{F}_\hbar( \stgF^{[1]}_{a,b}\otimes\stgE_{a,b})= \prod_{j=1}^{\substack{N-b+1 \\[-2pt] \longleftarrow}}\overline{F}_\hbar( \stgF^{[1]}_{a,b}\otimes\stgE^{[j]}_{a,b}).
\end{align}
Using~\eqref{eq:pairs2} to move quantum exponential factors from~\eqref{eq:leftsplit} to the right over those in~\eqref{eq:rightsplit} with which they commute, we rewrite the left-hand side of~\eqref{eq:lem-1} as
\beq
\label{eq:lem-1-2}
\prod_{j=N-b+2}^{\substack{N-a \\[-2pt] \longleftarrow}} \overline F_\hbar(\stgF^{[1]}_{a,a+1} \otimes \stgE^{[j]}_{a,a+1}) \!\!\!\prod\limits_{j=1}^{\substack{N-b+1 \\[-2pt] \longleftarrow}}\hc{\overline F_\hbar(\stgF^{[1]}_{a,a+1} \otimes \stgE^{[j]}_{a,a+1})\overline{F}_\hbar( \stgF^{[1]}_{a,b}\otimes\stgE^{[j]}_{a,b})}\overline{F}_\hbar( \stgF^{>1}_{a+1,b}\otimes\stgE_{a+1,b}).
\eeq
Now we use Lemma~\ref{lem:pent2} to make the replacement
\begin{multline*}
\overline F_\hbar(\stgF^{[1]}_{a,a+1} \otimes \stgE^{[1]}_{a,a+1})\overline{F}_\hbar( \stgF^{[1]}_{a,b}\otimes\stgE^{[1]}_{a,b})\overline{F}_\hbar( \stgF^{>1}_{a+1,b}\otimes\stgE_{a+1,b})\\=\overline{F}_\hbar( \stgF^{>1}_{a+1,b}\otimes\stgE_{a+1,b})\overline F_\hbar(\stgF^{[1]}_{a,a+1} \otimes \stgE^{[1]}_{a,a+1}),
\end{multline*}
and then use Lemma~\ref{lem:dilog-product} to split
$$
\overline{F}_\hbar( \stgF^{>1}_{a+1,b}\otimes\stgE_{a+1,b}) = \overline{F}_\hbar( \stgF^{>1}_{a+1,b}\otimes\stgE^{>1}_{a+1,b})\overline{F}_\hbar( \stgF^{>1}_{a+1,b}\otimes\stgE^{[1]}_{a+1,b})
$$
At this point, we can apply Lemma~\ref{lem:pent2} again to replace
\begin{multline*}
\overline F_\hbar(\stgF^{[1]}_{a,a+1} \otimes \stgE^{[2]}_{a,a+1})\overline{F}_\hbar( \stgF^{[1]}_{a,b}\otimes\stgE^{[2]}_{a,b})\overline{F}_\hbar( \stgF^{>1}_{a+1,b}\otimes\stgE^{>1}_{a+1,b})\\=\overline{F}_\hbar( \stgF^{>1}_{a+1,b}\otimes\stgE^{>1}_{a+1,b})\overline F_\hbar(\stgF^{[1]}_{a,a+1} \otimes \stgE^{[2]}_{a,a+1}),
\end{multline*}
and then split 
$$
\overline{F}_\hbar( \stgF^{>1}_{a+1,b}\otimes\stgE^{>1}_{a+1,b}) = \overline{F}_\hbar( \stgF^{>1}_{a+1,b}\otimes\stgE^{>2}_{a+1,b})\overline{F}_\hbar( \stgF^{>1}_{a+1,b}\otimes\stgE^{[2]}_{a+1,b}).
$$
Continuing this way, we rewrite the left-hand-side of~\eqref{eq:lem-1} as
$$
\prod_{j=N-b+2}^{\substack{N-a \\[-2pt] \longleftarrow}} \overline F_\hbar(\stgF^{[1]}_{a,a+1} \otimes \stgE^{[j]}_{a,a+1}) \cdot \prod\limits_{j=1}^{\substack{N-b+1 \\[-2pt] \longleftarrow}}\hc{\overline F_\hbar(\stgF^{>1}_{a+1,b} \otimes \stgE^{[j]}_{a+1,b})\overline F_\hbar(\stgF^{[1]}_{a,a+1} \otimes \stgE^{[j]}_{a,a+1})}.
$$
By~\eqref{eq:pairs3}, we can commute the factors $\overline F_\hbar(\stgF^{>1}_{a+1,b} \otimes \stgE^{[j]}_{a+1,b})$ all the way to the left, and regrouping the quantum exponentials we arrive at the right-hand-side of~\eqref{eq:lem-1}.
\end{proof}

With this lemma it becomes easy to prove:
\begin{prop}
\label{prop:R-factor}
There is an identity of unitaries
\beq
\label{eq:R-fin}
\Rc = \prod_{a=1}^{\substack{n \\[-2pt] \longrightarrow}} \prod_{b=a+1}^{\substack{n+1 \\[-2pt] \longrightarrow}} \overline{F}_\hbar(\stgF_{a,b} \otimes \stgE_{a,b})=\prod_{k=1}^{\substack{n \\[-2pt] \longleftarrow}} \prod_{a=k}^{\substack{n \\[-2pt] \longrightarrow}} \overline F_\hbar\hr{\stgF_{a;a+1}^{[k]} \otimes \stgE_{a,a+1}}.
\eeq
\end{prop}

\begin{proof}
 For $n=1$ the two factorizations are just identical, while the statement of the Proposition for rank $n-1$ implies the following equality in the rank $n$ setting:
$$
\prod_{a=2}^{\substack{n \\[-2pt] \longrightarrow}} \prod_{b=a+1}^{\substack{n +1\\[-2pt] \longrightarrow}} \overline F_\hbar\hr{\stgF^{>1}_{a,b} \otimes \stgE_{a,b}} = \prod_{k=2}^{\substack{n \\[-2pt] \longleftarrow}} \prod_{a=k}^{\substack{n \\[-2pt] \longrightarrow}}\overline F_\hbar\hr{\stgF_{a;a+1}^{[k]} \otimes \stgE_{a,a+1}}.
$$
So the proposition will follow by induction provided we can show that
\beq
\label{eq:ind-step}
\Rc = \hc{\prod_{a=2}^{\substack{n \\[-2pt] \longrightarrow}} \prod_{b=a+1}^{\substack{n+1 \\[-2pt] \longrightarrow}} \overline F_\hbar\hr{\stgF^{>1}_{a,b} \otimes \stgE_{a,b}} } \cdot \prod_{a=1}^{\substack{n \\[-2pt] \longrightarrow}} \overline F_\hbar\hr{\stgF_{a;a+1}^{[1]} \otimes \stgE_{a,a+1}}.
\eeq
To this end, by iterative application of Lemma~\ref{lem:1} to formula~\eqref{eq:R-init4} we have
\begin{align}
\label{eq:fact4}
\Rc = \prod_{a=1}^{\substack{n \\[-2pt] \longrightarrow}} \hc{\prod_{b=a+1}^{\substack{n \\[-2pt] \longrightarrow}} \overline F_\hbar\hr{\stgF_{a+1,b+1}^{>1} \otimes \stgE_{a+1,b+1}}} \overline F_\hbar\hr{\stgF_{a,a+1}^{[1]} \otimes \stgE_{a,a+1}}.
\end{align}
Since $\stgF_{a,a+1}^{[1]}\otimes \stgE_{a,a+1}$ strongly commutes with $\stgF_{c,d}^{>1}\otimes \stgE_{c,d}$ for any $d>c>a+1$, we 
can freely commute the factors $F_\hbar\hr{\stgF_{a,a+1}^{[1]} \otimes \stgE_{a,a+1}}$ in~\eqref{eq:fact4} rightwards over all such factors $F_\hbar\hr{\stgF_{c,d}^{>1}\otimes \stgE_{c,d}}$ to  arrive at~\eqref{eq:ind-step}, and this finishes the proof.
\end{proof}

The previous expression for the unitary $\mathcal{R}$ is reminiscent of the  familiar one in the algebraic setup (see e.g.\ \cite[Section 8.3.2]{KS97}).  We will now derive an alternative form of this unitary, identifying it with the quantum mutation operator associated to the Fock--Goncharov flip from cluster theory.

\begin{defn}
For any $1 \le i \le r \le b \le n$, we define on $\Hc \otimes \Hc$ the unitaries
\begin{equation}\label{DefBbri}
B^{b,r,i} 
 =  \varphi\hr{\nee_{b-i+1,b-r}\oplus \seee_{n-b+i,n-b}}.
\end{equation}
\end{defn}
\begin{lemma}\label{LemOrdering}
One has
\beq
\label{eq:comm-0}
\hs{B^{b,r,i}, B^{b',r,i'}} = 0, \qquad \forall 1 \le i,i' \le r \le b,b' \le n.
\eeq
\end{lemma}
\begin{proof}
Writing 
\[
s= b-i+1,\qquad s' = b'-i'+1,
\]
it follows immediately from \eqref{EqSameNil} that 
\begin{eqnarray*}
&& \hspace{-2cm} (\nee_{b-i+1,b-r}\oplus \seee_{n-b+i,n-b},\nee_{b'-i'+1,b'-r}\oplus \seee_{n-b'+i',n-b'})\\ 
&&= (\nee_{s,b-r},\nee_{s',b'-r}) + (\seee_{n-s+1,n-b},\seee_{n-s'+1,n-b'}) \\
&& =  (\nee_{s,b-r},\nee_{s',b'-r}) + (\nee_{n-s+1,n-b},\nee_{n-s'+1,n-b'}) \\
&& = 0.
\end{eqnarray*}
The lemma now follows from \eqref{EqCommComm}. 
\end{proof}

\begin{defn}
The \emph{braided Fock--Goncharov flip} $\Fc = \Fc_{\hbar} \in \Bc(\Hc\otimes \Hc)$ is defined as the unitary 
\begin{eqnarray}
\label{eq:F-1or}
\Fc &=& \prod_{1 \le r \le n}^{\longra} \prod_{r \le b \le n}^{\longra} \prod_{1 \le i \le r}^{\longra} B^{b,r,i}\\ 
&=& \label{eq:F-1or2} \prod_{1 \le r \le n}^{\longra} \prod_{1 \le s \le n}^{\longra} \underset{0\le (n-r)-k\le n-s}{\prod_{0 \le k \le s-1}^{\longra}}   \varphi\hr{\nee_{s,k}\oplus \seee_{N-s,(n-r)-k}}\\
&=& \label{eq:F-3or} \prod_{1 \le r \le n}^{\longra} \prod_{1 \le s \le n}^{\longra} \underset{0\le (n-r)-k\le n-s}{\prod_{0 \le k \le s-1}^{\longra}}   \varphi\hr{\nee_{N-s,(n-r)-k}\oplus \seee_{s,k}}.
\end{eqnarray}
\end{defn}
Note that by Lemma \ref{LemOrdering}, only the first product over $r$ in the above expressions must be ordered. 

\begin{remark}
   We note that $\Fc$ is nothing but a unitary implementing the quantum cluster transformation between cluster charts for the moduli space of framed $SL_N$ local systems corresponding to ideal triangulations related by a single diagonal flip, see~\cite[Section 10.3]{FG06}.
\end{remark}

\begin{lemma}
    The braided Fock--Goncharov flip $\mathcal{F}'$ coincides with the unitary $\mathcal{R}$ from~\eqref{eq:R-init2}.
\end{lemma}
\begin{proof}
We apply Lemma~\ref{lem:dilog-product} to decompose the factorization of $\mathcal{R}$ from the right-hand-side in Proposition~\ref{prop:R-factor} into a product of quantum dilogarithms of the monomials~\eqref{DefBbri}, and by Lemma~\ref{LemOrdering} the factors can be reordered to obtain the expression for $\mathcal{F}'$ in ~\eqref{eq:F-1or}.
\end{proof}

\begin{remark}
We stress that $\Fc$ is defined with respect to a fixed unitary $\hbar$-representation $\Hc = (\Hc,\pi)$ of $\wbtd{N}$. If we want to emphasize this, we will speak of a \emph{braided Fock--Goncharov flip on $\Hc$} or \emph{with respect to $(\Hc,\pi)$}. 
\end{remark}

We now introduce notations
\begin{align*}
B_{[12]}^{b,r,i} = B_{12}^{b,r,i}= 
\varphi\hr{\nee_{b-i+1,b-r} \oplus \seee_{n-b+i,n-b} \oplus 0}, \\
B_{[23]}^{b,r,i} = B_{23}^{b,r,i}= 
\varphi\hr{0 \oplus \nee_{b-i+1,b-r} \oplus \seee_{n-b+i,n-b}},
\end{align*}
and
$$
B_{[13]}^{b,r,i} = 
\varphi\hr{\nee_{b-i+1,b-r} \oplus \nee_{b-i+1,b-i+1} \oplus \seee_{n-b+i,n-b}},
$$
and we let $\Fc_{[12]},\Fc_{[23]}$ and $\Fc_{[13]}$ be their corresponding products in the order as in \eqref{eq:F-1or}.

\begin{theorem}\label{TheoFlip}
The following pentagon relation holds:
\beq
\label{EqBraidPent}
\Fc_{[23]} \Fc_{[12]} = \Fc_{[12]} \Fc_{[13]} \Fc_{[23]}.
\eeq
\end{theorem}

This result is well-known, see \cite{DGG16}, but for reader's convenience we present a direct algebraic proof in Appendix \ref{AppProofTheoFlip}.

\begin{remark}
One can show that $\Fc$ is a \emph{braided multiplicative unitary} as in \cite[Definition 3.2]{MRW17}. As we will not follow this approach further in this paper, we refrain from spelling out the details.
\end{remark} 


\subsection{Fock--Goncharov flip}

We resume the setting of Section \ref{SecFG}. We fix $\hbar\in \R^{\times}$ and a unitary $\hbar$-representation of $\wbtd{N}$ on a Hilbert space $\Hc$. 

By \eqref{EqVanishing} and \eqref{EqOrthogonalWeights}, the following definition is meaningful.

\begin{defn}
We define
\begin{equation}\label{EqGaussFG}
\Kc = \Kc_{\hbar} = \Gauss_{\hbar}\left(2\sum_{t=1}^n \nevarpi_t \otimes \seee_{N-t}\right) =  \Gauss_\hbar\left(2 \sum_{t=1}^n \nee_{t} \otimes \sevarpi_{N - t}\right). 
\end{equation}
\end{defn}

\begin{theorem}
The unitary
\begin{equation}\label{EqRealMU}
\Fct = \Fct_{\hbar} = \Kc\Fc
\end{equation}
is a multiplicative unitary, i.e.\ 
\begin{equation}\label{EqMUEq}
\Fct_{23}\Fct_{12} = \Fct_{12}\Fct_{13}\Fct_{23}.
\end{equation}
\end{theorem}
We refer to $\Fct$ as the \emph{Fock--Goncharov flip}. 
\begin{proof}
By \eqref{EqGausEqInv} and \eqref{EqDiffCarOpp}, we immediately get that the left hand side of \eqref{EqMUEq} equals
\[
\Fct_{23}\Fct_{12}  = \Kc_{23}\Kc_{12} \Fc_{23}\Fc_{12}. 
\]
On the other hand, for the factors on the right hand side of \eqref{EqMUEq} we find via \eqref{EqPairFundWeight} that
\begin{eqnarray*}
&& \hspace{-2cm} \varphi(\nee_{s,k}\oplus \seee_{N-s,(n-r)-k})_{12} \Kc_{13}\Kc_{23}\\  &&  = \Kc_{13}  \varphi(\nee_{s,k}\oplus \seee_{N-s,(n-r)-k}\oplus \seee_{N-s})\Kc_{23} \\
&& = \Kc_{13}\Kc_{[23]}\varphi(\nee_{s,k}\oplus \seee_{N-s,(n-r)-k})_{12},
\end{eqnarray*}
so $\Fc_{12}$ commutes with $\Kc_{13}\Kc_{23}$, while  from the last identity in \eqref{EqGaussFG} and \eqref{EqPairFundWeight} it follows that
\[
\varphi(\nee_{s,k}\oplus \seee_{N-s,(n-r)-k})_{[13]}\Kc_{23} = \Kc_{23} \varphi(\nee_{s,k}\oplus \nee_s \oplus \seee_{N-s,(n-r)-k}).
\]
Using the notation introduced above Theorem \ref{TheoFlip}, this gives that the right hand side of \eqref{EqMUEq} is 
\[
\Fct_{12}\Fct_{13}\Fct_{23} =  \Kc_{12}\Kc_{13}\Kc_{23} \Fc_{[12]}\Fc_{[13]}\Fc_{[23]},
\]
so since $\Fc_{[23]} = \Fc_{23}$ and $\Fc_{[12]} = \Fc_{12}$, the theorem follows from Theorem \ref{TheoFlip} once we show that
\begin{equation}\label{EqPentK}
\Kc_{23}\Kc_{12} = \Kc_{12}\Kc_{13}\Kc_{23}.
\end{equation}
But using \eqref{EqPairFundWeight} and \eqref{EqGausEq}  we have 
\[
\Kc (\eb(\seee_{s})\otimes 1)\Kc^* = \Eb(\seee_s\oplus \seee_s), 
\]
from which \eqref{EqPentK} directly follows.
\end{proof}

\begin{remark}
Note that if for $1\leq i\leq r \leq b \leq n$ we write 
\begin{equation}\label{DefCbri}
C^{b,r,i} = \varphi\hr{-\swe_{b-i+1,r-i}\oplus -\nwe_{n-b+i,i-1}},
\end{equation}
then by \eqref{EqGausEq}, \eqref{EqDiffOr} and \eqref{EqSameCar} we get the alternative formula 
\begin{equation}\label{EqRealMUAlt}
\Fct = \Fctt\Kc,
\end{equation}
where 
\begin{eqnarray}
\Fctt &=&  \prod_{1 \le r \le n}^{\longra} \prod_{r \le b \le n}^{\longra} \prod_{1 \le i \le r}^{\longra} C^{b,r,i}\\
&=& \label{eq:Fprime-1or2} \prod_{1 \le r \le n}^{\longra} \prod_{1\leq s \leq n}^{\longra} \underset{0\le r-k-1\le n-s}{\prod_{0 \le k \le s-1}^{\longra}}\varphi\hr{-\swe_{s,k}\oplus -\nwe_{N-s,r-k-1}}.
\end{eqnarray}  
\end{remark}

\section{Modularity of the Fock--Goncharov flip.}

Our goal in this section is to show that for suitable unitary $\hbar$-representations of $\wbtd{N}$, the Fock--Goncharov flip \eqref{EqRealMU} is \emph{modular}, in the sense of \cite[Definition 2.1]{SW01}.

\subsection{Modularity}
\begin{defn}\label{DefModular}
Let $\Wbb \in \Bc(\Hc\otimes \Hc)$ be a  unitary. We call $\Wbb$ \emph{modular} if there exist strictly positive operators $\Qb,\hat{\Qb}$ on $\Hc$ such that 
\begin{equation}\label{EqCommModStruct}
\Wbb^*(\hat{\Qb}\otimes \Qb)\Wbb = (\hat{\Qb}\otimes \Qb)
\end{equation}
and such that there exists a (necessarily unique) unitary operator $\wWbb \in \Bc(\overline{\Hc}\otimes \Hc)$ 
such that 
\begin{equation}\label{EqRightInvOp}
\langle \xi\otimes \eta',\Wbb(\eta\otimes \Qb\xi')\rangle = \langle \overline{\eta}\otimes \Qb\eta',\wWbb(\overline{\xi} \otimes \xi')\rangle,\qquad \forall \xi',\eta'\in \msD(\Qb),\xi,\eta\in \Hc.
\end{equation}
We then say that $\Wbb$ is \emph{modular with respect to the couple $(\Qb,\hat{\Qb})$}, and we call $(\Wbb,\Qb,\hat{\Qb},\wWbb)$ a \emph{modular datum}.
\end{defn}

\begin{remark}
\begin{enumerate}
\item Under the assumption of \eqref{EqCommModStruct}, it is easily seen that \eqref{EqRightInvOp} holds if and only if
\begin{equation}\label{EqLeftInvOp}
\langle \hat{\Qb}\xi\otimes \eta',\Wbb(\eta\otimes \xi')\rangle = \langle \overline{\hat{\Qb}\eta}\otimes \eta',\wWbb(\overline{\xi} \otimes \xi')\rangle,\qquad \forall \xi,\eta\in \msD(\hat{\Qb}),\xi',\eta'\in \Hc.
\end{equation}
\item The above definition is presented in a slightly different form than \cite[Definition 2.1]{SW01}, but the equality of the two definitions is shown in  \cite[Proposition 2.2]{SW01}.\footnote{We use in fact a slightly stronger version of that property when invoking the equivalence of conditions below, but this is immediately observed to hold from the proof of \cite[Proposition 2.2]{SW01}. We also take the opportunity to switch the inverse of the unbounded operator to the other side, which is as well simply a cosmetic operation. Finally, we formulate the condition for arbitrary unitaries on $\Hsp \otimes \Hsp$, instead of just multiplicative unitaries.}
\end{enumerate}
\end{remark}

\begin{remark}
One should think of $\hat{\Qb}$ as an auxiliary object implementing `the positive root of the antipode squared' on a quantum group (and $\Qb$ as the corresponding object for the dual quantum group). Indeed, assume for the moment that one has a  Hopf $*$-algebra $(M,\Delta)$ represented faithfully (in a $*$-preserving way) by bounded operators on a Hilbert space $\Hc$:
\[
\pi\colon  M \rightarrow \Bc(\Hc). 
\]
Then one wants to make a contragredient representation on the conjugate-linear Hilbert space $\overline{\Hsp}$ by means of the antipode $S$: 
\begin{equation}\label{EqContraRep}
\pi^c\colon M \rightarrow \Bc(\overline{\Hc}),\qquad a \mapsto \overline{\pi(S(a)^*)}.
\end{equation}
The problem is that this representation will not be $*$-preserving in general, as the antipode is not $*$-preserving if $S^2 \neq \id$: 
\[
S(a)^* = S^{-1}(a^*),\qquad a\in M.
\]
One hence tries to find an operator $\hat{\Qb}$ on $\Hc$ restoring the $*$-compatibility: 
\begin{equation}\label{EqContraGredRenorm}
\overline{\pi}\colon M \rightarrow \Bc(\overline{\Hc}),\qquad a \mapsto \overline{\hat{\Qb}\pi(S(a)^*)\hat{\Qb}^{-1}}
\end{equation}
(ignoring questions of unboundedness/domains of definition).

Now given a `good' multiplicative unitary $\Wbb$ (where `good' will turn out to be exactly the modularity condition), one intuitively thinks of its first leg $M$, consisting of (a completion of) elements of the form $(\id\otimes \omega)\Wbb$ for $\omega \in \Bc(\Hc)_*$, as forming a `topological' Hopf $*$-algebra with coproduct implemented by 
\begin{equation}\label{EqCoprodInt}
a \mapsto \Wbb^*(1\otimes a)\Wbb.
\end{equation}
(see Section \ref{SecvNMUMod}). In this setup, the antipode of $M$ should then act on the first leg of $\Wbb$ as 
\[
(S\otimes \id)\Wbb = \Wbb^*,
\]
which leads to \eqref{EqContraRep} becoming 
\[
\pi^c((\id\otimes \omega)\Wbb) = \overline{\pi((\id\otimes \overline{\omega})\Wbb)}, \qquad \overline{\omega}(x) = \overline{\omega(x^*)}. 
\]
Writing out the definition of $\overline{\pi}$ in  \eqref{EqContraGredRenorm} when $\pi$ is the identity representation and $\omega$ a vector functional, we arrive at the equation \eqref{EqLeftInvOp}, with 
\[
\overline{\pi}((\id\otimes \omega)\Wbb) = (\id\otimes \omega)\wWbb^*.
\]
Formally, the above manipulations then also lead to the conclusion that one should have 
\[
S^2(a) = \hat{\Qb}^2a\hat{\Qb}^{-2},\qquad a\in M,
\]
which explains the intuition of $\hat{\Qb}$ implementing `the positive root of the antipode squared'.

Note that making all this intuition rigorous requires a lot of hard work. The key point is however that in the analytic setting of multiplicative unitaries (as opposed to the setting of Hopf $*$-algebras), equation \eqref{EqLeftInvOp} should rather be seen as the starting point, from which a theory resembling the Hopf $*$-algebraic one can then be derived.
\end{remark}

The following theorem is known, see Appendix \ref{SecModAxB}. We will use the notation $\mathbf{X}$ and $\Pb = \frac{1}{2\pi i }d/dx$ for the infinitesimal self-adjoint generators of the unitary groups on $L^2(\R)$ determined by 
\begin{equation}\label{EqInfGenPosMom}
(e^{it\mathbf{X}}f)(x) = e^{itx}f(x),\qquad (e^{2\pi it\Pb}f)(x) = f(x+t),\qquad f\in L^2(\R).
\end{equation}
\begin{theorem}\label{TheoRank1}
Consider $\Hc = L^2(\R)$ together with 
\[
\stgF = e^{\mathbf{X}},\qquad \Lb = e^{-4\pi^2  \hbar \Pb},
\]
so that $\stgF$ and $\Lb$ form an $\hbar$-commuting pair:
\[
\stgF^{it}\Lb^{is} = e^{2\pi i \hbar st}\Lb^{is}\stgF^{it},\qquad s,t\in\R. 
\]
Put
\[
\Kb = \Fb^{-1}, \qquad \stgE = \Lb\star\Fb^{-1}. 
\]
Then, using the notation \eqref{EqQuantExp},  the unitary
\[
\Wbb = \exp\left(\frac{i}{2\pi \hbar}\ln(\Lb) \otimes \ln(\Kb)\right) \overline{F}_{\hbar}(\stgF\otimes \stgE)
\]
is a modular multiplicative unitary with respect to
\begin{equation}\label{EqModStrucRank1}
\Qb = \Kb^{-\frac{1+|\hbar|^{-1}}{2}},\qquad \hat{\Qb} = \Lb^{\frac{1+|\hbar|^{-1}}{2}},
\end{equation}
and $(\Wbb,\Qb,\hat{\Qb},\wWbb)$ is a modular datum for  
\[
\wWbb = \exp\left(\frac{i}{2\pi \hbar} \ln(\overline{\Lb}) \otimes \ln(\Kb)\right) F_{\hbar}(\overline{\stgF} \otimes \Kb^{-1}\star\stgE).
\]
\end{theorem}

We can rephrase this theorem as stating that in the case $n=1$, with associated diagram 
\[
C_2' = \adjustbox{scale=0.8}{\begin{tikzcd}
& 110 & \\
101 & &011
\arrow[from=2-1, to=1-2]
\arrow[from=1-2, to=2-3]
\arrow[from=2-3, to=2-1]
\end{tikzcd}},
\]
 the Fock--Goncharov unitary $\Fct$ is modular if we use the specific unitary $\hbar$-representation of $\wbtd{2}$ on $L^2(\R)$ given by 
\[
\eb(e_{101}) = \stgF,\qquad \eb(e_{110})=\stgE,\qquad \eb(e_{011}) = \Lb^{-1}.
\] 

\begin{remark}
The notation evokes a degenerate representation of the Heisenberg double of the Hopf $*$-algebra 
\begin{multline}\label{EqCoprodF}
U_q(\mfb^-) = \C\langle L,F\mid LF = q^2FL\rangle,\qquad F^* = F,L^* = L,\\ \Delta(L) = L\otimes L,\quad \Delta(F) = F\otimes L + 1 \otimes F,
\end{multline}
with $q = e^{\pi i \hbar}$ in this formalism. Indeed, if the dual of $U_q(\mfb^-)$ is realized as 
\begin{multline}
U_q(\mfb) = \C\langle K,E\mid KE = q^2EK\rangle,\qquad E^* = E,K^* = K,\\ \Delta(K) = K\otimes K,\quad \Delta(E) = E\otimes 1 + K \otimes E,
\end{multline}
then the (normalized) natural pairing between $U_q(\mfb)$ and $U_q(\mfb^-)$ leads to the Heisenberg commutation relations
\[
KL = q^2 LK,\quad KF = FK,\quad EL = q^2 LE,\quad EF - FE = (q-q^{-1})L,
\]
which is compatible with the relations between the operators in Theorem \ref{TheoRank1}.

Further note that the coproduct \eqref{EqCoprodInt} is compatible with the coproduct in \eqref{EqCoprodF}, through an easy application of the Kashaev pentagon equation. One then sees that the modular structure \eqref{EqModStrucRank1} amounts to the following modification of the contragredient of the identity representation $\pi$ as in \eqref{EqContraGredRenorm}:
\[
\overline{\pi}(\stgF) = \overline{\stgF}\star\overline{\Lb}^{-1},\qquad \overline{\pi}(\Lb) = \overline{\Lb}^{-1}. 
\]
\end{remark}

To extend the modularity in Theorem \ref{TheoRank1} to the general setting, it is convenient to single out the following auxiliary notion, focusing on Condition \eqref{EqLeftInvOp}.

\begin{defn}\label{DefWeaklyMod}
Let $\Gc,\Hc$ be Hilbert spaces. We call a unitary $U\in \Bc(\Gc \otimes\Hc)$ \emph{weakly modular} if there exists a strictly positive operator $\hat{\Qb}$ on $\Gc$ and a unitary operator $\wU \in \Bc(\overline{\Gc}\otimes \Hc)$ such that 
\begin{equation}\label{EqLeftInvOpGen}
\langle \hat{\Qb}\xi\otimes \eta',U(\eta\otimes \xi')\rangle = \langle \overline{\hat{\Qb}\eta}\otimes \eta',\wU(\overline{\xi} \otimes \xi')\rangle,\qquad \forall \xi,\eta\in \msD(\hat{\Qb}),\xi',\eta'\in \Hc.
\end{equation}
We then call $(U,\hat{\Qb},\wU)$ a \emph{weakly modular datum} for the Hilbert space couple $(\Gc,\Hc)$. We refer to $\hat{\Qb}$ as the \emph{modularity operator}.
\end{defn}

The following  (easy) observation will be crucial to show modularity of the Fock--Goncharov flip $\Fct$ in general rank. 

\begin{lemma}\label{LemMod}
Let $\Gc,\Kc,\Hc$ be Hilbert spaces, and let $(V,\hat{\Qb},\wV)$ and $(W,\hat{\Pb},\wW)$ be weakly modular data with respect to the respective pairs of Hilbert spaces $(\Gc,\Hc)$ and $(\Kc,\Hc)$. 

Then $(U,\hat{\Qb}\otimes \hat{\Pb},\wU)$ is a weakly modular datum with respect to $(\Gc\otimes \Kc,\Hc)$, where 
\begin{equation}\label{EqTensProdMod}
U = V_{13}W_{23} \in \Bc((\Gc\otimes \Kc)\otimes \Hc),\qquad \wU = \wV_{13}\wW_{23} \in \Bc((\overline{\Gc}\otimes \overline{\Kc})\otimes \Hc).
\end{equation}
\end{lemma}
\begin{proof}
For $\xi,\eta\in \msD(\hat{\Qb}),\theta,\zeta\in \msD(\hat{\Pb})$ and $\eta',\theta'\in \Hc$, we compute 
\begin{eqnarray*}
\langle \hat{\Qb}\xi\otimes \hat{\Pb}\theta\otimes \eta',V_{[13]}W_{[23]}(\eta\otimes \zeta\otimes\theta')\rangle &=& \langle \hat{\Qb}\xi\otimes \eta',V(\eta\otimes (\omega_{\hat{\Pb}\theta,\zeta}\otimes \id)(W)\theta')\rangle \\
&=& \langle \overline{\hat{\Qb}\eta}\otimes \eta',\wV(\overline{\xi}\otimes (\omega_{\hat{\Pb}\theta,\zeta}\otimes \id)(W)\theta')\rangle \\
&=& \langle \overline{\hat{\Qb}\eta} \otimes \hat{\Pb}\theta\otimes \eta',\wV_{[13]}W_{[23]}(\overline{\xi}\otimes \zeta\otimes \theta')\rangle \\
&=& \langle \hat{\Pb}\theta\otimes (\omega_{\overline{\xi},\overline{\hat{\Qb}\eta}}\otimes \id)(\wV^*)\eta',W(\zeta\otimes \theta')\rangle \\
&=& \langle \overline{\hat{\Pb}\zeta}\otimes (\omega_{\overline{\xi},\overline{\hat{\Qb}\eta}}\otimes \id)(\wV^*)\eta',\wW(\overline{\theta}\otimes \theta')\rangle \\
&=& \langle \overline{\hat{\Qb}\eta}\otimes \overline{\hat{\Pb}\zeta}\otimes \eta',\wV_{[13]}\wW_{[23]}(\overline{\xi}\otimes \overline{\theta}\otimes \theta')\rangle. 
\end{eqnarray*}
In other words, with $U,\wU$ as in \eqref{EqTensProdMod}, we find 
\[
\langle (\hat{\Qb}\xi\otimes \hat{\Pb}\theta)\otimes \eta',U((\eta\otimes \zeta)\otimes \theta')\rangle = \langle \overline{\hat{\Qb}\eta\otimes \hat{\Pb}\zeta}\otimes \eta',\wU(\overline{\xi\otimes \theta}\otimes \theta')\rangle. 
\]
Since the algebraic tensor product of $\msD(\hat{\Qb})$ and $\msD(\hat{\Pb})$ is a core for the tensor product $\hat{\Qb}\otimes \hat{\Pb}$,  we find that indeed $U$ is weakly modular, and that $(U,\hat{\Qb}\otimes \hat{\Pb},\wU)$ is a weakly modular datum with respect to $(\Gc\otimes \Kc,\Hc)$. 
\end{proof}

We also have the following trivial extension of Theorem \ref{TheoRank1}. 

\begin{theorem}\label{TheoRank1Gen}
Let $(\Lb,\Fb)$ and $(\Kb,\Eb)$ be \emph{any} two pairs of $\hbar$-commuting strictly positive operators on respective Hilbert spaces $\Hc$ and $\Gc$. Then 
\[
U = \exp\left(\frac{i}{2\pi \hbar}\ln(\Lb) \otimes \ln(\Kb)\right) \overline{F}_{\hbar}(\Fb\otimes \Eb)
\]
fits into a weakly modular datum $(U,\hat{\Qb},\widetilde{U})$ with
\[
\hat{\Qb} = \Lb^{\frac{1+|\hbar|^{-1}}{2}}
\]
and 
\[
\widetilde{U} = \exp\left(\frac{i}{2\pi \hbar} \ln(\overline{\Lb}) \otimes \ln(\Kb)\right) F_{\hbar}(\overline{\Fb} \otimes \Kb^{-1}\star\Eb).
\]
\end{theorem}
\begin{proof}
Let $\Gc,\Hc,\Kc,\mathcal{L}$ be Hilbert spaces. Then clearly $(U,\hat{\Qb},\widetilde{U})$ is a weakly modular datum for $(\Gc,\Hc)$ if and only if 
\[
(U_{[13]},\hat{\Qb}\otimes 1,\widetilde{U}_{[13]})
\]
is a weakly modular datum for $(\Gc\otimes \Kc,\Hc\otimes \mathcal{L})$. Similarly, if $u\colon \Gc\rightarrow \Kc$ and $v\colon \Hc\rightarrow \mathcal{L}$ are unitaries, then $(U,\hat{\Qb},\widetilde{U})$ is a weakly modular datum for $(\Gc,\Hc)$ if and only if 
\[
((u\otimes v)U(u^*\otimes v^*),u\hat{\Qb}u^*,(\overline{u}\otimes v)\widetilde{U}(\overline{u}^*\otimes v^*))
\]
is a weakly modular datum for $(\Kc,\mathcal{L})$. 

Hence the theorem follows immediately from Theorem \ref{TheoRank1} and the Stone-von Neumann theorem.
\end{proof}

\subsection{Modularity of Fock--Goncharov flips in general rank}\label{SecModularity}

We start by constructing a concrete unitary $\hbar$-representation of $\wbtd{N}$ with respect to which we will show modularity of the Fock--Goncharov flip. We then later show that modularity holds with respect to \emph{any} $\hbar$-representation of $\wbtd{N}$.  

Recall the notations introduced in Definition \ref{DefCNOr} and Definition \ref{DefBorel}.  We identify $\neC_N$ with the set
\[
I = I_0 \cup I_1,\quad 
I_0 = \{s = (s,s)\mid 1\leq s\leq n\},\quad I_1 = \{(s,k)\mid 1\leq s\leq n, 0\leq k < s\}
\]
by means of the map 
\[
I \cong \neC_N,\qquad (s,k) \mapsto (n-s+1,k,s-k).
\]

We totally order $I$ (and $\neC_N$) by requiring that $I_0 = [1,n]$ has the normal order, that elements in $I_0$ precede elements in $I_1$, and that $I_1$ is totally ordered by
\begin{equation}\label{EqOrderingSetI}
(s,k)< (s',k') \qquad \iff \qquad s-k< s'-k' \textrm{ or }(s-k = s'-k'\textrm{ and }s > s'),
\end{equation}
see Figure \ref{FigTriangOrder} for an example.

\begin{figure}[ht]
\adjustbox{scale=0.7,center}{%
\begin{tikzcd}
* && 1 && 2 && 3 && *\\
	&6  && 5 && 4 && * & \\
	&& 8 && 7 && * &&\\
&&& 9 && * &&&\\
&&&& * &&&&
	\arrow[dashed, from=1-3, to=1-1]
	\arrow[dashed, from=1-5, to=1-3]
	\arrow[dashed, from=1-7, to=1-5]
	\arrow[dashed, from=1-9, to=1-7]
	\arrow[dashed, from=1-1, to=2-2]
	\arrow[dashed, from=2-2, to=3-3]
	\arrow[dashed, from=3-3, to=4-4]
	\arrow[dashed, from=4-4, to=5-5]
	\arrow[dashed, from=2-4, to=1-5]
	\arrow[dashed, from=5-5, to=4-6]
\arrow[dashed, from=4-6, to=3-7]
\arrow[dashed, from=3-7, to=2-8]
\arrow[dashed, from=2-8, to=1-9]
	\arrow[from=2-2, to=1-3]
\arrow[from=2-4, to=1-5]
\arrow[from=2-6, to=1-7]
\arrow[from=3-3, to=2-4]
\arrow[from=3-5, to=2-6]
\arrow[from=4-4, to=3-5]
\arrow[from=2-4, to=2-2]
\arrow[from=2-6, to=2-4]
\arrow[from=2-8, to=2-6]
\arrow[from=3-7, to=3-5]
\arrow[from=3-5, to=3-3]
\arrow[from=4-6, to=4-4]
\arrow[from=1-3, to=2-4]
\arrow[from=1-5, to=2-6]
\arrow[from=1-7, to=2-8]
\arrow[from=2-4,to=3-5]
\arrow[from=2-6,to=3-7]
\arrow[from=3-5,to=4-6]
\end{tikzcd}
}
\caption{The ordering of $I$ identified with $\protect\neC_4$}\label{FigTriangOrder}
\end{figure}
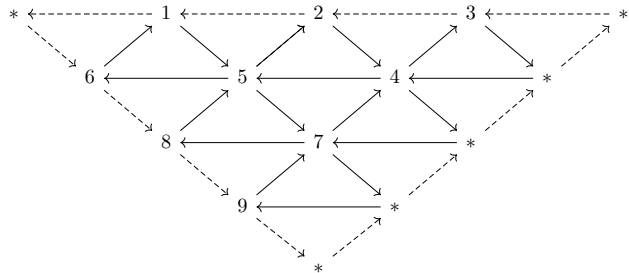

Consider the symplectic skew-symmetric space $W = \R f \oplus \R \varpi$ with
\begin{equation}\label{EqNormalisationfw}
(f,\varpi) = 1/2.
\end{equation}
We then consider $I$ copies $\{W_i\mid i \in I\}$ of $W$, whose respective vectors we denote $f_i,\varpi_i$. We further put
\begin{equation}\label{EqDirSum}
V_N = \oplus_{i\in I} W_i,
\end{equation}
keeping the same notation for the images of the $f_i,\varpi_i$ inside $V_N$. We can then by induction (on the reverse order) uniquely make an embedding of skew-symmetric spaces
\begin{equation}\label{EqEmbd}
\wwbtd{N} \subseteq V_N
\end{equation}
by requiring the following rules with respect to the basis $\{f_i,\varpi_i \mid i \in I\}$ of $V_N$:
\begin{enumerate}
\item if $i\in I_1$ and $j\in I$, then 
\begin{itemize}
\item for $j<i$, the $j$-th component of $\nee_i$ is zero, 
\item the $i$-th component of $\nee_i$ is $f$,
\item for $j>i$, the $j$-th component of $\nee_i$ lies in $\R \varpi$, and 
\end{itemize}
\item if $i\in I_0 \cong [1,n]$ and $j \in I$, then 
\begin{itemize}
\item for $j<i$, the $j$-th component of $\nevarpi_i$ is zero,
\item the $i$-th component of $\nevarpi_i$ is $\varpi$, and 
\item For $j>i$ the $j$-th component of $\nevarpi_i$ lies in $\R \varpi$, and is zero if $j \in I_0$.
\end{itemize}
\end{enumerate}
By \eqref{EqDiffCarOpp} and \eqref{EqDiffCart}, this embedding extends to $\wbtd{N}$ by putting
\begin{equation}\label{EqExtToTr}
\seee_{s} = -f_{N-s},\qquad s\in [1,n]. 
\end{equation}

Endow now $L^2(\R)$ with the unitary $\hbar$-representation of $W$ determined by 
\[
\eb(f) = \stgF,\qquad \eb(\varpi) = \Lb^{1/2},
\]
using the same notation as in Theorem \ref{TheoRank1}.  More generally, denote for $i\in I$ by $\Hc_i$ a copy of $L^2(\R)$, endowed with the corresponding unitary $\hbar$-representation of $W_i$,  and put 
\begin{equation}\label{EqTensProd}
\Hc = \otimes_{i\in I} \Hc_i,
\end{equation}
endowed with the product unitary $\hbar$-representation of $V_N = \oplus_{i\in I}W_i$. By restriction under the embedding \eqref{EqEmbd}, we can view this also as a unitary $\hbar$-representation of $\wbtd{N}$ or of $\wwbtd{N}$. 
  
Denote further, for $1\leq s\leq n$ and $0\leq k<s$,
\begin{equation}\label{EqPartFlipOne}
\bbU^{(s,k)} = \Gauss\left(2\varpi_{(s,k)} \otimes \seee_{N-s}\right)  \prod_{s \le r \le n}^{\longra}   \varphi\hr{f_{(s,k)}\oplus \seee_{N-s,n-r}} \in \Bc(\Hc_{(s,k)}\otimes \Hc), 
\end{equation}
\begin{equation}\label{EqPartFlipTwo}
\bbU^{(s)} =\Gauss\left(2\varpi_s \otimes \seee_{N-s}\right) \in  \Bc(\Hc_s\otimes \Hc). 
\end{equation}
We have the following alternative expression for  $\bbU^{(s,k)}$. Recall that $\boxplus$ denotes the form sum.  
\begin{lemma}
We have 
\begin{equation}\label{EqProdFormComp}
\bbU^{(s,k)} = \Gauss\left(2\varpi_{(s,k)} \otimes \seee_{N-s}\right)\overline{F}\hr{\eb(f_{(s,k)})\otimes  \left(\boxplus_{s \le r \le n} \eb(\seee_{N-s,n-r})\right)}.
\end{equation}
\end{lemma}
\begin{proof}
This follows immediately from Theorem \ref{TheoExponGen}. 
\end{proof}

The following is the main observation that will allow us to establish modularity of the Fock--Goncharov flip $\Fct$ as in  \eqref{EqRealMU}.

\begin{theorem}\label{TheoProdFG}
Inside $\Bc(\Hsp \otimes \Hsp) = \Bc\left(\left(\otimes_{i\in I} \Hc_i\right) \otimes \Hsp\right)$, we have an equality of unitaries
\begin{equation}\label{EqIdToProve}
\Fct = \overset{\longra}{\prod}_{i\in I} \bbU^{(i)}_{i,\bullet}. 
\end{equation}
\end{theorem}
Here we use leg numbering notation and keep $\bullet$ as the index for the $\Hsp$-factor at the end.
\begin{proof}
For the $\bbU^{(s,k)}_{(s,k),\bullet}$ and $\bbU^{(s)}_{s,\bullet}$,  we can still use the formulas \eqref{EqPartFlipOne} and \eqref{EqPartFlipTwo}, now with the first leg viewed as acting on the $i$-th leg of $\otimes_{i\in I} \Hc_i$ (and as the unit on all other legs). 

By shifting all the Gaussians through the quantum dilogarithms, using \eqref{EqGausEqInv}, we find that the right hand side of \eqref{EqIdToProve} becomes
\[
\overset{\rightarrow}{\prod}_{i\in I} \bbU^{(i)}_{i,\bullet} = \Lc \overset{\rightarrow}{\prod}_{i\in I_1} \bbV^{(i)}
\]
where 
\begin{equation}\label{EqGausPart}
\Lc = \Gauss\left(2\sum_{1\leq s \leq n}\left(\sum_{0\leq k \leq s} \varpi_{(s,k)} \right)\otimes \seee_{N-s}\right)
\end{equation}
and 
\begin{equation}\label{EqqDilogPart}
\bbV^{(i)} = \prod_{s \le r \le n}^{\longra}   \varphi\hr{f_{(s,k)}+ 2\sum_{(s',k')>(s,k)}(\seee_{N-s,n-r},\seee_{N-s'})\varpi_{(s',k')} \oplus \seee_{N-s,n-r}}.
\end{equation}

We however easily check that 
\begin{equation}\label{EqSumNEvect}
f_{(s,k)}+ 2\sum_{(s',k')>(s,k)}(\seee_{N-s,n-r},\seee_{N-s'})\varpi_{(s',k')} = \nee_{s,k}.
\end{equation}
\begin{equation}\label{EqSumFundWeights}
\sum_{0\leq k \leq s} \varpi_{(s,k)} = \nevarpi_s.
\end{equation}

Indeed, for \eqref{EqSumNEvect}, we just need to verify that 
\[
-(\seee_{n-s+1,n-r},\seee_{n-s'+1}) = (\nee_{s,k},\nee_{s',k'}),\qquad (s',k')>(s,k),
\]
but this is immediate from \eqref{EqSameNil}, \eqref{EqSameCar} and the definition of the ordering \eqref{EqOrderingSetI}. Similarly, for \eqref{EqSumFundWeights} we only need to check that for $1\leq s,s' \leq n$ and $0\leq k' < s'$, it holds that
\[
\sum_{0\leq k< s} (\varpi_{(s,k)},\nee_{s',k'}) =  (\nevarpi_s,\nee_{s',k'}).
\]
But the left hand side is 
\[
-\frac{1}{2}\delta_{s,s'} \delta_{k'<s} =  -\frac{1}{2}\delta_{s,s'},
\]
which is exactly the right hand side by \eqref{EqPairFundWeight}.

It follows from the above that $\Lc = \Kc$ as in \eqref{EqGaussFG}, and that
\begin{eqnarray*}
\overset{\rightarrow}{\prod}_{i\in I} \bbU^{(i)}_{i,\bullet} &=& \Kc  \prod^{\longra}_{(s,k)\in I_1}\prod_{s \le r \le n}^{\longra}   \varphi\hr{\nee_{s,k} \oplus \seee_{N-s,n-r}}.
\end{eqnarray*}
But by \eqref{eq:F-3}, we have
\[
 \prod^{\longra}_{(s,k)\in I_1}\prod_{s \le r \le n}^{\longra}   \varphi\hr{\nee_{s,k} \oplus \seee_{N-s,n-r}} = \Fc,
\]
the braided Fock--Goncharov flip.
 This finishes the proof.
\end{proof}

Combined with Theorem \ref{TheoRank1} and Lemma \ref{LemMod}, we now obtain:
\begin{theorem}\label{TheoModFGFin}
With respect to the unitary $\hbar$-representation of $\wbtd{N}$ on \eqref{EqTensProd}, the Fock--Goncharov flip $\Fct$ is a modular multiplicative unitary with respect to 
\begin{equation}\label{EqModularityMUFG}
\hat{\Qb} = \Lb_{\delta}^{1+|\hbar|^{-1}},\qquad \Qb = \Kb_{\delta}^{-(1+|\hbar|^{-1})},
\end{equation}
where
\[
  \Lb_{\delta} = \eb(\sum_{i=1}^n \nevarpi_i),\qquad \Kb_{\delta}=  \eb(\sum_{i=1}^n \sevarpi_i).
\]
Furthermore, $(\Fct,\Qb,\hat{\Qb},\tFct)$ is a modular datum for  
\begin{equation}\label{EqModFG}
\tFct = \tKc \tFc = \tFctt \tKc
\end{equation}
with
\begin{equation}\label{EqTildeK}
\tKc = \Gauss\left(2\sum_{t=1}^n \overline{\nevarpi_t} \otimes \nwe_{N-t}\right) 
\end{equation}
and 
\begin{equation}\label{EqTildeFPrime}
\tFc = \prod_{1 \le r \le n}^{\longra} \prod_{1 \le s \le n}^{\longra} \underset{0\le (n-r)-k\le n-s}{\prod_{0 \le k \le s-1}^{\longra}}   \overline{\varphi}\hr{\overline{\nee_{s,k}}\oplus -\nwe_{N-s,(n-r)-k}},
\end{equation}
\begin{equation}\label{EqTildeFDoublePrime}
\tFctt = \prod_{1 \le r \le n}^{\longra} \prod_{1 \le s \le n}^{\longra} \underset{0\leq r-k-1 \leq n-s}{\prod_{0 \le k \le s-1}^{\longra}}   \overline{\varphi}\hr{-\overline{\swe_{s,k}}\oplus \seee_{N-s,r-k-1}}.
\end{equation}
\end{theorem}
\begin{proof}
Note that
\[
(f_{(s,k)},2\varpi_{(s,k)}) = 1,\qquad (\seee_{N-s,n-r},\seee_{N-s}) = 1,
\]
so by Lemma \ref{LemSumSkewComm} we have that 
\[
(\eb(2\varpi_{(s,k)}),\eb(f_{(s,k)})),\qquad (\boxplus_{s \le r \le n} \eb(\seee_{N-s,n-r}),\eb(\seee_{N-s}))
\]
form $\hbar$-commuting pairs. From Theorem \ref{TheoRank1Gen}, we conclude  that $\bbU^{(s,k)}$ is weakly modular with respect to the operator
\[
\hat{\Qb}_{(s,k)} = \eb(\varpi_{(s,k)})^{1+|\hbar|^{-1}}.
\]
By Lemma \ref{LemMod} and using again \eqref{EqSumFundWeights}, it now follows that $\Fct$ is weakly modular for the modularity operator
\[
\hat{\Qb} = \otimes_{i \in I} \hat{\Qb}_i = \Lb_{\delta}^{1+|\hbar|^{-1}}.
\]
The proof of the modularity is concluded by the easy observation that $ \Lb_{\delta}\otimes \Kb_{\delta}^{-1}$ commutes with $\Fct$. 

To see that \eqref{EqModFG} completes the modular datum, we note that Lemma \ref{LemMod}, Theorem \ref{TheoRank1Gen} and Theorem \ref{TheoProdFG}, together with the expression \eqref{EqProdFormComp} and the first part of the proof, give that $(\Fct,\Lb_{\delta}^{1+|\hbar|^{-1}},\tFct)$ forms a weakly modular pair if we define
\begin{equation}\label{EqContrWeak}
\tFct = \prod_{i\in I}^{\longra} \wUbb^{(i)}_{i,\bullet},
\end{equation}
where 
\[
\wUbb^{(s)} = G_{\hbar}(2\overline{\varpi}_s\otimes \seee_{N-s}) = G_{-\hbar}(2\overline{\varpi}_s\otimes(- \nwe_{N-s}))
\]
and
\begin{eqnarray*}
\wUbb^{(s,k)} &=& G_{\hbar}(2\overline{\varpi_{(s,k)}} \otimes \seee_{N-s}) F_{\hbar}\left(\overline{\eb(f_{(s,k)})} \otimes \eb(-\seee_{N-s})\star \left(\boxplus_{s \le r \le n} \eb(\seee_{N-s,n-r})\right)\right)\\
&=& G_{-\hbar}(2\overline{\varpi_{(s,k)}} \otimes (-\seee_{N-s}))  \prod_{s \le r \le n}^{\longla}   \varphi_{-\hbar}\hr{\overline{f_{(s,k)}}\oplus \seee_{N-s,n-r}-\seee_{N-s}}\\
&=& G_{-\hbar}(2\overline{\varpi_{(s,k)}} \otimes (-\nwe_{N-s}))  \prod_{s \le r \le n}^{\longla}   \varphi_{-\hbar}\hr{\overline{f_{(s,k)}}\oplus -\nwe_{N-s,r-s}}\\
&=& G_{-\hbar}(2\overline{\varpi_{(s,k)}} \otimes (-\nwe_{N-s}))  \prod_{s \le r \le n}^{\longra}   \varphi_{-\hbar}\hr{\overline{f_{(s,k)}}\oplus -\nwe_{N-s,n-r}}\\
&=& \overline{G_{\hbar}(2\varpi_{(s,k)} \otimes \overline{-\nwe_{N-s}})  \prod_{s \le r \le n}^{\longra}   \varphi_{\hbar}\hr{f_{(s,k)}\oplus \overline{-\nwe_{N-s,n-r}}}}.
\end{eqnarray*}
As structurally the expression for \eqref{EqContrWeak} is of the same form as \eqref{EqIdToProve}, using that 
\[
(\overline{-\nwe_{s,k}},\overline{-\nwe_{s',k'}}) = (\seee_{s,k},\seee_{s',k'}),\qquad 1\leq s,s'\leq n,0\leq k \leq s,0\leq k'\leq s',
\]
we find that $\tFct$ is indeed of the form in \eqref{EqModFG}.

The alternative form $\tFct = \tFctt \tKc$ follows again straightforwardly from \eqref{EqGausEq} and the formulas in Lemma \ref{LemSkewProdFormDiff}.
\end{proof}

For the moment, we can not conclude yet that $\Fct$ will be modular with respect to an \emph{arbitrary} unitary $\hbar$-representation of $\wbtd{N}$. This turns out to be true, but to prove this we will need some more detailed information on C$^*$-algebras and von Neumann algebras associated to the Fock--Goncharov flip. This is what we will turn to in the next section. 

\section{von Neumann algebras associated to the Fock--Goncharov flip}\label{SecvNMUMod}

By \cite[Theorem 2.3]{SW01} (see also \cite[Theorem 5]{SW07}) one can associate to a modular multiplicative unitary $\Wbb \in \Bc(\Hc\otimes \Hc)$ a von Neumann algebra $M$, obtained as the $\sigma$-weak closure 
\begin{equation}\label{EqFormvN}
M = M_{\Wbb} := [(\id\otimes \omega)\Wbb\mid \omega\in \Bc(\Hc)_*]^{\sigma\textrm{-weak}} \subseteq \Bc(\Hc).
\end{equation}
This von Neumann algebra can further be enhanced into a \emph{von Neumann bialgebra}, i.e.\ $M$ comes equipped with a faithful normal unital $*$-homomorphism
\begin{equation}\label{EqComultiplication}
\Delta = \Delta_{\Wbb}\colon M \rightarrow M \bar{\otimes} M,\quad x \mapsto \Wbb^*(1\otimes x)\Wbb,
\end{equation}
satisfying the coassociativity condition. We refer to $(M,\Delta)$ as the  \emph{von Neumann bialgebra} associated to $\Wbb$. 

Now if $\Wbb$ is modular with respect to $(\Qb,\hat{\Qb})$, one has by \cite[Proposition 2.2]{SW01} that 
\[
\hat{\Wbb} = \Sigma \Wbb^* \Sigma = \Wbb_{21}^*
\]
is a modular multiplicative unitary with respect to $(\hat{\Qb},\Qb)$. 
We refer to $(\hat{M},\hat{\Delta}) = (M_{\hat{\Wbb}},\Delta_{\hat{\Wbb}})$ as the \emph{dual} von Neumann bialgebra. Explicitly, we can write 
\begin{equation}\label{EqFormvNDual}
\hat{M} = M_{\hat{\Wbb}} := [(\omega\otimes \id)\Wbb\mid \omega\in \Bc(\Hc)_*]^{\sigma\textrm{-weak}} \subseteq \Bc(\Hc),
\end{equation}
with the dual comultiplication
\begin{equation}\label{EqComultiplicationDual}
\hat{\Delta} = \Delta_{\hat{\Wbb}}\colon \hat{M} \rightarrow \hat{M} \bar{\otimes} \hat{M},\quad x \mapsto \hat{\Wbb}^*(1\otimes x)\hat{\Wbb}.
\end{equation}

The aim of this section is to show that in the case of the Fock--Goncharov flip, these von Neumann algebras are just completions of suitable \emph{quantum tori}.  To prove this, we will need to use some of the machinery in \cite{Rie93}, which we first recall. 

\subsection{\texorpdfstring{C$^*$-algebras}{C*-algebras} associated to a skew-symmetric space}\label{SecCStarRieff}

The following material is completely standard, but we use some care to introduce the necessary formalism. 

To a skew-symmetric space $V$ and $\hbar\in \R$ one can associate a C$^*$-algebra $C^*_{\hbar}(V)$ such that the unitary $\hbar$-representations of $V$ correspond to the non-degenerate $*$-representations of $C^*_{\hbar}(V)$. One may either view $C^*_{\hbar}(V)$ as a deformation of the group C$^*$-algebra of $C^*(V)$ or, in Pontryagin dual fashion, as a deformation $C_0^{\hbar}(\hat{V})$ of the C$^*$-algebra $C_0(\hat{V})$, with $\hat{V}$ the space of real linear functionals on $V$. Both viewpoints are useful to fuel intuition in the considerations below.

Write the canonical pairing between $V$ and its linear dual $\hat{V}$ as\footnote{In \cite{Rie93}, one identifies $\hat{V}$ and $V$ through an auxiliary positive-definite inner product on $V$ by which $w \mapsto \hat{w} = \langle w,-\rangle$ can actually be viewed as a concrete map instead of abstract notation for elements in the dual space. Though this auxiliary structure will be inessential in what follows, the reader is free to use such a concrete identification to more directly follow the arguments in \cite{Rie93}.} 
\[
\hat{w}(v) = \langle \hat{w},v\rangle = \langle v,\hat{w} \rangle = \hat{w} \cdot v = v\cdot \hat{w},\qquad v\in V, \hat{w} \in \hat{V}.
\] 
Put
\begin{equation}\label{EqJOperator}
\Jc = \Jc_{\hbar}\colon V \rightarrow \hat{V}, \qquad \langle v,\Jc w\rangle =   -\frac{1}{2}\hbar (v,w),\qquad \forall v,w\in V.
\end{equation}

We can then endow the Schwartz space $\Sc(\hat{V})$ with the $*$-algebra structure such that \[
f^*(\hat{v}) = \overline{f(\hat{v})}
\]
and 
\begin{eqnarray}\label{EqTwistProd}
(f\times_{\hbar} g)(\hat{v}) &:=& \dint_{V\times \hat{V}} f(\hat{v}-\Jc u) g(\hat{v}-\hat{w}) e^{2\pi i \hat{w} \cdot u}\rd u \rd \hat{w}\\ &=& \int_V f(\hat{v}-\Jc u) \hat{g}(u) e^{2\pi i \hat{v} \cdot u}\rd u  \label{EqTwistProd2}\\
&=&  \int_V \hat{f}(u) g(\hat{v} + \Jc u) e^{2\pi i \hat{v} \cdot u}\rd u. \label{EqTwistProd3}
\end{eqnarray}
Here we use the usual Fourier transform 
\begin{equation}\label{EqFourDualOrd}
\msF\colon \Sc(\hat{V})\rightarrow \Sc(V),\quad f\mapsto \hat{f},\quad  \hat{f}(v) = \int_{\hat{V}} f(\hat{w}) e^{-2\pi i v\cdot \hat{w}}\rd \hat{w},\qquad f \in \Sc(\hat{V}),v\in V,
\end{equation}
with the Lebesgue measures of $V$ and $\hat{V}$ normalized so that one has the reverse Fourier transform as 
\begin{equation}\label{EqInvFourDualOrd}
\msF^{-1}\colon \Sc(V)\rightarrow \Sc(\hat{V}),\quad f\mapsto \check{f},\quad \check{f}(\hat{v}) = \int_{V} f(w) e^{2\pi i \hat{v}\cdot w}\rd w,\qquad f \in \Sc(V),\hat{v}\in \hat{V}.
\end{equation}
In fact, to see that $(\Sc(\hat{V}),\times_{\hbar},*)$ is an associative $*$-algebra, it is easiest to consider the resulting $*$-algebra structure on $(\Sc(V),\ast_{\hbar},\sharp)$ on $\Sc(V)$ through the Fourier transform $\msF$ in \eqref{EqFourDualOrd}, which  leads to
\[
(f\ast_{\hbar}g)(v) = \int_V f(u)g(v-u) e^{2\pi i v\cdot \Jc u}\rd u,\qquad f^{\sharp}(u) = \overline{f(-u)},\qquad f\in \Sc(V),
\]
and so 
\begin{equation}\label{EqFourierDualCross}
\widehat{f\times_{\hbar}g} = \hat{f}\ast_{\hbar} \hat{g},\qquad \widehat{f^*} = \hat{f}^{\sharp}. 
\end{equation}
In fact, one easily checks that $(\Sc(V),*_{\hbar},\sharp)$ is a \emph{normed} associative $*$-algebra for the norm $f \mapsto \|\hat{f}\|_1$, which completes into the normed Banach $*$-algebra $(L^1(V),\ast_{\hbar},\sharp)$ with the twisted convolution $*$-algebra structure. Hence the following definition is meaningful, as any $*$-representation of a Banach $*$-algebra on a Hilbert space is contractive (= norm-decreasing).

\begin{defn}
We define  $C_{\hbar}^*(V)$ to be the C$^*$-algebra obtained as the norm-completion of $(\Sc(V),\ast_{\hbar},\#)$ with respect to its universal C$^*$-norm
\[
\|f\|_{\hbar} := \sup_{\pi} \{\|\pi(f)\|\},
\]
with $\pi$ ranging over all non-degenerate $*$-representations $\pi \colon L^1(V) \rightarrow \Bc(\Hc)$ of $(L^1(V),\ast_{\hbar},\sharp)$ on Hilbert spaces $\Hsp$. 

Correspondingly, we define $C_0^{\hbar}(\hat{V})$ to be the C$^*$-algebra completion of $(\Sc(\hat{V}),\times_{\hbar},*)$ through the C$^*$-norm 
\[
\|f\|_{\hbar} := \|\hat{f}\|_{\hbar},\qquad f\in \Sc(\hat{V}). 
\]
\end{defn}
By construction, we then have that $C_0^{\hbar}(\hat{V})$ and $C^*_{\hbar}(V)$ are isomorphic through 
\[
C_0^{\hbar}(\hat{V})\cong C^*_{\hbar}(V),\quad f\mapsto\hat{f}\textrm{ for }f\in \Sc(\hat{V}),
\]
and in what follows we can simply identify these two C$^*$-algebras. For $f\in \Sc(\hat{V})$, we denote by $f_{\hbar}$ its copy inside $C_{\hbar}^*(V)$ for emphasis, and we then denote the inclusion map of $\Sc(\hat{V})$ into $C^*_{\hbar}(V)$ by  
\begin{equation}\label{EqQuantMap}
\Qc\colon  \Sc(\hat{V}) \rightarrow C^*_{\hbar}(V),\qquad f \mapsto f_{\hbar}.
\end{equation}
On the other hand, there is less need to make such a differentiation on the Fourier dual side, so if $f\in \Sc(V)$ we simply write also $
f \in C^*_{\hbar}(V)$.

\begin{lemma}\label{LemUnivProp}
There is a one-to-one correspondence between unitary $\hbar$-representations $(\Hsp,\pi)$ of $V$ and non-degenerate $*$-representations $(\Hsp,\pi)$ of $C^*_{\hbar}(V)$, determined by 
\begin{equation}\label{EqExtRep}
\pi(f_{\hbar}) = \pi(\hat{f}) := \int_V \hat{f}(v) \pi(v)\rd v, \qquad f\in \Sc(\hat{V}).
\end{equation}
\end{lemma}
\begin{proof}
Well-known. The proof goes as follows: writing
\[
(u \cdot f)(v) =  e^{\pi i \hbar(u,v)}f(v-u),\qquad f\in L^1(V),u,v\in V, 
\]
one checks that, for any $f\in L^1(V)$, the map
\[
V\ni u \mapsto u\cdot f\in L^1(V),
\]
is continuous with respect to the $L^1$-norm. One further verifies that 
\[
L^1(V) \ni f\mapsto u\cdot f \in L^1(V)
\]
is an isometric linear map for any $u\in V$, with inverse
\[
f \mapsto u^{-1}\cdot f.
\]
As a last ingredient, one uses the easily checked \emph{multiplier property}
\[
g^{\sharp} \ast_{\hbar} (u\cdot f) = (u^{-1}\cdot g)^{\sharp}\ast_{\hbar} f,\qquad f,g\in L^1(V),u\in V. 
\]
\end{proof}

For example, with respect to the standard unitary $\hbar$-representation of $V$ given through \eqref{EqStandardDefDualRep}, the associated $C^*_{\hbar}(V)$-representation is determined by 
\begin{equation}\label{EqStandardDefDualRepCstar}
\pi(f) g = f\ast_{\hbar}g,\qquad f,g\in \Sc(V).
\end{equation}
If we conjugate by the unitary Fourier transform $\mathscr{F}$, we find the unitarily equivalent unitary $\hbar$-representation $\pi = \pi_{L^2(\hat{V})}$ of $V$ on $L^2(\hat{V})$ (which we will denote by the same symbol), which is now determined by 
\begin{equation}\label{EqStandardDefRep}
(\pi(v)g)(\hat{w}) = e^{2\pi i v\cdot \hat{w}} g(\hat{w} + \Jc v),\qquad v\in V,\hat{w}\in \hat{V},g\in \mathcal{S}(\hat{V}), 
\end{equation}
and where the associated $*$-representation of $C_{\hbar}^*(V)$ on $L^2(\hat{V})$ is  determined,
using \eqref{EqTwistProd3}, by 
\begin{equation}\label{EqStandardRep}
\pi(f_{\hbar})g = f\times_{\hbar} g,\qquad f,g\in \Sc(\hat{V}).
\end{equation}

\begin{defn}
We refer to either $(L^2(V),\pi)$ or $(L^2(\hat{V}),\pi)$, with $\pi$ as in \eqref{EqStandardDefDualRepCstar} or \eqref{EqStandardRep}, as the \emph{standard $*$-representation} of $C_{\hbar}^*(V)$, or as the \emph{standard unitary $\hbar$-representation} of $V$. We then put 
\begin{equation}\label{EqvNQT}
L_{\hbar}(V) := [\pi(f_{\hbar}) \mid f\in \Sc(\hat{V})]^{\sigma-\textrm{weak}} \subseteq \Bc(L^2(\hat{V})),
\end{equation}
which defines a von Neumann algebra called the \emph{twisted group von Neumann algebra} of $V$.

We call a non-degenerate $*$-representation of $C_{\hbar}^*(V)$ \emph{(faithfully) normal} if it extends to a (faithful) normal $*$-representation of $L_{\hbar}(V)$. We then also use this terminology for the associated unitary $\hbar$-representation of $V$. 
\end{defn}

If $V$ is symplectic with Lagrangian subspace $W \subseteq V$, we get through the construction in Example \ref{ExaHeisenberg} identifications 
\begin{equation}
C_{\hbar}^*(V) \cong \Kc(L^2(W)),\qquad L_{\hbar}(V) \cong \Bc(L^2(W)), 
\end{equation}
and \emph{any} unitary $\hbar$-representation give rise to a faithfully normal $*$-representation of $C^*_{\hbar}(V)$. 

In particular, $C^*_{\hbar}(V)$ is a nuclear C$^*$-algebra if $V$ is either symplectic or trivial. Since unitary $\hbar$-representations of a direct sum of skew-symmetric spaces are just pairs of pointwise commuting $\hbar$-representations of the components on a same Hilbert space, it follows from the decomposition \eqref{EqDecompRad} that 
\begin{equation}\label{EqTensProdCanRad}
C^*_{\hbar}(V)\cong C^*(\Rad(V))\otimes C^*_{\hbar}(V')
\end{equation}
is nuclear for any skew-symmetric space, and that 
\begin{equation}\label{EqDirSumTens}
C^*_{\hbar}(V\oplus W) \cong C^*_{\hbar}(V) \otimes C^*_{\hbar}(W)
\end{equation}
for any direct sum of skew-symmetric spaces. Moreover, the isomorphism \eqref{EqTensProdCanRad} also gives that the standard $*$-representation $\pi$ of $C^*_{\hbar}(V)$ is faithful for any skew-symmetric space $V$ and that, more generally, the decomposition \eqref{EqTensProdCanRad} is implemented concretely through an identification of (tensor products of) standard representations in the form \eqref{EqStandardDefDualRepCstar} via the canonical unitary 
\[
L^2(V\oplus W) \cong L^2(V) \otimes L^2(W).
\]
This then also establishes
\[
L_{\hbar}(V\oplus W) \cong L_{\hbar}(V)\bar{\otimes} L_{\hbar}(W).
\]

If we now, for a general skew-symmetric space $V$, consider a direct sum decomposition as in \eqref{EqDecompRadRefined}, and identify $\R\cong \R e_i$ as a Lagrangian subspace of $Z_i$, one gets in particular
isomorphisms of C$^*$-algebras and von Neumann algebras 
\begin{equation}\label{EqIsoSkewBilvN}
C_{\hbar}^*(V) \cong C_0(\widehat{\Rad(V)}) \otimes  \bigotimes_{i=1}^m\Kc(L^2(\R)),\qquad L_{\hbar}(V) \cong L^{\infty}(\widehat{\Rad(V)})\bar{\otimes}\overline{\bigotimes}_{i=1}^m\Bc(L^2(\R)). 
\end{equation}

The following is a direct consequence of the above decompositions.

\begin{lemma}\label{LemTypeIRieff}
If $V$ is a skew-symmetric space, then $L_{\hbar}(V)$ is a type $I$-von Neumann algebra. It is a factor if and only if $V$ is symplectic.
\end{lemma}

The following lemma is a consequence of the fact that if $Z \subseteq V$ is a subspace of a skew-symmetric space, we can always choose the decomposition \eqref{EqDecompRadRefined} of $V$ such that 
\[
Z = (Z\cap \Rad(V)) \oplus \bigoplus_{i=1}^m (Z\cap Z_i),
\]
with either $\emptyset, \{e_i\}$ or $\{e_i,\hat{e}_i\}$ a basis of $Z\cap Z_i$.

\begin{lemma}\label{LemIncGivesFaith}
If $Z \subseteq V$ is an inclusion of skew-symmetric spaces, then the  standard unitary $\hbar$-representation of $V$ restricts to a faithfully normal unitary $\hbar$-representation of $Z$. 
\end{lemma}

Concretely, the resulting embedding $L_{\hbar}(Z) \subseteq L_{\hbar}(V)$ is determined by 
\[
f_{\hbar}  =  \int_{Z} \hat{f}(z) \pi(z) \rd z,\qquad f\in \Sc(\hat{Z}).
\]

\subsection{Rieffel deformation of vector spaces with boundary}

For what follows, it will be important to know that one can extend the domain of the quantization map \eqref{EqQuantMap}. 

Fix again $V$ a skew-symmetric space, and consider $C_b^{\infty}(\hat{V})$, the space of uniformly smooth bounded functions on $\hat{V}$, i.e.\ those functions such that the Banach space-valued function $\hat{v}\mapsto f_{\hat{v}} \in C_b(\hat{V})$ is smooth, with $f_{\hat{v}}(\hat{w}) = f(\hat{w}-\hat{v})$. The following is then a special case of Proposition \ref{PropExtMultAlg}.

\begin{lemma}
The quantization map \eqref{EqQuantMap} extends uniquely to a map 
\begin{equation}\label{EqQuantMapGen}
\Qc\colon  C_b^{\infty}(\hat{V}) \rightarrow M(C^*_{\hbar}(V)),\qquad f \mapsto f_{\hbar},
\end{equation}
with $M(C^*_{\hbar}(V))$ the multiplier C$^*$-algebra of $C^*_{\hbar}(V)$, in such a way that, for $g \in \Sc(\hat{V})$, one has 
\[
f_{\hbar}g_{\hbar} = (f\times_{\hbar} g)_{\hbar},\qquad g_{\hbar}f_{\hbar} = (g\times_{\hbar} f)_{\hbar},\qquad f \in C_b^{\infty}(\hat{V}),g\in \Sc(\hat{V}),
\]
with $f\times_{\hbar}g$ and $g\times_{\hbar}f$ defined by respectively \eqref{EqTwistProd2} and \eqref{EqTwistProd3}.
\end{lemma}

For example, we have in $C_b^{\infty}(\hat{V})$ the function
\begin{equation}
e^{2\pi i v}\colon  \hat{V} \rightarrow \C,\qquad \hat{w}\mapsto e^{2\pi i v\cdot \hat{w}}.
\end{equation}
Then it is easily checked, e.g.\ from \eqref{EqTwistProd2} and \eqref{EqExtRep}, that for any unitary $\hbar$-representation $\pi$ we have, inversely to \eqref{EqExtRep},
\begin{equation}\label{EqCorrCanEl}
\pi(e_{\hbar}^{2\pi iv}) = \pi(v). 
\end{equation}
One further has 
\begin{equation}\label{EqExpAffRieff}
e_{\hbar}^v \,\eta\, C_{\hbar}^*(V),
\end{equation}
where $\eta$ denotes the affiliation relation for C$^*$-algebras, which we recall in Appendix \ref{SecMult}. The affiliation \eqref{EqExpAffRieff} is determined uniquely by
\begin{equation}\label{EqChare}
\pi(e_{\hbar}^{2\pi v}) = \eb(v), \qquad v\in V,
\end{equation}
for any unitary $\hbar$-representation of $V$ (see the discussion following Theorem \ref{TheoFunctRieffDef}). 

We apply this as follows. Assume we have dual bases $\{f_i\}$ and $\{\hat{f}_i\}$ of $V$ and $\hat{V}$, indexed by a set
\begin{equation}
I = I_0 \sqcup I_1.
\end{equation}
We identify $V \cong \R^I \cong \hat{V}$ through these bases via
\[
v  = \sum_i v_i f_i,\qquad \hat{v} = \sum_i \hat{v}_i\hat{f}_i,\qquad v\in V,\hat{v}\in \hat{V}.
\]

Consider the convex closed cone 
\begin{equation}\label{EqDualCone}
V \supseteq V_+ = V_+(I_0,I_1) := \{v\in V \mid v_i = \langle v,\hat{f}_i\rangle \geq 0,\forall i\in I_1\}
\end{equation}
as well as the dual extended convex closed cone 
\begin{equation}\label{EqExtCone}
\hat{V}\subseteq \hat{V}_+ = \hat{V}_+(I_0,I_1) := \R^{I_0} \times \R_{-\infty}^{I_1},\qquad \R_{-\infty} = [-\infty,\infty),
\end{equation}
endowed with the product topology and with $\R_{-\infty}$ having its topology generated by the open intervals in $\R$ as well as the intervals $[-\infty,a)$ for $a\in \R$.

Consider the inclusions
\begin{equation}\label{EqListInclusCont}
C_0(\hat{V}) \subseteq C_0(\hat{V}_+) \subseteq C_b(\hat{V}_+) \subseteq C_b(\hat{V}), 
\end{equation}
where the final inclusion is through restriction to $\hat{V}$. Since the translation action of $\hat{V}$ extends continuously to $\hat{V}_+$, the above inclusions are $\hat{V}$-equivariant, and we can consider also
\begin{equation}\label{EqListInclus}
C_0^{\infty}(\hat{V}) \subseteq C_0^{\infty}(\hat{V}_+) \subseteq C_b^{\infty}(\hat{V}_+) \subseteq C_b^{\infty}(\hat{V}). 
\end{equation}
Put
\begin{equation}\label{EqCStarCone}
C_{\hbar}^*(V_+) := [f_{\hbar}\mid f\in C_0^{\infty}(\hat{V}_+)] \subseteq M(C_{\hbar}^*(V)).
\end{equation}

\begin{theorem}\label{TheoAffil}
The space $C_{\hbar}^*(V_+)$ is a C$^*$-algebra. Moreover,  
\begin{equation}\label{EqAffExp}
e_{\hbar}^w \,\eta\, C_{\hbar}^*(V_+),\qquad w \in V_+.
\end{equation}
\end{theorem}
The proof of the above theorem can be found in Appendix \ref{SecRieff}.

This theorem will be important to have access to particular quotient maps. Namely, assume that 
\[
I' \subseteq I,\qquad \textrm{ with }I_i' = I_i\cap I',
\]
and assume that we have
\[
(f_i,f_j) = 0,\qquad \forall i \in I_0\setminus I_0',j\in I'.
\]
Put 
\[
V' = V(I_0',I_1') \subseteq V
\]
as a skew-symmetric subspace. Then although this map does not allow a splitting as skew-symmetric spaces, such a splitting exists when passing to the associated C$^*$-algebras:

\begin{theorem}\label{TheoDegRep}
There exists a unique non-degenerate C$^*$-algebra homomorphism 
\begin{equation}\label{EqRepDegBound}
\pi_{I'}\colon  C_{\hbar}^*(V_+) \rightarrow C_{\hbar}^*(V_+'),\qquad  f_{\hbar} \mapsto f_{I',\hbar}, \qquad f \in C_0^{\infty}(\hat{V}_+),
\end{equation}
where $f_{I'} =  f\circ \iota_{I'} \in C_0^{\infty}(\hat{V}_+')$ for
\begin{equation}\label{EqDegMap}
\iota_{I'}\colon  \hat{V}_+' \rightarrow \hat{V}_+,\qquad \iota_{I'}(\hat{v}')_i = \begin{cases}\hat{v}'_i &\text{if} \quad i\in I', \\
-\infty &\text{if} \quad i\in I_1\setminus I_1',\\
0 &\text{if} \quad i\in I_0\setminus I_0'.
\end{cases}
\end{equation}
Moreover, under this quotient map we have  
\begin{equation}\label{EqImageDegMap}
\pi_{I'}(e_{\hbar}^{f_i}) =  \begin{cases}e_{\hbar}^{f_i} &\text{if} \quad i\in I', \\
0 &\text{if} \quad i\in I_1\setminus I_1',\\
1 &\text{if} \quad i\in I_0\setminus I_0'.
\end{cases}
\end{equation}
\end{theorem}

The proof is presented at the end of Appendix \ref{SecRieff}.

\begin{remark}
To compare this with the algebraic setup, one should view $C^*_{\hbar}(V)$ as an analytic analogue of $\C_q[X_1^{\pm1},\ldots,X_n^{\pm1}]$, a skew Laurent polynomial algebra with (say)
\[
X_i X_j = q^{\varepsilon_{ij}}X_jX_i
\]
for $q$ a (formal) constant and $\varepsilon_{ij}\in \Z$. Then although no $X_i$ can be put to zero in this algebra, one can easily obtain such quotient maps for the associated skew polynomial algebra $\C_q[X_1,\ldots,X_n]$, which is the algebraic counterpart of $C_{\hbar}^*(V_+)$ when $I= I_1$. 
\end{remark}

\subsection{von Neumann algebras associated to Fock--Goncharov flips}\label{SecProofvNFGflip}

In this section, we determine explicitly the von Neumann algebras $M_{\Fct}$ and $\hat{M}_{\Fct} = M_{\hat{\Fct}}$ associated to a Fock--Goncharov flip $\Fct$ (as in \eqref{EqFormvN} and \eqref{EqFormvNDual}), where in particular 
\[
\hat{\Fct} = \Fct_{21}^*.
\]
We first show that the modularity of $\Fct$, as well as the associated von Neumann algebras $M_{\Fct}$ and $\hat{M}_{\Fct}$, are independent of the precise choice of unitary $\hbar$-representation of $\wbtd{N}$.

\begin{theorem}\label{EqFGGeneralMod}
Let $\Fct$ be an arbitrary Fock--Goncharov flip. Then with $\Qb,\hat{\Qb}$ and $\tFct$ as in Theorem \ref{TheoModFGFin}, we have that $(\Fct,\Qb,\hat{\Qb},\tFct)$ is a  modular datum.
\end{theorem}
\begin{proof}
Let $V_N$ be the symplectic space as in \eqref{EqDirSum}, with its associated $\hbar$-Heisenberg representation on the Hilbert space $\Hsp$. Then this representation is faithfully normal, and hence also its restriction to $\wbtd{N}$ is faithfully normal by Lemma \ref{LemIncGivesFaith}. We may hence in particular identify $C^*_{\hbar}(\wbtd{N}) \subseteq \Bc(\Hsp)$. 

Let now $\Fct$ be the associated Fock--Goncharov flip acting on $\Hsp$. By \eqref{EqExpAffRieff}, we have 
\begin{equation}\label{EqWeakAffiliation}
\Fct \in M(C^*_{\hbar}(\wbtd{N})\otimes C^*_{\hbar}(\wbtd{N}))
\end{equation}
Consider its associated modular datum $(\Fct,\Qb,\hat{\Qb},\tFct)$. If then $\pi$ is a unitary $\hbar$-representation of $\wbtd{N}$, seen as a non-degenerate $*$-representation of $C^*_{\hbar}(\wbtd{N})$, we are to show that 
\[
(\Fct_{\pi},\Qb_{\pi},\hat{\Qb}_{\pi},\tFct_{\pi}) := ((\pi\otimes \pi)\Fct,\pi(\Qb),\pi(\hat{\Qb}),(\overline{\pi}\otimes \pi)\tFct)
\]
is a modular datum. 

Now \eqref{EqRightInvOp} can be rewritten as 
\[
\omega((\id\otimes \omega_{\Qb^{-1}\xi,\Qb\eta})\Fct) = \omega(\overline{(\id\otimes \omega_{\xi,\eta})\tFct}^*),\qquad \forall \xi \in \msD(\Qb^{-1}),\forall \eta\in \msD(\Qb),\forall \omega \in \Bc(\Hsp)_*,
\]
which implies 
\[
(\id\otimes \omega_{\Qb^{-1}\xi,\Qb\eta})\Fct = \overline{(\id\otimes \omega_{\xi,\eta})\tFct}^*.
\]

As these are elements of $M(C^*_{\hbar}(\wbtd{N}))$, we can apply $\pi$, and then find 
\[
\langle \xi\otimes \eta',(\pi\otimes \id)(\Fct)(\eta\otimes \Qb\xi')\rangle = \langle \overline{\eta}\otimes \Qb\eta',(\overline{\pi}\otimes \id)(\tFct)(\overline{\xi} \otimes \xi')\rangle,\qquad \forall \xi',\eta'\in \msD(\Qb),\xi,\eta\in \Hc.
\]
Since $\hat{\Qb}_{\pi}\otimes \Qb$ clearly commutes with $(\pi\otimes \id)\Fct$, this implies as well that
\begin{multline*}
\langle \hat{\Qb}_{\pi}\xi\otimes \eta',(\pi\otimes\id)(\Fct)(\eta\otimes \xi')\rangle = \langle \overline{\hat{\Qb}_{\pi}\eta}\otimes \eta',(\overline{\pi}\otimes \id)(\tFct)(\overline{\xi} \otimes \xi')\rangle,\\\forall \xi,\eta\in \msD(\hat{\Qb}_{\pi}),\xi',\eta'\in L^2(\wbtd{N}).
\end{multline*}
So,
\[
(\omega_{\hat{\Qb}_{\pi}\xi,\eta}\pi\otimes \id)\Fct = (\omega_{\overline{\hat{\Qb}_{\pi}\eta},\overline{\xi}}\overline{\pi}\otimes \id)\tFct,\qquad \forall \xi,\eta\in \msD(\hat{\Qb}_{\pi}).
\]
As these are elements in $M(C^*_{\hbar}(\wbtd{N}))$, we can again apply $\pi$, leading to 
\[
\langle \hat{\Qb}_{\pi}\xi\otimes \eta',\Fct_{\pi}(\eta\otimes \xi')\rangle = \langle \overline{\hat{\Qb}_{\pi}\eta}\otimes \eta',\tFct_{\pi}(\overline{\xi} \otimes \xi')\rangle,\qquad \forall \xi,\eta\in \msD(\hat{\Qb}_{\pi}),\xi',\eta'\in L^2(\wbtd{N}).
\]
As $\hat{\Qb}_{\pi}\otimes \Qb_{\pi}$ commutes with $\Fct_{\pi}$, we are done. 
\end{proof}
Note that there is some substance to this result: for any $N\geq 2$ we have that $\Rad(\wbtd{N})\neq \{0\}$, so $C^*_{\hbar}(\wbtd{N})$ allows non-faithful non-degenerate $*$-representations. We show however that, nevertheless, all Fock--Goncharov flips lead to the same von Neumann bialgebras.

Recall $\wwbtd{N},\Bor_N^+$ as defined in \eqref{EqSpaceBNMin}.

\begin{lemma}\label{LemResFaith}
If $\pi$ is a unitary $\hbar$-representation of $\wbtd{N}$, its restrictions to $\Bor_{N}^{\pm}$ are faithfully normal. 
\end{lemma}
\begin{proof}
By Lemma \ref{LemNonDegPair}, we can find symplectic subspaces $Z^{\pm}$ such that 
\[
\Bor_N^{\pm} \subseteq Z^{\pm} \subseteq \wbtd{N}.
\]
The result now follows from Lemma \ref{LemIncGivesFaith} and the fact that any unitary $\hbar$-representation of a symplectic space is faithfully normal.
\end{proof}

If we consider the standard unitary $\hbar$-representation of $\wbtd{N}$ as in \eqref{EqvNQT}, we will refer to its associated Fock--Goncharov flip \eqref{EqRealMU} as the \emph{standard Fock--Goncharov flip}, and in this section we fix for this the notation 
\begin{equation}\label{EqStandardFG}
\Fct  \in L_{\hbar}(\wbtd{N}) \bar{\otimes} L_{\hbar}(\wbtd{N}).
\end{equation}
We then view 
\begin{equation}\label{EqvNFG}
M  = M_{\Fct} \subseteq L_{\hbar}(\wbtd{N}),\qquad \hat{M} = M_{\hat{\Fct}} \subseteq L_{\hbar}(\wbtd{N}).
\end{equation}
Note that it immediately follows from the formula for $\Fct$ that in fact 
\begin{equation}\label{EqFGLegs}
\Fct \in L_{\hbar}(\Bor_N^-)\bar{\otimes} L_{\hbar}(\Bor_N^+). 
\end{equation}
From Lemma \ref{LemResFaith}, we then obtain:

\begin{cor}\label{CorUnitRepArb}
Let $\Fct$ be the standard Fock--Goncharov flip, and $\pi$ a unitary $\hbar$-representation of $\wbtd{N}$. Then with $\Fct_{\pi} = (\pi\otimes \pi)\Fct$, we get
\[
(M_{\Fct},\Delta_{\Fct}) \cong (M_{\Fct_{\pi}},\Delta_{\Fct_{\pi}}),\qquad (M_{\hat{\Fct}},\Delta_{\hat{\Fct}}) \cong (M_{\hat{\Fct}_{\pi}},\Delta_{\hat{\Fct}_{\pi}}).
\]
\end{cor} 

Because of this, we from now on simply refer to 
\[
(M_{\Fct},\Delta_{\Fct}),\qquad (M_{\hat{\Fct}},\Delta_{\hat{\Fct}})
\]
as the \emph{Fock--Goncharov von Neumann bialgebra} and \emph{dual von Neumann bialgebra}.

Specializing Corollary \ref{CorUnitRepArb} further, we have: 
\begin{cor}\label{CorIrrMap}
Consider the symplectic quotient 
\[
\wbtd{N}\twoheadrightarrow D_N
\]
as in \eqref{EqDegHeisDoub}. Let $\pi$ be an arbitary unitary $\hbar$-representation of $D_N$. 

Then viewing $\pi$ as a unitary $\hbar$-representation of $\wbtd{N}$, the associated Fock--Goncharov flip $\Fct$ implements an isomorphic copy of the Fock--Goncharov von Neumann bialgebra and its dual.
\end{cor}

Note that from \eqref{EqFGLegs}, we can view 
\[
M_{\Fct} \subseteq L_{\hbar}(\Bor_N^-),\qquad M_{\hat{\Fct}} \subseteq L_{\hbar}(\Bor_N^+).
\]
Our next aim is to show that these are in fact equalities.

Endow $\Bor_N^{\pm}$ with the respective bases
\begin{equation}
\{\nef_i \mid i \in I\} = \{\nevarpi_t,\nee_{(s,k)}\mid 1\leq t,s\leq n,0\leq k <s\},
\end{equation}
\begin{equation}
\{\sef_i \mid i \in I\} = \{\sevarpi_t,\seee_{(s,k)}\mid 1\leq t,s\leq n,0\leq k <s\}.
\end{equation}
We use for $\Bor_N^{\pm}$, endowed with these bases, the notation introduced in \eqref{EqDualCone} and \eqref{EqExtCone}. 

We fix a unitary $\hbar$-representation of $\wbtd{N}$ with associated Fock--Goncharov flip $\Fct$. 

\begin{lemma}\label{LemAffil}
For $1\leq r \leq n$, $r\leq b \leq n$ and $1\leq i\leq r$, let $B^{b,r,i}$ be as in \eqref{DefBbri}. Then
\[
B^{b,r,i} \in M(C_{\hbar}^*(\wwbtd{N,+})\otimes C_{\hbar}^*(\wwwbtd{N,+})).
\]
In particular,
\begin{equation}\label{EqFGInMultRestr}
\Fct \in M(C_{\hbar}^*(\wwbtd{N,+})\otimes C^*_{\hbar}(\wwwbtd{N,+})).
\end{equation}
\end{lemma}
Note that \eqref{EqFGInMultRestr} is much stronger than \eqref{EqWeakAffiliation}, as when passing to multiplier algebras the inclusion $C_{\hbar}^*(\wwbtd{N}) \subseteq C_{\hbar}^*(\wwbtd{N,+})$ switches the order of inclusion, cf.\ \eqref{EqListInclusCont} for the classical intuition for this.
\begin{proof}
By \eqref{EqAffExp}, we have that 
\[
\Eb(\nee_{b-i+1,b-r}) = e_{\hbar}^{2\pi \nee_{b-i+1,b-r}}\, \eta \, C_{\hbar}^*(\Bor_{N,+}^-),
\]
and similarly 
\[
\Eb(\seee_{n-b+i,n-b}) = e_{\hbar}^{2\pi \seee_{n-b+i,n-b}}\, \eta \, C_{\hbar}^*(\Bor_{N,+}^+).
\]
Then by the observation \eqref{EqAffiliTens}, we have 
\[
\Eb(\nee_{b-i+1,b-r}\oplus \seee_{n-b+i,n-b}) = \Eb(\nee_{b-i+1,b-r})\otimes \Eb(\seee_{n-b+i,n-b})\,\eta\,C_{\hbar}^*(\wwbtd{N,+})\otimes C^*_{\hbar}(\wwwbtd{N,+}). 
\]
Now as $\overline{F}$ is a continuous function on $[0,\infty)$, we get by functional calculus (Lemma \ref{LemFunctCal} and the remark following it) that 
\begin{multline*}
B^{b,r,i} = \varphi(\nee_{b-i+1,b-r}\oplus \seee_{n-b+i,n-b}) = \overline{F}(\Eb(\nee_{b-i+1,b-r})\otimes \Eb(\seee_{n-b+i,n-b}))\\  \in M(C_{\hbar}^*(\wwbtd{N,+})\otimes C^*_{\hbar}(\wwwbtd{N,+})).
\end{multline*}
The statement for $\Fct$ follows immediately.
\end{proof}

\begin{lemma}\label{LemCommGauss}
Let $\Ab$ and $\Bb$ be two self-adjoint operators on Hilbert spaces $\Hc$ and $\Gc$, and assume $\Spec(\Ab) = \R$. Assume that $u$ is a unitary in $\Bc(\Gc)$ such that 
\[
(1\otimes u)\Gauss_{\hbar}(\Ab\otimes \Bb) = \Gauss_{\hbar}(\Ab\otimes \Bb)(1\otimes u). 
\]
Then $u$ commutes with all $e^{it\Bb}$, and hence commutes with $\Bb$. 
\end{lemma}
\begin{proof}
Immediate by specializing $\Ab$ to $t \in \R$ using functional calculus.
\end{proof}

\begin{lemma}\label{LemCommQDil}
Let $\Ab$ and $\Bb$ be two strictly positive operators on Hilbert spaces $\Hc$ and $\Gc$, and assume $\Spec(\Ab) = \R_{\geq0}$. Assume that $u$ is a unitary in $\Bc(\Gc)$ such that 
\[
(1\otimes u)F_{\hbar}(\Ab\otimes \Bb) = F_{\hbar}(\Ab\otimes \Bb)(1\otimes u). 
\]
Then $u$ commutes with all $\Bb^{it}$, and hence commutes with $\Bb$. 
\end{lemma}  
\begin{proof}
Again by functional calculus, the assumption implies that 
\[
u F_{\hbar}(t\Bb) = F_{\hbar}(t\Bb)u,\qquad \forall t>0. 
\]
In other words, writing $\Hb = \ln(\Bb)$, we have 
\begin{equation}\label{EqInvCond}
u\overline{\varphi}_{\hbar}(\Hb + t)u^* = \overline{\varphi}_{\hbar}(\Hb + t),\qquad \forall t\in \R. 
\end{equation}

Assume now that $\hbar>0$. Then as $\overline{\varphi}_{\hbar}$ extends to a holomorphic function on the upper half plane $\bbH$, the equality \eqref{EqInvCond} continues to hold (as an equality of unbounded normal operators) for any $t\in \bbH$. 
As $\overline{\varphi}_{\hbar}$ satisfies, rewriting \eqref{EqFunctEq1}, the functional equation 
\begin{equation}\label{EqFuncEqDilog}
\overline{\varphi}_{\hbar}(z+2\pi i) = (1+e^{i\pi \hbar^{-1}} e^{\hbar^{-1} z})\overline{\varphi}_{\hbar}(z),
\end{equation}
we can conclude from \eqref{EqInvCond} that 
\[
u e^{\hbar^{-1} \Hb} u^*=  e^{\hbar^{-1} \Hb}.
\]
So $u$ commutes with $\Hb$, and the lemma follows. 

The proof in the case $\hbar<0$ follows by taking adjoints.
\end{proof}

Note now that by Lemma \ref{LemSkewProdFormDiff}, we can apply Theorem \ref{TheoDegRep} to obtain non-degenerate $*$-representations
\begin{equation}\label{EqProjMapsPii}
\pi_t = \pi_{\{t\}},\qquad \pi_{(s,k)} = \pi_{\{s,(s,0)\}},\qquad t\in I_0, (s,k)\in I_1
\end{equation}
of $C_{\hbar}^*(\wwbtd{N,+})$. Writing 
\[
I_0'(t) = \{t\},\quad I_1'(t) = \emptyset,\quad I_0'(s,k) = \{s\},\qquad I_1'(s,k)= \{(s,0)\},
\]
we may identify (for respectively $i = t$ or $i = (s,k)$)  
\[
\wwbtd{}(I_0'(i),I_1'(i)) \subseteq W_i,\qquad \nevarpi_t \mapsto \varpi_t,\quad \varpi_s \mapsto \varpi_{(s,k)},\quad \nee_{s,0}\mapsto f_{(s,k)},
\]
with $W_i$ as in \eqref{EqDirSum}, and hence view the above $\pi_i$ as non-degenerate $*$-representations of $C_{\hbar}^*(\wwbtd{N,+})$ on $\Hsp_i = L^2(\R)$. By Lemma \ref{LemAffil}, we can apply these projection maps to the first leg of $\Fct$. Then one sees in fact immediately that 
\begin{equation}\label{EqProjFct}
(\pi_i \otimes \id)\Fct = \bbU^{(i)},\qquad i \in I, 
\end{equation}
with $\bbU^{(i)}$ as in \eqref{EqPartFlipOne} and \eqref{EqPartFlipTwo}. This is the main ingredient to prove the following theorem:

\begin{theorem}\label{TheovNFGFlip}
Let $\hbar\in \R^{\times}$, and let $\Fct$ be a Fock--Goncharov flip with associated dual von Neumann algebra $\hat{M} = M_{\hat{\Fct}}$. Then 
\[
\hat{M} = L_{\hbar}(\wwwbtd{N}).
\]
\end{theorem}
\begin{proof}
Assume that $\Fct \subseteq \Bc(\Hsp\otimes \Hsp)$. As we already know that $\hat{M} \subseteq L_{\hbar}(\wwwbtd{N})$, we are to show 
\[
\hat{M}'  \subseteq  L_{\hbar}(\wwwbtd{N})' \subseteq \Bc(\Hsp).
\]
But assume that $x \in \hat{M}'$, i.e.\ 
\begin{equation}\label{EqCommx}
\Fct(1\otimes x) = (1\otimes x) \Fct. 
\end{equation}
 Then \eqref{EqCommx} gives us 
\begin{equation}
\bbU^{(i)}(1\otimes x) = (1\otimes x)\bbU^{(i)},\qquad i\in I. 
\end{equation}
By the expressions \eqref{EqPartFlipTwo}  and \eqref{EqProdFormComp} for the $\bbU^{(i)}$, we find that 
\[
[(1\otimes x),G(2\varpi \otimes \seee_{N-s})]=0 = \left[(1\otimes x),\overline{F}\hr{\eb(f)\otimes  \left(\boxplus_{s \le r \le n} \eb(\seee_{N-s,n-r})\right)}\right]
\]
for all $1 \leq s \leq n$ and $s \leq r \leq n$. Applying Lemma \ref{LemCommGauss} and Lemma \ref{LemCommQDil}, we find, using the notation $[-,-]=0$ to denote strong commutation between two operators, that 
\[
[x,\eb(\seee_{N-s})]=0 = [x,\boxplus_{s \le r \le n} \eb(\seee_{N-s,n-r})],\qquad \forall 1\leq s\leq n.
\]
The theorem now follows from the first part of the proposition below. 
\end{proof}

Recall the standard generators from Definition \ref{Defstandardgen}, where we abbreviate
\begin{equation}\label{EqGenEs}
\stgE_s := \stgE_{s,s+1} =  \boxplus_{s \le r \le n} \eb(\seee_{N-s,n-r}),\qquad 1\leq s\leq n.
\end{equation}
We also use the notation 
\begin{equation}\label{EqCartanK}
\Kb_s := \eb(\seee_{N-s}),\qquad 1\leq s \leq n.
\end{equation}

\begin{prop}\label{PropStandardGen}
\begin{enumerate}
\item
The operators $\{\stgE_s\mid 1\leq s\leq n\}$ generate $L_{\hbar}(\Nilp_N^+)$.
\item The operators $\{\Kb_s,\stgE_s\mid 1\leq s \leq n\}$ generate $L_{\hbar}(\Bor_N^+)$.
\end{enumerate}
\end{prop} 
\begin{proof}
(1) Assume that $x$ is an operator which strongly commutes with all $\stgE_s$ for $1\leq s\leq n$. We show by induction on $N-s$ that $x$ then also commutes with all $\eb(\seee_{s,k})$ for $0\leq k <s$, and hence with all $\eb(e_{a,b,c})^{it}$ for $(a,b,c)\in \wseC_N$, which is enough to prove the lemma. 

The base case is automatic as $\stgE_n = \eb(\seee_{1,0})$. 

Assume then that, for fixed $s$, we know that $x$ commutes with all $\eb(\seee_{s,k})$ for $0\leq k <s$, as well as with $\stgE_{N-s-1}$. Putting 
\[
\Ab = \boxplus_{N-s-1\leq r <n} \eb(\seee_{s+1,n-r}),\qquad \Bb = \eb(\seee_{s+1,0}),
\]
we have that $\Ab$ and $\Bb$ $-\hbar$-commute, so by \eqref{EqBoxPlusDef} we can write 
\[
\stgE_{N-s-1} = \Ab \boxplus \Bb = F_{\hbar}(\Ab\star\Bb^{-1})\Bb F_{\hbar}(\Ab\star\Bb^{-1})^*.
\]
Hence from \eqref{EqSameNil} and Lemma \ref{LemSumSkewComm} we deduce 
\begin{equation}\label{EqConjSInf}
 e^{-\pi \hbar s} \Eb(\seee_{s,0})^{it} \stgE_{N-s-1} \Eb(\seee_{s,0})^{-it} =F_{\hbar}(e^{-2\pi \hbar s}\Ab\star\Bb^{-1})\Bb F_{\hbar}(e^{-2\pi \hbar s}\Ab\star\Bb^{-1})^*.
\end{equation}
As $x$ commutes with the left hand side by assumption, and as $F_{\hbar}(e^{-s}\Cb)$ converges strongly to $1$ for any positive operator $\Cb$ (by the dominated convergence theorem and \eqref{EqValueF0}), we see by letting $\hbar s$ tend to $-\infty$ on the right hand side of \eqref{EqConjSInf} that $x$ must also commute with $\Bb = \Eb(\seee_{s+1,0})$. 

Similarly, writing now rather 
\[
\stgE_{N-s-1} = \Bb \boxplus \Ab = F_{\hbar}(\Bb\star\Ab^{-1})^*\Ab F_{\hbar}(\Bb\star\Ab^{-1}),
\]
a same argument gives that $x$ must also commute with $\Ab = \boxplus_{N-s-1\leq r <n} \eb(\seee_{s+1,n-r})$. 

We can now conclude by a second induction on the variable $k$: if we assume that $x$ commutes with $\boxplus_{N-s-1\leq r \leq n-k} \eb(\seee_{s+1,n-r})$, the same argument as above, conjugating now however with $\Eb(\seee_{s,k})^{it}$, gives that $x$ also commutes with $\eb(\seee_{s+1,n-k})$ and $\boxplus_{N-s-1\leq r \leq n-k-1} \eb(\seee_{s+1,n-r})$. This finishes the proof part (1) of the lemma.

(2) This follows immediately from part (1).
\end{proof}

We can similarly determine the von Neumann algebra $M = M_{\Fct}$ associated to $\Fct$, by interpreting $M$ as the dual von Neumann algebra associated to $\hat{\Fct}$. First, we have the following analogue of Proposition \ref{PropStandardGen}. We introduce for this the notation \begin{equation}\label{EqGenFs}
\stgF_s := \stgF_{s,s+1} =  \boxplus_{0 \le r \le s-1} \eb(\nee_{s,r}),\qquad 1\leq s\leq n,
\end{equation}
\begin{equation}\label{EqCartanL}
\Lb_s := \eb(\nee_s),\qquad 1 \leq s \leq n.
\end{equation}

\begin{prop}\label{PropStandardGenDual}
\begin{enumerate}
\item
The operators $\{\stgF_s\mid 1\leq s\leq n\}$ generate $L_{\hbar}(\Nilp_N^-)$.
\item The operators $\{\Lb_s,\stgF_s\mid 1\leq s \leq n\}$ generate $L_{\hbar}(\Bor_N^-)$.
\end{enumerate}
\end{prop}
\begin{proof}
The proof is completely similar to the one for Proposition \ref{PropStandardGen}.
\end{proof}

To get an analogue of Theorem \ref{TheovNFGFlip}, we first obtain a direct formula for $\hat{\Fct}$.

\begin{lemma}\label{LemDualFG}
We have 
\begin{equation}\label{EqDualMUFG}
\hat{\Fct} = \hat{\Kc}\hat{\Fc}
\end{equation}
where 
\[
\hat{\Kc} = \overline{\Gauss}(2\sum_{t=1}^n  (-\nwvarpi_t)\otimes (-\swe_{N-t}))
\]
and
\[
\hat{\Fc} =  \prod_{1 \le r \le n}^{\longra} \prod_{1 \le s \le n}^{\longra} \underset{0\le (n-r)-k\le n-s}{\prod_{0 \le k \le s-1}^{\longra}}\overline{\varphi}(-\nwe_{N-s,(n-r)-k}\oplus -\swe_{s,k}).
\]
\end{lemma}
\begin{proof}
This follows immediately from \eqref{EqRealMUAlt}.
\end{proof} 

Now note that we have an isomorphism of skew-symmetric spaces
\begin{equation}\label{EqFormtheta}
\theta\colon  \btd{N} \rightarrow \overline{\btd{N}},\qquad e_{a,b,c} \mapsto \overline{-e_{b,a,c}}
\end{equation}
and that under this isomorphism
\[
- \swe_{s,k} \mapsto \overline{\seee_{s,k}},\qquad - \nwe_{s,k}\mapsto \overline{\nee_{s,k}},\qquad 1\leq s\leq n,0\leq k \leq s. 
\]
This leads to a $*$-isomorphism of von Neumann algebras
\begin{equation}\label{EqThetaIso}
\Theta\colon  L_{\hbar}(\wbtd{N}) \rightarrow \overline{L_{\hbar}(\wbtd{N})},\qquad \eb(v) \mapsto \eb(\theta(v))
\end{equation}
such that
\begin{equation}\label{EqThetaIm}
\Theta(L_{\hbar}(\Bor_N^-)) = \overline{L_{\hbar}(\Bor_N^+)},\qquad \Theta(L_{\hbar}(\Bor_N^+)) = \overline{L_{\hbar}(\Bor_N^-)}.
\end{equation}
On the other hand, from \eqref{EqDualMUFG} and a comparison with \eqref{eq:F-3or} we get for the standard Fock--Goncharov flip that
\begin{equation}\label{EqSelfDual}
(\Theta\otimes \Theta)(\hat{\Fct}) = \overline{\Fct},\qquad (\Theta\otimes \Theta)(\Fct) = \overline{\hat{\Fct}}.
\end{equation}
Then through Corollary \ref{CorUnitRepArb}, we get the following analogue of Theorem \ref{TheovNFGFlip}:

\begin{theorem}\label{TheovNFGFlipOr}
Let $\hbar\in \R^{\times}$, and let $\Fct$ be a Fock--Goncharov flip. Then 
\[
M_{\Fct} = L_{\hbar}(\Bor_N^-).
\]
\end{theorem}

To end this section, we determine how the coproducts \eqref{EqComultiplication} and \eqref{EqComultiplicationDual} act on the standard generators. Note that by Proposition \ref{PropStandardGen}
 and Proposition \ref{PropStandardGenDual}, this then completely determines the coproduct.

\begin{prop} \label{Propstandardcoproducts}
On the generators $\Lb_s,\stgF_s$ of $M$, resp.\ $\Kb_s,\stgE_s$ of $\hat{M}$, we have 
\[
\Delta(\Lb_s) = \Lb_s \otimes \Lb_s,\qquad \hat{\Delta}(\Kb_s) = \Kb_s \otimes \Kb_s,
\]
\[
\Delta(\stgF_s) = (\stgF_s \otimes \Lb_s) \boxplus (1 \otimes \stgF_s),\qquad \hat{\Delta}(\stgE_s) = (\stgE_s \otimes 1)\boxplus (\Kb_s \otimes \stgE_s).
\]
\end{prop}
\begin{proof}
Recall the identity \eqref{EqProjFct}. Then applying $\pi_{i}$ to the first leg of the pentagon identity 
\[
\Fct_{23}\Fct_{12}\Fct_{23}^* = 
\Fct_{12}\Fct_{13},
\]
we find 
\begin{equation}\label{EqCoprodOnGen}
(\id\otimes \hat{\Delta})\bbU^{(i)} = \bbU^{(i)}_{13}\bbU^{(i)}_{12}. 
\end{equation}
If we take $i = s \in I_0$, this gives 
\begin{eqnarray*}
(\id\otimes \hat{\Delta})\Gauss(2\varpi_s \otimes \seee_{N-s}) &=& \Gauss(2\varpi_s \otimes \seee_{N-s})_{13}\Gauss(2\varpi_s \otimes \seee_{N-s})_{12}\\
&=& \Gauss(2\varpi_s \otimes (\seee_{N-s}\otimes 1 \dotplus 1\otimes \seee_{N-s})),
\end{eqnarray*}
so by functional calculus we get 
\[
\hat{\Delta}(\seee_{N-s}) = \seee_{N-s}\otimes 1 \dotplus 1\otimes \seee_{N-s}, 
\]
and hence $\hat{\Delta}(\Kb_s) = \Kb_s \otimes \Kb_s$. Alternatively, we could also verify directly from e.g.\ \eqref{EqRealMU} that 
\[
(\Kb^{-it}\otimes \Kb^{-it})\Fct (\Kb^{it}\otimes 1)= \Fct,\qquad t\in \R. 
\]

Now taking on the other hand $i = (s,0)$, we find from \eqref{EqCoprodOnGen} and \eqref{EqProdFormComp} that 
\begin{multline*}
(\id\otimes \hat{\Delta})\left(\Gauss\left(2\varpi_{(s,k)} \otimes \seee_{N-s}\right)\overline{F}\hr{\eb(f_{(s,k)})\otimes  \stgE_s}\right) \\ = \Gauss\left(2\varpi_{(s,k)} \otimes \seee_{N-s}\right)_{12}\overline{F}\hr{\eb(f_{(s,k)})\otimes  \stgE_s}_{13}\Gauss\left(2\varpi_{(s,k)} \otimes \seee_{N-s}\right)_{12}\overline{F}\hr{\eb(f_{(s,k)})\otimes  \stgE_s}_{12}.
\end{multline*}
Then from \eqref{EqGausEqInv}, \eqref{EqNormalisationfw} and the first part of the proof, this gives 
\begin{eqnarray*}
(\id\otimes \hat{\Delta})\overline{F}\hr{\eb(f_{(s,k)})\otimes  \stgE_s}&=& \overline{F}\hr{\eb(f_{(s,k)})\otimes \Eb(\seee_{N-s})\otimes  \stgE_s} \overline{F}\hr{\eb(f_{(s,k)})\otimes  \stgE_s}_{12}\\
&=& \overline{F}\hr{\eb(f_{(s,k)})\otimes ((\Kb_s \otimes \stgE_s) \boxplus (\stgE_s \otimes 1))}, 
\end{eqnarray*}
where in the last step we used \eqref{EqExpoGenProp}.
Again by functional calculus (cf.\ the proof of Lemma \ref{LemCommQDil}), this let's us conclude that 
\[
\hat{\Delta}(\stgE_s) = (\stgE_s\otimes 1)\boxplus(\Kb_s\otimes \stgE_s).
\]

Now note that, with $\Theta$ as in \eqref{EqThetaIso}, we have 
\[
\overline{\Theta}(\overline{\Kb_{N-s}}) = \Lb_s^{-1},\qquad \overline{\Theta}(\overline{\stgE_{N-s}}) = \Lb_s^{-1}\star \stgF_s. 
\]
Then from \eqref{EqSelfDual}, it follows that 
\[
\Delta(\Lb_s) = \Lb_s \otimes \Lb_s,\qquad \Delta(\Lb_s^{-1} \star \stgF_s) = (\Lb_s^{-1}\star \stgF_s \otimes 1)\boxplus (\Lb_s^{-1}\otimes \Lb_s^{-1}\star \stgF_s),
\]
and hence 
\[
\Delta(\stgF_s) = (\Lb_s\otimes \Lb_s)\star (\Delta(\Lb^{-1}\star \stgF_s)) = (\stgF_s \otimes \Lb_s)\boxplus (1\otimes \stgF_s). 
\]
\end{proof}

\subsection{Modular duality}

In this subsection, we briefly explain how modular duality arises in this setting. 

Let $(V,\epsilon)$ be an $n$-dimensional skew-symmetric space. Let $\pi_{\hbar}$ be the standard $\hbar$-representation on $L^2(V)$ as in Example \ref{ExaStandardRep}, which we now explicitly index by $\hbar$ for differentation. We also label the associated infinitesimal generator by $\hbar$, so $\eb_{\hbar}(v)^{it} = \pi_{\hbar}(tv)$ for $t\in \R,v\in V$. 

Let $u_{\hbar}$ be the unitary 
\[
u_{\hbar} \colon L^2(V) \rightarrow L^2(V),\qquad (ug)(w) = |\hbar|^{n/2}g(\hbar w) ,\qquad g\in L^2(V),w\in V. 
\]
Then we have a normal $*$-isomorphism 
\begin{equation}\label{EqIsohbarhinvbar}
\Theta_{\hbar} = \Ad(u_{\hbar})\colon L_{\hbar^{-1}}(V) \cong L_{\hbar}(V),\qquad \eb_{1/\hbar}(v) \mapsto u_{\hbar}\eb_{1/\hbar}(v)u_{\hbar}^* = \eb_{\hbar}(v)^{1/\hbar} = \eb_{\hbar}(\hbar^{-1}v).
\end{equation}

\begin{theorem}
Let $(L_{\hbar}(\Bor_N^-),\Delta_{\hbar})$ be the Fock--Goncharov von Neumann bialgebra at parameter $\hbar$. Then $\Theta_{\hbar}$ as in \eqref{EqIsohbarhinvbar} provides an isomorphism
\[
(L_{1/\hbar}(\Bor_N^-),\Delta_{1/\hbar}) \cong (L_{\hbar}(\Bor_N^-),\Delta_{\hbar}).
\]
\end{theorem}
\begin{proof}
Note that when writing $\vb_{\hbar} = \ln(\eb_{\hbar}(v))$, the isomorphism $\Theta_{\hbar}$ is simply acting via 
\[
\Theta_{\hbar}(\vb_{1/\hbar}) = \hbar^{-1}\vb_{\hbar}. 
\]

Consider now the decomposition \eqref{EqRealMU} for the Fock--Goncharov flip. Then on the one hand, we see that 
\[
(\Theta_{\hbar}\otimes \Theta_{\hbar})(\Kc_{\hbar^{-1}}) = \Gauss_{1/\hbar}\left(2\hbar^{-2}\sum_{t=1}^n \nevarpi_t \otimes \seee_{N-t}\right) =  \Gauss_{\hbar}\left(2\sum_{t=1}^n \nevarpi_t \otimes \seee_{N-t}\right) = \Kc_{\hbar}. 
\]
On the other hand, from \eqref{EqModDualityV} it follows that
\[
\varphi_{1/\hbar}(z/\hbar) = \varphi_{\hbar}(z),\qquad z\in \R,
\]
so for any of the factors $B^{b,r,i} = B_{\hbar}^{b,r,i}$ of $\Fc$, defined in \eqref{DefBbri}, we get 
\begin{multline*}
(\Theta_{\hbar}\otimes \Theta_{\hbar})(B_{1/{\hbar}}^{b,r,i}) = 
(\Theta_{\hbar}\otimes \Theta_{\hbar})(\varphi_{1/\hbar}(\nee_{b-i+1,b-r}\oplus \seee_{n-b+i,n-b})) \\ = \varphi_{1/\hbar}(\hbar^{-1}(\nee_{b-i+1,b-r}\oplus \seee_{n-b+i,n-b}))  =  \varphi_{\hbar}(\nee_{b-i+1,b-r}\oplus \seee_{n-b+i,n-b}) = B_{\hbar}^{b,r,i},
\end{multline*}
hence $(\Theta_{\hbar}\otimes \Theta_{\hbar})(\Fc_{1/\hbar}) = \Fc_{\hbar}$. 

It follows from the above that $(\Theta_{\hbar}\otimes \Theta_{\hbar})(\Fct_{1/\hbar}) = \Fct_{\hbar}$, and so in particular we get for any $x\in L_{1/\hbar}(\Bor_N^-)$ that 
\[
(\Theta_{\hbar}\otimes \Theta_{\hbar})(\Delta_{1/\hbar}(x)) = (\Theta_{\hbar}\otimes \Theta_{\hbar})(\Fct_{1/\hbar}^*(1\otimes x)\Fct_{1/\hbar}) = \Fct_{\hbar}^*(1\otimes \Theta_{\hbar}(x))\Fct_{\hbar} = \Delta_{\hbar}(\Theta_{\hbar}(x)),
\]
finishing the proof.
\end{proof}

\section{Quantized Borel groups as locally compact quantum groups}

In this section, we show that the Fock--Goncharov von Neumann bialgebra $(L_{\hbar}(\Bor_N^-),\Delta)$ is a locally compact quantum group, in the sense of \cite{KV03}.

\subsection{Invariant weights associated to modular multiplicative unitaries}

For a brief general overview of weight theory for von Neumann algebras, we refer to Appendix \ref{AppWeights}.

Fix $\Hc$ a Hilbert space with associated trace $\Tr$ on $\Bc(\Hc)$ as in \eqref{EqTrace}. For $\Tb$ a positive operator on $\Hc$, consider the deformation $\Tr_{\Tb}$ as in \eqref{EqDeformTrace}, so with $\chi_{[0,r]}$ the characteristic function of $[0,r]$ and 
\[
\Tb_r = \Tb \chi_{[0,r]}(\Tb),
\]
we have 
\[
\Tr_{\Tb}(x) = \sup_{r>0} \sum_i \omega_{\Tb_r^{1/2}e_i,\Tb_r^{1/2}e_i}(x) = \sup_{r>0} \sum_i \|x^{1/2}\Tb_r^{1/2}e_i\|^2,\qquad x\geq 0,
\] 
with $\{e_i\}$ any orthonormal basis of $\Hc$. 

Let $\Wbb$ be a multiplicative unitary on $\Hc\otimes \Hc$, and $(\Wbb,\Qb,\hat{\Qb},\wWbb)$ a modular datum. Rewriting \eqref{EqRightInvOp}, we see that
\begin{equation}\label{EqRightInvOpAlt}
\langle \xi'\otimes \eta,\Wbb(\eta'\otimes \Qb\xi)\rangle = \langle \xi'\otimes \overline{\xi},\overline{\wWbb}^*(\eta'\otimes \overline{\Qb\eta})\rangle,\qquad \forall \xi,\eta\in \msD(\Qb),\xi',\eta'\in \Hc.
\end{equation}

The following is then a variation on \cite[Theorem 1.1]{Wor03}, see also \cite{VVD03}. 

\begin{prop}\label{PropTechn}
Let $\Wbb\in \Bc(\Hc\otimes \Hc)$ be a modular multiplicative unitary with respect to $(\Qb,\hat{\Qb})$, with associated von Neumann algebras $M,\hat{M}$ as in \eqref{EqFormvN} and \eqref{EqFormvNDual}. Let $a \in \hat{M}'$. Then the normal weight 
\[
\varphi = \Tr_{\Qb^2,a}\colon  \Bc(\Hc)_+ \rightarrow [0,\infty],\qquad x \mapsto \Tr_{\Qb^2}(axa^*)
\]
is left invariant with respect to 
\[
\Delta = \Delta_{\Wbb}\colon  \Bc(\Hc)\rightarrow \Bc(\Hc)\bar{\otimes}\Bc(\Hc),\qquad x \mapsto\Delta(x) = \Wbb^*(1\otimes x)\Wbb,\qquad x\in \Bc(\Hsp),
\]
in the sense that 
\begin{equation}\label{EqLeftInvLem}
\varphi((\omega \otimes \id)\Delta(x)) = \varphi(x),\qquad \forall x \in \mfm_{\varphi}^+, \textrm{ all states }\omega \in \Bc(\Hc)_*^+.
\end{equation}
\end{prop} 

\begin{proof}
Assume that $x\in \mfm_\varphi^+$. If then $\xi$ is a unit vector in $\Hc$, it is sufficient to show \eqref{EqLeftInvLem} with respect to $\omega = \omega_{\xi,\xi}$. Since $x$ is the supremum of the directed set of sums of positive rank $1$ operators below $x$, it is moreover sufficient to prove \eqref{EqLeftInvLem} for $x = \theta_{\eta,\eta}$ (as in \eqref{EqRank1Op}) with $\|\eta\|=1$ and $axa^* = \theta_{a\eta,a\eta} \in \mfm_{\Tr_{\Qb^2}}^+$. Note that the latter implies $a\eta\in \msD(\Qb)$.

Since $\Wbb \in M \bar{\otimes} \hat{M}$, we now compute that indeed, for $\{e_n\}$ an orthonormal basis of $\Hc$, 
\begin{eqnarray*}
\varphi((\omega_{\xi,\xi}\otimes \id)\Delta_{\Wbb}(x)) &=& \sup_{r>0} \sum_n (\omega_{\xi,\xi}\otimes \omega_{a^*\Qb_re_n,a^*\Qb_re_n})(\Wbb^*(1\otimes \theta_{\eta,\eta})\Wbb)\\
 &=& \sup_{r>0} \sum_{m,n} (\omega_{\xi,\xi}\otimes \omega_{a^*\Qb_re_n,a^*\Qb_re_n})(\Wbb^*(\theta_{e_m,e_m}\otimes \theta_{\eta,\eta})\Wbb)\\
&=& \sup_{r>0} \sum_{m,n} |\langle e_m \otimes \eta,\Wbb(\xi\otimes a^*\Qb_re_n)\rangle|^2\\
&=& \sup_{r>0} \sum_{m,n} |\langle e_m \otimes a\eta,\Wbb(\xi\otimes \Qb_re_n)\rangle|^2\\
&\underset{\eqref{EqRightInvOpAlt}}{=}& \sup_{r>0} \sum_{m,n} |\langle e_m \otimes \overline{\chi_{[0,r]}(\Qb)e_n},\overline{\wWbb}^*(\xi\otimes \overline{\Qb a\eta})\rangle|^2\\
&=& \sup_{r>0} \|(1\otimes \overline{\chi_{[0,r]}(\Qb)})\overline{\wWbb}^*(\xi\otimes \overline{\Qb a\eta})\|^2 \\
&=&  \|\overline{\wWbb}^*(\xi\otimes \overline{ \Qb a\eta})\|^2 \\
&=& \|\xi\otimes \overline{ \Qb a\eta}\|^2\\
&=& \|\Qb a\eta\|^2\\
&=& \varphi(x). 
\end{eqnarray*}
\end{proof}

\subsection{Invariant weights on Rieffel deformations}

To apply Proposition \ref{PropTechn} in the setting of the Fock--Goncharov flip, we need some information on weights on twisted group von Neumann algebras of skew-symmetric spaces.

We start with the following well-known observations. Fix $\hbar\in \R^{\times}$, and fix $V$ a skew-symmetric space with its standard unitary $\hbar$-representation $(L^2(\hat{V}),\pi)$ and associated von Neumann algebra
\[
L_{\hbar}(V) \subseteq \Bc(L^2(\hat{V}))
\]
as in \eqref{EqvNQT}. We consider also the (ordinary) unitary representations of $V$ and $\hat{V}$ on $L^2(\hat{V})$ given by 
\begin{equation}\label{EqExpoGen}
(e_{v}f)(\hat{w}) = e^{2\pi i v\cdot \hat{w}}f(\hat{w}),\qquad f\in L^2(\hat{V}),v\in V,\hat{w}\in \hat{V},
\end{equation}
\begin{equation}\label{EqTransDual}
(\theta_{\hat{v}}f)(\hat{w}) = f(\hat{w}-\hat{v}),\qquad f\in L^2(\hat{V}),\hat{v},\hat{w}\in \hat{V}.
\end{equation}
Note that from \eqref{EqStandardDefRep} we get
\begin{equation}\label{EqStructurepi}
\pi(v) = e_v \theta_{-\Jc v},
\end{equation}
with $\Jc$ as in \eqref{EqJOperator}.

\begin{lemma}
There exists a unique $\sigma$-weakly continuous action $\alpha$ of $\hat{V}$ on $L_{\hbar}(V)$ such that 
\begin{equation}\label{EqAlphaActGroup}
\alpha_{\hat{v}}(\pi(v)) = e^{-2\pi i \hat{v} \cdot v}\pi(v),\qquad \forall v\in V,\hat{v} \in\hat{V}. 
\end{equation}
Moreover, this action is ergodic, and for $f \in \mathcal{S}(\hat{V})$ one has
\begin{equation}\label{EqAlphaActAlg}
\alpha_{\hat{v}}(f_{\hbar}) = (f_{\hat{v}})_{\hbar},\qquad f_{\hat{v}}(\hat{w}) = f(\hat{w}-\hat{v}),\qquad f\in \Sc(\hat{V}),\hat{v},\hat{w}\in \hat{V}.
\end{equation}

\end{lemma}
\begin{proof}
With the $\theta_{\hat{v}}$ as in \eqref{EqTransDual}, one sees that 
\[
x \mapsto \alpha_{\hat{v}}(x) =  \theta_{\hat{v}}x\theta_{\hat{v}}^*
\]
satisfies \eqref{EqAlphaActGroup}, and hence indeed implements a $\hat{V}$-action on $L_{\hbar}(V)$. Using e.g.\ \eqref{EqExtRep}  and  \eqref{EqStandardDefRep}, one finds \eqref{EqAlphaActAlg}.

To see that the action is ergodic, consider also the unitaries
\[
\pi'(v) = e_v\theta_{\Jc v},\qquad v\in V.
\]
From \eqref{EqStructurepi} one sees that
\[
[\pi(v),\pi'(w)] = 0,\qquad v,w\in V,
\]
and hence the $\pi'(w)$ also commute with all $x\in L_{\hbar}(V)$. So, if $x\in L_{\hbar}(V)$ and $\alpha_{\hat{v}}(x) = x$ for all $\hat{v}\in \hat{V}$, then $x$ commutes with all $\theta_{\hat{v}}$ and all $e_v\theta_{\Jc v}$, hence with all $\theta_{\hat{v}}$ and $e_v$. But the latter generate $\Bc(L^2(\hat{V}))$ as a von Neumann algebra, so $x$ must be a scalar.
\end{proof}

Recall now the construction and notation from \eqref{EqIntegratedWeight}. 

\begin{lemma}
The weight
\[
\int_{\hat{V}}^{\hbar} := \varphi_{\alpha}\colon  L_{\hbar}(V)_+ \rightarrow [0,\infty] 
\]
associated to $\alpha$ is a tracial nsf weight. 
\end{lemma}
\begin{proof}
Faithfulness of $\varphi_{\alpha}$ follows from the general theory and construction of $\varphi_{\alpha}$. One can also prove semi-finiteness with general machinery, but let us give a direct approach: first note that if $f,g\in\Sc(\hat{V}) \subseteq L^2(\hat{V})$, then 
\[
k_{f,g}\colon  v \mapsto \langle f,\pi(v)g\rangle = \int_{\hat{V}} e^{2\pi i v\cdot \hat{w}}g(\hat{w}+\Jc v)\overline{f(\hat{w})}\rd \hat{w} = \int_V e^{-2\pi i v\cdot \Jc w}\hat{g}(w-v)\overline{\hat{f}(w)} \rd w
\]
lies in $\Sc(V)$. Take then $f,g \in \Sc(\hat{V})$ with $\|g\|_2 = 1$, and put 
\[
h = f^*\times_{\hbar} f,\qquad k = k_{g,g}.
\]
Then with $\omega = \omega_{g,g}$, and noting that $k\hat{h}\in \Sc(V)$, we find 
\[
\int_{\hat{V}}^{\hbar} h_{\hbar} =  \int_{\hat{V}} \omega(\alpha_{\hat{v}}(h_{\hbar})) \rd \hat{v} =  \int_{\hat{V}} \int_V k(v)\hat{h}(v) e^{-2\pi i v\cdot \hat{v}} \rd v  \rd \hat{v} = k(0)\hat{h}(0) = \hat{h}(0) = \int_{\hat{V}} h(\hat{v})\rd\hat{v}.
\]
In particular, $h_{\hbar} \in \mfm_{\varphi_{\alpha}}^+$, and since such $h_{\hbar}$ generate $L_{\hbar}(V)$ we find that $\int_{\hat{V}}^{\hbar}$ is semi-finite.

Finally, note that the definition of $\int_{\hat{V}}^{\hbar}$ and \eqref{EqAlphaActGroup} immediately give that 
\[
\int_{\hat{V}}^{\hbar}\pi(v)x\pi(v)^*  = \int_{\hat{V}}^{\hbar} x,\qquad x\in L_{\hbar}(V)_+,v\in V. 
\]
This implies that the modular automorphism group of $\int_{\hat{V}}^{\hbar}$ vanishes on the $\pi(v)$. But as the latter generate $L_{\hbar}(V)$, it follows that the modular automorphism group of $\int_{\hat{V}}^{\hbar}$ is trivial, i.e.\ $\int_{\hat{V}}^{\hbar}$ is a trace.
\end{proof}
In particular, we record from the above proof that $f,g \in \Sc(\hat{V})$ satisfy $f_{\hbar}g_{\hbar}\in \mfm_{\varphi_{\alpha}}$ with 
\begin{equation}\label{EqIntegralDef}
\int_{\hat{V}}^{\hbar} f_{\hbar}g_{\hbar} = \int_{\hat{V}}f(\hat{v})g(\hat{v})\rd \hat{v}. 
\end{equation}

\begin{lemma}\label{LemIdentGNS}
The GNS-space of $\int_{\hat{V}}^{\hbar}$ can be uniquely identified with $L^2(\hat{V})$ in such a way that the associated GNS-map satisfies
\begin{equation}\label{EqGNSIdent}
\Lambda(f_{\hbar}) = f\in L^2(\hat{V}),\qquad f\in \Sc(\hat{V}). 
\end{equation}
\end{lemma}
\begin{proof}
We use below the notion of a (unimodular) Hilbert algebra with respect to a trace, see \cite[Chapter VI]{Tak03}.

The $*$-algebra $(\mathcal{S}(\hat{V}),\times_{\hbar},*)$ forms a Hilbert algebra with respect to the $2$-norm
\[
\|f\|_2^2 = \int_{\hat{V}}^{\hbar} f_{\hbar}^*f_{\hbar} = \int_{\hat{V}} |f(\hat{v})|^2\rd \hat{v}, 
\]
where the last equality is \eqref{EqIntegralDef}. From e.g.\ \cite[Proposition VIII.3.16]{Tak03}, one sees that the associated weight on $L_{\hbar}(V)$ must equal $\int_{\hat{V}}^{\hbar}$, which is enough to conclude the lemma. 
\end{proof}

We suspect the following lemma is known, but we do not know a direct reference. 
\begin{lemma}\label{LemCharCorrSub}
Assume $Z \subseteq V$ is a subspace, and put 
\[
\hat{W} = \Ker(Z) = \{\hat{v}\in \hat{V}\mid \hat{v}_{\mid Z} = 0\}\subseteq \hat{V}.
\]
Then for $x \in L_{\hbar}(V)$, one has 
\begin{equation}\label{EqRightCoact}
L_{\hbar}(Z) = \{x\in L_{\hbar}(V) \mid \alpha_{\hat{w}}(x) = x\textrm{ for all }\hat{w}\in \hat{W}\}.
\end{equation}
\end{lemma}

The proof of this lemma will unfortunately be a bit technical, but it is simply a quantized version of the following well-known fact: we can identify
\[
L^{\infty}(\hat{V}/\hat{W}) = \{f\in L^{\infty}(\hat{V})\mid \forall \hat{w}\in \hat{W}, \textrm{for almost all }\hat{v}\in \hat{V}: f(\hat{v}+\hat{w}) = f(\hat{v})\}. 
\]
Indeed, if we now identify $\hat{Z}= \hat{V}/\hat{W}$ and consider the Pontryagin dual picture, we get 
\[
L(Z) = L^{\infty}(\hat{Z}),\quad L(V) = L^{\infty}(\hat{V}),
\]
which is the classical counterpart to \eqref{EqRightCoact}.

\begin{proof}
Write the right hand side of \eqref{EqRightCoact} as $N$. Then clearly $N \subseteq L_{\hbar}(V)$ is a von Neumann subalgebra with $L_{\hbar}(Z) \subseteq N$. It is also clear that the $\hat{V}$-action $\alpha$ restricts to $N$, where it factorizes over $\hat{Z} = \hat{V}/\hat{W}$. Let us specifically write this action of $\hat{Z}$ as $\beta$, which then restricts to the usual $\alpha$-action of $\hat{Z}$ on $L_{\hbar}(Z)$. 

Obviously the $\hat{Z}$-action $\beta$ on $N$ is still ergodic. It is also still integrable, as the associated weight 
\[
\varphi_{\beta}(x) = \int_{\hat{Z}} \beta_{\hat{z}}(x) \rd \hat{z},\qquad x\in \Nilp_+
\]
restricts to the tracial semi-finite weight $\int_{\hat{Z}}^{\hbar}$ on $L_{\hbar}(Z)$, so that all elements of the form
\[
xyz,\qquad y\in N,x,z\in \mfm_{\int_{\hat{Z}}^{\hbar}} =(\mfn_{\int_{\hat{Z}}^{\hbar}})^2  = (\mfn_{\int_{\hat{Z}}^{\hbar}})^*\mfn_{\int_{\hat{Z}}^{\hbar}}
\]
are integrable. 

The same argument as in the proof of traciality in Lemma \ref{LemCharCorrSub} moreover gives that the modular automorphism $\sigma_t^{\varphi_{\beta}}$ is trivial on $L_{\hbar}(Z)$. By Takesaki's Theorem \cite[Theorem X.4.2]{Tak03} it follows that there exists a conditional expectation $ E\colon  N \rightarrow L_{\hbar}(Z) $ such that 
\[
\varphi_{\beta} = \int_{\hat{Z}}^{\hbar} \circ E.
\]
This conditional expectation is uniquely determined via 
\[
E(x)p = pxp,\qquad x\in N,
\]
for $p$ the projection of $L^2(N)$ onto $L^2(\hat{Z})\subseteq L^2(N)$, with this isometric inclusion given via
\[
L^2(\hat{Z}) \subseteq L^2(N),\quad f\mapsto \Lambda_{\varphi_{\beta}}(f_{\hbar}),\qquad f\in \Sc(\hat{Z}).
\]

To show that $L_{\hbar}(Z) = N$, it suffices now to show that $p = 1$. But note that by ergodicity and integrability, one can uniquely define an isometric map 
\begin{multline}\label{EqGaloisMap}
\Gc\colon  L^2(N) \otimes L^2(N) \rightarrow L^2(\hat{Z},L^2(N)),\quad \Lambda_{\varphi_{\beta}}(x) \otimes \Lambda_{\varphi_{\beta}}(y) \mapsto \left(\hat{z} \mapsto \Lambda_{\varphi_{\beta}}(\beta_{-\hat{z}}(y)x)\right),\\ \qquad x,y \in \mfn_{\varphi_{\beta}}.
\end{multline}
However, using the above identification \eqref{EqGNSIdent} for the GNS-map of $(L_{\hbar}(Z),\int_{\hat{Z}}^{\hbar})$, one sees by using 
\[
\beta_{\hat{z}}(f_{\hbar}) = \alpha_{\hat{z}}(f_{\hbar}) = (f_{\hat{z}})_{\hbar},\qquad f\in \Sc(\hat{Z}),\hat{z}\in \hat{Z},
\] 
that the restriction of $\Gc$ to $L^2(N) \otimes L^2(\hat{Z})$ satisfies
\[
\Gc(\Lambda_{\varphi_{\beta}}(x) \otimes \Lambda(f_\hbar))=  \left(\hat{z}\mapsto \Lambda_{\varphi_{\beta}}(((f_{-\hat{z}})_{\hbar})x)\right),\qquad x\in \mfn_{\varphi_{\beta}},f\in \Sc(\hat{Z}).
\]
So, conjugating with the Fourier transform $ \msF $ and identifying $L^2(N) \otimes L^2(Z) \cong L^2(Z,L^2(N))$, we see that $\hat{\Gc} = (1\otimes \msF)\Gc(1\otimes \msF^*)$ gives the unitary transformation 
\[
\hat{\Gc}\colon  L^2(Z,L^2(N)) \rightarrow L^2(Z,L^2(N)),\qquad F \mapsto \left(z \mapsto \pi(z)F(z)\right).
\]
Since  $\Gc$ is isometric on $L^2(N) \otimes L^2(N)$ and $\Gc_{\mid L^2(N)\otimes L^2(\hat{Z})}$ is unitary onto the range of $\Gc$,  this implies that indeed $L^2(N) = L^2(\hat{Z})$, finishing the proof.
\end{proof}

\begin{lemma}
Let $Z \subseteq V$ be a vector subspace, and put $\hat{W} = \Ker(Z) \subseteq \hat{V}$. There exists a unique nsf operator-valued weight 
\begin{equation}\label{EqOVW}
\msE^{\hbar}_{\hat{W}}\colon  L_{\hbar}(V)_+\rightarrow L_{\hbar}(Z)_+^{\ext}
\end{equation}
such that, for any normal state $\omega$ on $L_{\hbar}(V)$, it holds that   
\[
\omega_{\mid L_{\hbar}(Z)}\left(\msE^{\hbar}_{\hat{W}} (x)\right) = \int_{\hat{W}} \omega(\alpha_{\hat{w}}(x))\rd \hat{w},\qquad x\in L_{\hbar}(V)_+. 
\]
Moreover, with the Lebesgue measure on $\hat{Z}$ suitably normalized, it holds that 
\[
\int^{\hbar}_{\hat{V}} = \int^{\hbar}_{\hat{Z}}\circ\, \msE^{\hbar}_{\hat{W}}.
\]
\end{lemma}
\begin{proof}
Considering $L_{\hbar}(V)$ with the restricted $\hat{W}$-action $\alpha_{\mid \hat{W}}$, we simply define 
\[
\msE_{\hat{W}}^{\hbar} = T_{\alpha_{\hat{W}}}
\]
as in \eqref{EqOpValWeight}, which by Lemma \ref{LemCharCorrSub} is indeed an operator valued weight onto
\[
L_{\hbar}(V)^{\alpha_{\mid \hat{W}}} = L_{\hbar}(Z).
\]

Normalize now the Lebesgue measure on $\hat{Z}$ such that for positive $f\in C_c(\hat{V})$ we have
\[
\int_{\hat{Z}} E_{\hat{W}}(f)(\hat{z})\rd\hat{z} = \int_{\hat{V}}f(\hat{v})\rd\hat{v},\qquad \textrm{where } E_{\hat{W}}(f)(\hat{v}+\hat{W}) = \int_{\hat{W}}f(\hat{v}+\hat{w})\rd\hat{w}.
\]
Choose any linear splitting $s\colon  \hat{Z}\rightarrow \hat{V}$ of the quotient map $\hat{V}\twoheadrightarrow \hat{Z}$. Noting that $\alpha_{s(\hat{z})}$ restricts to $\alpha_z$ on $L_{\hbar}(Z) \subseteq L_{\hbar}(V)$, we then find for $x\in L_{\hbar}(V)_+$ and $\omega \in L_{\hbar}(V)_*^+$ a state: 
\begin{eqnarray*}
\int_{\hat{Z}}^{\hbar}\left(\msE_{\hat{W}}^{\hbar} (x)\right) &=&  \int_{\hat{Z}} (\omega \circ \alpha_{\hat{z}})\left(\msE_{\hat{W}}^{\hbar} (x)\right) \rd \hat{z} \\
&=& \int_{\hat{Z}} \int_{\hat{W}} (\omega\circ \alpha_{s(\hat{z})}\circ \alpha_{\hat{w}})(x) \rd \hat{w} \rd \hat{z}\\
&=& \int_{\hat{V}} (\omega\circ \alpha_{\hat{v}})(x) \rd \hat{v}\\
&=& \int_{\hat{V}}^{\hbar} x.
\end{eqnarray*}
\end{proof}

\begin{remark}\label{RemWeight}
If $W \subseteq V$ is any vector space complementary to $Z$, so $V = Z\oplus W$, we easily see that 
\[
\msE_{\hat{W}}^{\hbar} (x) = \left(\int_{\hat{W}}^{\hbar}x\right)1,\qquad x\in L_{\hbar}(W)_+ \subseteq L_{\hbar}(V)_+,
\]
where we consider $\hat{W}$ as the dual space of $W$ upon restriction.
\end{remark} 

In the following, we will want to further perturb the weight $\int_{\hat{V}}^{\hbar}$. Namely, if $v\in V$, we can form the nsf weight
\[
\int_{\hat{V}}^{\hbar,\eb(v)} = \varphi_{\alpha,\eb(v)}
\]
as defined through \eqref{EqDeformTrace}. This is meaningful as indeed $\int_{\hat{V}}^{\hbar}$ is an nsf tracial weight and $\eb(v)$ is affiliated to $L_{\hbar}(V)$. 
 
Recall now from Lemma \ref{LemTypeIRieff} that $L_{\hbar}(V)$ is a type $I$-factor if $V$ is a symplectic. Let $\Tr$ be the trace on $L_{\hbar}(V)$, normalized so that minimal projections have trace $1$. By uniqueness of the trace, we must have (up to suitable normalisation) 
\[
\Tr = \int_{\hat{V}}^{\hbar}.
\]

The following is the key result that we need.
\begin{lemma}\label{LemInvWeightQT}
Assume $V$ is symplectic. Let $Z \subseteq V$, put $\hat{W} = \Ker(Z) \subseteq \hat{V}$, and choose a subspace $W \subseteq V$ such that
\[
V = Z\oplus W
\]
as vector spaces. Assume $z\in Z,w\in W$. Then there exists $a \in L_{\hbar}(W) \subseteq L_{\hbar}(V)$ such that  
\begin{equation}\label{EqFact}
\int_{\hat{Z}}^{\hbar,\eb(z)} =  \left(\Tr_{\Eb(z+w)}(a-a^*)\right)_{\mid L_{\hbar}(Z)_+}.
\end{equation}
\end{lemma} 
\begin{proof}
In the following, we again identify $\hat{W}$ with the dual of $W$ by the restriction map.

Consider the $\sigma$-weakly continuous one-parameter group of automorphisms 
\[
\gamma_t\colon  L_{\hbar}(W) \rightarrow L_{\hbar}(W),\qquad x \mapsto \Eb(z)^{it}x \Eb(z)^{-it}.
\]
This is clearly well-defined, and scales the weight $\int_{\hat{W}}^{\hbar,\Eb(w)}$:
\[
\int_{\hat{W}}^{\hbar,\Eb(w)} \circ \gamma_t = e^{2\pi\hbar t(w,z)}\int_{\hat{W}}^{\hbar,\Eb(w)}.
\]
Taking a non-zero positive $c \in \mfn_{\int_{\hat{W}}^{\hbar}}$ and putting 
\[
b = \int_{\R} e^{-t^2/2} \gamma_t (\widetilde{c})\rd t,\qquad \widetilde{c} = \chi_{[0,1]}(\Eb(w))c \chi_{[0,1]}(\Eb(w)),
\]
we see that we can find a non-zero $b \in L_{\hbar}(W)$ which is square $\int_{\hat{W}}^{\hbar,\Eb(w)}$-integrable, with $\Eb(w)^{1/2}b$ bounded, and with $b$ and  $\Eb(w)^{1/2}b$ analytic for $\gamma_t$. Normalize $b$ such that 
\begin{equation}\label{EqNormb}
\int_{\hat{W}}^{\hbar,\Eb(w)}bb^* =1,
\end{equation}
and put then 
\[
a = \gamma_{i/2}(b) \in L_{\hbar}(W).
\]
Then also $\Eb(w)^{1/2}a$ is bounded, and $\Eb(w)^{1/2}a$ is analytic for $\gamma_t$.

Now for $y \in L_{\hbar}(V)$ and $\Tb \in L_{\hbar}(V)_+^{\ext}$, let us write $y^*\cdot \Tb \cdot y \in L_{\hbar}(V)_+^{\ext}$ for the element 
\[
\omega(y^*\cdot \Tb\cdot y) = \omega(y^*-y)(\Tb),\qquad \omega \in M_*^+.
\]
We then claim that 
\begin{equation}\label{EqEqualPosOp}
\msE_{\hat{W}}^{\hbar}(a^*\cdot \eb(z+w)\cdot a) = \eb(z). 
\end{equation}
Indeed, take $\xi \in \msD(\eb(z)^{1/2}) \subseteq L^2(\hat{V})$. Then we compute, using the notation \eqref{EqTransDual}, that
\begin{eqnarray}
\omega_{\xi,\xi}(\eb(z)) 
&=& \|\eb(z)^{1/2}\xi\|^2 \\
&\underset{\eqref{EqNormb}}{=}&
\left(\int_{\hat{W}}^{\hbar,\Eb(w)} bb^*\right) \|\eb(z)^{1/2}\xi\|^2\\
&=& \left(\int_{\hat{W}}^{\hbar} (\Eb(w)^{1/2}b)^*(\Eb(w)^{1/2}b)\right)  \|\eb(z)^{1/2}\xi\|^2\\
&\underset{\textrm{Remark } \ref{RemWeight}}{=}& 
\omega_{\Eb(z)^{1/2}\xi,\Eb(z)^{1/2}\xi}\left(\msE_{\hat{W}}^{\hbar} \left((\Eb(w)^{1/2}b)^*(\Eb(w)^{1/2}b)\right) \right)\\
&=& \int_{\hat{W}} \|\Eb(w)^{1/2}b\theta_{\hat{w}}^* \eb(z)^{1/2}\xi\|^2\rd \hat{w}\\
&=& \int_{\hat{W}} \|\Eb(w)^{1/2}b \eb(z)^{1/2}\theta_{\hat{w}}^*\xi\|^2\rd \hat{w}\\
 &=& \int_{\hat{W}} \|\Eb(w)^{1/2} \Eb(z)^{1/2}a \theta_{\hat{w}}^*\xi\|^2 \rd\hat{w} \\
 &=& \int_{\hat{W}} \| \Eb(z+w)^{1/2}a \theta_{\hat{w}}^*\xi\|^2 \rd\hat{w} \\
&=& 
\omega_{\xi,\xi}\left(\int_{\hat{W}}^{\hbar} a^*\cdot \eb(z+w)\cdot a\right).
\end{eqnarray}
Writing the left hand side of \eqref{EqEqualPosOp} as $\Tb$, this shows by \eqref{EqEqualFormPos} that $\Eb(z)^{1/2} \subseteq \Tb^{1/2}$, which by self-adjointness of both sides is sufficient to conclude that they are equal. Hence equality holds in \eqref{EqEqualPosOp}.

Choose now $x\in L_{\hbar}(Z)$. Then we compute that 
\begin{eqnarray*}
\Tr_{\Eb(z+w)}(axx^*a^*) &=& \int_{\hat{V}}^{\hbar} x^*a^* \cdot  \eb(z+w)\cdot ax\\
&=& \int_{\hat{Z}}^{\hbar} \msE_{\hat{W}}^{\hbar} \left(x^*\cdot (a^* \cdot \eb(z+w)\cdot a)\cdot x\right)\\
&=& \int_{\hat{Z}}^{\hbar}x^*\cdot \left( \msE_{\hat{W}}^{\hbar} \left( a^*\cdot \eb(z+w)\cdot a\right)\right)\cdot x\\
&\underset{\eqref{EqEqualPosOp}}{=}& \int_{\hat{Z}}^{\hbar}x^*\cdot \eb(z)\cdot x\\
&=& \int_{\hat{Z}}^{\hbar,\Eb(z)}xx^*.
\end{eqnarray*} 
\end{proof}

\begin{cor}\label{CorInvWeightQT}
Assume $V$ is symplectic, and assume
\[
V = Z\oplus W
\]
as vector spaces. Let $v\in V$. Then there exists $a \in L_{\hbar}(W) \subseteq L_{\hbar}(V)$ such that the restriction of 
\[
\Tr_{\Eb(v),a} = \Tr_{\Eb(v)}(a-a^*)
\]
to $L_{\hbar}(Z)$ is an nsf weight.
\end{cor} 

\begin{remark}\label{RemGNSRest}
If we denote 
\[
\varphi= (\Tr_{\Eb(v),a})_{\mid L_{\hbar}(Z)},
\]
then clearly the restriction of the GNS-map $\Lambda$ of $\Tr_{\Eb(v),a}$ to $\mfn_{\varphi}$ is a GNS-map for $\varphi$, giving an embedding 
\[
L^2(L_{\hbar}(Z),\varphi)\subseteq L^2(L_{\hbar}(V),\Tr_{\Eb(v),a}),\qquad \Lambda(x)\mapsto \Lambda(x),\qquad x\in \mfn_{\varphi},
\]
which intertwines the standard $L_{\hbar}(Z)$-bimodule structures (cf.\ \cite[Theorem IX.4.2]{Tak02}).
\end{remark}

\subsection{Unitary corepresentations}

In this section, we revisit the notion of modularity in the following setting. 

\begin{defn}
Let $\Wbb$ be a multiplicative unitary, represented on a Hilbert space $\Hc$. A \emph{unitary corepresentation} for $\Wbb$ on a Hilbert space $\Gc$ is a unitary $\Ubb$ on $\Hc \otimes \Gc$ such that 
\begin{equation}\label{EqPentagonCorep}
\Ubb_{23}\Wbb_{12} = \Wbb_{12}\Ubb_{13}\Ubb_{23}.
\end{equation}
\end{defn}

Note that if $\Wbb$ is modular, with associated  von Neumann bialgebra $(M,\Delta)$, then slicing \eqref{EqPentagonCorep} on the second and third leg reveals that 
\[
\Ubb \in M \bar{\otimes} \Bc(\Gc), 
\]
and \eqref{EqPentagonCorep} can then be rewritten as 
\[
(\Delta\otimes \id)\Ubb = \Ubb_{13}\Ubb_{23}.
\]

\begin{prop}\label{PropContraIsUnitCorep}
Let $\Ubb \in \Bc(\Hc\otimes \Gc)$ be a unitary corepresentation for a multiplicative unitary $\Wbb$. Assume that there exists a unitary operator $\Ubb^c \in \Bc(\Hc\otimes \overline{\Gc})$ and a strictly positive operator $\hat{\Qb}$ on $\Gc$ such that \begin{equation}\label{EqObsWeakContra}
(\Ubb_{21}^*,\hat{\Qb},\Ubb_{21}^c)
\end{equation}
is a weakly modular datum. Then also $\Ubb^c$ is a unitary corepresentation. 
\end{prop} 
\begin{proof}
The modularity condition can be rewritten as
\begin{equation}\label{EqRightInvOpAltCorep}
\langle \xi'\otimes \eta,\Ubb(\eta'\otimes \hat{\Qb}\xi)\rangle = \langle \xi'\otimes \overline{\xi},\Ubb^{c,*}(\eta'\otimes \overline{\hat{\Qb}\eta})\rangle,\qquad \forall \xi,\eta\in \msD(\hat{\Qb}),\xi',\eta'\in \Hc.
\end{equation}

Let then $\{e_i\}$ be an orthonormal basis of $\Hc$. On the one hand, we have for $\xi_i',\eta_i' \in \Hc$ and $\xi,\eta\in \msD(\hat{\Qb})\subseteq \Gc$ that 
\begin{eqnarray*} 
&&\hspace{-2cm} \langle \xi_1' \otimes \xi_2' \otimes \eta,\Wbb_{12}^*\Ubb_{23}\Wbb_{12}(\eta_1'\otimes \eta_2'\otimes \hat{\Qb}\xi)\rangle \\
&& = \langle (\Wbb(\xi_1' \otimes \xi_2')) \otimes \eta,\Ubb_{23}(\Wbb(\eta_1'\otimes \eta_2'))\otimes \hat{\Qb}\xi\rangle \\
&& = \sum_i \langle (\omega_{e_i,\xi_1'}\otimes \id)(\Wbb)\xi_2' \otimes \eta,\Ubb((\omega_{e_i,\eta_1'}\otimes \id)(\Wbb)\eta_2'\otimes \hat{\Qb}\xi)\rangle\\
&& = \sum_i \langle (\omega_{e_i,\xi_1'}\otimes \id)(\Wbb)\xi_2' \otimes \overline{\xi},\Ubb^{c,*}((\omega_{e_i,\eta_1'}\otimes \id)(\Wbb)\eta_2'\otimes \overline{\hat{\Qb}\eta})\rangle\\
&& = \langle \xi_1' \otimes \xi_2' \otimes \overline{\xi},\Wbb_{12}^*\Ubb_{23}^{c,*}\Wbb_{12}(\eta_1'\otimes \eta_2'\otimes \overline{\hat{\Qb}\eta})\rangle.
\end{eqnarray*}
On the other hand, choosing now an orthonormal basis $\{e_k\}$ of $\Gc$  and putting $p_r = \chi_{(r,r^{-1})}(\hat{\Qb})$ for $r>0$, we have that 
\begin{eqnarray*} 
&&\hspace{-1cm} \langle \xi_1' \otimes \xi_2' \otimes \eta,\Ubb_{13}(1\otimes 1 \otimes p_r)\Ubb_{23}(\eta_1'\otimes \eta_2'\otimes \hat{\Qb}\xi)\rangle \\
&& = \sum_{i,j,k}\langle \xi_1' \otimes \xi_2' \otimes \eta,\Ubb_{13}e_i \otimes e_j \otimes p_re_k\rangle \langle e_i \otimes e_j \otimes p_re_k,\Ubb_{23}(\eta_1'\otimes \eta_2'\otimes \hat{\Qb}\xi)\rangle\\
&& = \sum_{i,j,k}\langle \xi_1' \otimes \xi_2' \otimes \overline{\hat{\Qb}^{-1}p_re_k},\Ubb_{13}^{c,*}e_i \otimes e_j \otimes \overline{\hat{\Qb}\eta}\rangle \langle e_i \otimes e_j \otimes \overline{\xi},\Ubb_{23}^{c,*}(\eta_1'\otimes \eta_2'\otimes \overline{\hat{\Qb}p_re_k})\rangle\\
\\
&& = \sum_k \langle \xi_1'\otimes \xi_2', (\id\otimes \omega_{\overline{\hat{\Qb}^{-1}p_re_k},\overline{\hat{\Qb}\eta}})(\Ubb^{c,*})\eta_1' \otimes (\id\otimes \omega_{\overline{\xi},\overline{\hat{\Qb}p_re_k}})(\Ubb^{c,*})\eta_2'\rangle \\
&& = \sum_k \langle \overline{\hat{\Qb}^{-1}p_re_k}, (\omega_{\xi_1',\eta_1'}\otimes \id)(\Ubb^{c,*})\overline{\hat{\Qb}\eta}\rangle \langle \overline{\xi}, (\omega_{\xi_2',\eta_2'}\otimes \id)(\Ubb^{c,*})\overline{\hat{\Qb}p_re_k}\rangle\\
&&=    \langle \overline{\xi}, (\omega_{\xi_2',\eta_2'}\otimes \id)(\Ubb^{c,*}) p_r (\omega_{\xi_1',\eta_1'}\otimes \id)(\Ubb^{c,*})\overline{\hat{\Qb}\eta}\rangle \\
&& = \langle \xi_1' \otimes \xi_2' \otimes \overline{\xi},\Ubb_{23}^{c,*}(1\otimes 1 \otimes p_r)\Ubb_{13}^{c,*}(\eta_1'\otimes \eta_2'\otimes \overline{\hat{\Qb}\eta})\rangle.
\end{eqnarray*}
Taking the limit as $r \rightarrow 0$, this gives 
\[
\langle \xi_1' \otimes \xi_2' \otimes \eta,\Ubb_{13}\Ubb_{23}(\eta_1'\otimes \eta_2'\otimes \hat{\Qb}\xi)\rangle = \langle \xi_1' \otimes \xi_2' \otimes \overline{\xi},\Ubb_{23}^{c,*}\Ubb_{13}^{c,*}(\eta_1'\otimes \eta_2'\otimes \overline{\hat{\Qb}\eta})\rangle.
\]
Comparing the above two expressions and using the corepresentation property of $\Ubb$, we find 
\[
\Wbb_{12}^*\Ubb_{23}^{c,*}\Wbb_{12} = \Ubb_{23}^{c,*}\Ubb_{13}^{c,*},
\]
i.e.\ $\Ubb^c$ is a unitary corepresentation.
\end{proof}

We will apply this as follows. Consider $\Fct$ the Fock--Goncharov flip in the faithfully normal representation of $L_{\hbar}(\wbtd{N})$ as constructed in \eqref{SecModularity}. Let $\hat{\Fct} = \Fct_{21}^*$ be its dual. Then $\hat{\Fct}$ can be viewed as a unitary corepresentation with respect to itself.

\begin{cor}\label{CorContraRegLeft}
Let $\Fct$ be a Fock--Goncharov flip, and let $\Qb,\hat{\Qb},\tFct$ be as in Theorem \ref{EqFGGeneralMod}. Then $\tFct_{21}$ is a unitary corepresentation for $\hat{\Fct}$. 
\end{cor} 

\begin{proof}
This follows immediately from Proposition \ref{PropContraIsUnitCorep}, the fact that $\hat{\Fct}$ is a unitary corepresentation for $\hat{\Fct}$ and the fact that $(\Fct,\hat{\Qb},\tFct)$ forms a weakly modular datum.  
\end{proof} 

\begin{remark}
Alternatively, one can show that if $\Wbb$ is modular, the unitary $\Ubb^c$ in \eqref{PropContraIsUnitCorep} coincides with the contragredient of $\Ubb$ as defined in \cite[Definition 11]{SW07}, from which Corollary \ref{CorContraRegLeft} then also immediately follows.
\end{remark}

\subsection{Locally compact quantum group associated to the Fock--Goncharov flip}

We are now ready for the main result of this paper. 

We first make the following observation: we have an isomorphism of skew-symmetric spaces
\begin{equation}\label{EqVarTheta}
\vartheta\colon  \wbtd{N}\cong \overline{\wbtd{N}},\quad e_{a,b,c} \mapsto \overline{-e_{c,b,a}}.
\end{equation}
This gives an associated $*$-isomorphism 
\begin{equation}\label{EqFlipSo}
\varTheta\colon  L_{\hbar}(\wbtd{N}) \rightarrow \overline{L_{\hbar}(\wbtd{N})},\qquad \eb(v) \mapsto \eb(\vartheta(v)). 
\end{equation}

\begin{lemma}\label{LemRightMUreg} 
The map $\varTheta$ restricts to an isomorphism 
\begin{equation}\label{EqVarThetOr}
(L_{\hbar}(\Bor_N^+),\hat{\Delta}^{\opp}) \cong (\overline{L_{\hbar}(\Bor_N^+)},\overline{\hat{\Delta}}).
\end{equation}
\end{lemma}
\begin{proof}
With $\Fct$ the standard Fock--Goncharov flip and $\tFct$ as in Theorem \ref{TheoModFGFin}, we easily verify that
\[
(\id\otimes \varTheta)\tFct = \overline{\Fct}.
\] 
Hence 
\[
\hat{\Vbb} := (\varTheta^{-1}\otimes \id)\tFct= (\varTheta^{-1}\otimes \varTheta^{-1})(\overline{\Fct})
\]
is also a multiplicative unitary. 

On the other hand, it follows from Corollary \ref{CorContraRegLeft} that $\hat{\Vbb}_{21}$ is a unitary $\hat{\Fct}$-corepresentation, so 
\[
\hat{\Fct}_{12}^* \hat{\Vbb}_{32} \hat{\Fct}_{12} = \hat{\Vbb}_{31} \hat{\Vbb}_{32},\qquad \textrm{i.e.}\qquad  \hat{\Fct}_{23}^* \hat{\Vbb}_{13} \hat{\Fct}_{23} = \hat{\Vbb}_{12} \hat{\Vbb}_{13}.
\]

It now follows that 
\[
\hat{\Vbb}_{23}\hat{\Vbb}_{12} \hat{\Vbb}_{23}^* = \hat{\Vbb}_{12}\hat{\Vbb}_{13} = \hat{\Fct}_{23}^*\hat{\Vbb}_{13}\hat{\Fct}_{23}.
\]
Since the second leg of $\hat{\Vbb}$ equals $L_{\hbar}(\Bor_N^+)$, we find 
\[
\hat{\Delta}(x) = \hat{\Vbb}(x\otimes 1)\hat{\Vbb}^* = \hat{\Fct}^*(1\otimes x)\hat{\Fct},\qquad x\in L_{\hbar}(\Bor_N^+). 
\]
Now $\varTheta\colon  L_{\hbar}(\Bor_N^+) \cong \overline{L_{\hbar}(\Bor_N^+)}$, and we compute 
\begin{multline*}
\hat{\Delta}^{\opp}(x) =\Sigma \hat{\Vbb}(x\otimes 1)\hat{\Vbb}^*\Sigma = (\varTheta^{-1}\otimes \varTheta^{-1})(\overline{\Fct}_{21}(1\otimes \varTheta(x))\overline{\Fct}_{21}^*) \\
= (\varTheta^{-1}\otimes \varTheta^{-1}))(\overline{\hat{\Fct}}^*(1\otimes \varTheta(x))\overline{\hat{\Fct}})
,\qquad x\in L_{\hbar}(\Bor_N^+).
\end{multline*}
This proves the lemma.
\end{proof}

We recall now the notion of a locally compact quantum group \cite{KV00,KV03}. 

\begin{defn}
A \emph{locally compact quantum group} consists of a von Neumann bialgebra $(M,\Delta)$ for which there exists a left invariant nsf weight $\varphi$ and a right invariant nsf weight $\psi$: 
\begin{equation}\label{EqLeftInvWeight}
\varphi((\omega \otimes \id)\Delta(x)) = \varphi(x),\qquad \forall x \in \mfm_{\varphi}^+, \textrm{ all states }\omega \in M_*^+,
\end{equation}
\begin{equation}\label{EqRightInvWeight}
\psi((\id \otimes \omega)\Delta(x)) = \psi(x),\qquad \forall x \in \mfm_{\psi}^+, \textrm{ all states }\omega \in M_*^+.
\end{equation}
\end{defn}
These left and right invariant nsf weights are then unique up to multiplication with a positive scalar.

\begin{theorem}\label{TheoLCQG}
The Fock--Goncharov von Neumann bialgebra $(L_{\hbar}(\Bor_N^-),\Delta)$ is a locally compact quantum group.
\end{theorem} 
\begin{proof}
Put
\[
V = D_N,\qquad Z = \Bor_N^-
\]
as in Corollary \ref{CorHeisDeg}. Then by Lemma \ref{LemNonDegPair}, we have that we can pick a complementary subspace $W \subseteq D_N$ to $Z$ such that 
\begin{equation}\label{EqDecompVOrthBNPlus}
V = Z\oplus W,\qquad (w,b) = 0,\qquad \forall b\in \Bor_N^+.
\end{equation}
In particular, $L_{\hbar}(W)$ commutes pointwise with $L_{\hbar}(\Bor_N^+)$.

Implement $(L_{\hbar}(\Bor_N^-),\Delta)$ via a Fock--Goncharov flip with respect to an irreducible normal $*$-representation of $L_{\hbar}(D_N)$ on a Hilbert space $\Hsp$ as in Corollary \ref{CorIrrMap}. Then since we can identify $L_{\hbar}(D_N) = \Bc(\Hsp)$, we are in the setting of Proposition \ref{PropTechn}, and combined with Corollary \ref{CorInvWeightQT} applied with respect to \eqref{EqDecompVOrthBNPlus}, we find that $(L_{\hbar}(\Bor_N^-),\Delta)$ admits a left invariant nsf weight.

By  the  self-duality \eqref{EqSelfDual}, also $(L_{\hbar}(\Bor_N^+),\hat{\Delta})$ then has a left invariant nsf weight. But by Lemma \ref{LemRightMUreg}, this means that $(L_{\hbar}(\Bor_N^+),\hat{\Delta}^{\opp})$ has a left invariant nsf weight, i.e.\ $(L_{\hbar}(\Bor_N^+),\hat{\Delta})$ has a right invariant nsf weight. Again by self-duality, it follows that $(L_{\hbar}(\Bor_N^-),\Delta)$ has a right invariant nsf weight, and we are done.
\end{proof}

\section{Modular structure of 
\texorpdfstring{$(L_{\hbar}(\Bor_N^-),\Delta)$}{the upper Borel quantum group}} \label{secmodularstructure}

We resume the setting and notations of the previous sections. 

Theorem \ref{TheoLCQG} shows that the Fock--Goncharov von Neumann bialgebra $(L_{\hbar}(\Bor_N^-),\Delta)$ is a locally compact quantum group. Following common practice in quantum group theory, we designate this locally compact quantum group with a symbol such that the above von Neumann algebra is its \emph{function algebra}, and the \emph{group von Neumann algebra} of its dual. Given the duality between $\Bor_N^-$ and $\Bor_N^+$, this motivates the notation 
\begin{equation}\label{EqLCQGFG}
(L^{\infty}(\Bbbb_{N,\hbar}^+),\Delta)  := (L_{\hbar}(\Bor_N^-),\Delta),
\end{equation}
and likewise for the Pontryagin dual:
\begin{equation}\label{EqLCQGFGDual}
(L^{\infty}(\Bbbb_{N,\hbar}^-),\hat{\Delta}) := (L_{\hbar}(\Bor_N^+),\hat{\Delta}).
\end{equation}

Now by the general theory of \cite{KV00,KV03}, the locally compact quantum group $\Bbbb_{N,\hbar}^+$ comes with a variety of associated \emph{modular data}, that we will determine in this section. 

\subsection{Left Haar measure for  \texorpdfstring{$\Bbbb_{N,\hbar}^+$}{the upper Borel quantum group}.}

We start off by determining more explicitly the left invariant nsf weight on $L_{\hbar}(\Bor_N^-)$.

Write $\hat{\Bor}_N^- = \widehat{\Bor_N^-}$ for the dual vector space. By Lemma \ref{LemIdentGNS}, we can uniquely identify
\[
L^2(M) \cong L^2(\hat{\Bor}_N^-)
\]
in such a way that, with $\Lambda$ the GNS-map associated to $\int_{\hat{\Bor}_N^-}^{\hbar}$, we have 
\begin{equation}\label{EqGNSNonTwisted}
\Lambda(f_{\hbar}) = f \in L^2(\hat{\Bor}_N^-),\qquad f\in \Sc(\hat{\Bor}_N^-).
\end{equation}
The associated standard representation of $L_{\hbar}(\Bor_N^-)$ is then uniquely determined by 
\[
f_{\hbar} g = f\times_{\hbar}g,\qquad f,g\in \Sc(\hat{\Bor}_N^-),
\]
which has associated unitary $\hbar$-representation of $\Bor_N^-$ given by \eqref{EqStandardDefRep} (or equivalently \eqref{EqStructurepi}). As $\int_{\hat{\Bor}_N^-}^{\hbar}$ is tracial, we also know that the modular conjugation of $L_{\hbar}(\Bor_N^-)$ with respect to this standard representation is given by 
\begin{equation}\label{EqModConjB}
J g = \overline{g},\qquad g\in \Sc(\hat{\Bor}_N^-).
\end{equation}

Now by Corollary \ref{CorIrrMap}, Proposition \ref{PropTechn} and Lemma \ref{LemInvWeightQT}, we have that $(L_{\hbar}(\Bor_N^-),\Delta)$ has a left invariant nsf weight 
\[
\varphi = \int_{\hat{\Bor}_N^-}^{\hbar,\Eb(2d_l)},
\]
for a particular $d_l \in \Bor_N^-$. More precisely, by these results, together with the fact that 
\[
\Qb = \Kb_{\delta}^{-(1+|\hbar|^{-1})} = \Eb(-(1+|\hbar|^{-1})\sum_i \sevarpi_i), 
\] 
we can take any $d_l \in \Bor_N^-$ such that there exists $w\in \wbtd{N}$, vanishing on $\Bor_N^+$, with 
\[
2d_l + w = -2(1+|\hbar|^{-1})\sum_i \sevarpi_i.
\]
By Lemma \ref{LemNonDegPair}, $d_l$ must then be the unique element in $\Bor_N^-$ such that 
\begin{equation}\label{EqChardl}
(d_l,z) = -(1+|\hbar|^{-1})\sum_{i=1}^n (\sevarpi_i,z),\qquad \forall z\in \Bor_N^+.
\end{equation}

\begin{lemma}
We have 
\[
2d_l = (1+|\hbar|^{-1})\sum_{(a,b,c)\in \neC_N} ac\,e_{a,b,c} = (1+|\hbar|^{-1})\underset{0\leq k<s}{\sum_{1\leq s\leq n}} (N-s)\nee_{s,k}. 
\]
\end{lemma} 
\begin{proof}
Using \eqref{EqChardl}, it follows by a direct computation from Lemma \ref{LemSkewProdFormDiff}.
\end{proof}

So, summarizing, we have:
\begin{theorem}\label{TheoLeftNSF}
The left invariant nsf weight for $(L_{\hbar}(\Bor_N^-),\Delta)$ is 
\begin{equation}\label{EqModOper}
\varphi = \int_{\hat{\Bor}_N^-}^{\hbar,\Eb(2d_l)},\qquad \Eb(2d_l) =  \Eb(\underset{0\leq k<s}{\sum_{1\leq s\leq n}} (N-s)\nee_{s,k})^{1+|\hbar|^{-1}}. 
\end{equation}
\end{theorem}

Now by general theory \cite{PT73,Vae01}, we can identify also 
\begin{equation}\label{EqIdentGNSvNTwist}
L^2(L_{\hbar}(\Bor_N^-),\varphi)\cong L^2(\hat{\Bor}_N^-)
\end{equation}
by means of the unique GNS-map 
\[
\Lambda_{\varphi}\colon  \mfn_{\varphi} \rightarrow L^2(\hat{\Bor}_N^-)
\]
such that 
\[
\Lambda_{\varphi}(x) = \Lambda(x\Eb(d_l)),\qquad x\in L_{\hbar}(\Bor_{N}^-) \textrm{ and }x\Eb(d_l)\textrm{ bounded with closure in }\mfn_{\int_{\hat{\Bor}_N^-}^{\hbar}}.
\]
\begin{cor}
In the above GNS-representation, the modular operator of the weight $\varphi$ in \eqref{EqModOper} is given by 
\[
\nabla_{\varphi} = \Eb(2d_l)J \Eb(2d_l)^{-1}J,\qquad \nabla_{\varphi}^{it} =  \Eb(2d_l)^{it}J \Eb(2d_l)^{it}J.
\]
\end{cor}

\subsection{Right Haar measure for  \texorpdfstring{$\Bbbb_{N,\hbar}^+$}{the upper Borel quantum group}.}

We now determine explicitly the right invariant nsf weight for $(L_{\hbar}(\Bor_N^-),\Delta)$. 

First note that the maps 
$\theta$ and $\Theta$ in \eqref{EqThetaIm} and \eqref{EqSelfDual} lead to  isomorphisms 
\begin{equation}\label{EqThet}
\Theta\colon (L_{\hbar}(\Bor_N^+),\hat{\Delta}) \cong (\overline{L_{\hbar}(\Bor_N^-)},\overline{\Delta}),\qquad \Theta\colon  (L_{\hbar}(\Bor_N^-),\Delta)\cong (\overline{L_{\hbar}(\Bor_N^+)},\overline{\hat{\Delta}}).
\end{equation}
On the other hand, by Lemma \ref{LemRightMUreg} we have that the maps $\vartheta$ and $\varTheta$ in \eqref{EqVarTheta} and \eqref{EqFlipSo} lead to an isomorphism 
\begin{equation}\label{EqVarThet}
\varTheta\colon (L_{\hbar}(\Bor_N^+),\hat{\Delta}^{\opp}) \cong (\overline{L_{\hbar}(\Bor_N^+)},\overline{\hat{\Delta}}).
\end{equation}
By composition, \eqref{EqThet} and \eqref{EqVarThet} lead to 
\[
\upsilon= \theta\vartheta^{-1}\theta\colon  \wbtd{N}\cong \overline{\wbtd{N}},\quad e_{a,b,c} \mapsto \overline{-e_{a,c,b}},\qquad 
\Upsilon\colon  L_{\hbar}(\wbtd{N}) \cong \overline{L_{\hbar}(\wbtd{N})},\qquad \eb(v) \mapsto \eb(\upsilon(v))
\]
inducing an isomorphism
\[
\Upsilon\colon  (L_{\hbar}(\Bor_N^-),\Delta^{\opp}) \cong (\overline{L_{\hbar}(\Bor_N^-)},\overline{\Delta}).
\]
Now 
\[
\overline{\int_{\hat{\Bor}_N^-}^{\hbar}} \circ \Upsilon = \int_{\hat{\Bor}_N^-}^{\hbar},
\]
so since 
\[
\upsilon^{-1}(\overline{d_l}) = d_r,\qquad 2d_r = -(1+|\hbar|^{-1})\sum_{(a,b,c)\in \neC_N} ab \,e_{a,b,c} = -(1+|\hbar|^{-1})\underset{0\leq k<s}{\sum_{1\leq s\leq n}} (N-s)\swe_{s,k},
\]
we find:
\begin{theorem}\label{TheoRightNSF}
The right invariant nsf weight for $(L_{\hbar}(\Bor_N^-),\Delta)$ is 
\[
\psi = \int_{\hat{\Bor}_N^-}^{\hbar,\Eb(2d_r)},\qquad \Eb(2d_r) = \Eb(\underset{0\leq k<s}{\sum_{1\leq s\leq n}} (N-s)\swe_{s,k})^{-(1+|\hbar|^{-1})}.
\]
\end{theorem}

\subsection{Modular element for  \texorpdfstring{$\Bbbb_{N,\hbar}^+$}{the upper Borel quantum group}}

If $(M,\Delta)$ is a locally compact quantum group with left and right nsf weights $\varphi,\psi$, there exists a unique strictly positive operator $\delta\,\eta\,M$ (i.e.\ $\delta^{it}\in M$ for all $t\in \R$) which is \emph{group-like}, i.e.\
\[
\Delta(\delta) = \delta\otimes \delta,\qquad \Delta(\delta^{it}) = \delta^{it}\otimes \delta^{it},
\]
and such that, possibly rescaling $\psi$, we have
\[
\psi = \varphi_{\delta},
\]
where $\varphi_{\delta}$ is the unique nsf weight such that 
\begin{equation}\label{EqVarphidelta}
\varphi_{\delta}(x^*x) = \varphi(\delta^{1/2}x^*x \delta^{1/2}),\qquad \forall x \in M\textrm{ with }x\circ \delta^{1/2} \textrm{ bounded and with closure }x\delta^{1/2} \in \mfn_{\varphi}.
\end{equation}
One calls $\delta$ the \emph{modular element} of $(M,\Delta)$.

\begin{lemma}\label{LemGrouplike}
The elements $\Eb(v)$ for $v \in \Tor_N^-$ are group-like.
\end{lemma}
\begin{proof}
Fix $v\in \Tor_N^-$. We have to compute 
\[
\Delta(\Eb(v)^{it}) = \Fct^*(1\otimes \Eb(v)^{it})\Fct.
\]
But using the expression \eqref{EqRealMUAlt} and noting that 
\[
(\nwe_{s,k},v)=0,\qquad 0\leq k <s\leq n
\]
by \eqref{EqDiffCar} and \eqref{EqDiffCart}, we find by \eqref{EqGausEqInv} that indeed
\[
\Delta(\Eb(v)^{it}) = \Kc^* \Fctt^*(1\otimes \Eb(v)^{it})\Fctt\Kc = \Kc^*(1\otimes \Eb(v)^{it}) \Kc = \Eb(v)^{it}\otimes \Eb(v)^{it}
\]
as claimed. 
\end{proof}
 
\begin{theorem}
The modular element of $(L_{\hbar}(\Bor_N^-),\Delta)$ equals 
\begin{equation}\label{EqModEl}
\delta = \Lb_{\delta}^{-2(1+|\hbar|^{-1})},\qquad \Lb_{\delta} = \Eb(\sum_{i=1}^n \nevarpi_i).
\end{equation}
\end{theorem}
\begin{proof}
This follows directly from Theorem \ref{TheoLeftNSF} and Theorem \ref{TheoRightNSF}, together with the fact that
\begin{multline*}
2d_r-2d_l = -(1+|\hbar|^{-1})\underset{0\leq k<s}{\sum_{1\leq s\leq n}} (N-s)(\swe_{s,k}  + \nee_{s,k}) =- (1+|\hbar|^{-1})\sum_{1\leq s\leq n} (N-s)s\nee_s \\ =- 2(1+|\hbar|^{-1})\sum_{i=1}^n \nevarpi_i,
\end{multline*}
and the fact that $\Lb_{\delta}$ is indeed group-like by Lemma \ref{LemGrouplike}.
\end{proof}

Note that with $\varphi$ and $\psi$ scaled as in Theorem \ref{TheoLeftNSF} and Theorem \ref{TheoRightNSF}, we then indeed get 
\[
\psi = \varphi_{\delta}. 
\]

\subsection{Scaling constant of  \texorpdfstring{$\Bbbb_{N,\hbar}^+$}{the upper Borel quantum group}} 

If $(M,\Delta)$ is a locally compact quantum group with left invariant nsf weight $\varphi$, then associated to $\varphi$ one has the modular automorphism group $(\sigma_t^{\varphi})_{t\in \R}$ of $M$, determined by 
\[
\sigma_t^{\varphi}(x) = \nabla_{\varphi}^{it}x\nabla_{\varphi}^{-it},\qquad x\in M.
\]
There then exists a unique number $\nu>0$, called the \emph{scaling constant} of $(M,\Delta)$, such that if $\delta$ is the modular element of $(M,\Delta)$, we have
\[
\sigma_t^{\varphi}(\delta) = \nu^t \delta
\]
(see e.g.\ \cite[Proposition 2.13.(18)]{KV03}).

Now in the case of $(L_{\hbar}(\Bor_N^-),\Delta)$, we immediately see that
\[
\sigma_t^{\varphi}(x) = \nabla_{\varphi}^{it}x\nabla_{\varphi}^{-it} = \Eb(2d_l)^{it}x \Eb(2d_l)^{-it},\qquad x\in L_{\hbar}(\Bor_N^-),t\in \R.
\]
Together with the concrete expression \eqref{EqModEl} and the simple computation 
\[
2\underset{0\leq k <s}{\sum_{1\leq s,t\leq n}} (N-s)(\nee_{s,k},\nevarpi_t) = \underset{0\leq k <s}{\sum_{1\leq s\leq n}}(N-s) = \sum_{1\leq s\leq n} (N-s)s = \tau_n,
\]
with $\tau_n$ the $n$-th tetrahedral number
\[
\tau_n = \binom{n+2}{3},
\]
this then leads to:

\begin{theorem}
The scaling constant of $(L_{\hbar}(\Bor_N^-),\Delta)$ equals 
\[
\nu = e^{-2\pi \beta_{\hbar}  \tau_n},\qquad \beta_{\hbar} =  \sgn(\hbar)(|\hbar|^{1/2}+|\hbar|^{-1/2})^2.
\]
\end{theorem} 

\begin{remark}
This shows in particular that $\Bbbb_N^+$ cannot be a $*$-algebraic quantum group as in \cite{VD98}, as the latter have scaling constant equal to $1$ \cite[Theorem 3.4]{DCVD10}. 
\end{remark}

\subsection{Antipode}

Let $\Wbb \in \Bc(\Hsp\otimes \Hsp)$ be a modular multiplicative unitary with associated modular datum $(\Wbb,\Qb,\hat{\Qb},\hat{\Wbb})$. Let $(M,\Delta)$ be its associated von Neumann bialgebra. Then by \cite[Theorem 2.3]{SW01} (applied to the dual of $\Wbb$), the Banach space 
\begin{equation}\label{EqCstarMU}
A = A_{\Wbb} =  [(\id\otimes \omega)\Wbb\mid \omega \in \Bc(\Hsp)_*] 
\end{equation}
is a C$^*$-algebra, to which $\Delta_{\Wbb}$ restricts as a non-degenerate $*$-homomorphism
\[
\Delta_{\Wbb}\colon  A \rightarrow M(A\otimes A).
\]
There then exists a unique closed unbounded operator 
\[
S\colon  \mathscr{D}(S) \subseteq A \rightarrow A 
\]
such that $\{(\id\otimes \omega)\Wbb\mid \omega \in \Bc(\Hsp)_*\}$ is a core for $S$ and satisfies there 
\begin{equation}\label{EqModularAntipodeCrit}
S((\id\otimes \omega)\Wbb) = (\id\otimes \omega)(\Wbb^*),\qquad \omega \in \Bc(\Hsp). 
\end{equation}
Moreover, by this same theorem there then exists a unique involutive, anti-multiplicative automorphism
\[
R\colon  M\rightarrow M
\]
the \emph{unitary antipode}, and a $\sigma$-weakly continuous one-parameter group of automorphisms
\begin{equation}\label{EqFormScalGroup}
\tau_t\colon  M \rightarrow M,\qquad t\in \R,
\end{equation}
called the \emph{scaling group}, such that $\tau_t$ is norm-continuous on $A$, and such that on $A$ we have the equality
\[
S = R\circ \tau_{-i/2},
\]
where $\tau_{-i/2}$ has as domain all $a\in A$ for which $t \mapsto \tau_t(a)$ has a norm-continuous extension to the halfstrip $\{z\in \C\mid -i/2\leq \Imm(z) \leq 0\}$, and is holomorphic on the interior of that strip. The scaling group is then implemented concretely through 
\[
\tau_t(x) =  \hat{\Qb}^{2it} x \hat{\Qb}^{-2it},\qquad x\in M,t\in \R.  
\]

The following theorem tells us that the above data only depends on $(M,\Delta)$. It is a small variation of \cite[Theorem 5]{SW07}. 

\begin{theorem} \label{TheoIndepofantipode}
Let $\Wbb,\Vbb$ be modular multiplicative unitaries giving rise to isomorphic von Neumann bialgebras
\begin{equation}\label{EqIsovN}
(M_{\Wbb},\Delta_{\Wbb})\underset{\theta}{\cong} (M_{\Vbb},\Delta_{\Vbb}). 
\end{equation}
Then $\theta$ intertwines the respective unitary antipode and scaling group. 
\end{theorem} 
\begin{proof}
For the first part of the theorem, it is enough to show that $\theta$ intertwines the C$^*$-algebra $A_{\Wbb}$ in \eqref{EqCstarMU} with the C$^*$-algebra $A_{\Vbb}$, for one can then simply apply \cite[Theorem 5]{SW07}.  

Now since any isomorphism of von Neumann algebras can always be implemented through unitary conjugation, upon tensoring their defining normal representation with a multiplicity Hilbert space, we may assume that we are simply given a von Neumann bialgebra $(M,\Delta)$ with $M \subseteq \Bc(\Hsp)$, and two modular multiplicative unitaries 
\[
\Wbb,\Vbb\in \Bc(\Hsp \otimes \Hsp)
\]
such that 
\[
M = [(\id\otimes \omega)\Wbb\mid \omega\in \Bc(\Hc)_*]^{\sigma\textrm{-weak}} = [(\id\otimes \omega)\Vbb\mid \omega\in \Bc(\Hc)_*]^{\sigma\textrm{-weak}},
\]
and such that
\[
\Delta(x) = \Wbb^*(1\otimes x)\Wbb = \Vbb^*(1\otimes x)\Vbb,\qquad x\in M. 
\]
Now $\Vbb \in M \bar{\otimes}\Bc(\Hsp)$ is a unitary corepresentation for $(M,\Delta)$, in the sense that 
\[
(\Delta\otimes \id)\Vbb = \Vbb_{13}\Vbb_{23}. 
\]
But since $\Delta = \Delta_{\Wbb}$, this entails 
\[
\Vbb_{13} = \Wbb_{12}^*\Vbb_{23}\Wbb_{12}\Vbb_{23},
\]
and so from $\Wbb \in M(A_{\Wbb}\otimes \Kc(\Hsp))$ (\cite[Theorem 2.3]{SW01} applied to $\hat{\Wbb} = \Wbb_{21}^*$) we see that $\Vbb \in M(A_{\Wbb}\otimes \Kc(\Hsp))$. So $\Vbb$ is also a representation of $(A_{\Wbb},\Delta_{\Wbb})$ in the sense of \cite{SW07}, and moreover it is right absorbing, in the sense that for any other unitary corepresentation $U$ of $(M_{\Wbb},\Delta_{\Wbb})$ we have a unitary conjugation
\[
\Vbb_{13}U_{23} \cong \Vbb_{13},
\]
using the identity 
\[
\Vbb_{13}U_{23}  =  U_{32}\Vbb_{13}U_{32}^*,
\]
which follows immediately from the fact that $U$ is a unitary corepresentation and $\Delta = \Delta_{\Vbb}$. 

It then follows from the arguments at the beginning of \cite[Section 4]{SW07} that there exists a normal isomorphism 
\[
\hat{\theta}\colon  \hat{M}_{\Wbb} \rightarrow \hat{M}_{\Vbb},\qquad (\id\otimes \hat{\theta})\Vbb = \Wbb. 
\]
But this implies immediately 
\[
A_{\Wbb} = [(\id\otimes \omega)\Wbb\mid \omega \in \Bc(\Hsp)_*] = [(\id\otimes \omega)\Vbb\mid \omega \in \Bc(\Hsp)_*]  = A_{\Vbb}.
\] 
\end{proof}

Now if $(M,\Delta)$ is a locally compact quantum group with left invariant nsf weight $\varphi$, there exist a unique multiplicative unitary $\Wbb$ on $L^2(M)\otimes L^2(M)$, called the \emph{left regular representation}, such that 
\begin{equation}\label{EqLeftRegularMultUnit}
(\omega\otimes \id)(\Wbb^*)\Lambda_{\varphi}(x) = \Lambda_{\varphi}((\omega\otimes \id)\Delta(x)),\qquad \forall \omega \in M_*,\forall x\in \mfn_{\varphi}. 
\end{equation}

\begin{cor}
If $\Wbb$ is a modular multiplicative unitary such that $(M_{\Wbb},\Delta_{\Wbb})$ defines a locally compact quantum group, then its unitary antipode and scaling group, as defined by \eqref{EqModularAntipodeCrit}, agree with the ones defined in \cite{KV03}. 
\end{cor}
\begin{proof}
By \cite[Proposition 6.10]{KV00}, the left regular multiplicative unitary attached to a locally compact quantum group is modular. The corollary then follows directly from Theorem \ref{TheoIndepofantipode}, since the antipode for a locally compact quantum group is determined by \eqref{EqModularAntipodeCrit} with respect to the left regular representation.
\end{proof}

Let us return now again to the locally compact quantum group $\Bbbb_N^+ = (L_{\hbar}(\Bor_N^-),\Delta)$. Invoking \eqref{EqFormScalGroup}, Theorem \ref{TheoModFGFin} and Theorem \ref{EqFGGeneralMod}, we find: 
\begin{theorem} \label{TheoScalGroup}
The scaling group of $\Bbbb_N^+$ is given by 
\begin{equation}
\tau_t(x) = \Lb_{\delta}^{2i(1+|\hbar|^{-1})t}x \Lb_{\delta}^{-2i(1+|\hbar|^{-1})t} ,\qquad x\in L_{\hbar}(\Bor_N^-).
\end{equation}
In particular, the unitary one-parameter group implementing the scaling group (\cite[Definition 6.9]{KV00}) is 
\begin{equation}
P^{it} = \Lb_{\delta}^{2i(1+|\hbar|^{-1})t} J \Lb_{\delta}^{2i(1+|\hbar|^{-1})t} J.
\end{equation}
\end{theorem}

From \cite[Theorem 2.3]{SW01} it can also be gathered that, writing $x^T = \overline{x^*}$, the unitary antipode $R$ is uniquely determined by  
\begin{equation}
(R\otimes T)(\Fct) = \overline{\tFct}, 
\end{equation}
with $\Fct$ a Fock--Goncharov flip and $\tFct$ as in Theorem \ref{TheoModFGFin}. From the concrete formulas for $\Fct$ and $\tFct$, we then easily obtain:
\begin{theorem}
The unitary antipode of $\Bbbb_N^+$ is determined by 
\begin{equation}\label{EqFormUnitAntip}
R(\Eb(\nee_{s,k})) = \Eb(-\swe_{s,k}),\qquad 0\leq k \leq s \leq n.
\end{equation}
\end{theorem}

Recalling the notation \eqref{EqVarphidelta}, general theory \cite[Section 7]{KV00} gives us that
\[
\varphi R = \varphi_{\delta}.
\]

\subsection{Dual modular structure}

Given a locally compact quantum group $(M,\Delta)$, we have already mentioned that its associated left regular multiplicative unitary $\Wbb$, defined by \eqref
{EqLeftRegularMultUnit}, is modular. One further shows that $(\hat{M},\hat{\Delta})$ (as defined through \eqref{EqFormvNDual} and \eqref{EqComultiplicationDual}) is again a locally compact quantum group. There is then a canonical way to identify 
\begin{equation}\label{EqIdentDualGNS}
L^2(\hat{M})\cong L^2(M).
\end{equation}
Namely, write
\[
\lambda(\omega) = (\omega \otimes \id)\Wbb,\qquad \omega \in M_*.
\]
Consider 
\[
\mathcal{I} = \{\omega\in M_* \mid \exists C\geq0: |\omega(x^*)| \leq C\|\Lambda_{\varphi}(x)\|\textrm{ for all }x\in \mfn_{\varphi}\},
\]
and let then $\hat{\Lambda}(\lambda(\omega))$, for $\omega \in \mathcal{I}$, be the unique vector in $L^2(M)$ such that 
\[
\langle \Lambda_{\varphi}(x),\hat{\Lambda}(\lambda(\omega))\rangle  = \omega(x^*),\qquad x\in \mfn_{\varphi},\omega \in\mathcal{I}.  
\]
Then one shows that one can pick the left invariant nsf weight for $(\hat{M},\hat{\Delta})$ such that $\lambda(\mathcal{I}) \subseteq \mfn_{\hat{\varphi}}$, and such that the map
\[
\Lambda_{\hat{\varphi}}(\lambda(\omega)) \mapsto \hat{\Lambda}(\lambda(\omega)),\qquad \omega \in \mathcal{I}
\]
extends to a unitary, giving the identification \eqref{EqIdentDualGNS}.

Let us now return to the concrete locally compact quantum group $\Bbbb_N^+$ determined by \eqref{EqLCQGFG}, with dual $\Bbbb_N^-$. Write $\rho$ for the linear involution 
\begin{equation}
\rho\colon  \Bor_N^- \rightarrow \Bor_N^-,\qquad 
\rho(\nee_{s,k}) = -\swe_{s,k},\qquad 0\leq k \leq s \leq N,
\end{equation}
and write 
\[
\hat{\rho}\colon  \hat{\Bor}_N^- \rightarrow \hat{\Bor}_N^-,\qquad \hat{v}\mapsto \hat{v}\circ \rho.
\]

\begin{prop}
Let $\hat{J}$ be the modular conjugation for $L^{\infty}(\Bbbb_N^-)$, implemented on $L^2(\hat{B}_N^-)$ through \eqref{EqIdentDualGNS} and \eqref{EqIdentGNSvNTwist}.  Then
\begin{equation}\label{EqFormHatJ}
\hat{J}g = \overline{g\circ \hat{\rho}},\qquad g\in L^2(\hat{\Bor}_N^-).
\end{equation}
\end{prop}
\begin{proof}
For a general locally compact quantum group $(M,\Delta)$ with left invariant nsf weight $\varphi$ and unitary antipode $R$, a GNS-map $\Gamma$ for the right invariant nsf weight $\psi = \varphi R$ can be realized on $L^2(M)$ in a unique way such that 
\[
\Gamma(x) = \Lambda_{\varphi}(x\delta^{1/2}),\qquad x\in M\textrm{ with } x\circ \delta^{1/2} \textrm{ bounded and }x\delta^{1/2} \in \mfn_{\varphi}, 
\]
and $\hat{J}$ is then determined, for such $x$ as above, by 
\begin{equation}\label{EqFormulaDualConj}
\hat{J}\Gamma(x) = \Lambda_{\varphi}(R(x)^*).
\end{equation}

Consider now the elements 
\[
\pi(f) = \int_V f(v) \pi(v) \rd v \in L_{\hbar}(\Bor_N^-), \qquad f\in L^1(\Bor_N^-)
\]
as in \eqref{EqStandardDefDualRepCstar}, so that with respect to \eqref{EqGNSNonTwisted} we have
\[
\Lambda(\pi(f)) = \check{f},\qquad f \in \Sc(V). 
\]
Then it is easily seen that with 
\[
\Ec(\Bor_N^-) = \{f\colon  v\mapsto P(v)e^{-|v|^2/2}\mid P\textrm{ polynomial}\}
\]
(choosing an arbitrary Euclidian norm on $\Bor_N^-$), we have $\pi(f)\circ \Eb(w)$ bounded for any $w\in \Bor_N^-$, with 
\[
\pi(f)\Eb(w) = \pi(\kappa_wf),\qquad (\kappa_wf)(v) = e^{\pi \hbar (v,w)}f(v+iw).
\]
It then follows that for $f\in \Ec(\Bor_N^-)$ we have 
\[
\pi(f) \in \mfn_{\varphi},\quad \pi(f)\circ \delta^{1/2}\textrm{ bounded and }\pi(f)\delta^{1/2} \in \mfn_{\varphi},
\]
and that, using $-(1+|\hbar|^{-1})\sum_{i=1}^n \nevarpi_i = d_r-d_l$, 
\[
\Lambda_{\varphi}(\pi(f)) = (\kappa_{d_l}f)^{\vee},\qquad \Gamma(\pi(f)) = (\kappa_{d_r}f)^{\vee},\qquad f \in \Ec(\Bor_N^-).
\]
It then follows from \eqref{EqFormulaDualConj} and \eqref{EqFormUnitAntip} that 
\begin{equation}\label{EqInBetwFormHatJ}
\hat{J}((\kappa_{d_r}f)^{\vee}) = (\kappa_{d_l} \overline{f\circ -\rho})^{\vee}.
\end{equation}
Using that 
\[
\rho(d_l)  = d_r,\qquad (\rho v,w) = -(v,\rho w),\qquad v,w\in \Bor_N^-,
\]
\eqref{EqInBetwFormHatJ} simplifies to 
\[
\hat{J}(f^{\vee}) = (\overline{f\circ -\rho})^{\vee},\qquad f\in \Ec(\Bor_N^-). 
\]
A further simplification, together with the fact that $\Ec(\Bor_N^-)$ is dense in $L^2(\hat{\Bor}_N^-)$, leads to \eqref{EqFormHatJ}.
\end{proof}

\section{Left regular representation of  \texorpdfstring{$\Bbbb_{N,\hbar}^+$}{the upper Borel quantum group} and duality}

\subsection{Left regular representation}

In this section, we determine more explicitly the left regular representation $\Wbb$ of $(L_{\hbar}(\Bor_N^-),\Delta)$ (see \eqref{EqLeftRegularMultUnit}). We start with the following lemma.

\begin{lemma}
The standard right representation $\rho$ of $L_{\hbar}(\Bor_N^-)$ on $L^2(\hat{\Bor}_N^-)$, defined by 
\[
\rho(x)\xi = J x^*J\xi,\qquad \xi \in L^2(\hat{\Bor}_N^-), 
\]
is concretely determined by 
\begin{equation}\label{EqRightRepB}
\rho(\Eb(v)) = \Eb(\overline{v}),
\end{equation}
where we consider the unitary $\hbar$-representation of $\overline{\Bor_N^-}$ on $L^2(\hat{\Bor}_N^-)$ given by 
\begin{equation}\label{EqRightAction}
(\Eb(\overline{v})^{it}g)(\hat{w}) = e^{2\pi i tv\cdot \hat{w}}g(\hat{w}-t\Jc v),\qquad g \in L^2(\hat{\Bor}_N^-),\hat{w}\in \hat{\Bor}_N^-,v\in \Bor_N^-, t\in \R.
\end{equation}
\end{lemma} 
\begin{proof}
This is immediate from \eqref{EqStandardDefRep} and \eqref{EqModConjB}.
\end{proof}
In general, if $M$ is a von Neumann algebra, we will view the standard right representation of $M$ on $L^2(M)$ as inducing an ordinary normal $*$-representation of $\overline{M}$ on $L^2(M)$ via 
\[
\overline{x}\xi = \rho(x^*)\xi = J xJ\xi,\qquad x\in M,\xi\in L^2(M). 
\]
Then compatibly with \eqref{EqRightRepB}, we get that
\[
\overline{\eb(v)} = \eb(\overline{v}),\qquad v\in \Bor_N^-. 
\]

Let us now further equip $L^2(\hat{\Bor}_N^-)$ with the unitary $\Tor_N^+$-representation such that 
\begin{equation}\label{EqActionDualHPlus}
(\Eb(v)^{it}g)(\hat{v}) = g(\hat{v}+2t \Jc v),\qquad g\in L^2(\hat{\Bor}_N^-). 
\end{equation}

Recall the unitaries introduced in \eqref{eq:F-1or}, \eqref{EqGaussFG} and \eqref{EqTildeFDoublePrime}.

\begin{theorem} \label{TheoKacTakesakiop}
The left regular representation $\Wbb$ of $(L_{\hbar}(\Bor_N^-),\Delta)$ is given by the unitary 
\begin{equation}\label{EqFormRegMultUni}
\Wbb =  \overline{\tFctt} \Kc \Fc\in \Bc(L^2(\hat{\Bor}_N^-)\otimes L^2(\hat{\Bor}_N^-)). 
\end{equation}
\end{theorem}
Note that since 
\[
\Fc \in L_{\hbar}(\Bor_N^-)\bar{\otimes}L_{\hbar}(\Bor_N^-),\quad \overline{\tFctt}\in L_{\hbar}(\Bor_N^-)\bar{\otimes} \overline{L_{\hbar}(\Bor_N^-)},\quad \Kc \in L_{\hbar}(\Bor_N^-) \bar{\otimes} L_{\hbar}(\Tor_N^+),
\]
we can indeed make sense of this formula through the standard action of $L_{\hbar}(\Bor_N^-)$ and the actions \eqref{EqRightAction} and \eqref{EqActionDualHPlus} of respectively $\overline{L_{\hbar}(\Bor_N^-)}$ and $L_{\hbar}(\Tor_N^+)$. 
\begin{proof}
Let $\Hsp$ be an irreducible unitary $\hbar$-representation of $\wbtd{N}$, and $\Fct$ the associated Fock--Goncharov flip, with associated modular datum $(\Fct,\Qb,\hat{\Qb},\tFct)$. Pick 
\[
a\in L_{\hbar}(\Bor_N^+)' \subseteq \Bc(\Hsp).
\]
By Proposition \ref{PropTechn}, the normal weight 
\[
\varphi = \Tr_{\Qb^2,a}
\]
is left invariant with respect to 
\[
\Delta\colon  \Bc(\Hsp) \rightarrow \Bc(\Hsp) \bar{\otimes}\Bc(\Hsp),\quad x \mapsto \Fct^*(1\otimes x)\Fct.
\]
By \eqref{CorInvWeightQT}, we can choose $a$ such that $\varphi$ is still semi-finite on $L_{\hbar}(\Bor_N^-)$, and by Remark \ref{RemGNSRest} we may view the GNS-map $\Lambda_{\varphi}$ for $L_{\hbar}(\Bor_N^-)$ as the restriction of the one for $\Bc(\Hsp)$, compatible with the standard left and right representations. Note that since $a\in L_{\hbar}(\Bor_N^+)'$, we have that $a$ and $\Qb$ (strongly) commute.

Now the GNS-map for $\varphi$ on $\Bc(\Hsp)$ can be concretely realized on $\Hsp \otimes \overline{\Hsp}$ via 
\begin{equation}
L^2(\Bc(\Hsp))\cong \Hsp \otimes \overline{\Hsp},\qquad \Lambda_{\varphi}(\theta_{\xi,\eta}) \cong \xi \otimes \overline{\Qb a\eta} \in \Hsp \otimes \overline{\Hsp},\qquad \xi\in \Hsp,\eta\in \msD(\Qb).
\end{equation}
Further observe that if $x\in \mfn_{\varphi}$, then also $(\omega\otimes \id)\Delta(x) \in \mfn_{\varphi}$ for any $\omega \in \Bc(\Hsp)_*$, by left invariance of $\varphi$. So if $\xi,\xi',\zeta,\zeta'\in \Hsp$ and $\{e_k\}$ is an orthonormal basis of $\Hsp$, we compute for $\eta,\eta' \in \msD(\Qb^2)$ that: 
\begin{eqnarray*}
\langle \Lambda_{\varphi}(\theta_{\xi,\eta}) , \Lambda_{\varphi}((\omega_{\zeta,\zeta'}\otimes \id)\Delta(\theta_{\xi',\eta'}))\rangle &=& \langle \xi,(\omega_{\zeta,\zeta'}\otimes \id)(\Delta(\theta_{\xi',\eta'}))a^*\Qb^2a\eta\rangle \\
&=& \langle \zeta\otimes \xi,\Delta(\theta_{\xi',\eta'})(\zeta'\otimes a^*\Qb^2a\eta)\rangle \\
&=& \langle \zeta\otimes \xi,\Fct^*(1\otimes \theta_{\xi',\eta'})\Fct(\zeta'\otimes a^*\Qb^2a\eta)\rangle\\
&=& \sum_k \langle \Fct(\zeta\otimes \xi),e_k\otimes \xi'\rangle \langle e_k\otimes \eta',\Fct(\zeta'\otimes a^*\Qb^2a \eta)\rangle \\
&=& \sum_k \langle \Fct(\zeta\otimes \xi),e_k\otimes \xi'\rangle \langle e_k\otimes a\eta',\Fct(\zeta'\otimes \Qb^2a \eta)\rangle\\
&\underset{\eqref{EqRightInvOp}}{=}& \sum_k \langle \Fct(\zeta\otimes \xi),e_k\otimes \xi'\rangle \langle e_k\otimes \overline{\Qb a\eta},\overline{\tFct}^{*}(\zeta'\otimes \overline{\Qb a \eta'})\rangle\\
&=& \langle \xi \otimes \overline{\Qb a\eta},(\omega_{\zeta,\zeta'}\otimes \id\otimes \id)(\Fct^*_{12} \overline{\tFct}^*_{13})(\xi'\otimes \overline{\Qb a \eta'})\rangle\\
&=& \langle \Lambda_{\varphi}(\theta_{\xi,\eta}),(\omega_{\zeta,\zeta'}\otimes \id)(\Fct^* \overline{\tFct}^*) \Lambda_{\varphi}(\theta_{\xi',\eta'})\rangle. 
\end{eqnarray*} 
So, we find that 
\begin{equation}\label{EqOnBH}
\Lambda_{\varphi}((\omega_{\zeta,\zeta'}\otimes \id)\Delta(y))\rangle =  (\omega_{\zeta,\zeta'}\otimes \id)(\Fct^* \overline{\tFct}^*) \Lambda_{\varphi}(y),\qquad y\in \mfn_{\varphi},\zeta,\zeta'\in \Hsp.
\end{equation}

Now using \eqref{EqRealMU} and \eqref{EqModFG}, we can write 
\[
\Fct^* \overline{\tFct}^* = \Fc^* \Kc^*\overline{\tKc}^* \overline{\tFctt}^*.
\]
As also $(\Kc,\Qb,\hat{\Qb},\Lc)$ is a modular datum, we can see from the same computation as above that $\Kc^* \overline{\tKc}^*$ restricts to $L^2(\hat{\Bor}_N^-)\otimes L^2(\hat{\Bor}_N^-)$ as the unitary operator $\Pbb^*$ such that 
\[
(\omega\otimes \id)(\Pbb^*)\Lambda_{\varphi}(x) = \Lambda_{\varphi}((\omega\otimes \id)(\Kc^*(1\otimes x)\Kc)),\qquad \forall \omega \in L_{\hbar}(\Bor_N^-)_*,\forall x\in \mfn_{\varphi}. 
\]
But extending the proof of Lemma \ref{LemGrouplike}, we easily see that 
\begin{equation}\label{EqActionK}
\Kc^*(1\otimes \Eb(v))\Kc = \Eb(\sum_{t=1}^n 2(v,\seee_{N-t})\nevarpi_t \oplus v),\qquad v \in \Bor_N^-.
\end{equation}
By a small calculation, we then see that indeed $\Pbb = \Kc$, with its second leg acting through \eqref{EqActionDualHPlus}.

As 
\[
\Fc \in L_{\hbar}(B_N^-) \bar{\otimes} L_{\hbar}(B_N^-), \qquad \tFctt \in \overline{L_{\hbar}(B_N^-)}\bar{\otimes} L_{\hbar}(B_N^-), 
\]
compatibility of the standard representations of $L_{\hbar}(B_N^-)$ on $L^2(L_{\hbar}(B_N^-))\subseteq L^2(\Bc(\Hsp))$ then lets us conclude from \eqref{EqOnBH} that \eqref{EqFormRegMultUni} holds.
\end{proof}

\subsection{Dual representation} \label{subsection:dual}
We next aim to find a formula for $\Wbb$ which is more in line with the expression in \eqref{EqRealMU}, with $\Fc$ in the form  $\Rc$ as in \eqref{eq:R-init2}. 

Recall first from 
\eqref{TheoIndepofantipode} that if we view 
\[
\Fct \in L_{\hbar}(\Bor_N^-) \bar{\otimes}L_{\hbar}(\Bor_N^+), 
\]
and $L_{\hbar}(\Bor_N^-)$ is represented in its standard form on $L^2(\hat{\Bor}_N^-)$, then there exists a unique normal $*$-representation $\hat{\pi}$ of $L_{\hbar}(\Bor_N^+)$ on $L^2(\hat{\Bor}_N^-)$ such that 
\[
(\id\otimes \hat{\pi})\Fct = \Wbb. 
\]
If we denote the image of the standard generators of \eqref{EqStandardGenE} as 
\[
\eb(\hat{\varpi}_r) := \hat{\pi}(\eb(\sevarpi_r)),\qquad \hat{\stgE}_{rs} := \hat{\pi}(\stgE_{rs}),\qquad 1\leq r<s\leq N,
\]
then we trivially have 
\[
\Wbb =   \Gauss_\hbar\left(2 \sum_{t=1}^n \nee_{t} \otimes \hat{\varpi}_{N - t}\right) \prod_{r=1}^{\substack{n \\[-2pt] \longrightarrow}} \prod_{s=r+1}^{\substack{n+1 \\[-2pt] \longrightarrow}} \overline{F}_\hbar(\stgF_{rs} \otimes \hat{\stgE}_{rs}).
\]
Our goal will be to determine the operators $\eb(\hat{\varpi}_r)$ and $\hat{\stgE}_{rs}$ more explicitly. 

We start with the following trivial observation, which is immediate from the definitions of the structures involved (see \eqref{EqRightAction} for the second part of the statement):

\begin{lemma}
The standard unitary $\hbar$-representation of $\Bor_N^-$ and the unitary $\hbar$-representation of $\Tor_N^+$ in \eqref{EqActionDualHPlus} extend to a unitary $\hbar$-representation of $\wbtd{N}$: 
\[
\Eb(v)^{it} \Eb(w)^{is} = e^{2\pi i \hbar st (v,w)}\Eb(w)^{is}\Eb(v)^{it},\qquad w\in \Bor_N^-,v\in \Tor_N^+.
\]
The similar rule holds for the $\overline{L_{\hbar}(\Bor_N^-)}$-representation:
\[
\Eb(v)^{it} \Eb(\overline{w})^{is} = e^{2\pi i \hbar st (v,w)}\Eb(\overline{w})^{is}\Eb(v)^{it},\qquad w\in \Bor_N^-,v\in \Tor_N^+.
\]
\end{lemma}

If we hence write 
\[
\Eb(\overline{v}) = \Eb(-v),\qquad v\in \Tor_N^+,
\]
then we also obtain a unitary $\hbar$-representation of $\overline{\wbtd{N}}$. So, if we put
\[
\Heiss(\wbtd{N}) =  \wbtd{N}\widetilde{\oplus} \overline{\wbtd{N}},
\]
where $\widetilde{\oplus}$ indicates that we take the unique skew-symmetric bilinear form restricting to the usual one on $\wbtd{N}$ and $\overline{\wbtd{N}}$, and with cross form
\[
(u,\overline{v}) = (w,\overline{z}) = 0,\quad (u,\overline{z}) = -(u,z),\quad (w,\overline{v}) = (w,v),\qquad u,v\in \Bor_N^-,w,z\in \Tor_N^+, 
\]
the above formulas combine into a unitary $\hbar$-representation of $H(\wbtd{N})$ with 
\begin{equation}\label{EqRepTwistDirSum}
\eb(v\oplus \overline{w}) = \eb(v)\star \eb(\overline{w}),\qquad v,w\in \wbtd{N}. 
\end{equation}

Now from the above, we can by \eqref{EqModFG} rewrite \eqref{EqFormRegMultUni} as 
\begin{equation}\label{EqFormRegMultUniAlt}
\Wbb =  \Kc\overline{\tFc} \Fc \in \Bc(L^2(\hat{\Bor}_N^-)\otimes L^2(\hat{\Bor}_N^-)),
\end{equation}
with $\tFc$ as in \eqref{EqTildeFPrime}. Then using the notation of Section \ref{SecModularity} and of \eqref{EqProjMapsPii}, we get with $i = s$ that 
\[
(\pi_i \otimes \id)\Wbb = \Gauss(2\varpi_s \otimes \seee_{N-s}),
\]
which is enough to conclude the following lemma: 
\begin{lemma}
One has
\begin{equation}
\eb(\hat{\varpi}_s) = \eb(\sevarpi_s),\qquad 1\leq s\leq n,
\end{equation}
with $\eb(\sevarpi_s)$ as determined by \eqref{EqActionDualHPlus}. 
\end{lemma}

Now on the other hand, if $i = (s,k)$, we get that 
\begin{eqnarray*}
(\pi_i \otimes \id)\Wbb &=& \Gauss(2\varpi_i \otimes \seee_{N-s}) \prod_{s \le r \le n}^{\longra} \varphi\hr{f_i\oplus \overline{-\nwe_{N-s,n-r}}}   \prod_{s \le r \le n}^{\longra}\varphi\hr{f_i\oplus \seee_{N-s,n-r}}\\
&=&  \Gauss(2\varpi_i \otimes \seee_{N-s})  \overline{F}\hr{\eb(f_i)\otimes  \left(\boxplus_{s \le r \le n} \eb(\overline{-\nwe_{N-s,n-r}})\right)}\\
&& \hspace{5cm} \times \overline{F}\hr{\eb(f_i)\otimes  \left(\boxplus_{s \le r \le n} \eb(\seee_{N-s,n-r})\right)}.
\end{eqnarray*}
Since we can write 
\[
\Eb(\overline{-\nwe_{N-s,n-r}}) = \Eb(\overline{\seee_{N-s,r-s}})*\Eb(\seee_{N-s}),
\]
and all $\Eb(\overline{\seee_{N-s,r-s}})$ and $\Eb(\seee_{N-s,n-r'})$ strongly commute, it follows that $\boxplus_{s \le r \le n} \eb(\overline{-\nwe_{N-s,n-r}})$ and $\boxplus_{s \le r \le n} \eb(\seee_{N-s,n-r})$ $-\hbar$-commute, and we can further meaningfully simplify the above expression for $\Wbb$ (using commutativity of the form sum) to 
\begin{eqnarray*}
(\pi_i \otimes \id)\Wbb &=&  \Gauss(2\varpi_i \otimes \seee_{N-s})\\
&& 
\times \overline{F}\hr{\eb(f_i)\otimes  \left(\left(\boxplus_{s \le r \le n} \eb(\overline{-\nwe_{N-s,n-r}}) \right)\boxplus \left(\boxplus_{s \le r \le n} \eb(\seee_{N-s,n-r})\right)\right)}.
\end{eqnarray*}
On the other hand, this needs to equal  also 
\[
(\pi_i\otimes \id)\Wbb = \Gauss(2\varpi_i \otimes \seee_{N-s})
\overline{F}\hr{\eb(f_i)\otimes  \hat{\Eb}_s}.
\]
We hence deduce:
\begin{theorem}
For all and $1\leq s\leq n$, it holds that
\begin{equation}
\hat{\Eb}_s = \left(\boxplus_{s \le r \le n} \eb(\overline{-\nwe_{N-s,n-r}}) \right)\boxplus \left(\boxplus_{s \le r \le n} \eb(\seee_{N-s,n-r})\right).
\end{equation}
\end{theorem}

We now show how this formula can be extended to all standard generators, in the form of partition functions.

Note first that we have two embeddings of $\Nilp_N^+$ into $\Heiss(\wbtd{N})$, namely 
\begin{equation}\label{EqTwoEmb}
\Nilp_N^+\rightarrow \Heiss(\wbtd{N}),\quad  e_{abc} \mapsto e_{abc} =  e_{abc} \oplus 0,\qquad e_{abc} \mapsto f_{abc} = 0 \oplus \vartheta(e_{abc}), 
\end{equation}
with $\vartheta$ as in \eqref{EqVarTheta}. We obtain in this way a twisted direct sum 
\begin{equation}\label{EqEmbTwistSum}
\Nilp_N^+\widetilde{\oplus}\Nilp_N^+ \subseteq \Heiss(\wbtd{N}).
\end{equation}
\begin{lemma}
The resulting cross relations on $\Nilp_N^+\widetilde{\oplus}\Nilp_N^+$ are determined by 
\begin{equation}\label{EqCrossNilp}
(e_{abc},f_{a'b'c'}) = 0,\qquad c\neq 0\textrm{ or }c'\neq 0,\qquad 
(e_{ab0},f_{a'b'0}) = \left\{\begin{array}{lll} 1 & \textrm{if}& a=a',\\ -1/2 &\textrm{if}& |a-a'|=1\\
0&\textrm{if}&|a-a'|\geq 2.\end{array}\right.
\end{equation}
\end{lemma}
\begin{proof}
It follows from a straightforward verification.
\end{proof}

The diagram $\wseC_N(2)$ associated to $\Nilp_N^+\widetilde{\oplus}\Nilp_N^+$ (with respect to its given basis) is obtained by stacking a second (horizontally reflected) copy of $\wseC_N$ (the second summand) on top of $\wseC_N$ (the first summand), with connecting full arrows going up between vertices at the same position, and half-arrows going down between vertices at distance one apart. See Figure \ref{FigTriangDoublePrimeDouble}.

\begin{figure}[ht]
\adjustbox{scale=0.7,center}{%
\begin{tikzcd}
	&&  && f_{112}&& \\
	& && f_{211}&& f_{121} & \\
&& f_{310} && f_{220} && f_{130}  \\
&& e_{310} && e_{220} && e_{130}  \\
	& && e_{211} && e_{121} & \\
	&&  && e_{112} && 
	\arrow[dashed, from=3-5, to=3-3]
	\arrow[dashed, from=3-7, to=3-5]
\arrow[from=2-4, to=3-5]
\arrow[from=2-6, to=3-7]
\arrow[from=1-5, to=2-6]
\arrow[from=2-6, to=2-4]
\arrow[from=3-3, to=2-4]
\arrow[from=3-5, to=2-6]
\arrow[from=2-4,to=1-5]
	\arrow[dashed, from=4-5, to=4-3]
	\arrow[dashed, from=4-7, to=4-5]
\arrow[from=5-4, to=4-5]
\arrow[from=5-6, to=4-7]
\arrow[from=6-5, to=5-6]
\arrow[from=5-6, to=5-4]
\arrow[from=4-3, to=5-4]
\arrow[from=4-5, to=5-6]
\arrow[from=5-4,to=6-5]
\arrow[from=4-3,to=3-3]
\arrow[from=4-5,to=3-5]
\arrow[from=4-7,to=3-7]
\arrow[dashed, from=3-3, to=4-5]
\arrow[dashed, from=3-5, to=4-7]
\arrow[dashed, from=3-7, to=4-5]
\arrow[dashed, from=3-5, to=4-3]
\end{tikzcd}
}
\caption{The diagram $\protect\wseC_4(2)$.}
\label{FigTriangDoublePrimeDouble}
\end{figure}
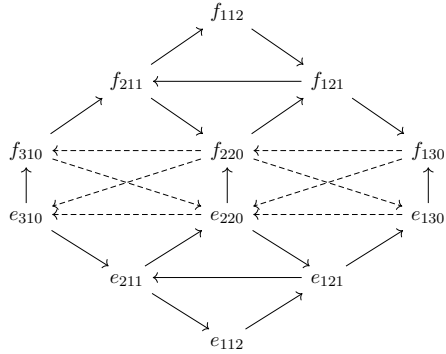

Recall now again the setting of Section \ref{SecStandGenPrel} and Section \ref{SecStandGen}, and the $\btd{N}$-labelled graph $\bbE = (\Gamma_{\bbE},\ell)$ as in Figure  \ref{fig:E-graph}, corresponding to the word $\mathbf{w}_0 \in B_N$ as \eqref{EqLongestWord1}. We slightly rewrite the labelling as to make it explicitly $\Bor_N^+$-valued (Figure \ref{fig:E-graphNilp}).
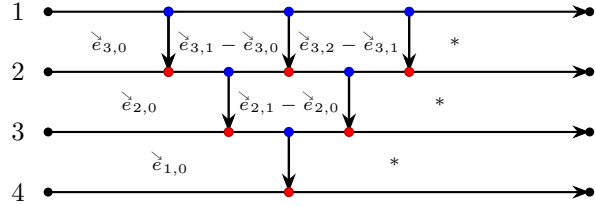
\begin{figure}[ht]
\begin{tikzpicture}[>=Stealth, line cap=round, x=8mm, y=8mm]
  \tikzset{
    wire/.style   ={line width=0.8pt},
    post/.style   ={line width=0.9pt,->},
    reddot/.style ={circle,draw=red!70!black,fill=red,inner sep=1.2pt},
    bluedot/.style={circle,draw=blue!70!black,fill=blue,inner sep=1.2pt},
    elab/.style   ={font=\small},
    blackdot/.style={circle,fill=black,inner sep=1.2pt},
    elab/.style   ={font=\small}
  }

  \draw[wire,->] ( -1,3) -- (8,3);
  \draw[wire,->] (  -1,2) -- ( 8,2);
  \draw[wire,->] (  -1,1) -- ( 8,1);
  \draw[wire,->] (  -1,0) -- ( 8,0);

  \foreach \x in {1,3,5,}{
    \draw[post] (\x,3) -- (\x,2);
    \node[bluedot] at (\x,3) {};
    \node[reddot]  at (\x,2) {};
  }
  \foreach \x in {2,4}{
    \draw[post] (\x,2) -- (\x,1);
    \node[bluedot] at (\x,2) {};
    \node[reddot]  at (\x,1) {};
  }
  \foreach \x in {3}{
    \draw[post] (\x,1) -- (\x,0);
    \node[bluedot] at (\x,1) {};
    \node[reddot]  at (\x,0) {};
  }
  
   \node[blackdot] at (-1,0) {};
   \node[blackdot] at (-1,1) {};
   \node[blackdot] at (-1,2) {};
   \node[blackdot] at (-1,3) {};
   \node[blackdot] at (8,0) {};
   \node[blackdot] at (8,1) {};
   \node[blackdot] at (8,2) {};
   \node[blackdot] at (8,3) {};
  \node[elab] at (-1.5,3) {$1$};
  \node[elab] at (-1.5,2) {$2$};
  \node[elab] at (-1.5,1) {$3$};
  \node[elab] at (-1.5,0) {$4$};

  \node[elab, anchor=east] at (1.5,0.5) {{\tiny $\seee_{1,0}$}};
  \node[elab, anchor=west] at (4.5,0.5) {{\tiny $*$}};

  \node[elab, anchor=east] at (1,1.5) {{\tiny $\seee_{2,0}$}};
  \node[elab]              at (3.0,1.5) {{\tiny $\seee_{2,1}-\seee_{2,0}$}};
  \node[elab, anchor=west] at (5.25,1.5) {{\tiny $*$}};

  \node[elab, anchor=east] at (.5,2.5) {{\tiny $\seee_{3,0}$}};
  \node[elab]              at (2.0,2.5) {{\tiny $\seee_{3,1}-\seee_{3,0}$}};
  \node[elab]              at (4.0,2.5) {{\tiny $\seee_{3,2}-\seee_{3,1}$}};
  \node[elab, anchor=west] at (5.5,2.5) {{\tiny $*$}};

\end{tikzpicture}
\caption{The $\Nilp_4^+$-labelled graph $\mathbb{E}_4 = (\Gamma_{\mathbb{E}_4},\ell)$.}
\label{fig:E-graphNilp}
\end{figure}
Note that the labelling on the utmost right is not needed to compute our partition functions, so we don't need to specify it. In particular, for all intents and purposes we can consider the labelling as being $\Nilp_N^+$-valued.

Consider now again the word $\mathbf{w}_0\mathbf{w}_0$ which concatenates $\mathbf{w}_0$ with itself. We then label the associated planar directed graph $\Gamma_{\mathbf{w}_0\mathbf{w}_0}$ with elements of $\Nilp_N^+(2)$ via a function $\ell(2)$ (Figure \ref{fig:E-graphDoubled}). We write the resulting $\Nilp_N^+(2)$-colored graph as $\mathbb{E}(2)$. 

\begin{figure}[ht]
\begin{tikzpicture}[>=Stealth, line cap=round, x=8mm, y=8mm]
  \tikzset{
    wire/.style   ={line width=0.8pt},
    post/.style   ={line width=0.9pt,->},
    reddot/.style ={circle,draw=red!70!black,fill=red,inner sep=1.2pt},
    bluedot/.style={circle,draw=blue!70!black,fill=blue,inner sep=1.2pt},
    elab/.style   ={font=\small}, 
    blackdot/.style={circle,fill=black,inner sep=1.2pt},
    elab/.style   ={font=\small}
  }

  \draw[wire,->] ( -1,3) -- (14,3);
  \draw[wire,->] (  -1,2) -- ( 14,2);
  \draw[wire,->] (  -1,1) -- ( 14,1);
  \draw[wire,->] (  -1,0) -- ( 14,0);

  \foreach \x in {1,3,5,7,9,11,}{
    \draw[post] (\x,3) -- (\x,2);
    \node[bluedot] at (\x,3) {};
    \node[reddot]  at (\x,2) {};
  }
  \foreach \x in {2,4,8,10}{
    \draw[post] (\x,2) -- (\x,1);
    \node[bluedot] at (\x,2) {};
    \node[reddot]  at (\x,1) {};
  }
  \foreach \x in {3,9}{
    \draw[post] (\x,1) -- (\x,0);
    \node[bluedot] at (\x,1) {};
    \node[reddot]  at (\x,0) {};
  }
   \node[blackdot] at (-1,0) {};
   \node[blackdot] at (-1,1) {};
   \node[blackdot] at (-1,2) {};
   \node[blackdot] at (-1,3) {};
   \node[blackdot] at (14,0) {};
   \node[blackdot] at (14,1) {};
   \node[blackdot] at (14,2) {};
   \node[blackdot] at (14,3) {};
  \node[elab] at (-1.5,3) {$1$};
  \node[elab] at (-1.5,2) {$2$};
  \node[elab] at (-1.5,1) {$3$};
  \node[elab] at (-1.5,0) {$4$};

  \node[elab, anchor=east] at (1.5,0.5) {{\tiny $\seee_{1,0}$}};
  \node[elab] at (6,0.5) {{\tiny $\sef_{1,0}-\seee_{1,0}$}};
  \node[elab, anchor=west] at (10.5,0.5) {{\tiny $*$}};

  \node[elab, anchor=east] at (1,1.5) {{\tiny $\seee_{2,0}$}};
  \node[elab]              at (3.0,1.5) {{\tiny $\seee_{2,1}-\seee_{2,0}$}};
  \node[elab] at (6,1.5) {{\tiny $\sef_{2,0}-\seee_{2,1}$}};
  \node[elab]              at (9,1.5) {{\tiny $\sef_{2,1}-\sef_{2,0}$}};
  \node[elab, anchor=west] at (11,1.5) {{\tiny $*$}};

  \node[elab, anchor=east] at (.5,2.5) {{\tiny $\seee_{3,0}$}};
  \node[elab]              at (2.0,2.5) {{\tiny $\seee_{3,1}-\seee_{3,0}$}};
  \node[elab]              at (4.0,2.5) {{\tiny $\seee_{3,2}-\seee_{3,1}$}};
  \node[elab] at (6.0,2.5) {{\tiny $\sef_{3,0}-\seee_{3,2}$}};
  \node[elab]              at (8.0,2.5) {{\tiny $\sef_{3,1}-\sef_{3,0}$}};
  \node[elab]              at (10.0,2.5) {{\tiny $\sef_{3,2}-\sef_{3,1}$}};
  \node[elab, anchor=west] at (11.5,2.5) {{\tiny $*$}};
\end{tikzpicture}
\caption{The $\Nilp_4^+(2)$-labelled graph $\mathbb{E}_4(2) = (\Gamma_{\mathbb{E}_4(2)},\ell(2))$.}
\label{fig:E-graphDoubled}
\end{figure}
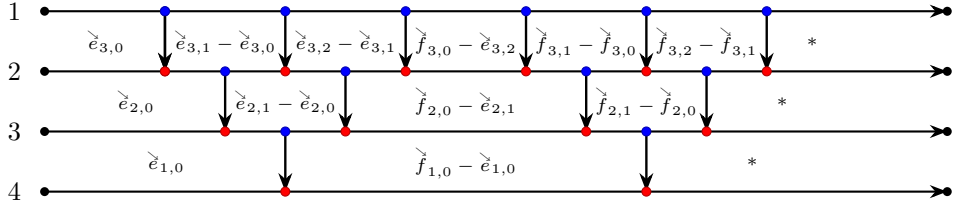

It is elementary to check that this labelling is indeed consistent with the graph structure (again, the utmost right labelling is just for padding, and the precise choice of these vectors are not relevant for the further computations).

Note now that the unitary $\hbar$-representation \eqref{EqEmbTwistSum} of $\Heiss(\wbtd{N})$ on $L^2(\hat{\Bor}_N^-)$ restricts via \eqref{EqRepTwistDirSum} to a unitary $\hbar$-representation of $\Nilp_N^+\widetilde{\oplus}\Nilp_N^+$. Let $Z_{\bullet}$ be the partition operator associated to a labelled graph $\bullet$. 

\begin{theorem}
We have 
\[
\hat{\stgE}_{rs} = 
Z_{\bbE_{[r,s]}(2)}.
\]
\end{theorem} 
\begin{proof}
Consider the embedding 
\begin{equation}
\iota \colon \Nilp_N^+ \rightarrow\Nilp_N^+\widetilde{\oplus} \Nilp_N^+,\qquad v \mapsto v \oplus 0. 
\end{equation}
We will show that there exists a unitary $u$ such that
\begin{equation}\label{EqUnitaryConj}
Z_{\bbE_{[r,s]}(2)} = u Z_{\Gamma_{\bbE_{[r,s]}},\iota \circ \ell} u^*.
\end{equation}
This is enough to establish the theorem: indeed, then both 
\[
\stgE_{rs} \mapsto Z_{\bbE_{[r,s]}}\qquad \textrm{ and }\qquad \stgE_{rs}\mapsto \hat{\stgE}_{rs}
\]
arise from normal $*$-representations of $L_{\hbar}(\Nilp_N^+)$. As these coincide on the $\eb_r$, they must be equal by  Proposition \ref{PropStandardGen}. 

By Lemma \ref{lem:Zmut}, we will have shown \eqref{EqUnitaryConj} if we can reduce $(\mathbf{w}_0\mathbf{w}_0,\ell(2))$ to $(\mathbf{w}_0,\ell)$ by a particular sequence of mutations (ignoring the rightmost columns of the latter). But this is exactly the content of Lemma \ref{LemLongestDouble}. 
\end{proof}

\section{Quantum  \texorpdfstring{$SL^+(N,\R)$}{special linear group} as a locally compact quantum group}

From any locally compact quantum group $ \Bbbb = (L^{\infty}(\Bbbb,\Delta)$ one can construct its Drinfeld double $D\Bbbb$, which is again a locally compact quantum group \cite{BV05}. 

More precisely, writing the dual of $\Bbbb$ as $\hat{\Bbbb} = (L(\Bbbb),\hat{\Delta})$, one considers
\[
L^\infty(D\Bbbb) = L^\infty(\Bbbb) \bar{\otimes} L(\Bbbb) \subset \Bc(L^2(\Bbbb)\otimes L^2(\Bbbb))
\]
with the comultiplication 
\[
\Delta_{D\Bbbb} = (\id \otimes \sigma \otimes \id)(\id \otimes \ad(\Wbb) \otimes \id)(\Delta \otimes \hat{\Delta}),
\]
where $ \ad(\Wbb)(x) = \Wbb x \Wbb^* $ denotes conjugation by the left regular multiplicative unitary $ \Wbb $ and $ \sigma $ is the tensor flip. The unitary antipode and scaling group of $ L^\infty(D\Bbbb) $ are given by 
\[
R_{D\Bbbb} = (R \otimes \hat{R}) \ad(\Wbb) = \ad(\Wbb^*) (R \otimes \hat{R}), \quad 
\tau^{D\Bbbb}_t = \tau_{-t} \otimes \hat{\tau}_t,  
\]
respectively. 

We are interested in the case $ \Bbbb = \Bbbb_{N, \hbar}^- $, and we will identify $ L^\infty(\Bbbb) = L(\Bbbb_{N, \hbar}^+) = L_\hbar(\Bor_N^+) $ and $ L(\Bbbb) = L(\Bbbb_{N, \hbar}^-) = L_\hbar(\Bor_N^-) $, using the notation introduced at the start of Section \ref{secmodularstructure}, so that $ L^\infty(D \Bbbb_{N, \hbar}^-) = L_\hbar(\Bor_N^+) \bar{\otimes} L_\hbar(\Bor_N^-) $.

Let us write $ L_\hbar(\Tor_N) $ for the diagonal copy of $ L_\hbar(\Tor_N^\pm) $ inside $ L_\hbar(\Bor_N^+) \bar{\otimes} L_\hbar(\Bor_N^-)$ generated by the operators
\[
\Db_s = \Kb_s \otimes \Lb_s,\qquad 1 \leq s \leq n,
\]
compare \eqref{EqCartanK}, \eqref{EqCartanL}, and consider the von Neumann subalgebra $ M \subseteq L^\infty(D \Bbbb_{N, \hbar}^-) $ generated by $ L_\hbar(\Nilp_N^+) \bar{\otimes} L_\hbar(\Nilp_N^-) $ and $ L_\hbar(\Tor_N) $. 

\begin{prop}
The von Neumann algebra $ M $ is a Baaj--Vaes subalgebra of $ L^\infty(D \Bbbb_{N, \hbar}^-) $. That is, we have $ \Delta_{D \Bbbb_{N, \hbar}^-}(M) \subseteq M \bar{\otimes} M $, and $ M $ is invariant under the unitary antipode and the scaling group of $ L^\infty(D \Bbbb_{N, \hbar}^-) $. 
\end{prop}

\begin{proof}
Consider the Fock--Goncharov flip $ \Fct = \Kc \Fc = \Fctt \Kc $. 
Using \eqref{eq:F-1or2} one easily checks that   
$ \Fc \in L_\hbar(\Nilp_N^-) \bar{\otimes} L_\hbar(\Nilp_N^+) $, or equivalently, $ \Fc_{21} \in L_\hbar(\Nilp_N^+) \bar{\otimes} L_\hbar(\Nilp_N^-) \subseteq M $. 
Moreover, by inspecting \eqref{eq:Fprime-1or2} one obtains $ \Fctt_{21} \in M $, but note that $ \Fctt_{21} $ is not contained in $ L_\hbar(\Nilp_N^+) \bar{\otimes} L_\hbar(\Nilp_N^-) $. 
It follows from the proof of Theorem \ref{TheoIndepofantipode} that we can similarly write $\Wbb = \Kc \Wbb' = \Wbb'' \Kc $ with $ \Wbb'_{21} \in L_\hbar(\Nilp_N^+) \bar{\otimes} L_\hbar(\Nilp_N^-) \subseteq M $ and $ \Wbb''_{21} \in M $, see also the discussion at the start of section \ref{subsection:dual}. 

Let us now abbreviate $ D = D \Bbbb_{N, \hbar}^- $ and observe that the left regular representation of $ \Bbbb_{N, \hbar}^- $ is given by $ \widehat{\Wbb} = \Wbb^*_{21} = (\Wbb')^*_{21} \Kc_{21}^* = \Kc_{21}^* (\Wbb'')^*_{21} $. 
According to Proposition \ref{Propstandardcoproducts} we have $ \Delta(\Lb_s) = \Lb_s \otimes \Lb_s $ and $ \hat{\Delta}(\Kb_s) = \Kb_s \otimes \Kb_s $ for $ 1 \leq s \leq n $. 
Using \eqref{EqVanishing} one checks that $ (\hat{\Delta} \otimes \Delta)(\Db_s) $ strongly commutes with $ \Kc_{32} $ for $ 1 \leq s \leq n $, and as a consequence we obtain 
\[
\Delta_D(L_\hbar(\Tor_N)) \subseteq (\Wbb')^*_{23} (L_\hbar(\Tor_N) \bar{\otimes} L_\hbar(\Tor_N)) (\Wbb')_{23} \subseteq M \bar{\otimes} M.
\]

We next claim $ \Delta_D(1 \otimes L_\hbar(\Nilp_N^-)) \subseteq M \bar{\otimes} M $. By Proposition \ref{PropStandardGenDual} it is enough to check $ \Delta_D(1 \otimes \Fb_s^{it}) \in M \bar{\otimes} M $ for the standard generators $ \Fb_s $ for $ 1 \leq s \leq n $ defined in \eqref{EqGenFs}. Due to Proposition \ref{Propstandardcoproducts} we obtain 
\begin{align*}
\Delta_D(1 \otimes \Fb_s) &= (\Wbb')^*_{23} \Kc^*_{23} (1 \otimes \Fb_s \otimes 1 \otimes \Lb_s \boxplus 1 \otimes 1 \otimes 1 \otimes \Fb_s) \Kc_{23} \Wbb'_{23} \\
&= (\Wbb')^*_{23} (1 \otimes \Fb_s \otimes \Kb_s \otimes \Lb_s \boxplus 1 \otimes 1 \otimes 1 \otimes \Fb_s)\Wbb'_{23},   
\end{align*}
using that $ \Kc^*( \Eb(\nee_{s, k}) \otimes 1)\Kc = \Eb(\nee_{s, k}) \otimes \Eb(\seee_{N - s}) $ for all $ 0 \leq k \leq s - 1 $ by \eqref{EqPairFundWeight}. Combining this again with $ \Wbb' \in L_\hbar(\Nilp_N^-) \bar{\otimes} L_\hbar(\Nilp_N^+) $ we conclude that $ \Delta_D(1 \otimes \Fb_s^{it}) \in M \bar{\otimes} M $ for $ 1 \leq s \leq n $ as required. Using Proposition \ref{PropStandardGen} one verifies $ \Delta_D(L_\hbar(\Nilp_N^+) \otimes 1) \subseteq M \bar{\otimes} M $ in a similar way, which allows us to conclude that $ \Delta_D(M) \subseteq M \bar{\otimes} M $.

Next we use \eqref{EqFormUnitAntip} to show that $ M $ is invariant under the unitary antipode of the Drinfeld double. Firstly, we observe $ (\hat{R} \otimes R)(\Db_s) = \Db_s^{-1} $ for $ 1 \leq s \leq n $, and using that $ \Db_s^{-1} $ strongly commutes with $ \Kc_{21} $ combined with $ \Wbb_{21}'' \in M $ we get $ R_D(L_\hbar(\Tor_N)) \subseteq M $. 
For the standard generators $ \Fb_s $ for $ 1 \leq s \leq n $ we obtain 
\begin{align*}
R_D(1 \otimes \Fb_s) &= \Wbb_{21}'' \Kc_{21} (\hat{R} \otimes R)(1 \otimes \Fb_s) \Kc_{21}^* (\Wbb_{21}'')^* \\
&= \Wbb_{21}'' \Kc_{21} (1 \otimes \Fb_s * \Lb_s^{-1}) \Kc^*_{21} (\Wbb_{21}'')^* \\ 
&= \Wbb_{21}'' (\Kb_s^{-1} \otimes \Fb_s * \Lb_s^{-1}) (\Wbb_{21}'')^*, 
\end{align*}
and using again $ \Wbb_{21}'' \in M $ and Proposition \ref{PropStandardGenDual} we deduce $ R_D(1 \otimes L_\hbar(\Nilp_N^-)) \subseteq M $. 
Analogously one verifies $ R_D(L_\hbar(\Nilp_N^+) \otimes 1) \subseteq M $.  

Finally, recall from Theorem \ref{TheoScalGroup} and Theorem \ref{TheoModFGFin} 
that the scaling groups of $ L_\hbar(\Bor_N^\pm) $ are implemented by conjugation with $ \hat{\Qb}^{it} $ and $ \Qb^{it} $, respectively. Here $ \hat{\Qb} = \Lb_\delta^{1 + |\hbar|^{-1}} $ where $ \Lb_\delta = \Eb(\sum \nevarpi_i) $, and $ \Qb = \Kb_\delta^{-(1 + |\hbar|^{-1})} $ where $ \Kb_\delta = \Eb(\sum \sevarpi_i) $. Clearly, each $ \Db_s $ is fixed by $ \tau^D_t $ for all $ t \in \mathbb{R} $. Moreover, from  \eqref{EqPairFundWeight} we see that the standard generators $ \Fb_s $ are preserved under conjugation by $ \hat{\Qb}^{it} $ up to a scalar. Similarly, the operators $ \Eb_s $ are preserved under conjugation by $ \Qb^{it} $ up to a scalar. We conclude that $ M $ is invariant under the scaling group. 
\end{proof}

According to \cite[Proposition A.5]{BV05} we conclude that $ M = L^\infty(\G) $ for a uniquely determined locally compact quantum group $ \G $. 

\begin{defn}
We call $ \G = SL_\hbar^+(N, \mathbb{R}) $ the totally positive quantum $ SL(N, \mathbb{R}) $-group. 
\end{defn}

The general theory from \cite{BV05} allows one to describe the basic structure of $ D \Bbbb_{N, \hbar}^- $ and $ SL_\hbar^+(N, \mathbb{R}) $ in terms of the quantized Borel groups $ \Bbbb_{N, \hbar}^\pm $.  
A more detailed study of $ SL_\hbar^+(N, \mathbb{R}) $ and its representation theory will be left to future work.

\appendix

\section{The quantum dilogarithm function}\label{Apqdilog}

For $a_+,a_- \in \R^{\times}$, put
\begin{equation}\label{EqGenFormW}
W_{a_+,a_-}(z) = -\frac{\pi i}{2} \int_{\Omega} \frac{dy}{y}\, \frac{e^{-2 iyz}}{\sinh(a_+y)\sinh(a_-y)},\qquad z \in \R,
\end{equation}
with $\Omega$ any curve in $\C$ going from $-\infty$ to $+\infty$ with a small bump in the upper half plane around the origin. Gathering leading terms, this can be reduced to the alternative expression
\begin{equation}\label{EqGenFormWAlt}
W_{a_+,a_-}(z) = \frac{\pi^2}{12}\left(\frac{a_+}{a_-} + \frac{a_-}{a_+}\right) + \frac{\pi^2z^2}{a_+a_-} -2\pi \widetilde{W}_{a_+,a_-}(z), 
\end{equation}
where $\widetilde{W}_{a_+,a_-}(z)$ is defined through the real integral
\begin{equation}\label{EqGenFormWReal}
\widetilde{W}_{a_+,a_-}(z) = \int_0^{\infty} \frac{dy}{y}\left(\frac{\sin(2yz)}{2\sinh(a_+y)\sinh(a_-y)} - \frac{z}{a_+a_-y}\right).
\end{equation}

Now recalling \eqref{EqFunctWTheta}, we see that
\begin{equation}\label{EqWorToGen}
W_{a_+,a_-}(z) = W_{a_+/a_-}(2\pi z/|a_+|).
\end{equation}
Writing 
\begin{equation}\label{EqTwoVarPhi}
\varphi_{a_+,a_-}(z) = \exp\left(\frac{i}{2\pi}W_{a_+,a_-}(z)\right) = \exp\left(\frac{1}{4}\int_{\Omega} \frac{dy}{y}\, \frac{e^{-2 iyz}}{\sinh(a_+y)\sinh(a_-y)}\right),
\end{equation}
the resulting function for $a_+ =a_-^{-1}=:b$ is the quantum dilogarithm as it appears in e.g. \cite{Fad95,Kas01,FG09}. By \eqref{EqWorToGen}, we have the transformation formula 
\begin{equation}\label{EqWorToGenExp}
\varphi_{a_+,a_-}(z) = \varphi_{a_-/a_+}(2\pi z/|a_+|).
\end{equation}
Finally, 
\begin{equation}
G(a_+,a_-;z) = \exp(i \widetilde{W}_{a_+,a_-}(z))
\end{equation}
gives the \emph{hyperbolic gamma function} as introduced in \cite{Rui97} (see also \cite[Appendix A]{Rui05}). 

Note that 
\[
W_{1/\hbar}(t) = \sgn(\hbar) W_{1/|\hbar|}(t),\qquad t\in \R,\hbar\in \R^{\times}.
\]
We then have the following alternative expression for $W_{1/|\hbar|}$, see \cite[Proposition 2.4]{GK21}.

\begin{lemma}\label{LemOtherDescW}
With $\theta>0$, we have
\begin{equation}\label{EqDifferentExpW}
W_{\theta}(t) = \int_0^{\infty} \frac{\ln(1+s^{-\theta})}{s+e^{-t}}\rd s,\qquad t \in \R.
\end{equation}
\end{lemma}
This is the expression for $W_{\theta}$ as it appears in \cite{Wor00}. 

In the following, we recall some of the structural properties of the quantum dilogarithm that we will need.

Fix $\theta >0$. First note that from \eqref{EqFunctWTheta}, we see that $W_{\theta}$ extends holomorphically to the strip 
\[
\{z\in \C\mid |\Imm(z)|<\pi(1+\theta^{-1})\}.
\]
It furthermore satisfies the functional identity 
\[
W_{\theta}(x+\pi i ) - W_{\theta}(x-\pi i) = 2\pi i \ln(1+e^{\theta x}),\qquad x\in \R 
\]
(see \cite[Lemma 1.1]{Wor00}). Hence
\begin{equation}\label{EqDefV}
V_{\theta}(z) := \exp\left(\frac{1}{2\pi i}W_{\theta}(z)\right) =  \varphi_{-\theta^{-1}}(z)= 1/\varphi_{\theta^{-1}}(z)
\end{equation}
admits a meromorphic extension 
\[
V_{\theta}\colon  \C\rightarrow \hat{\C}.
\]
(We use the notation \eqref{EqDefV} to be in agreement with \cite{Wor00}.) Note that writing 
\begin{equation}\label{DefChi}
\chi= \pi(\theta+\theta^{-1})/24,
\end{equation}
we get from \eqref{EqGenFormWAlt} the identity 
\begin{equation}\label{EqIdentifSpecial}
V_{\theta}(z) = e^{-i\chi} e^{-i\theta z^2/8\pi}G(2\pi,2\pi/\theta;z)
\end{equation}
as claimed in \cite[Equation (A.18)]{Rui05}.

The following Lemma gathers results from \cite[Section 1 and Appendix A]{Wor00} and  \cite[Appendix A]{Rui05} (the latter reference extends certain results of \cite{Wor00} to a larger domain for the deformation parameter). 

\begin{prop}\label{PropFundPropert}
Let $\theta>0$. The meromorphic function $V_{\theta}$ satisfies the following properties:
\begin{enumerate}
\item\label{Prop1} Its zeroes are located at 
\[
Z = \{\pi i \left((2k+1) + (2l+1)\theta^{-1}\right)\mid k,l\in \Z_{\geq0}\},
\]
while its poles are located at $-Z$. The multiplicity of a zero $z$ is counted by the number of couples $(k,l)\in \Z_{\geq0}^2$ for which $z= \pi i((2k+1) + (2l+1)\theta^{-1})$, and similarly for poles. In particular, all zeroes and poles are simple if $\theta$ is irrational. 
\item\label{Prop2} It satisfies the following functional equations: for $z\in \C$,
\begin{equation}\label{EqFunctEq1}
V_{\theta}(z+2\pi i) = (1+ e^{i \pi \theta}e^{\theta z})V_{\theta}(z),
\end{equation}
\begin{equation}\label{EqFunctEq2}
V_{\theta}(z+2\pi i/\theta) = (1+e^{\pi i/\theta}e^z)V_{\theta}(z).
\end{equation}
\item\label{Prop3} It satisfies the following functional equation: for $z\in \C$,
\begin{equation}\label{EqComplexConj}
V_{\theta}(z)V_{\theta}(-z)= V_{\theta}(z)/\overline{V_{\theta}(-\overline{z})} = e^{-2i\chi}e^{-i\theta z^2/4\pi}.
\end{equation}
\item\label{Prop4} The functions $V_{\theta}$ and $V_{1/\theta}$ are related through 
\begin{equation}\label{EqModDualityV}
V_{1/\theta}(z) = V_{\theta}(z/\theta),\qquad z\in \R.
\end{equation}
\item\label{Prop5}
For fixed $0<\rho < \min\{1,\theta\}$ and with $y$ uniformly varying over a compact subset of $\R$, we have
\begin{equation}\label{EqEstima1}
V_{\theta}(x+iy) = 1+ \Oc(e^{\rho x}),\qquad x \rightarrow -\infty,
\end{equation}
\begin{equation}\label{EqEstima2}
e^{i\theta (x+iy)^2/4\pi}V_{\theta}(x+iy) = e^{-2i\chi} + \Oc(e^{-\rho x}),\qquad x \rightarrow + \infty.
\end{equation}
In particular, for any strip $S_b = \{z\in \C\mid -b\leq\Imm(z)\leq b\}$, there exists $a\geq 0$ such that, for any $\lambda>0$, the function $z \mapsto V_{\theta}(z)e^{-\lambda z^2}$ is bounded on $\{z\in S_b\mid |\Ree(z)|\geq a\}$. 
\end{enumerate}
\end{prop}

To prove the modularity of particular multiplicative unitaries, we will need some information about the ($\theta$-rescaled) Fourier dual of $V_{\theta}$ and its inverse. 

Let $\Sc(\R)$ be the class of Schwartz space functions with its natural Fréchet topology, and $\Sc'(\R)$ its dual, the space of tempered distributions. We write the pairing of a distribution and a function alternatively as 
\[
\langle f,g\rangle = \langle f(x),g(x)\rangle = \int_{\R} f(x)g(x)\rd x,\qquad f\in \Sc'(\R),g\in \Sc(\R).
\]
We will use below rescaled Fourier transforms, so 
\begin{equation}\label{EqFourTrans}
\hat{g}(x) =  \frac{\sqrt{\theta}}{2\pi} \int_{\R} e^{-i\theta xy/2\pi}g(y)\rd y,\qquad \check{g}(y) = \frac{\sqrt{\theta}}{2\pi} \int_{\R} e^{i\theta xy/2\pi}g(x)\rd x,\qquad g\in \Sc(\R),
\end{equation}
and similarly for their extension to the space of tempered distributions, 
\begin{multline*}
\int_{\R} \hat{f}(y)g(y)\rd y :=  \int_{\R} f(x) \hat{g}(x)\rd x,\qquad \int_{\R} \check{f}(y)g(y)\rd y :=  \int_{\R} f(x) \check{g}(x)\rd x,\\ f\in \Sc'(\R),g \in \Sc(\R).
\end{multline*}
In particular, we consider $L^{\infty}(\R) \subseteq \Sc'(\R)$, and as such can interpret Fourier duals of elements in $L^{\infty}(\R)$ as tempered distributions.

Since $V_{\theta}$, restricted to $\R$, has values on the unit circle and hence $1/V_{\theta}= \overline{V}_{\theta}$, its complex conjugate, we can consider the inverse Fourier duals of $V_{\theta}$ and $\overline{V}_{\theta}$ in the above sense of tempered distributions. 

\begin{prop}\label{PropFourDualV}
Let $\theta>0$, and put
\[
a = \pi(1+\theta^{-1}).
\]
Then 
\begin{equation}\label{EqDefNeedEqual}
\hat{\overline{V}}_{\theta}(y) = \alpha_{\theta}(y)\hat{V}_{\theta}(y),
\end{equation}
where 
\begin{equation}\label{EqDefAlphaMult}
\alpha_{\theta}(y) = e^{a\theta y/2\pi} e^{-i\theta y^2/4\pi}.
\end{equation}
\end{prop} 
Here the right hand side in \eqref{EqDefNeedEqual} is a priori defined on the Fréchet dense domain $\msD \subseteq \Sc(\R)$ of $f$ such that $\alpha_{\theta}f$ again lies in $\Sc(\R)$. What the result says in particular, is that the resulting functional uniquely extends to a tempered distribution, i.e.\ $\alpha_{\theta}$ is a multiplier for $\hat{V}_{\theta}$.

The result in Proposition \ref{PropFourDualV} again appears in \cite[Appendix B]{Wor00}, but under the restriction $\theta>2$. We use the results of \cite[Appendix A]{Rui05} to extend this result to arbitrary $\theta >0$. 
\begin{proof}
With $\chi$ as in \eqref{DefChi}, write 
\[
\widetilde{V}_{\theta}(y) = e^{-i\pi/4} e^{-2i\chi}  e^{i\theta y^2/4\pi}  V_{\theta}(y-ia).
\]
Then $\widetilde{V}_{\theta}$ is continuous  on $\R$ except for a simple pole at $0$, and extends to a meromorphic function that is holomorphic on the upper half plane. By the estimates \eqref{EqEstima1} and \eqref{EqEstima2}, we know that $\widetilde{V}_{\theta}$ is bounded at infinity, and we can hence make sense of the tempered distribution 
\begin{equation}\label{EqDistributionDual}
D(y) = P.V.(\widetilde{V}_{\theta}(y)) - \pi i \Res_0(\widetilde{V}_{\theta})\delta_0(y).
\end{equation}
We claim that this is exactly the distribution $\hat{V}_{\theta}(y)$. Indeed, as both distributions are tempered, it is enough to establish equality on
\[
\Ec(\R) = \{g\colon  x \mapsto P(x) e^{-x^2/2}\mid P \textrm{ polynomial}\},
\]
which is a Fréchet dense subspace of $\Sc(\R)$ (\cite[Theorem V.13]{RS81}), stable under Fourier duality.

Now from \cite[(A.21)]{Rui05} we get that, for any $0<\delta<a$, 
\[
V_{\theta}(z) = \frac{\sqrt{\theta}}{2\pi} \int_{\R+i\delta} e^{i\theta yz/2\pi}\widetilde{V}_{\theta}(y)\rd y,\qquad -a<\Imm(z) <-\delta,
\]
where we note that the integrand is $L^1$ on $\R+i\delta$ by  the behaviour of $x\mapsto V_{\theta}(x+iy)$ at infinity in Proposition \ref{PropFundPropert}.\eqref{Prop5}. We then find for $g\in \Ec(\R)$ and $\Omega$ as before that, again by Proposition \ref{PropFundPropert}.\eqref{Prop5}, the following string of equalities is justified for $0<\varepsilon$ small: 
\begin{eqnarray*}
\langle D(y),g(y)\rangle &=& \int_{\Omega} \widetilde{V}_{\theta}(y) g(y) \rd y\\
&=& \int_{\R+ i\delta} \widetilde{V}_{\theta}(y)g(y)\rd y \\
&=& \frac{\sqrt{\theta}}{2\pi}\int_{\R-i\varepsilon} \hat{g}(x) \int_{\R+ i\delta} \widetilde{V}_{\theta}(y) e^{i\theta xy/2\pi} \rd y \rd x \\
&=& \int_{\R-i\varepsilon} \hat{g}(x) V_{\theta}(x)\rd x \\
&=& \langle \hat{V}_{\theta}(y),g(y)\rangle.
\end{eqnarray*} 
In the following, we can hence write 
\begin{equation}\label{EqFourTransPVV}
\hat{V}_{\theta}(y) = D(y).
\end{equation}

To finish the proof, it is now sufficient to show that $\hat{V}_{\theta}$ is in the definition domain of multiplication with $\alpha_{\theta}$, and that \eqref{EqDefNeedEqual} holds. But clearly from  \eqref{EqDistributionDual} and \eqref{EqFourTransPVV} we have
\[
\hat{\overline{V}}_{\theta}(y) = P.V.(\overline{\widetilde{V}}_{\theta}(-y)) + \pi  i \overline{\Res_0(\widetilde{V}_{\theta})}\delta_0(y).
\]
So \eqref{EqDefNeedEqual} follows by observing that  \eqref{EqComplexConj} gives 
\[
\overline{\widetilde{V}}_{\theta}(-y) = \alpha_{\theta}(y)\widetilde{V}_{\theta}(y),\qquad y \in \R^{\times}, 
\]
and that $\alpha_{\theta}(0) = 1$.
\end{proof}

\section{Pentagon equation}\label{SecPentQE}

In this section, we give some details on the proof of Theorem \ref{TheoPentagon}. There are basically two arguments in the literature that we are aware of: the abstract operator algebraic approach in \cite{Wor00} (whose proof however has a restriction on the parameter range), and the approach in \cite{FKV01} obtained via considering the concrete operators $\Xb,\Pb$ and proving \eqref{TheoPentagon} via an integral transform with the deformation parameter extended to a particular region in the complex domain. In this section, we argue that, with a small modification, the approach in \cite{Wor00} can be completed to give a full proof of the theorem, and then comment on how this result is equivalent to the one obtained in \cite{FKV01}. 

Fix $\hbar\in \R^{\times}$, and let $\Ab= e^{\mathbf{a}},\Bb= e^{\mathbf{b}}$ be two strictly positive $\hbar$-commuting operators on a Hilbert space, i.e.\ satisfying \eqref{EqSkewCommDef}. Put $\mathbf{a} \dotplus \mathbf{b}$ the unique self-adjoint operator such that 
\[
e^{it(\mathbf{a}\dotplus \mathbf{b})} = (\Ab\star\Bb)^{it},\qquad t\in \R, 
\]
with $\Ab\star\Bb$ as in \eqref{EqStarProd}. It can indeed easily be shown that $\mathbf{a} \dotplus \mathbf{b}$ is the unique self-adjoint extension of the sum $\mathbf{a}+\mathbf{b}$ defined on the intersection of their domains, by considering the concrete presentation \eqref{EqHeisPicSvN} below and arguing that $\Xb+\Pb$ is essentially self-adjoint on $\Sc(\R)$ by means of the identity
\begin{equation}\label{EqTransformSum}
(e^{-\pi i \Xb^2}\Pb e^{\pi i \Xb^2})f = \Xb f +\Pb f,\qquad f \in \Sc(\R).
\end{equation}

We have to prove that, with $\varphi_{\hbar}$ as in \eqref{EqQuantExp}, it holds that
\begin{equation}\label{EqPentagonProof}
\varphi_{\hbar}(\mathbf{a})\varphi_{\hbar}(\mathbf{b}) = \varphi_{\hbar}(\mathbf{b})\varphi_{\hbar}(\mathbf{a}\dotplus \mathbf{b})\varphi_{\hbar}(\mathbf{a}),\qquad \forall \hbar>0 
\end{equation}
(the case $\hbar <0$ follows by symmetry, using $\overline{\varphi}_{\hbar} = \varphi_{-\hbar}$).

Alternatively, pick $b \in \R_{>0}$ with $b^2  = \hbar$. Then by the Stone-von Neumann theorem, we may as well take 
\begin{equation}\label{EqHeisPicSvN}
\Hsp = L^2(\R),\qquad \mathbf{a} = 2\pi b\Pb,\qquad \mathbf{b} = 2\pi b \Xb,
\end{equation}
with $\Xb,\Pb$ as in \eqref{EqInfGenPosMom}. Then using the notation \eqref{EqTwoVarPhi} and writing
\[
e_b(z) = \varphi_{b^{-1},b}(z) =  \varphi_{b^2}(2\pi bz),\qquad z\in \C,
\]
we get that \eqref{EqPentagonProof} is equivalent to proving
\begin{equation}\label{EqPentagonProofAlt}
e_b(\Pb)e_b(\Xb) = e_b(\Xb) e_b(\Xb\dotplus \Pb) e_b(\Pb). 
\end{equation}

Now \eqref{EqPentagonProof} is proven in \cite[Theorem 6.1]{Wor00} for $0<\hbar<1/2$. In fact, a careful inspection shows that the proof still works for $0<\hbar<1$. Indeed, the only property that is used crucially for that part of the argument is that the function 
\[
x \mapsto \frac{1-\varphi_{\hbar}(\ln(x))}{x}
\]
stays bounded in the neighborhood of $0$, which is guaranteed for $0<\hbar<1$ by Proposition \ref{PropFundPropert}.(4).

Now from  \eqref{EqFunctWTheta} we see that
\begin{equation}\label{EqFunctEquationW}
W_{1/\theta}(x) = W_{\theta}(x/\theta),\qquad x\in \R.
\end{equation}
Hence \eqref{EqPentagonProof} holds for all $\hbar\in \R_{>0}\setminus\{1\}$, or alternatively, \eqref{EqPentagonProofAlt} holds for all $b\in \R_{>0}\setminus\{1\}$. The missing case $b=1$ is now easily derived from the fact that $b \mapsto e_b(z)$ is continuous on $\R_{>0}$ for each $z\in \R$, which (by the dominated convergence theorem) is enough to guarantee that $b \mapsto e_b(\mathbf{a})$ is strongly continuous for each strictly positive operator $\mathbf{a}$. Hence both sides of \eqref{EqPentagonProofAlt} converge strongly as $b \rightarrow 1$ to their instances at $b=1$, giving that \eqref{EqPentagonProofAlt} holds for all $b>0$. 

Let us now comment on the proof of \cite{FKV01}. First note that $W_{b^{-1},b}(z)$ as in \eqref{EqGenFormW} is still well-defined for any $b\in \C\setminus i\R$ (taking the half-circle in the contour $\Omega$ of radius $\delta < |b|^{\pm1}$), and then defines on the halfplane $H_+ = \{z\in \C\mid \Ree(z)>0\}$ a function 
\[
W\colon  H_+ \times \R \rightarrow \C,\quad (b,z) \mapsto W_{b,b^{-1}}(z) 
\]
which is holomorphic in $b$ for each $z$. 

In fact, we easily see from \eqref{EqGenFormW}  that $W(b,x)$ remains bounded for $b$ ranging over a compact region, uniformly over all $x$ on the negative axis, and that we have there continuity in $b$,  uniformly over all $x$. We hence obtain a Banach space valued holomorphic function 
\[
W^-\colon  H_+ \rightarrow C_b(\R_{\leq 0}),\qquad W^-(b)(x) = W_{b,b^{-1}}(x),
\]
and consequently also a holomorphic function
\[
e^-\colon  H_+ \rightarrow C_b(\R_{\leq 0}),\qquad e^-(b)(x) = e_{b}(x).
\]
On the other hand, the identity \eqref{EqComplexConj}, the renormalisation \eqref{EqWorToGenExp} and holomorphicity in $b$ lead to the identity
\[
e_{b}(-x) = e^{-2\pi i(b^2+b^{-2})/24}e^{ix^2/2}\overline{e_{\overline{b}}(x)},\qquad x\in \R,b\in H_+,
\]
from which it is immediately concluded that we obtain in fact a Banach space valued holomorphic function 
\[
e\colon  H_+ \rightarrow C_b(\R),\qquad e(b)(x) = e_b(x).
\]
It is hence sufficient to show that \eqref{EqPentagonProofAlt} holds on any subset of $b$'s with a cluster point in $H_+$. From the result in \cite[Theorem 6.1]{Wor00}, one can take $(0,1/\sqrt{2})$ as such a set. In \cite[Section 6]{FKV01}, one rather shows \eqref{EqPentagonProofAlt} directly on the quadrant $\R_{>0} + i\R_{>0} \subseteq H_+$.

\section{Modularity of the rank  \texorpdfstring{$1$}{1} Fock--Goncharov flip}\label{SecModAxB}

In this section we present a proof of Theorem \ref{TheoRank1}. It is simply the proof of \cite[Proposition 2.3]{WZ02} with no restriction on the deformation parameter. 

We will need the notion of spectral measure: recall that if $\Hsp$ is a Hilbert space, $\Hb$ a self-adjoint operator on $\Hsp$ and $\xi,\xi'\in \Hsp$, the associated \emph{spectral measure} is the (absolutely) bounded (complex) measure $\mu_{\xi',\xi}$ on $\R$ determined by 
\[
\int_{\R} g(x) \rd \mu_{\xi',\xi}(x) = \langle \xi',g(\Hb)\xi\rangle,\qquad g \in C_b(\R). 
\]

We also recall that if $\Hsp$ is a Hilbert space and $\Hb$ a self-adjoint operator on $\Hsp$, then for any $\xi \in \cap_{z\in \C} \msD(e^{z\Hb})$ the map
\[
\C\rightarrow \Hsp,\qquad z\mapsto e^{z \Hb}\xi
\]
is automatically a Banach space valued holomorphic function.

\begin{prop}\label{PropDenseSubspCor}
Consider on $L^2(\R)$ the operators 
\[
\Fb = e^{\Xb},\qquad \Lb = e^{-4\pi^2 \hbar \Pb},\qquad \Eb = \Lb\star\Fb^{-1},\qquad  \Kb = \Fb^{-1}
\]
as in the statement of Theorem \ref{TheoRank1}. Then there exists a subspace $V \subseteq L^2(\R)$ such that 
\begin{itemize}
\item $V \subseteq \msD(\Fb^z) \cap \msD(\Lb^z) \cap \msD(\Eb^z) \cap \msD(\Kb^z)$ for all $z\in \C$,
\item $V$ is stable under each $\Fb^z,\Lb^z,\Eb^z,\Kb^z$, and
\item  for each $z\in \C$ the space $V$ is a core for any of the operators $\Fb^z,\Lb^z,\Eb^z$ and $\Kb^z$. 
\end{itemize} 
\end{prop}
\begin{proof}
We can take $V$ to be the linear span of all functions of the form $x\mapsto e^{\alpha x +\beta x^2}$ with $\alpha\in \C$ and $\Ree(\beta)<0$. Indeed, $V$ is stable under Fourier duality, and is easily seen to form a core of $e^{z\Xb}$ for each $z\in \C$, left stable by each such operator. By Fourier duality this then also holds for $e^{z\Pb}$ for each $z\in\C$. Finally, by \eqref{EqTransformSum} it also follows that $V$ is a core for each such $\Eb^z$, and stable under all $\Eb^z$. 
\end{proof} 

Consider now the setting and notation of Theorem \ref{TheoRank1}. In the terminology of Definition \ref{DefWeaklyMod}, it is sufficient to show that the triple $(\Wbb,\Lb^{\frac{1+|\hbar|^{-1}}{2}},\wWbb)$ is a weakly modular datum, since clearly $\Wbb$ commutes with 
\[
\hat{\Qb}\otimes \Qb = \Lb^{\frac{1+|\hbar|^{-1}}{2}} \otimes \Kb^{-\frac{1+|\hbar|^{-1}}{2}}.
\]

We slightly rewrite \eqref{EqLeftInvOp} as in \eqref{EqRightInvOpAltCorep} to be more in line with the setup of \cite{WZ02}. Namely, put  
\begin{equation}\label{EqAltZExpr}
Z = \Wbb_{21}^* = F_{\hbar}(\Eb\otimes \Fb)\exp\left(-\frac{i}{2\pi \hbar} \ln(\Kb)\otimes \ln(\Lb)\right)
\end{equation}
 and put 
\[
\widetilde{Z} = {\wWbb}_{,21} = \exp\left(\frac{i}{2\pi \hbar} \ln(\Kb) \otimes \ln(\overline{\Lb})\right) F_{\hbar}(\Kb^{-1}\star\Eb\otimes \overline{\Fb}).
\]
Then with $\hat{\Qb} = \Lb^{\frac{1+|\hbar|^{-1}}{2}}$, we need to show that
\begin{equation}\label{EqHelpInvRep}
\langle \xi'\otimes \overline{\hat{\Qb}^{-1}\eta'},\widetilde{Z}^*(\xi\otimes \overline{\hat{\Qb}\eta})\rangle = \langle \xi'\otimes \eta,Z(\xi\otimes \eta')\rangle,\qquad \forall \xi,\xi'\in L^2(\R),\eta\in \msD(\hat{\Qb}),\eta'\in \msD(\hat{\Qb}^{-1}). 
\end{equation}

We start with the following easy lemma, which follows immediately by functional calculus.

\begin{lemma}
Assume $\mathbf{a},\mathbf{b}$ are selfadjoint operators on a Hilbert space $\Hc$. Then for all $\xi,\xi',\eta,\eta'\in \Hc$, we have
\begin{equation}\label{EqAntiSwitch}
\langle \xi'\otimes  \overline{\eta'}, \exp\left(\frac{i}{2\pi\hbar}\mathbf{a}\otimes \overline{\mathbf{b}}\right)(\xi\otimes \overline{\eta})\rangle \\
= \langle \xi'\otimes \eta, \exp\left(\frac{i}{2\pi \hbar}\mathbf{a}\otimes \mathbf{b}\right)(\xi\otimes \eta')\rangle.
\end{equation}
\end{lemma}

We now show \eqref{EqHelpInvRep}.
Note first that in this identity, we may restrict to $\xi,\xi',\eta,\eta' \in V$, 
with $V$ as in Proposition \ref{PropDenseSubspCor}, as this space forms a core for $\hat{\Qb}^{\pm 1}$. In the following, we fix such $\eta,\eta',\xi,\xi'$. 

Put 
\begin{multline*}
\Cb = \Eb\otimes \Fb,\qquad \widetilde{\Cb}= \Kb^{-1}\star\Eb\otimes \overline{\Fb},\\
E_{\hbar} = \exp\left(-\frac{i}{2\pi \hbar}(\ln(\Kb)\otimes \ln(\Lb)\right),\quad \widetilde{E}_{\hbar} = \exp\left(-\frac{i}{2\pi \hbar}(\ln(\Kb)\otimes \ln(\overline{\Lb})\right), 
\end{multline*}
and for $y \in \R$ write 
\[
\varphi(y) = \langle \xi' \otimes \eta,\Cb^{iy/2\pi \hbar}E_{\hbar}(\xi\otimes \eta')\rangle,\qquad 
\psi(y) = \langle \xi'\otimes \overline{\hat{\Qb}^{-1}\eta'},\widetilde{\Cb}^{iy/2\pi\hbar}\widetilde{E}_{\hbar}(\xi\otimes \overline{\hat{\Qb}\eta})\rangle.
\]
By our choice of $\xi',\eta,\eta'$ we have that 
\[
z \mapsto \Cb^{iz} (\xi'\otimes \eta),\qquad z \mapsto \widetilde{\Cb}^{iz}(\xi'\otimes \overline{\hat{\Qb}^{-1}\eta'}) 
\]
are holomorphic maps $\C\rightarrow \Hsp$, which implies that $\varphi,\psi$ extend to holomorphic functions.

From the skew commutation between $\Lb$ and $\Fb$, resp.\ $\overline{\Lb}$ and $\overline{\Fb}$, we get
\begin{equation}
\langle \xi' \otimes \Lb^{-\overline{z}}\eta,\Cb^{iy/2\pi \hbar}E_{\hbar}(\xi\otimes \Lb^{z}\eta')\rangle =e^{y z}\varphi(y),
\end{equation}
\begin{equation}\label{EqUnifExp}
\langle \xi'\otimes \overline{\Lb}^{-\overline{z}}\,\overline{\hat{\Qb}^{-1}\eta'},\widetilde{\Cb}^{iy/2\pi\hbar}\widetilde{E}_{\hbar}(\xi\otimes \overline{\Lb}^{z}\,\overline{\hat{\Qb}\eta})\rangle = e^{-y z}\psi(y).
\end{equation}
Since the left hand sides and their derivatives are bounded in $y\in \R$ (for $z\in \R$ fixed), we find that $\varphi,\psi$ and their derivatives have faster than exponential decay at infinity, and so in particular $\varphi,\psi\in\Sc(\R)$. Moreover, if we denote by $\mu$ the spectral measure for $\ln(\Cb)$ associated to $\xi'\otimes \eta$ and $E_{\hbar}(\xi\otimes \eta')$, and by $\nu$ the spectral measure for $\ln(\widetilde{\Cb})$ associated to $\xi'\otimes \overline{\hat{\Qb}^{-1}\eta'}$ and $\widetilde{E}_{\hbar}(\xi\otimes \overline{\hat{\Qb}\eta})$, we see that
\[
\varphi(y) = \int_{\R} e^{ixy/2\pi \hbar} \rd\mu(x),\qquad \psi(y) = \int_{\R}e^{ixy/2\pi\hbar} \rd\nu(x),
\]
so it follows that $\mu,\nu$ are absolutely continuous with respect to the Lebesgue measure. Then, using the notation \eqref{EqFourTrans} with $\theta := |\hbar|^{-1}$, and putting $\varepsilon = \sgn(\hbar)$, we see that
\begin{equation}\label{EqSpecMeasQD}
\rd \mu(x) = (4\pi^2 |\hbar|)^{-1/2}\hat{\varphi}(\varepsilon x)\rd x,\qquad \rd\nu(x) = (4\pi^2 |\hbar|)^{-1/2}\hat{\psi}(\varepsilon x)\rd x.
\end{equation}

Now from \eqref{EqUnifExp} we get in particular, putting $z =- (1+\theta)/2$, that
\[
\psi(y) = e^{-(1+\theta)y/2}\langle \xi'\otimes \overline{\eta'},\widetilde{\Cb}^{iy/2\pi\hbar}\widetilde{E}_{\hbar}(\xi\otimes \overline{\eta})\rangle 
\]
But
\begin{eqnarray*}
&& \hspace{-2cm} \langle \xi'\otimes \overline{\eta'},\widetilde{\Cb}^{iy/2\pi\hbar}\widetilde{E}_{\hbar}(\xi\otimes \overline{\eta})\rangle  \\ &&= \langle (\Kb^{-1}\star\Eb)^{-iy/2\pi\hbar}\xi'\otimes \overline{\Fb}^{-iy/2\pi\hbar} \overline{\eta'}, \widetilde{E}_{\hbar}(\xi\otimes \overline{\eta})\rangle \\
&& =\langle (\Kb^{-1}\star\Eb)^{-iy/2\pi\hbar}\xi'\otimes \overline{\Fb^{iy/2\pi\hbar}} \overline{\eta'},\exp\left(-\frac{i}{2\pi \hbar}(\ln(\Kb)\otimes \ln(\overline{\Lb})\right)(\xi\otimes \overline{\eta})\rangle \\
&&\underset{\eqref{EqAntiSwitch}}{=} \langle (\Kb^{-1}\star\Eb)^{-iy/2\pi\hbar}\xi'\otimes \eta, \exp\left(-\frac{i}{2\pi \hbar}(\ln(\Kb)\otimes \ln(\Lb)\right)(\xi\otimes \Fb^{iy/2\pi\hbar}\eta')\rangle \\
&&= \langle \xi'\otimes \eta, ((\Kb^{-1}\star\Eb)^{iy/2\pi\hbar}\otimes 1) E_{\hbar}(1\otimes \Fb^{iy/2\pi\hbar})(\xi\otimes \eta')\rangle \\
&&\underset{\eqref{EqGausEqInv}}{=} \langle \xi'\otimes \eta, ((\Kb^{-1}\star\Eb)^{iy/2\pi\hbar}\otimes 1)(\Kb\otimes \Fb)^{iy/2\pi\hbar} E_{\hbar}(\xi\otimes \eta')\rangle\\
&&\underset{\eqref{EqStarProd}}{=} e^{i\pi\hbar (y/2\pi\hbar)^2} \langle \xi'\otimes \eta,(\Eb^{iy/2\pi\hbar}\otimes \Fb^{iy/2\pi\hbar})E_{\hbar}(\xi\otimes \eta')\rangle \\
&&= e^{iy^2/4\pi\hbar}\varphi(y). 
\end{eqnarray*}

We hence conclude that 
\begin{equation}\label{EqIdPsiPhi}
\psi(y) = e^{-(1+\theta) y/2}e^{iy^2/4\pi\hbar}\varphi(y). 
\end{equation}
Now recall again the notation and results of Proposition \ref{PropFourDualV}. Then from \eqref{EqIdPsiPhi} we see that  
\begin{equation}\label{EqIdPHPS}
\alpha_{\theta}^{\varepsilon}(\varepsilon y) \psi(y) = \varphi(y).
\end{equation}
We can now compute that 
\begin{eqnarray*}
\langle \xi'\otimes \overline{\hat{\Qb}^{-1}\eta'},\widetilde{Z}^*(\xi\otimes \overline{\hat{\Qb}\eta})\rangle &=& 
\langle \xi'\otimes \overline{\hat{\Qb}^{-1}\eta'}, \overline{F}_{\hbar}(\widetilde{\Cb}) \widetilde{E}_{\hbar}(\xi\otimes \overline{\hat{\Qb}\eta})\rangle \\
&=& \int_{\R} V_{\theta}^{-\varepsilon}(x)\rd \nu(x)
\\
&\underset{\eqref{EqSpecMeasQD}}{=}& (4\pi^2|\hbar|)^{-1/2}\langle V_{\theta}^{-\varepsilon}(x),\hat{\psi}(\varepsilon x)\rangle
\\
&=& (4\pi^2|\hbar|)^{-1/2}\langle (V^{-\varepsilon}_{\theta})^{\wedge}(\varepsilon y),\psi(y)\rangle \\
&\underset{\eqref{EqDefNeedEqual}}{=}&  (4\pi^2|\hbar|)^{-1/2}\langle (V^{\varepsilon}_{\theta})^{\wedge}(\varepsilon y),\alpha_{\theta}^{\varepsilon}(\varepsilon y)\psi(y)\rangle \\
&\underset{\eqref{EqIdPHPS}}{=}&  (4\pi^2|\hbar|)^{-1/2} \langle (V_{\theta}^{\varepsilon})^{\wedge}(\varepsilon y),\varphi(y)\rangle \\
&=&  (4\pi^2|\hbar|)^{-1/2} \langle V_{\theta}^{\varepsilon}(x),\hat{\varphi}(\varepsilon x)\rangle \\
&\underset{\eqref{EqSpecMeasQD}}{=}& \int_{\R} V_{\theta}^{\varepsilon}(x)\rd \mu (x)\\ 
&=&  \langle \xi'\otimes \eta,F_{\hbar}(\Cb) E_{\hbar}(\xi\otimes \eta')\rangle \\
&=&  \langle \xi'\otimes \eta,Z(\xi\otimes \eta')\rangle.
\end{eqnarray*}

\section{Multiplier  \texorpdfstring{C$^*$-algebras}{C*-algebras} and affiliated elements}\label{SecMult}
\begin{defn}
A \emph{multiplier} for a C$^*$-algebra $A$ is a linear map $m\colon  A \rightarrow A$ for which there exists a linear map $m^*\colon  A \rightarrow A$ (necessarily unique) such that 
\[
(m^*b)^* a = b^*(ma),\qquad \forall b,a \in A. 
\]
\end{defn}
By the closed graph theorem, multipliers are automatically norm-continuous. The space $M(A)$ of multipliers then becomes a unital C$^*$-algebra under composition and the above $*$-structure. Identifying 
\[
A \subseteq M(A),\qquad a \mapsto (a' \mapsto aa'),
\]
we have that $M(A)$ contains $A$ as a norm-closed $2$-sided ideal.

If $B$ is a C$^*$-algebra containing $A$ as a $*$-ideal, we get a $*$-homomorphism 
\[
B \rightarrow M(A), \qquad b \mapsto (a\mapsto ba).
\]
If $A$ is an \emph{essential} ideal, so $b\in B$ and $bA=0$ implies $b=0$, we get a $*$-embedding 
\[
B \hookrightarrow M(A),
\]
acting as identity on $A$. So, $M(A)$ is the largest  C$^*$-algebra containing $A$ as an essential ideal.

For example, if $X$ is a locally compact Hausdorff space, we have an identification
\[
C_b(X) \cong M(C_0(X)),\qquad f \mapsto (g\mapsto fg).
\]
As another example, let $\Hc$ be a Hilbert space. Then 
\[
\Bc(\Hsp) \cong M(\Kc(\Hsp)),\qquad x \mapsto (y \mapsto xy).
\]

If $A,B$ are C$^*$-algebras, a  $*$-homomorphism 
\[
\pi\colon  A \rightarrow M(B)
\]
is called \emph{non-degenerate} if 
\[
B = [\pi(A)B].
\]
If $A$ is a $*$-algebra and $\pi\colon  A \rightarrow \Bc(\Hsp)$ is a $*$-representation, one calls $\pi$ \emph{non-degenerate} if 
\[
\Hsp = [\pi(A)\Hsp].
\]
Viewing $\Bc(\Hc) = M(\Kc(\Hc))$, this is then the same as asking that $\pi\colon  A \rightarrow M(\Kc(\Hsp))$ is non-degenerate.

Whenever $\pi\colon  A \rightarrow M(B)$ is a non-degenerate $*$-homomorphism between C$^*$-algebras, there exists a unique extension of $\pi$ to a unital $*$-homomorphism 
\[
\pi\colon  M(A) \rightarrow M(B). 
\]

We will also have need of the affiliation relation and affiliated elements \cite{Wor91} (see also \cite{WN92,Wor95}). 

\begin{defn}
Let $A$ be a C$^*$-algebra. Let $z \in M(A)$ be such that $\|z\|\leq 1$ and 
\begin{equation}\label{EqDenseRange}
A = [(1-z^*z)^{1/2}A].
\end{equation}
Then we define $T_z$ to be the unbounded, densely defined operator 
\begin{equation}\label{EqDefT}
\msD(T_z) := (1-z^*z)^{1/2}A \rightarrow A,\qquad (1-z^*z)^{1/2}a \mapsto za.
\end{equation}
\end{defn}
It is easily seen from \eqref{EqDenseRange} that $T_z$ is well-defined. 

\begin{defn}
A densely defined linear operator $T\colon  \msD(T) \rightarrow A$ is said to be \emph{affiliated to $A$} if $T = T_z$ for some $z$ as above. Notation: 
\[
T \eta A,\qquad A^{\eta} = \{T \mid T \eta A\}.
\]
\end{defn}
\begin{remark}
\begin{itemize}
\item One can show that $T_z= T_w$ if and only if $w=z$, see the remark following \cite[Proposition 2.2]{Wor91}. Hence if $T \eta A$, it makes sense to write $z = z_T \in M(A)$ for the unique contraction satisfying \eqref{EqDenseRange} and $T = T_z$. 
\item Any $T\in M(A)$ is affiliated to $A$, by taking 
\[
z = T(1+T^*T)^{-1/2}.
\]
Conversely, if $T \eta A$ and $T$ is norm-bounded, then $T \in M(A)$. 
\end{itemize}
\end{remark}

For example, if $X$ is a locally compact Hausdorff space, one shows (\cite[Example 2]{Wor91}) that 
\[
C_0(X)^{\eta} = C(X) = \{\textrm{continuous functions }X \rightarrow \C\},
\]
where we view $f \in C(X)$ as an operator $T = M_f$ on $C_0(X)$ via multiplication: 
\[
\msD(M_f) = \{g\in C_0(X) \mid fg \in C_0(X)\},\qquad M_fg = fg,\qquad g\in C_0(X). 
\]
As another example one shows (\cite[Example 3]{Wor91}) that, for $\Hsp$ a Hilbert space,
\[
\Kc(\Hc)^{\eta} = \{\textrm{closed, densely defined operators on }\Hsp\},
\]
by sending $T$ to $z_T(1-z_T^*z_T)^{-1/2}$ (the latter in the sense of unbounded operators). 

We will need the following simple observations on affiliated elements, which we gather in the following lemmas. See \cite[Theorem 1.2]{Wor91} and  \cite[Theorem 6.1]{WN92} for the proof of the respective statements.

\begin{lemma}\label{LemHomCorr}
Let $A,B$ be C$^*$-algebras.
\begin{itemize}
\item If $\pi\colon  A \rightarrow M(B)$ is a non-degenerate $*$-homomorphism, then for any $T \,\eta\, A$ there exists a unique $\pi(T) \eta B$ such that 
\[
z_{\pi(T)} = \pi(z_T).
\]

\item If $T\, \eta\, A$ and $S\, \eta\, B$, there exists a unique $ (T \otimes S) \,\eta\, (A\otimes B)$ such that 
\begin{equation}\label{EqAffiliTens}
\msD(T)\odot \msD(S) \subseteq \msD(T\otimes S),\quad (T\otimes S)(a\otimes b) = Ta\otimes Sb,\qquad a\in \msD(T),b\in \msD(S).
\end{equation}
\end{itemize}
\end{lemma}

If $A$ is a C$^*$-algebra, we call
an affiliated element $T \,\eta\, A$ \emph{normal} if $z_T$ is normal. We will then need the following form of functional calculus (see \cite[Theorem 1.5, Theorem 1.6]{Wor91} and the discussion following it). Note first that if 
\[
\alpha\colon  \{z\mid |z|<1\} \rightarrow \C,\quad z\mapsto z(1-|z|^2)^{-1/2}
\]
and $g \in C_0(\C)$, then $g\circ \alpha$ is a $C_0$-function on the open unit disk, and in particular extends to a continuous function on the closed unit disk. 

\begin{lemma}[Functional calculus]\label{LemFunctCal}
Assume $A$ is a C$^*$-algebra and $T \,\eta\, A$ normal. Then there exists a unique map 
\[
C(\C) \rightarrow A^{\eta},\qquad f \mapsto f(T)
\]
such that $f(T) = \pi(f)$ for $\pi$ the non-degenerate $*$-homomorphism
\[
\pi\colon  C_0(\C) \rightarrow M(A),\qquad g \mapsto (g\circ \alpha)(z_T).
\]
\end{lemma} 

For example, if $A \subseteq \Bc(\Hsp)$ is a non-degenerate faithful $*$-representation of $A$, then any normal $T \,\eta\, A$ can be seen as an (unbounded) normal operator $\Tb$ on $\Hsp$, and the operator corresponding to $f(T)$ will be $f(\Tb)$ as obtained through the usual functional calculus for normal operators on $\Hsp$. In particular, we only need to have $f$ defined as a continuous function on the closed set 
\[
\Spec(\Tb) = \{\lambda \in \C \mid \Tb- \lambda \textrm{ bijective as a map }\msD(\Tb) \rightarrow \Hsp\}. 
\]
We need this observation in particular in case $\Tb$ is a self-adjoint positive operator on $\Hsp$, so
\[
\Spec(\Tb) \subseteq [0,\infty).
\]

\section{Rieffel's deformation theory}\label{SecRieff}

Although the material in Section \ref{SecCStarRieff} only refers to the deformation of very specific function algebras, we recall some of the general framework of  \cite{Rie93} to have easy access to the results obtained there. 

Let $A$ be any C$^*$-algebra with a (not necessarily norm-continuous) action $\alpha$ of $\hat{V}$ by $*$-automorphisms. We then let $A  ^u \subseteq A$ be the C$^*$-subalgebra of norm-continuous elements, i.e.\ those $a\in A$ for which 
\begin{equation}\label{EqActionMap}
\hat{v}\mapsto \alpha_{\hat{v}}(a)
\end{equation}
is norm-continuous. We let $A^{\infty} \subseteq A^u$ be the $*$-algebra of smooth functions, i.e.\ those for which the Banach-space valued functions \eqref{EqActionMap} have infinite partial derivatives.

If now $A$ is any C$^*$-algebra, we can endow the C$^*$-algebra $C_b(\hat{V},A)$ of bounded functions $\hat{V}\rightarrow A$ with the translation action\footnote{We change the direction of the translation action with respect to \cite{Rie93}, to be in line with the other conventions in our paper. The resulting sign changes that appear are of course inconsequential.}
\[
(\tau_{\hat{v}}f)(\hat{w}) = f(\hat{w}-\hat{v}),\qquad \hat{v},\hat{w}\in \hat{V},
\]
and consider the associated $*$-algebra $C^{\infty}_b(\hat{V},A)$ as above. Let $\Sc(\hat{V},A) \subseteq C^{\infty}_b(\hat{V},A)$ be the corresponding space of Schwartz functions $f$, i.e.\ those functions such that for any iterated partial derivative $g$ of $f$, the function $\hat{v}\mapsto p(\hat{v})g(\hat{v})$ is bounded for any polynomial function $p$ on $\hat{V}$. Then we obtain a well-defined map 
\[
\times_{\hbar}\colon  C^{\infty}_b(\hat{V},A) \times \Sc(\hat{V},A) \rightarrow \Sc(\hat{V},A),\qquad (f \times_{\hbar} g)(\hat{v}) = \int_V f(\hat{v} - \Jc u)\hat{g}(u) e^{2\pi i \hat{v}\cdot u} \rd u,
\]
see \cite[Proposition 3.3]{Rie93}. There then exists a unique associative algebra structure $\times_{\hbar}$ on $C_b^{\infty}(\hat{V},A)$ which turns the above map into a module, and $C_b^{\infty}(\hat{V},A)$ becomes a pre-C$^*$-algebra through the pointwise $*$-structure and the norm 
\[
\|f\| = \sup_{\|g\|_A\leq 1} \|f\times_{\hbar} g\|_A,\qquad \|g\|_A^2 = \left\|\int_{\hat{V}} g(\hat{v})^*g(\hat{v}) \rd \hat{v}\right\|,\qquad f \in C_b^{\infty}(\hat{V},A),g\in \Sc(\hat{V},A),
\]
see \cite[Chapter 4]{Rie93}. We denote
\[
C_{b,\hbar}^{\infty}(\hat{V},A) = \textrm{ norm-completion of }C_{b}^{\infty}(\hat{V},A).
\]
If then $A$ is \emph{any} C$^*$-algebra with a norm-continuous action $\alpha$ of $\hat{V}$, we denote
\[
\widetilde{\alpha}\colon  A^{\infty} \rightarrow C_b^{\infty}(\hat{V},A),\qquad a \mapsto (\hat{v}\mapsto \alpha_{\hat{v}}(a)),
\]
and we endow $A^\infty$ with the pre-C$^*$-algebra structure inherited from $C_b^{\infty}(\hat{V},A)$ through his map. We write $A_{\hbar}$ for the completion of $A^{\infty}$, and we denote the resulting emedding of $A^{\infty}$ into this completion as 
\[
\Qc_A\colon  A^{\infty} \rightarrow A_{\hbar},\qquad a \mapsto a_{\hbar}. 
\]
For example, if $A = C_0(\hat{V})$, it is not hard to show that $A^{\infty} = \Sc(\hat{V})$, and that the resulting quantization $C_{0,\hbar}(\hat{V})$ is precisely $C_{\hbar}^*(V)$ as constructed before. 

What will be important for us, are the structural properties of the assignment 
\[
A \mapsto A_{\hbar}. 
\]
For this, we first recall the following result, which is \cite[Proposition 5.9]{Rie93}. 

\begin{prop}\label{PropExtMultAlg}
Let $B$ be a C$^*$-algebra with a norm-continuous $\hat{V}$-action. Consider $C = M(B)$ with its natural (not necessarily norm-continuous) $\hat{V}$-action. Then there exists a unique injective $*$-homomorphism 
\[
B_{\hbar} \rightarrow C_{\hbar},\qquad b_{\hbar}\mapsto b_{\hbar},\qquad b \in B^{\infty},
\]
realizing $B_{\hbar}$ as an essential two-sided $*$-ideal of $C_{\hbar}$. 
\end{prop}

In particular, we may view $C_{\hbar}\subseteq M(B_{\hbar})$. 

Together with the previous proposition, the following theorem now follows from combining \cite[Theorem 5.7, Proposition 5.8 and Proposition 5.10]{Rie93}.

\begin{theorem}\label{TheoFunctRieffDef}
Let $A,B$ be C$^*$-algebras with norm-continuous $\hat{V}$-actions, and assume that $\theta\colon  A \rightarrow M(B)$ is a non-degenerate $\hat{V}$-equivariant $*$-homomorphism. Then there exists a unique non-degenerate $*$-homomorphism
\[
\theta_{\hbar}\colon A_{\hbar}\rightarrow M(B_{\hbar}),\qquad a_{\hbar} \mapsto \theta(a)_{\hbar},\qquad a\in A^{\infty}. 
\]
Moreover, if $\theta$ is faithful, also $\theta_{\hbar}$ is faithful, and if $\theta$ has range in $B$ (resp.\ is surjective with range $B$), then also $\theta_{\hbar}$ has range in $B_{\hbar}$ (resp.\ is surjective with range $B_{\hbar}$).

The above construction is functorial, in the sense that if $\theta\colon  A \rightarrow M(B)$ and $\kappa\colon  B \rightarrow M(C)$ are non-degenerate $*$-homomorphisms, then 
\[
(\kappa \circ\theta)_{\hbar} = \kappa_{\hbar}\circ \theta_{\hbar}.
\]
\end{theorem}

We will apply this as follows. 

First assume that we have a subspace $W \subseteq V$. Then the corresponding projection map $q\colon  \hat{V}\twoheadrightarrow \hat{W}$ leads to the non-degenerate $*$-homomorphism 
\[
\pi\colon  C_0(\hat{W})\rightarrow C_b(\hat{V}) = M(C_0(\hat{V})),\qquad f\mapsto f\circ q. 
\]
Moreover, this $*$-homomorphism is $\hat{V}$-equivariant when endowing $\hat{W}$ with the translation action of $\hat{V}$ through the quotient map. This leads to a non-degenerate $*$-homomorphism
\[
C_{0,\hbar}(W) \rightarrow M(C_{0,\hbar}(V)) = M(C_{\hbar}(V)). 
\]
A priori, the left hand side could be different from the canonical quantization $C_{\hbar}^*(W)$ of $C_0(\hat{W})$ viewing $W$ as a skew-symmetric space in its own right. However, one checks that the above map, considered on elements of the form $f_{\hbar}$ for $f\in \mathcal{S}(\hat{W})$, coincides with the following map as defined on these elements: 
\[
C_{\hbar}^*(W) \rightarrow M(C^*_{\hbar}(V)),\qquad f_{\hbar} \mapsto \int_{W} \hat{f}(w) e_{\hbar}^{2\pi i w}\rd w, \qquad f\in \mathcal{S}(W),
\]
using notation as in \eqref{EqCorrCanEl}. Hence we can unambiguosly interpret $C_{0,\hbar}(W) = C^*_{\hbar}(W)$. 

In particular, if $v \in V$ and we take $W =\R v$, then we can canonically identify 
\begin{equation}\label{EqIdentHatWR}
\hat{W} \cong \R,\qquad \hat{w}\mapsto \hat{w}\cdot v 
\end{equation}
with associated quotient map 
\begin{equation}\label{EqIdentHatWRQuot}
\hat{V} \twoheadrightarrow \R,\quad \hat{w}\mapsto \hat{w}\cdot v.
\end{equation}
We then find a canonical non-degenerate inclusion
\begin{equation}\label{EqVectIncl}
\pi_{v,\hbar}\colon  C_0(\R) = C_0(\hat{W}) = C_{0,\hbar}(\R) = C_{\hbar}^*(\R v) \hookrightarrow M(C^*_{\hbar}(V)).
\end{equation}
So, if $f\in C(\R)$, then we can consider $f \,\eta \,C_0(\R)$, and through Lemma \ref{LemHomCorr} we obtain a corresponding element 
\[
f_{\hbar} = \pi_v(f) \in C^*_{\hbar}(V).
\]
If in particular the function 
\[
z = z_f\colon  \R \hat{v}\rightarrow \C,\quad \hat{w} \mapsto z(\hat{w}) = \frac{f(\hat{w})}{(1+|f(\hat{w})|^2)^{1/2}}
\]
lies in $C_b^{\infty}(\R \hat{v})$, then the functoriality of our constructions gives concretely that
\[
f_{\hbar} = T_{z_{\hbar}}.
\]
For example, applying this to the identity function $\id_{\R}\colon  \R\rightarrow \R$, we find (after applying the functional calculus as in Lemma \ref{LemFunctCal} to the exponential map) the element $e_{\hbar}^v$ affiliated to $C_{\hbar}^*(V)$ as in \eqref{EqChare}.

Resume now the setting as at the end of Section \ref{SecCStarRieff}, where we have chosen dual bases $\{f_i\},\{\hat{f}_i\}$ of $V$ and $\hat{V}$, together with a partitioning $I = I_0 \sqcup I_1$. Then the translation action of $\hat{V}$ on itself extends to a continuous action of $\hat{V}$ on $\hat{V}_+$ (where $-\infty$ acts as a fixed point). Hence from the theory developed above, we see that \eqref{EqListInclus} quantizes into a sequence of inclusions
\begin{equation}\label{EqListInclusQuant}
C_{0,\hbar}(\hat{V}) \subseteq C_{0,\hbar}(\hat{V}_+) \subseteq C_{u,\hbar}(\hat{V}_+) \subseteq C_{u,\hbar}(\hat{V}) \subseteq M(C_{0,\hbar}(\hat{V}))
\end{equation}
(where $C_u(\hat{V})$, resp.\ $C_u(\hat{V}_*)$ denotes the norm-continuous elements for the translation action in $C_b(\hat{V})$, resp.\ $C_b(\hat{V}_+)$). By construction, this already gives that $C^*_{\hbar}(V_+)$, as defined by \eqref{EqCStarCone}, is indeed a C$^*$-algebra, equal to $C_{0,\hbar}(\hat{V}_+)$ in the notation above. 

Let us now conclude the remainder of the proof of Theorem \ref{TheoAffil}. Assume that $v\in V_+$. Then \eqref{EqIdentHatWRQuot} extends to a quotient map 
\[
\hat{V}_+ \rightarrow [-\infty,\infty),\qquad \hat{w}\mapsto v\cdot \hat{w}.
\]
The resulting non-degenerate $*$-homomorphism 
\[
C_0([-\infty,\infty)) = C_{0,\hbar}([-\infty,\infty)) \rightarrow M(C^*_{\hbar}(V_+))
\]
is then simply the restriction of the corresponding map 
\[
\pi_{v,\hbar}\colon  C_b(\R) = M(C_0(\R)) \rightarrow M(C^*_{\hbar}(V)). 
\]
Hence it is sufficient to observe that the exponential function $x \mapsto e^x$ extends to a continuous function on $[-\infty,\infty)$, and so determines an element affiliated to $C_0([-\infty,\infty))$.   

Theorem \ref{TheoDegRep} follows similarly through functoriality. Indeed, assume first that $I_0 = I_0'$. Then $\iota_{I'}$ as in \eqref{EqDegMap} is a $\hat{V}$-equivariant closed embedding when endowing $\hat{V}'$ with the translation action of $\hat{V}$ under the quotient map $\hat{V}\twoheadrightarrow \hat{V}'$. Hence \eqref{EqRepDegBound} is provided through Theorem \ref{TheoFunctRieffDef}. The formula \eqref{EqImageDegMap} then follows from the fact that, again through functoriality, $\pi_{I'}(e_{\hbar}^{f_i})$ must coincide with $e_{\hbar}^{f_i \circ \iota_{I'}}$. 

For the general theorem, we may now already assume that $I_1' = I_1$, and that $I_0'$ is such that $(f_i,f_j) =0$ for each $i \in I_0'\setminus I_0$. But in this case the $f_i$ with $i\in I_0\setminus I_0'$ lie in the kernel of $\epsilon$, and the inclusion map $V'\subseteq V$ \emph{does} allow a splitting as a map between skew-symmetric spaces, through the section
\[
s\colon  V \rightarrow V',\qquad \sum_{i\in I} v_i f_i \mapsto \sum_{i\in I'} v_{i'}f_i'. 
\]
It follows that any unitary $\hbar$-representation $\pi$ of $V'$ leads to a unitary $\hbar$-representation $\pi$ of $V$ through
\[
\pi(v) :=  \pi(s(v)),
\]
and we hence obtain a surjective $*$-homomorphism 
\[
\pi_{I'}\colon  C_{\hbar}^*(V) \rightarrow C_{\hbar}^*(V'), \quad f_{\hbar} \mapsto (f\circ \hat{s})_{\hbar},\qquad f \in \mathcal{S}(\hat{V}). 
\]
Extending this map to $M(C_{\hbar}^*(V))$ and restricting to $C_{\hbar}^*(V_+)$ leads to the map $\pi_{I'}$ as in Theorem \ref{TheoDegRep}. As we can then also apply the extended map $\pi_{I'}$ to the element $e_{\hbar}^v$ as affiliated to $C^*(V)$, this then also immediately proves \eqref{EqImageDegMap} in this case.

\section{von Neumann algebras and their weight theory}\label{AppWeights}

As a main reference for von Neumann algebras and their weight theory, we refer to \cite{Tak02,Tak03}.

\subsection{von Neumann algebras}

A von Neumann algebra $M$ is any unital C$^*$-algebra allowing a predual as a Banach space. The resulting weak$^*$-topology on $M$ is independent of the choice of predual, and called the \emph{$\sigma$-weak} or \emph{ultraweak} topology. Canonically, one can then choose the predual of $M$ to be the space $M_* \subseteq M^*$ of $\sigma$-weakly continuous functionals.

For example, if $\Hsp$ is a Hilbert space, then $\Bc(\Hsp)$ is a von Neumann algebra with $\Bc(\Hsp)_*$ the functionals of the form 
\[
\omega(x) = \sum_i \langle \xi_i,x\eta_i\rangle,\qquad \sum_i \|\xi_i\|^2 <\infty,\sum_i \|\eta_i\|^2 <\infty. 
\]

The following are equivalent for a unital $*$-subalgebra $M \subseteq \Bc(\Hsp)$: 
\begin{itemize}
\item $M \subseteq \Bc(\Hsp)$ is a von Neumann algebra whose $\sigma$-weak topology is induced by the one on $\Bc(\Hsp)$.
\item $M$ is $\sigma$-weakly closed in $\Bc(\Hsp)$. 
\item $M= M''$, where for a subset $S \subseteq \Bc(\Hsp)$ we put 
\[
S' = \{x\in \Bc(\Hsp)\mid \forall y\in S: xy = yx\}.
\]
\end{itemize}
Usually, one only calls $M \subseteq \Bc(\Hsp)$ a von Neumann algebra if the above conditions are satisfied, i.e.\ $M \subseteq \Bc(\Hsp)$ is what is called \emph{normally embedded}.  

More generally, a unital $*$-homomorphism $\pi\colon  M \rightarrow N$ between von Neumann algebras is called \emph{normal} if it is $\sigma$-weak-$\sigma$-weak continuous, which is the same as asking that it is dual to a norm-continuous linear map $\pi_*\colon  N_*\rightarrow M_*$. For example, any $*$-isomorphism between von Neumann algebras is automatically normal.

A \emph{normal unital $*$-representation} of a von Neumann algebra $M$ on a Hilbert space $\Hsp$ is any normal unital $*$-homomorphism $\pi\colon  M \rightarrow \Bc(\Hsp)$. One can show that any von Neumann algebra $M$ has a faithful normal unital $*$-representation $\pi$ on a Hilbert space $\Hsp$, and then necessarily $\pi(M) \subseteq \Bc(\Hsp)$ is a (normally embedded) von Neumann algebra with $M \cong \pi(M)$. 

If $\pi\colon  M \rightarrow \Bc(\Hsp)$ is a normal unital $*$-representation and $\Gc$ is another Hilbert space, then we obtain another normal $*$-representation 
\[
\pi \otimes 1\colon  M \rightarrow \Bc(\Hsp \otimes \Gc),\quad x\mapsto \pi(x)\otimes 1,
\]
called \emph{amplification} of $\pi$. 

If on the other hand $\pi\colon  M \rightarrow \Bc(\Hsp)$ is a normal unital $*$-representation and $p \in \pi(M)'$ is a self-adjoint projection, then we obtain another normal $*$-representation 
\[
\pi_p\colon  M \rightarrow \Bc(p\Hsp),\quad x\mapsto \pi(x)p,
\]
called \emph{cut-down} of $\pi$ (by $p$). One can then show:
\begin{lemma}\label{LemCutAmp}
Given a faithful normal unital $*$-representation $\pi$ of a von Neumann algebra $M$, every other normal unital $*$-representation of $M$ is, up to unitary equivalence, a cut-down of an amplification of $\pi$. 
\end{lemma}

\subsection{General theory of weights}

Fix $M \subseteq \Bc(\Hc)$ a von Neumann algebra on a Hilbert space $\Hsp$ and let $M_+$ be the cone of positive elements in $M$.
\begin{defn}
A \emph{weight} on $M$ is any map 
\[
\varphi\colon  M_+ \rightarrow [0,\infty]
\]
such that
\[
\varphi(x+y) = \varphi(x)+\varphi(y),\qquad \varphi(0) = 0,\qquad  \varphi(ax) = a\varphi(x),\qquad a\in \R_{>0}, x,y\in M_+.
\]
One calls the weight 
\begin{itemize}
\item \emph{normal} if for any strongly converging increasing net $x_{i}\nearrow x$ in $M_+$, it holds that 
\[
\varphi(x) = \lim_{i} \varphi(x_i),
\]
\item \emph{faithful} if $\varphi(x)=0$ for $x\in M_+$ implies $x=0$, 
\item \emph{semi-finite} if the linear span $\mathfrak{m}_{\varphi}$ of 
\[
\mathfrak{m}_{\varphi}^+ = \{x\in M_+\mid \varphi(x)<\infty\}
\]
is $\sigma$-weakly dense in $M$, and
\item \emph{tracial} if 
\[
\varphi(xx^*) = \varphi(x^*x),\qquad \forall x\in M.
\]
\end{itemize}
We denote by $\mathfrak{W}_M$ the set of all weights, and by $\mathfrak{W}_M^{\nsf}$ the set of all normal, semi-finite, faithful (=nsf) weights (we also use the obvious variations).
\end{defn}

Clearly any weight which is a pointwise supremum of normal states is normal, and conversely for any normal weight $\varphi$ one has
\[
\varphi(x) = \sup_{\omega\in \mathcal{G}_{\varphi}}\omega(x),\qquad \forall x\in M_+,
\]
where
\[
\mathcal{G}_{\varphi} = \{\omega \in M_*^+\mid \omega \leq \varphi\},
\]
and where for $\omega \in M_*^+$ we write $\omega \leq \varphi$ if this holds pointwise on $M_+$. 

\begin{example}
Assume $\Hc$ is a Hilbert space. Then we obtain a normal, semi-finite, faithful, tracial weight on $\Bc(\Hc)$ via the usual trace 
\begin{equation}\label{EqTrace}
\Tr(x) = \sum_i \langle e_i,xe_i\rangle,\qquad x\geq0,
\end{equation}
where $\{e_i\}$ is an orthonormal basis of $\Hc$. Any other orthonormal basis gives the same trace. 
\end{example}

\begin{example}\label{ExaTwistTraceBH}
More generally, if $\Tb$ is a (strictly) positive operator on $\Hc$, one defines a normal semi-finite (faithful) weight
\begin{equation}\label{EqDeformTraceOrd}
\Tr_{\Tb}(x) = \sup_{r\geq0} \Tr(\Tb_r^{1/2}x\Tb_r^{1/2}) =\sup_{r\geq 0} \Tr(x^{1/2}\Tb_rx^{1/2}),\qquad x\in \Bc(\Hsp)_+,
\end{equation}
where $\Tb_r = \Tb \chi_{[0,r)}(\Tb)$ for $\chi_Y$ the indicator function on a set $Y$. In fact, any semi-finite normal weight on $\Bc(\Hsp)$ is of this form, for a unique $\Tb$. 
\end{example}

\begin{example} 
The construction in Example \ref{ExaTwistTraceBH} can be modified to an arbitrary von Neumann algebra $M \subseteq \Bc(\Hc)$ with an nsf trace $\tau$. Namely, assume that $\Tb$ is a positive operator on $\Hc$ which is \emph{affiliated} to $M$ in the sense of von Neumann algebras:
\begin{equation}\label{EqAffiliation}
u\Tb u^* = \Tb,\qquad \forall u \textrm{ unitary in }M'. 
\end{equation}
Then $\Tb_r \in M$ for all $0\leq r$, and one can deform $\tau$ into a semifinite normal weight $\tau_{\Tb}$, exactly as before:
\begin{equation}\label{EqDeformTrace}
\tau_{\Tb}(x) = \sup_{r\geq0} \tau(\Tb_r^{1/2}x\Tb_r^{1/2}) =\sup_{r\geq 0} \tau(x^{1/2}\Tb_rx^{1/2}),\qquad x\in M_+.
\end{equation}
\end{example}

We now recall the GNS-construction: with $\varphi$ an nsf weight, put 
\[
\mfn_{\varphi} = \{x\in M \mid x^*x\in \mfm_{\varphi}^+\}. 
\]
Then $\mfn_{\varphi}$ is a left-sided ideal in $M$, and in case $\varphi$ is tracial also a two-sided ideal with $\mfn_{\varphi}^* = \mfn_{\varphi}$. In general, one has 
\[
\mfm_{\varphi}= \mfn_{\varphi}^*\mfn_{\varphi}. 
\]
By polarisation, it follows that $\varphi$ can be extended to a linear functional on $\mfm_{\varphi}$, and $\mfn_{\varphi}$ becomes a Hilbert space with respect to the inner product 
\[
\langle x,y \rangle_{\varphi} :=\varphi(x^*y),\qquad x,y\in \mfn_{\varphi}.
\]
One writes $L^2(M,\varphi)$ for the completion, or simply $L^2(M)$ if $\varphi$ is understood. The natural embedding of $\mfn_{\varphi}$ into $L^2(M)$ is written 
\[
\Lambda = \Lambda_{\varphi}\colon  \mfn_{\varphi} \rightarrow L^2(M). 
\]
One further shows that there exists a unique normal faithful unital $*$-representation of $M$ on $L^2(M)$ such that 
\[
x\Lambda(y) = \Lambda(xy),\qquad x\in M,y\in \mfn_{\varphi}. 
\]
The celebrated \emph{Tomita--Takesaki} theorem states that there exist a unique anti-linear involutive isometry $J_{\varphi}$ (the modular conjugation) and strictly positive operator $\nabla_{\varphi}$ (the modular operator) on $L^2(M)$ such that 
\[
\mfn_{\varphi}\cap \mfn_{\varphi}^* \subseteq \msD(\nabla_{\varphi}^{1/2}),\qquad J_{\varphi}\nabla_{\varphi}^{1/2}\Lambda_{\varphi}(x) = \Lambda_{\varphi}(x^*),\qquad x\in \mfn_{\varphi}\cap \mfn_{\varphi}^*.
\]
Moreover, one has 
\[
\nabla_{\varphi}^{it}M\nabla_{\varphi}^{-it} = M,\qquad J_{\varphi} M J_{\varphi} = M'. 
\]
The resulting one-parameter group of automorphisms of $M$ is called the \emph{modular} \emph{one-parameter-group}: 
\[
\sigma_t^{\varphi}\colon  M \rightarrow M,\quad x \mapsto \nabla_{\varphi}^{it}x\nabla_{\varphi}^{-it},\qquad t\in \R. 
\]
It compensates for the non-traciality of arbitrary nsf weights: whereas an nsf weight $\varphi$ is tracial if and only if $\sigma_t^{\varphi} = \id$ for all $t\in \R$, a general nsf weight satisfies
\[
\varphi \circ \sigma_t^{\varphi}= \varphi,\qquad \varphi(xx^*) = \varphi(\sigma_{-i/2}^{\varphi}(x)^*\sigma_{-i/2}^{\varphi}(x)),\qquad \forall t\in \R,x\in \mfn_{\varphi}\cap \mfn_{\varphi}^*,
\]
where the notation indicates that $\mfn_{\varphi}^*\cap\mfn_{\varphi} \subseteq \msD(\sigma_{-i/2}^{\varphi})$. Here $x\in \msD(\sigma_{z}^{\varphi})$ for a $z\in \C$ if there exists a (necessarily unique) $\sigma$-weakly continuous function $w \rightarrow \sigma_w^{\varphi}(x)$ on the closed strip 
\[
\{w \mid \Imm(w) \textrm{ between }0\textrm{ and }\Imm(z)\}
\]
which is holomorphic on its interior (as a Banach-space valued function) and extends $\sigma_t^{\varphi}(x)$ for $t\in \R$. 

\begin{example}
Resume the setting of Example \ref{ExaTwistTraceBH}. Then if $\Tb$ is strictly positive and $\varphi = \Tr_{\Tb}$, one can identify 
\[
L^2(\Bc(\Hsp),\varphi)\cong \Hsp \otimes \overline{\Hsp}
\]
in a unique way such that 
\[
\Lambda(\theta_{\xi,\eta}) = \xi\otimes \overline{\Tb^{1/2}\eta},\qquad \xi\in \Hsp, \eta\in \msD(\Tb^{1/2}). 
\]
The modular conjugation and modular operator are given in this case by 
\[
J_{\varphi}(\xi\otimes \overline{\eta}) = \eta\otimes \overline{\xi},\qquad \nabla_{\varphi}^{it}(\xi\otimes \overline{\eta}) = \Tb^{it}\xi\otimes \overline{\Tb^{it}\eta},\qquad \xi,\eta\in \Hsp,
\]
so in particular the modular one-parameter group is 
\[
\sigma_t^{\varphi}(x) = \Tb^{it}x\Tb^{-it},\qquad x\in \Bc(\Hsp),t\in \R. 
\]
\end{example}

We also briefly comment on how affiliated positive operators (as in \eqref{EqAffiliation}) can be seen as being  \emph{dual} to the set of weights. Namely, define the \emph{extended positive cone} $M_+^{\mathrm{ext}}$ of a von Neumann algebra $M$ as the set of all
\[
x\colon  M_*^+ \rightarrow [0,\infty],\qquad \omega \mapsto \omega(x)
\]
such that there exists an increasing net $x_i \in M_+$ with 
\[
\omega(x) = \sup_{i}\omega(x_i),\qquad \omega \in M_*^+.
\]
Necessarily one then has 
\[
(\omega+\omega')(x) = \omega(x)+\omega'(x),\qquad 0(x) = 0,\qquad  (a\omega)(x) = a\omega(x),\qquad a\in \R_{>0}, \omega\in M_+^*,
\]
and $x$ satisfies the normality condition 
\[
\omega(x) = \lim_{\alpha} \omega_{\alpha}(x)
\]
whenever $\omega_{\alpha}$ is an increasing net in $M_*^+$ converging in norm to $\omega$. One then calls $x\in M^{\ext}_+$ \emph{semi-finite} if 
\[
M_*^+ = \textrm{closure of }\{\omega \in M_*^+ \mid \omega(x)<\infty\},
\]
and \emph{faithful} if $\omega(x) = 0$ implies $\omega = 0$ for $\omega \in M_*^+$. 

Then if $M \subseteq \Bc(\Hc)$, there is a one-to-one correspondence between the (faithful) semi-finite $x \in M^{\ext}_+$ and the (strictly) positive operators $\mathbf{x}$ on $\Hsp$ which are affiliated to $M$, determined through 
\begin{equation}\label{EqEqualFormPos}
\omega_{\xi,\xi}(x) = \begin{cases} \|\mathbf{x}^{1/2}\xi\|^2&\text{ if } \xi \in \msD(\mathbf{x}^{1/2}),\\
\infty &\text{ if } \xi \notin\msD(\mathbf{x}^{1/2}).
\end{cases}
\end{equation}
The abstract notion of the extended cone makes it easy to directly evaluate weights on its elements: for a general von Neumann algebra, we have an $\R_{>0}$-bilinear pairing 
\[
\mathfrak{W}_M^{\nsf} \times M_+^{\ext} \rightarrow [0,\infty],\qquad \varphi(x) = \sup_{\omega,y} \omega(y),
\]
where $\omega$ ranges over $\mathcal{G}_{\varphi}$ and $y$ ranges over 
\[
\mathcal{G}_x = \{y \in M_+ \mid y\leq x\},
\]
where for $y \in M_+$ we write $y \leq x$ if this holds pointwise on $M_*^+$.

\subsection{Form sums}\label{SecPropSkewComm}

If $\Hsp$ is a Hilbert space and $\Tb,\Sb$ are positive operators on $\Hsp$, then 
\[
\Tr_{\Tb}+ \Tr_{\Sb}
\]
is again a normal weight on $\Bc(\Hsp)$, which is faithful if $\Tb,\Sb$ are strictly positive. By the result in Example \ref{ExaTwistTraceBH}, we can make the following definition:  
\begin{defn}\label{DefFormSum}
If $\Tb,\Sb$ are strictly positive operators on a Hilbert space $\Hsp$ such that $\Tr_{\Tb}+ \Tr_{\Sb}$ is semi-finite, their \emph{form sum}\footnote{The notation $\Tb\dotplus \Sb$ is more common, but we prefer to use the latter notation for the closure of the operator $\Tb+\Sb$ with domain $\msD(\Tb)\cap \msD(\Sb)$.} $\Tb\boxplus \Sb$ is the unique strictly positive operator such that 
\[
\Tr_{\Tb\boxplus \Sb} = \Tr_{\Tb}+ \Tr_{\Sb}.
\]
\end{defn}
We have that $\Tr_{\Tb}+ \Tr_{\Sb}$ is semi-finite if and only if $\msD(\Tb)\cap \msD(\Sb)$ is dense in $\Hsp$.

Consider $\Xb,\Pb$ as in \eqref{EqInfGenPosMom}, as well as the rescaled operator
\begin{equation}\label{EqRescMomOp}
\Pb_{\hbar} = 4\pi^2\hbar \Pb.
\end{equation}
By the Stone-von Neumann theorem, Proposition \ref{PropSkewComm} will follow once we prove the following: 
\begin{prop}
The following identity of positive operators on $L^2(\R)$ holds:
\begin{equation}\label{EqToProve}
e^{-\Pb_{\hbar}} \boxplus e^{-\Pb_{\hbar}\dotplus \Xb} = \varphi_{\hbar}(\Xb) e^{-\Pb_{\hbar}} \varphi_{\hbar}(\Xb)^*.
\end{equation}
\end{prop}

We will need some preparations.

For $r>0$ we consider the strips 
\begin{equation}\label{NotStrips}
S_{r} = \{z\in \C\mid 0< \Imm(z) <r\},\qquad S_{-r} = \{z\in \C\mid -r <\Imm(z) < 0\},  
\end{equation}
\begin{equation}
T_{r} = \{z\in \C\mid 0\leq \Imm(z) \leq r\},\qquad T_{-r} = \{z\in \C\mid -r \leq \Imm(z) \leq 0\}. 
\end{equation}

For a general Hilbert space $\Hsp$ with self-adjoint operator 
\[
\Tb\colon  \msD(\Tb) \rightarrow \Hsp,\qquad \msD(\Tb) \subseteq \Hsp, 
\]
and any $r\in \R^{\times}$, we define 
\[
\Hol(T_r,\Hsp) =  \{F\colon  T_r \rightarrow \Hsp \mid F\textrm{ continuous on $ T_r $ and holomorphic on }S_r\}
\] 
and
\[
H_r(\Tb) = \{F \in \Hol(T_r,\Hsp)\mid \forall z\in T_r, t\in \R: F(z+t) = e^{it\Tb}F(z)\}.
\]
The following is \cite[Lemma VI.2.3]{Tak03}.
\begin{lemma}
The map
\begin{equation}\label{EqMapEvZero}
H_r(\Tb) \rightarrow \msD(e^{-r\Tb}),\qquad F \mapsto F(0)
\end{equation}
is a well-defined linear isomorphism, with
\[
e^{-r\Tb}(F(0))= F(ir),\qquad F \in H_r(\Tb).
\]
\end{lemma}
We will write the inverse of \eqref{EqMapEvZero} as
\begin{equation}\label{EqIdDom}
\msD(e^{-r\Tb}) \rightarrow  H_r(\Tb),\qquad \xi \mapsto F_{\xi},
\end{equation}
so that
\[
e^{-r\Tb}\xi = F_{\xi}(ir),\qquad \xi\in \msD(e^{-r\Tb}).
\]
Moreover, then clearly $\msD(e^{-r\Tb}) \subseteq\msD(e^{z\Tb})$ for all $z\in T_r$, with 
\begin{equation}\label{EqRestF}
e^{iz\Tb}\xi = F_{\xi}(z),\qquad \xi \in \msD(e^{-r\Tb}).
\end{equation}

We now return to the specific case of the $L^2(\R)$, and gather some extra information on the domain of $e^{-\Pb_{\hbar}/2}$ with $\Pb_{\hbar}$ as in \eqref{EqRescMomOp}. 

For $r\in \R^{\times}$, write 
\[
\Hol(S_{r}) = \{F\colon  S_{r} \rightarrow \C\mid F \textrm{ holomorphic}\},
\]
and 
\[
H(S_{r})= \{f \in \Hol(S_{r}) \mid \sup_{0<b<1} \int_{\R} |f(x+ir b)|^2\rd x <\infty\} \subseteq \Hol(S_{r}) .
\]
\begin{lemma}\label{LemSch}
If $0<b$ and $f \in H(S_{\pi b\hbar})$, there exists $\xi \in \msD(e^{-b\Pb_{\hbar}/2})$ such that 
\begin{equation}\label{EqLimLow}\xi(x) \overset{\textrm{a.e.}}{=} \lim_{n\rightarrow \infty} f(x+i/n^2),
\end{equation}
and we obtain in this way a linear isomorphism 
\begin{equation}\label{EqIdSchH}
H(S_{\pi b\hbar}) \cong \msD(e^{-b\Pb_{\hbar}/2}),\quad f \mapsto f_{\xi}.
\end{equation}
Moreover, in this case we have 
\begin{equation}\label{EqLimSup}
(e^{-b\Pb_{\hbar}/2}\xi)(x) \overset{\textrm{a.e.}}{=} \lim_{n\rightarrow \infty} f(x+ \pi i b\hbar - i/n^2) 
\end{equation}
and, for all $z\in H(S_{\pi b\hbar})$, 
\begin{equation}\label{EqRelfF}
(e^{-z\Pb_{\hbar}/2}\xi)(x) \overset{\textrm{a.e.}}{=} f(z+x).
\end{equation}
\end{lemma} 
\begin{proof}
This is just a reformulation of \cite[Lemma 1.1]{Sch94} (which is an application of the Paley--Wiener theorem). More precisely, \cite[Lemma 1.1]{Sch94} states \eqref{EqLimLow}, the surjectivity of \eqref{EqIdSchH}, and \eqref{EqLimSup}. The identity \eqref{EqRelfF} and the injectivity of \eqref{EqIdSchH}  then follow by applying this first part to $f$ restricted to $S_{\pi b'\hbar}$ with $0<b'<b$. 
\end{proof}

If $f \in H(S_{\pi\hbar})$ and $\xi\in \msD(e^{-\Pb_{\hbar}/2})$ are related through \eqref{EqLimLow}, we write 
\[
f = f_{\xi}.
\]
Similarly, if $f \in H(S_{\pi\hbar})$, we write $F\in H_{\pi\hbar}(2\pi \Pb)$ for the unique element related to $f$ via \eqref{EqRelfF}, so
\begin{equation}\label{EqRelfFNew}
F(z)(x) \overset{\textrm{a.e.}}{=} f(z+x).
\end{equation}
Then indeed $f_{\xi}$ and $F_{\xi}$ are related under this correspondence for $\xi\in \msD(e^{-\Pb_{\hbar}/2})$. We then also extend the definition domain of $f\in H(S_{\pi\hbar})$ to $T_{\pi\hbar}$ as
\[
f(x+ i\pi\hbar) = F(i\pi\hbar)(x),\qquad f(x) = F(0)(x),
\]
with the understanding that this is only well-defined almost everywhere.

\begin{lemma}\label{LemPL}
Let $\chi$ be a holomorphic function on a domain $\Omega \supseteq T_{\hbar\pi}$, and assume there exist $C,D>0$ with \begin{equation}\label{EqCondPL}
|\chi(x+iy)| \leq Ce^{D|x|},\qquad \forall x+iy \in T_{\pi \hbar}.
\end{equation}
Assume that $f \in H(S_{\pi\hbar})$, and define $g\in \Hol(S_{\pi\hbar})$ by 
\[
g(z) = \chi(z)f(z),\qquad z\in S_{\pi\hbar}.
\]
Then 
\begin{equation}\label{EqPL}
\sup_{0<b<1} \int_{\R} |g(x+ib)|^2\rd x = \max\left\{\int_{\R}|g(x)|^2\rd x, \int_{\R}|g(x+\pi i \hbar)|^2\rd x\right\}.
\end{equation}
\end{lemma}
\begin{proof}
For any $\eta\in C_c(\R)$ with $\|\eta\|_2\leq 1$ we obtain a holomorphic function 
\[
z \mapsto \int_{\R} g(z+t)\overline{\eta(t)}\rd t =  \langle \overline{\chi}(z+ \Xb)\eta,F(z)\rangle.
\]
Using \eqref{EqCondPL}, together with the uniform boundedness of $\|F(z)\|_2$, this function satisfies the conditions for the Phragmen-Lindel\"{o}f theorem, by which it then follows that 
\begin{multline*}
\sup_{0<b<1} \left|\int_{\R} g(t+\pi \hbar b) \overline{\eta(t)}\rd t \right|^2 \leq \sup_{z \in \R \cup (\R+\pi i \hbar)} |\langle \overline{\chi}(z+\Xb)\eta,F(z)\rangle|^2 \\ \leq \max\left\{\int_{\R}|g(x)|^2\rd x, \int_{\R}|g(x+\pi i \hbar)|^2\rd x\right\}.
\end{multline*}
Taking the supremum over all such $\eta$, we find 
\[
\sup_{0<b<1} \int_{\R} |g(t+\pi \hbar b)|^2\rd t  \leq \max\left\{\int_{\R}|g(x)|^2\rd x, \int_{\R}|g(x+\pi i \hbar)|^2\rd x\right\}.
\]

The converse direction follows from \eqref{EqLimLow},\eqref{EqLimSup} and Fatou's lemma.
\end{proof}

\begin{lemma}
We have 
\begin{equation}\label{EqCharDomAlt}
\msD(e^{-\Pb_{\hbar}/2\dotplus \Xb/2}) = e^{i\Xb^2/4\pi \hbar}\msD(e^{-\Pb_{\hbar}/2}),
\end{equation}
with 
\begin{equation}\label{EqIdent}
e^{-\Pb_{\hbar}/2\dotplus \Xb/2}= e^{i\Xb^2/4\pi \hbar} e^{-\Pb_{\hbar}/2} e^{-i\Xb^2/4\pi \hbar},
\end{equation}
\end{lemma}
\begin{proof}
The identity \eqref{EqIdent} follows from \eqref{EqTransformSum}, and  \eqref{EqCharDomAlt} is then an immediate consequence.
\end{proof}

\begin{lemma}\label{LemDomainInt}
We have $\xi\in \msD(e^{-\Pb_{\hbar}/2}) \cap \msD(e^{-\Pb_{\hbar}/2\dotplus \Xb/2})$ if and only if there exists $f \in H(S_{\pi\hbar})$ with $f_{\mid \R} = \xi$ and 
\begin{equation}\label{EqFundEst}
\int_{\R} (1+e^{x}) |f(x+\pi i \hbar)|^2\rd x<\infty.
\end{equation}
\end{lemma} 
\begin{proof}
Assume first that $\xi \in \msD(e^{-\Pb_{\hbar}/2}) \cap \msD(e^{-\Pb_{\hbar}/2\dotplus \Xb/2})$. By Lemma \ref{LemSch} and equation \eqref{EqCharDomAlt}, we can then find $f,g\in H(S_{\pi\hbar})$ with associated $F,G \in H_{\pi\hbar}(2\pi \Pb)$ such that 
\[
F(0) = \xi,\qquad G(0) = e^{-i\Xb^2/4\pi\hbar} \xi.
\]
We claim that then 
\begin{equation}\label{EqEqualfg} 
g(z) = e^{-iz^2/4\pi \hbar}f(z),\qquad \forall z\in S_{\pi\hbar}. 
\end{equation}
Indeed, pick $\eta \in C_c(\R)$. Then we obtain the holomorphic function 
\[
S_{\pi\hbar}\ni z \mapsto \int_{\R} g(z+t)\overline{\eta}(t) \rd t  = \langle \eta,G(z)\rangle.
\]
On the other hand, also 
\[
S_{\pi\hbar}\ni z \mapsto \int_{\R} f(z+t)e^{-i(z+t)^2/4\pi \hbar}\overline{\eta}(t) \rd t  = \langle e^{i(\overline{z}+\Xb)^2/4\pi\hbar}\eta,F(z)\rangle
\]
is holomorphic. Since these functions extend continuously to the same function on $\R$, they must be equal. This immediately implies \eqref{EqEqualfg}.

Since now $g \in H(S_{\pi\hbar})$ and $\chi(z) = e^{-iz^2/4\pi\hbar}$ satisfies the conditions in Lemma \ref{LemPL}, we can conclude from equation \eqref{EqPL} that \eqref{EqFundEst} holds. 

Conversely, if $\xi \in \msD(e^{-\Pb_{\hbar}/2})$ with associated function $f = f_{\xi}$ that satisfies \eqref{EqFundEst}, define $g\in \Hol(S_{\pi\hbar})$ via \eqref{EqEqualfg}. Applying again Lemma \ref{LemPL},  we find $g\in H(S_{\pi \hbar})$. But through the identity \eqref{EqLimLow}, this implies precisely 
\[
\xi \in e^{i\Xb^2/4\pi\hbar}\msD(e^{-\Pb_{\hbar}/2}) = \msD(e^{-\Pb_{\hbar}/2\dotplus \Xb/2})
\]
as required.
\end{proof}

For the next lemma, we recall that $\varphi_{\hbar}$ is unimodular on the real line, so that 
\begin{equation}\label{EqProdAdjo}
\overline{\varphi_{\hbar}(\overline{z})}\varphi_{\hbar}(z) = 1,\qquad z\in \C. 
\end{equation}
On the other hand, by the identity \eqref{EqFunctEq2} we have
\begin{equation}\label{EqFunctEq2n}
\varphi_{\hbar}(z-\pi i\hbar) = (1+e^z)\varphi_{\hbar}(z+\pi i \hbar),\qquad z\in \C,
\end{equation}
so combined with \eqref{EqProdAdjo} we find
\begin{equation}\label{EqValueShift}
|\varphi_{\hbar}(x+\pi i \hbar)| = (1+e^x)^{-1/2},\qquad x\in \R.
\end{equation}

\begin{lemma}\label{LemEqIdDomains}
We have 
\begin{equation}\label{EqIdDomains}
\msD(e^{-\Pb_{\hbar}/2}) \cap \msD(e^{-\Pb_{\hbar}/2\dotplus \Xb/2}) =   \varphi_{\hbar}(\Xb)\msD(e^{-\Pb_{\hbar}/2}).
\end{equation}
\end{lemma}
\begin{proof}
Assume first that $\xi \in \msD(e^{-\Pb_{\hbar}/2}) \cap \msD(e^{-\Pb_{\hbar}/2\dotplus \Xb/2})$, and put $f = f_{\xi}$. Then by holomorphicity of $\varphi_{\hbar}^{-1}$ on the strip $S_{\pi \hbar}$ we get that 
\[
z\mapsto g(z):= f(z)/\varphi_{\hbar}(z)
\]
defines a function in $\Hol(S_{\pi\hbar})$. Now by the growth estimates in \eqref{EqEstima1} and \eqref{EqEstima2}, the function
$\chi(z) = 1/\varphi(z)$ satisfies the conditions in Lemma \ref{LemPL}. Then since by Lemma \ref{LemDomainInt} the inequality \eqref{EqFundEst} holds, we conclude from the identity \eqref{EqValueShift} and Lemma \ref{LemPL} that $g \in H(S_{\pi\hbar})$. But then $\xi \in \varphi_{\hbar}(\Xb)\msD(e^{-\Pb_{\hbar}/2})$, since
\[
\xi(x)/\varphi_{\hbar}(x) \overset{\textrm{a.e.}}{=} \lim_{n\rightarrow 0} g(x+ \frac{i}{n^2}).
\]

Assume now conversely that $\xi = \varphi_{\hbar}(\Xb)\eta$ for $\eta\in \msD(e^{-\Pb_{\hbar}/2})$. Let $g = f_{\eta}$ be its associated function in $H(S_{\pi\hbar})$, and $G = F_{\eta}\in H_{\pi\hbar}(2\pi \Pb)$. Put 
\[
z\mapsto f(z):= \varphi_{\hbar}(z)g(z)
\]
as a function in $\Hol(S_{\pi\hbar})$. Then we can extend $f$ to a measurable function on $\R+\pi i \hbar$ such that 
\[
f(x+\pi i \hbar) \overset{\textrm{a.e.}}{=} \lim_{n\rightarrow \infty} f(x+ \pi i \hbar - i/n^2), 
\]
and using again \eqref{EqValueShift} we compute
\[
\int_{\R} (1+e^x)|f(x+\pi i \hbar)|^2\rd x =  \int_{\R} \frac{|f(x+\pi i \hbar)|^2}{|\varphi_{\hbar}(x+\pi i \hbar)|^2}\rd x =  \int_{\R}|g(x+\pi i \hbar)|^2\rd x =  \|G(\pi i \hbar)\|^2<\infty.
\]
Moreover, we then also have that $f\in H(S_{\pi\hbar})$, by another application of Lemma \ref{LemPL}. Since $\xi$ realizes the boundary value of $f$ on $\R$, we can hence conclude by Lemma \ref{LemDomainInt} that $\xi \in \msD(e^{-\Pb_{\hbar}/2}) \cap \msD(e^{-\Pb_{\hbar}/2\dotplus \Xb/2})$. 
\end{proof}

We can now prove  the identity \eqref{EqToProve}. Indeed, it is enough to show that with $\theta_{\xi}$ the projection onto a unit vector $\xi\in L^2(\R)$, we have
\begin{equation}\label{EqFormSumComp}
\Tr_{e^{-\Pb_{\hbar}}}(\theta_{\xi}) + \Tr_{e^{-\Pb_{\hbar}\dotplus \Xb}}(\theta_{\xi}) = \Tr_{\varphi_{\hbar}(\Xb) e^{-\Pb_{\hbar}}\varphi_{\hbar}(\Xb)^*}(\theta_{\xi}). 
\end{equation}
Now the left hand side is finite if and only if 
\begin{equation}\label{EqDomainInt}
\xi \in \msD(e^{-\Pb_{\hbar}/2}) \cap \msD(e^{-\Pb_{\hbar}/2\dotplus \Xb/2}),
\end{equation}
while the right hand side is finite if and only if $\varphi_{\hbar}(\Xb)^*\xi\in \msD(e^{-\Pb_{\hbar}/2})$. By Lemma \ref{LemEqIdDomains}, these are the same sets of vectors. Moreover, for a vector in this set, we then easily compute, using again \eqref{EqValueShift} and the computations in the above lemmas, that
\begin{eqnarray*}
\Tr_{\varphi_{\hbar}(\Xb) e^{-\Pb_{\hbar}}\varphi_{\hbar}(\Xb)^*}(\theta_{\xi}) &=& \| e^{-\Pb_{\hbar}/2} \varphi_{\hbar}(\Xb)^*\xi\|^2\\
&=& \int_{\R}\frac{|f_{\xi}(x+\pi i \hbar)|^2}{|\varphi_{\hbar}(x+\pi i \hbar)|^2} \rd x\\
&=& \int_{\R}(1+ e^{x}) |f_{\xi}(x+\pi i \hbar)|^2 \rd x\\
&=& \int_{\R} |f_{\xi}(x+\pi i\hbar)|^2 \rd x+  \int_{\R} e^{x} |f_{\xi}(x+\pi i \hbar)|^2 \rd x\\ 
&=& \Tr_{e^{-\Pb_{\hbar}}}(\theta_{\xi}) + \Tr_{e^{-\Pb_{\hbar}\dotplus \Xb}}(\theta_{\xi}).
\end{eqnarray*}

\subsection{Actions on von Neumann algebras}

Assume that $G$ is a locally compact Hausdorff group. Let $M$ be a von Neumann algebra. Then we say $G$ acts on $M$, and write $G \underset{\alpha}{\curvearrowright} M$, if we are given $*$-automorphisms $\alpha_g$ of $M$ for $g\in G$ giving an algebraic action and satisfying the continuity condition that for all $x\in M$, the map 
\[
\alpha(x)\colon  G \rightarrow M,\qquad g \mapsto \alpha_g(x)
\]
is $\sigma$-weakly continuous. Then the $G$-fixed points of $M$ form a von Neumann subalgebra
\[
M^{\alpha} = \{x\in M\mid \forall g\in G: \alpha_g(x) = x\}.
\]
Then with $\lambda$ a left Haar integral on $M$, one can define on $M$ a unique map 
\begin{equation}\label{EqOpValWeight}
T_{\alpha}\colon  M_+^{\ext} \rightarrow M_+^{\alpha,\ext}
\end{equation}
such that 
\[
\omega_{M^{\alpha}}(T_{\alpha}(x)) = \int_G (\omega\circ \alpha_g)(x)\rd \lambda(g),\qquad \omega \in M_*^+,x\in M_+^{\ext}.
\]
One refers to $T_{\alpha}$ as an  \emph{operator valued weight}. It again satisfies the normality condition that if $x_i \nearrow x$ strongly, then $T_{\alpha}(x_i) \nearrow T_{\alpha}(x)$ strongly. 

In particular, if $\alpha$ is ergodic, that is, $M^{\alpha} = \C 1_M$, then $T_{\alpha}$ is itself a normal weight, which we write as $\varphi_{\alpha}$. It hence satisfies that, for $\omega$ an \emph{arbitrary} normal state, on $M$,
\begin{equation}\label{EqIntegratedWeight}
\varphi_{\alpha}(x) = \int_G (\omega\circ \alpha_g)(x)\rd \lambda(g),\qquad x\in M_+^{\ext}.
\end{equation}
This weight is necessarily faithful. One then calls an ergodic action \emph{integrable} if $\varphi_{\alpha}$ is also semi-finite.

\section{Proof of Theorem \ref{TheoFlip}}\label{AppProofTheoFlip}

Recall that
\beq
\label{eq:F-1}
\Fc = \prod_{1 \le r \le n}^{\longra} \prod_{r \le b \le n}^{\longra} \prod_{1 \le i \le r}^{\longra} B^{b,r,i}.
\eeq

By similar direct calculations to those in Lemma~\ref{LemOrdering}, one shows the following result.

\begin{lemma}
\label{lem:bri-bri'}
Assuming $1 \le i \le r \le b \le n$ and similarly for the triple $(b',r',i')$ we have
\beq
\label{bri-bri'}
\hs{B^{b,r,i},B^{b',r',i'}}
\eeq
whenever any of the following conditions is satisfied:
\begin{enumerate}
\item $r'=r$;
\item $\hm{b'-i'-b+i}>1$;
\item $b'-i'=b-i+1$, \; $b'<b$, \; $b'-r' \le b-r$;
\item $b'-i'=b-i+1$, \; $b'\ge b$, \; $b'-r' > b-r$;
\item $b'-i'=b-i$, \; $(b'-b)(b'-r'-b+r)>0$.
\end{enumerate}
\end{lemma}

\begin{cor}
\label{cor:bri-bri'}
Relation~\eqref{bri-bri'} holds whenever $r'>r$ and any of the following conditions is satisfied:
\begin{enumerate}
\item $b'<b$;
\item $r'-i'<r-i$;
\item $r'-i'=r-i$ and $b'-r'>b-r$.
\end{enumerate}
\end{cor}

\begin{proof}
We prove part (1), the other two cases are analogous. First, note that $b'<b$ and $r'>r$ imply $b'-r'<b-r$. By Lemma~\ref{lem:bri-bri'}, we only need to consider possibilities
\begin{align*}
    b'-i'&=b-i, \\
    b'-i'&=b-i\pm1.
\end{align*}
In all three case, the result follows from Lemma~\ref{lem:bri-bri'}.
\end{proof}

Let us set 
$$
B^{b,r} = \prod_{i=1}^r B^{b,r,i}
\qquad\text{and}\qquad
B^b = \prod_{1 \le r \le b}^{\longra} B^{b,r}.
$$
This way we have
$$
\Fc = \prod_{1 \le r \le n}^{\longra} \prod_{r \le b \le n}^{\longra} B^{b,r}.
$$
Applying part~(1) of Corollary~\ref{cor:bri-bri'}, we see that that the flip $\Fc$ can be written as
\beq
\label{eq:F-2}
\Fc = \prod_{1 \le b \le n}^{\longra} B^b.
\eeq

In what follows, we shall also use the following expressions for $\Fc$. The first one is
\beq
\label{eq:F-4}
\Fc = \prod_{1 \le r \le n}^{\longra} \prod_{1 \le j \le r}^{\longra} \prod_{1 \le c \le n-r+1}^{\longra} B^{n-c+1,r,r-j+1}.
\eeq
and can be obtained from~\eqref{eq:F-1} by using part~(1) of Lemma~\ref{lem:bri-bri'} to write
\beq
\label{eq:G-aux}
\Fc = \prod_{1 \le r \le n}^{\longra} \prod_{1 \le i \le r}^{\longleftarrow} \prod_{r \le b \le n}^{\longleftarrow} B^{b,r,i},
\eeq
and re-parametrizing $j=r-i+1$, $c=n-b+1$.

The second one is
\beq
\label{eq:F-3}
\Fc = \prod_{1 \le b \le n}^{\longra} \prod_{1 \le r \le n-b+1}^{\longra} \prod_{1 \le i \le r}^{\longra} B^{n-r+i,b+i-1,i},
\eeq
and can be obtained from~\eqref{eq:G-aux} by commuting factors $B^{b',r',i'}$ past $B^{b,r,i}$ whenever $r'>r$ and conditions~(2) or~(3) of Corollary~\ref{cor:bri-bri'} hold.

Finally, the third one is derived in an analogous fashion and reads
\beq
\label{eq:F-5}
\Fc = \prod_{1 \le b \le n}^{\longra} \prod_{1 \le r \le n-b+1}^{\longra} \prod_{1 \le i \le r}^{\longra} B^{b+r-1,b+i-1,i}.
\eeq

Now the entire proof of Theorem \ref{TheoFlip} only consists of spelling out a picture in \cite{DGG16}. (Here we feel obliged to remark that in our opinion mathematical thought ought to go the opposite way --- from formulas to pictures).

We start by using formulas~\eqref{eq:F-2} and~\eqref{eq:F-3} to rewrite the left hand side as
$$
\prod_{1 \le b \le n}^{\longra} \prod_{1 \le r \le b}^{\longra} \prod_{1 \le i \le r}^{\longra} B_{[23]}^{b,r,i} \cdot \prod_{1 \le b \le n}^{\longra} \prod_{1 \le r \le n-b+1}^{\longra} \prod_{1 \le i \le r}^{\longra} B_{[12]}^{n-r+i,b+i-1,i}.
$$
We will now be transforming this expression step by step, so that after $s$ steps we arrive at
\begin{align*}
&\prod_{1 \le b \le s}^{\longra} \prod_{1 \le r \le b}^{\longra} \prod_{1 \le i \le n-b+1}^{\longra} B_{[12]}^{n-i+1,b,b-r+1} \cdot \prod_{1 \le b \le n-s-1}^{\longra} \prod_{1 \le r \le b}^{\longra} \prod_{1 \le i \le r}^{\longra} B_{[23]}^{b,r,i} \cdot \\
&\prod_{1\le a \le s+1}^{\longra} \Bigg( \prod_{1 \le r \le n-s}^{\longra} \prod_{1 \le i \le r}^{\longra} B_{[23]}^{n+a-s-1,r,i} \cdot
\prod_{1 \le r \le n-s}^{\longra} \prod_{1 \le i \le r}^{\longra} B_{[12]}^{n-r+i,s+i,s+i-a+1} \cdot \\
&\prod_{n-s+1 \le r \le n-a+1}^{\longra} \prod_{1 \le i \le r}^{\longra} B_{[13]}^{a+r-1,a+i-1,i} \Bigg) \\
&\prod_{s+2 \le b \le n}^{\longra} \prod_{1 \le r \le n-b+1}^{\longra} \prod_{1 \le i \le r}^{\longra} B_{[12]}^{n-r+i,b+i-1,i} \prod_{n-s+1 \le r \le n}^{\longra} \prod_{r \le b \le n}^{\longra} \prod_{1 \le i \le r}^{\longra} B_{[23]}^{b,r,i}.
\end{align*}

After $n$ steps we arrive at expression
\begin{align*}
&\prod_{1 \le b \le n}^{\longra} \prod_{1 \le r \le b}^{\longra} \prod_{1 \le i \le n-b+1}^{\longra} B_{[12]}^{n-i+1,b,b-r+1} \cdot \\
&\prod_{1\le a \le n}^{\longra} \prod_{1 \le r \le n-a+1}^{\longra} \prod_{1 \le i \le r}^{\longra} B_{[13]}^{a+r-1,a+i-1,i} \cdot \\
&\prod_{1 \le r \le n}^{\longra} \prod_{r \le b \le n}^{\longra} \prod_{1 \le i \le r}^{\longra} B_{[23]}^{b,r,i},
\end{align*}
which equals the right hand side, as the three factors are respectively by formulas~\eqref{eq:F-4},~\eqref{eq:F-5}, and~\eqref{eq:F-1}. Performing the $s$-th step consists of applying Lemma~\ref{lem:step} and shuffling commuting factors. Namely, after applying the lemma, inside the parenthesis, where the product over $a$ is taken, we get
\begin{align*}
\prod_{1 \le r \le n-s}^{\longra} B_{[12]}^{n-r+1,s+1,c+2} \prod_{1 \le r \le n-s-1}^{\longra} \prod_{1 \le i \le r}^{\longra} B_{[12]}^{n-r+i,s+i+1,c+i+2} \cdot \prod_{1 \le r \le n-s}^{\longra} B_{[13]}^{n-c-1,s+r-c-1,r} \\
\prod_{1 \le r \le n-s-1}^{\longra} \prod_{1 \le i \le r}^{\longra} B_{[23]}^{n-c-1,r,i} \prod_{1 \le i \le n-s}^{\longra}B_{[23]}^{n-c-1,n-s,i} \cdot \prod_{n-s+1 \le r \le n-a+1}^{\longra} \prod_{1 \le i \le r}^{\longra} B_{[13]}^{a+r-1,a+i-1,i}.
\end{align*}
Then we observe that the factors $B_{[23]}$ in the second line commute with the factors $B_{[13]}$. This way $\Fc_{[23]}\Fc_{[12]}$ equals
\begin{align*}
&\prod_{1 \le b \le s}^{\longra} \prod_{1 \le r \le b}^{\longra} \prod_{1 \le i \le n-b+1}^{\longra} B_{[12]}^{n-i+1,b,b-r+1} \cdot \prod_{1 \le b \le n-s-1}^{\longra} \prod_{1 \le r \le b}^{\longra} \prod_{1 \le i \le r}^{\longra} B_{[23]}^{b,r,i} \cdot \\
&\prod_{1\le a \le s+1}^{\longra} \Bigg( \prod_{1 \le r \le n-s}^{\longra} B_{[12]}^{n-r+1,s+1,s-a+2} \cdot \prod_{1 \le r \le n-s-1}^{\longra} \prod_{1 \le i \le r}^{\longra} B_{[12]}^{n-r+i,s+i+1,s+i-a+2} \cdot \\
&\prod_{n-s \le r \le n-a+1}^{\longra} \prod_{1 \le i \le r}^{\longra} B_{[13]}^{a+r-1,a+i-1,i} \cdot \prod_{1 \le r \le n-s-1}^{\longra} \prod_{1 \le i \le r}^{\longra} B_{[23]}^{n+a-s-1,r,i} \cdot \prod_{1 \le i \le n-s}^{\longra} B_{[23]}^{n+a-s-1,n-s,i}\Bigg) \\
&\prod_{s+2 \le b \le n}^{\longra} \prod_{1 \le r \le n-b+1}^{\longra} \prod_{1 \le i \le r}^{\longra} B_{[12]}^{n-r+i,b+i-1,i} \prod_{n-s+1 \le r \le n}^{\longra} \prod_{r \le b \le n}^{\longra} \prod_{1 \le i \le r}^{\longra} B_{[23]}^{b,r,i}.
\end{align*}
Using ~\eqref{eq:q-comm}, one can now easily verify the following. First, factors
$$
\prod_{1 \le r \le n-s}^{\longra} B_{[12]}^{n-r+1,s+1,s-a+2}
$$
can be commuted all the way to the left past the product
$$
\prod_{1 \le b \le n-s-1}^{\longra} \prod_{1 \le r \le b}^{\longra} \prod_{1 \le i \le r}^{\longra} B_{[23]}^{b,r,i},
$$
so that the first big factor in our expression becomes
\begin{multline*}
\prod_{1 \le b \le s}^{\longra} \prod_{1 \le r \le b}^{\longra} \prod_{1 \le i \le n-b+1}^{\longra} B_{[12]}^{n-i+1,b,b-r+1} \cdot \prod_{1\le a \le s+1}^{\longra} \prod_{1 \le r \le n-s}^{\longra} B_{[12]}^{n-r+1,s+1,s-a+2} = \\
= \prod_{1 \le b \le s+1}^{\longra} \prod_{1 \le r \le b}^{\longra} \prod_{1 \le i \le n-b+1}^{\longra} B_{[12]}^{n-i+1,b,b-r+1}.
\end{multline*}
Second, the factors
$$
\prod_{1 \le i \le n-s}^{\longra} B_{[23]}^{n+a-s-1,n-s,i}
$$
can be similarly commuted all the way to the right past the factor
$$
\prod_{s+2 \le b \le n}^{\longra} \prod_{1 \le r \le n-b+1}^{\longra} \prod_{1 \le i \le r}^{\longra} B_{[12]}^{n-r+i,b+i-1,i},
$$
so that the very last factor becomes
$$
\prod_{1\le a \le s+1}^{\longra} \prod_{1 \le i \le n-s}^{\longra} B_{[23]}^{n+a-s-1,n-s,i} \cdot  \prod_{n-s+1 \le r \le n}^{\longra} \prod_{r \le b \le n}^{\longra} \prod_{1 \le i \le r}^{\longra} B_{[23]}^{b,r,i} = \prod_{n-s \le r \le n}^{\longra} \prod_{r \le b \le n}^{\longra} \prod_{1 \le i \le r}^{\longra} B_{[23]}^{b,r,i}.
$$
At this point we have
\begin{align*}
&\prod_{1 \le b \le s+1}^{\longra} \prod_{1 \le r \le b}^{\longra} \prod_{1 \le i \le n-b+1}^{\longra} B_{[12]}^{n-i+1,b,b-r+1} \cdot \prod_{1 \le b \le n-s-2}^{\longra} \prod_{1 \le r \le b}^{\longra} \prod_{1 \le i \le r}^{\longra} B_{[23]}^{b,r,i} \cdot \\
&\prod_{1 \le a \le s+2}^{\longra} \Bigg( \prod_{1 \le r \le n-s-1}^{\longra} \prod_{1 \le i \le r}^{\longra} B_{[23]}^{n+a-s-2,r,i} \cdot \prod_{1 \le r \le n-s-1}^{\longra} \prod_{1 \le i \le r}^{\longra} B_{[12]}^{n-r+i,s+i+1,s+i-a+2} \cdot \\
&\prod_{n-s \le r \le n-a+1}^{\longra} \prod_{1 \le i \le r}^{\longra} B_{[13]}^{a+r-1,a+i-1,i} \Bigg) \\
&\prod_{s+3 \le b \le n}^{\longra} \prod_{1 \le r \le n-b+1}^{\longra} \prod_{1 \le i \le r}^{\longra} B_{[12]}^{n-r+i,b+i-1,i} \prod_{n-s \le r \le n}^{\longra} \prod_{r \le b \le n}^{\longra} \prod_{1 \le i \le r}^{\longra} B_{[23]}^{b,r,i},
\end{align*}
which is precisely the expression $\Fc_{[23]}\Fc_{[12]}$ after $s+1$ steps. This completes the proof.

\begin{lemma}
\label{lem:step}
The following equality holds (here $c=s-a$):
\begin{multline*}
\prod_{1 \le r \le n-s}^{\longra} \prod_{1 \le i \le r}^{\longra} B_{[23]}^{n-c-1,r,i} \cdot \prod_{1 \le r \le n-s}^{\longra} \prod_{1 \le i \le r}^{\longra} B_{[12]}^{n-r+i,s+i,c+i+1} = \\
\prod_{1 \le r \le n-s}^{\longra} B_{[12]}^{n-r+1,s+1,c+2} \prod_{1 \le r \le n-s-1}^{\longra} \prod_{1 \le i \le r}^{\longra} B_{[12]}^{n-r+i,s+i+1,c+i+2}\\
\prod_{1 \le r \le n-s}^{\longra} B_{[13]}^{n-c-1,s+r-c-1,r} \prod_{1 \le r \le n-s-1}^{\longra} \prod_{1 \le i \le r}^{\longra} B_{[23]}^{n-c-1,r,i} \prod_{1 \le i \le n-s}^{\longra}B_{[23]}^{n-c-1,n-s,i}.
\end{multline*}
\end{lemma}

\begin{proof}
As always using~\eqref{eq:q-comm}, we see that the left hand side is equal to
$$
\prod_{1 \le r \le n-s}^{\longra} \prod_{1 \le i \le r}^{\longra} B_{[23]}^{n-c-1,r,i} B_{[12]}^{n-r+i,s+i,c+i+1},
$$
which in turn equals
$$
\prod_{1 \le r \le n-s}^{\longra} \prod_{1 \le i \le r}^{\longra} \varphi\hr{0 \oplus \nee_{n-c-i,n-c-r-1} \oplus \seee_{c+i+1,c+1}}
\varphi\hr{\nee_{n-c-r,n-s-r} \oplus \seee_{c+r+1,r-i} \oplus 0}
$$
Applying the pentagon relation for quantum dilogarithms (one easily verifies that \eqref{EqCommKashaev} is valid), we get
$$
\prod_{1 \le r \le n-s}^{\longra} \prod_{1 \le i \le r}^{\longra} C_{[12]}^{r,i} C_{[13]}^{r,i} C_{[23]}^{r,i}.
$$
where
\begin{align*}
C_{[12]}^{r,i} &= B_{[12]}^{n-r+i,s+i,c+i+1} = \varphi\hr{\nee_{n-c-r,n-s-r} \oplus \seee_{c+r+1,r-i} \oplus 0} \\
C_{[13]}^{r,i} &= \varphi\hr{\nee_{n-c-r,n-s-r} \oplus \hr{\seee_{c+r+1,r-i} + \nee_{n-c-i,n-c-r-1}} \oplus \seee_{c+i+1,c+1}} \\
C_{[23]}^{r,i} &= B_{[23]}^{n-c-1,r,i}  = \varphi\hr{0 \oplus \nee_{n-c-i,n-c-r-1} \oplus \seee_{c+i+1,c+1}}.
\end{align*}
Note that
$$
C_{[13]}^{r,r} = B_{[13]}^{n-c-1,s+r-c-1,r}.
$$
Using our favorite commutation relations~\eqref{eq:q-comm} we then arrive at
$$
\prod_{1 \le r \le n-s}^{\longra} C_{[12]}^{r,1} \prod_{1 \le k \le n-s-1}^{\longra} \hr{C_{[13]}^{k,k} \prod_{1 \le j \le k}^{\longra} \hr{C_{[23]}^{k,j} C_{[13]}^{k+1,j} C_{[12]}^{k+1,j+1}} } C_{[13]}^{n-s,n-s} \prod_{1 \le i \le n-s}^{\longra}C_{[23]}^{n-s,i}.
$$
Using the pentagon relation once again, we obtain
$$
\prod_{1 \le r \le n-s}^{\longra} C_{[12]}^{r,1} \prod_{1 \le k \le n-s-1}^{\longra} \hr{C_{[13]}^{k,k} \prod_{1 \le j \le k}^{\longra} \hr{C_{[12]}^{k+1,j+1} C_{[23]}^{k,j}} } C_{[13]}^{n-s,n-s} \prod_{1 \le i \le n-s}^{\longra}C_{[23]}^{n-s,i},
$$
and further commute it to the form
$$
\prod_{1 \le r \le n-s}^{\longra} C_{[12]}^{r,1} \prod_{1 \le k \le n-s-1}^{\longra} \prod_{1 \le j \le k}^{\longra} C_{[12]}^{k+1,j+1} \prod_{1 \le k \le n-s}^{\longra} C_{[13]}^{k,k} \prod_{1 \le k \le n-s-1}^{\longra} \prod_{1 \le j \le k}^{\longra} C_{[23]}^{k,j} \prod_{1 \le i \le n-s}^{\longra}C_{[23]}^{n-s,i}.
$$
This in turn equals the desired expression.
\end{proof}


\begin{thebibliography}{GHKK18}
\bibitem[BS93]{BS93} S.\,Baaj and G.\,Skandalis. ``Unitaires multiplicatifs et dualit\'{e} pour les produits crois\'{e}s de C$^*$-alg\`{e}bres.'' \emph{Annales scientifiques de l'\'Ecole Normale Sup\'erieure} \textbf{26} (4) (1993): 425--488.
\bibitem[BV05]{BV05} S.\,Baaj and S.\,Vaes. ``Double crossed products of locally compact quantum groups.'' \emph{Journal of the Institute of Mathematics of Jussieu} \textbf{4} (1) (2005): 135--173.
\bibitem[BZBJ18]{BZBJ18}
D.\,Ben-Zvi, A.\,Brochier, and D.\,Jordan.
``Integrating quantum groups over surfaces.''
\emph{Journal of Topology} \textbf{11} (4) (2018): 874-917.
\bibitem[BZ05]{BZ05} A.\,Berenstein and A.\,Zelevinsky. ``Quantum cluster algebras.'' \emph{Advances in Mathematics} \textbf{195} (2) (2005): 405--455.
\bibitem[BT03]{BT03} A.\,G.\,Bytsko and J.\,Teschner. ``$R$-Operator, Co-Product and Haar-Measure for the Modular Double of $U_q(\mathfrak{sl}(2,\R))$.'' \emph{Communications in mathematical physics} \textbf{240} (1) (2003): 171--196.
\bibitem[Com68]{Com68} F.\,Combes. ``Poids sur une C$^*$-alg\`{e}bre.'' \emph{Journal de Math\'ematiques Pures et Appliqu\'ees} \textbf{47} (1) (1968): 57--100.
\bibitem[DM21]{DM21} B.\,Davison and T.\,Mandel.
``Strong positivity for quantum theta bases of quantum cluster algebras.'' \emph{Inventiones mathematicae} \textbf{226} (3) (2021): 725--843.
\bibitem[DCVD10]{DCVD10} K.\,De Commer and A.\,Van Daele. ``Multiplier Hopf Algebras Imbedded in Locally Compact Quantum Groups.'' \emph{The Rocky Mountain Journal of Mathematics} 
\textbf{40} (4) (2010): 1149--1182.
\bibitem[DGG16]{DGG16} T.\,Dimofte, M.\,Gabella, and A.\,B.\,Goncharov. ``$K$-Decompositions and $3d$ Gauge Theories'', \emph{ournal of High Energy Physics} \textbf{2016} (11) (2016): 1-147.
\bibitem[Dri87]{Dri87} V.\,G.\,Drinfeld, ``Quantum groups''. In: \emph{Proceedings of the International Congress of Mathematicians} \textbf{1} (Berkeley,
California, 1986), \emph{American Mathematical Society} (1987): 798--820.
\bibitem[Fad95]{Fad95} L.\,D.\,Faddeev, ``Discrete Heisenberg--Weyl group and modular group.'' \emph{Letters in Mathematical Physics} \textbf{34} (3) (1995): 249--254.
\bibitem[FK94]{FK94} L.\,D.\,Fadeev and R.\,M.\,Kashaev. ``Quantum dilogarithm.'' \emph{Modern Physics Letters A} \textbf{9} (5) (1994): 427--434.
\bibitem[FKV01]{FKV01}  L.\,D.\,Faddeev, R.\,M.\,Kashaev, and A.\,Yu.\,Volkov. ``Strongly Coupled Quantum Discrete Liouville Theory. I: Algebraic Approach and Duality.'' \emph{Communications in Mathematical Physics} \textbf{219} (1) (2001): 199--219.
\bibitem[FG06]{FG06} V.\,Fock and A.\,Goncharov.
``Moduli spaces of local systems and higher Teichm\"uller theory.''
\emph{Publications Math\'ematiques de l'IH\'ES} \textbf{103} (2006): 1--211.
\bibitem[FG09]{FG09} V.\,V.\,Fock and A.\,B.\,Goncharov. ``The quantum dilogarithm and representations of
quantum cluster varieties.'' \emph{ Inventiones mathematicae} \textbf{175} (2) (2009): 223--286.
\bibitem[FZ02]{FZ02} S.\,Fomin and A.\,Zelevinsky. ``Cluster algebras. I. Foundations.'' \emph{Journal of the American mathematical society} \textbf{15} (2) (2002): 497--529.
\bibitem[FI14]{FI14} I.\,B.\,Frenkel and I.\,Ip. ``Positive representations of split real quantum groups and future perspectives.''
\emph{International Mathematics Research Notices} \textbf{8} (2014): 2126--2164. 
\bibitem[GK21]{GK21} S.\,Garoufalidis and R.\,Kashaev. ``Resurgence of Faddeev’s quantum dilogarithm.'' In: \emph{Topology and Geometry: A collection of essays dedicated to Vladimir G. Turaev. IRMA lectures in Mathematics and Theoretical Physics} \textbf{33} (2021): 257--272.
\bibitem[GHKK18]{GHKK18} M.\,Gross, P.\,Hacking, S.\,Keel, and M.\,Kontsevich. ``Canonical bases for cluster algebras.'' \emph{ournal of the American Mathematical Society} \textbf{31} (2) (2018): 497--608.
\bibitem[Ip12]{Ip12} I.\,Ip. ``Positive representations and harmonic analysis of split real quantum groups.'' \emph{ProQuest LLC, Ann Arbor, MI. Thesis (Ph.D.) -- Yale University} (2012).
\bibitem[Ip13]{Ip13} I.\,Ip. ``Representation of the quantum plane, its quantum double and harmonic analysis on $GL_q^+(2,\R)$.'' \emph{Selecta Mathematica New Series} \textbf{19} (4) (2013): 987--1082.
\bibitem[Ip15]{Ip15} I.\,Ip. ``Positive representations of non-simply-laced split real quantum groups.'' \emph{Journal of Algebra} \textbf{425} (2015): 245--276. 
\bibitem[Ip17]{Ip17} I.\,Ip. ``Positive representations, multiplier Hopf algebra, and continuous canonical basis.'' \emph{String theory, integrable systems and representation theory, RIMS Kokyuroku Bessatsu} B62 (2017): 71--86.
\bibitem[Jim85]{Jim85} M.\,Jimbo. ``A $q$-difference analogue of $U(\mathfrak{g})$ and the Yang--Baxter equation.'' \emph{Letters in Mathematical Physics} \textbf{10} (1) (1985): 247--252.
\bibitem[Kas01]{Kas01} R.\,Kashaev. ``On the spectrum of Dehn twists in quantum Teichm\"{u}ller Theory.'' \emph{Physics and Combinatorics} (2001): 63--81.
\bibitem[KS97]{KS97} A.\,Klimyk and K.\,Schm\"{u}dgen. ``Quantum groups and their representations.'' \emph{Texts and Monographs in Physics}, \emph{Springer} (1997).
\bibitem[KK03]{KK03} E.\,Koelink and J.\,Kustermans. ``A locally compact quantum group analogue of the normalizer of
$SU(1,1)$ in $SL(2,\C)$.'' \emph{Communications in mathematical physics} \textbf{233} (2) (2003): 231--296.
\bibitem[Kus05]{Kus05} J.\,Kustermans. ``Locally compact quantum groups''. In: \emph{Quantum independent increment processes, Lecture Notes in Mathematics} \textbf{1865} (2005): 99--180.
\bibitem[KV00]{KV00} J.\,Kustermans and S.\,Vaes. ``Locally compact quantum groups.'' \emph{Annales Scientifiques de l’Ecole Normale Sup\'erieure}
\textbf{33} (6) (2000): 837--934.
\bibitem[KV03]{KV03} J.\,Kustermans and S.\,Vaes. ``Locally compact quantum groups in the von Neumann algebraic setting.'' \emph{Mathematica Scandinavica} \textbf{92} (1) (2003): 68--92.
\bibitem[KVD97]{KVD97} J.\,Kustermans and A.\,Van Daele. ``{$C^*$}-algebraic quantum groups arising from algebraic quantum groups.''
\emph{International Journal of Mathematics} \textbf{8} (8) (1997): 1067--1139. 
\bibitem[Lus93]{Lus93} G.\,Lusztig. ``Introduction to quantum groups.'' \emph{Progress in Mathematics} \textbf{110}, Birkha\"{u}ser Boston, MA (1993). 
\bibitem[MRW17]{MRW17} R.\,Meyer, S.\,Roy, and S.\,L.\,Woronowicz. ``Semidirect Products of C$^*$-Quantum Groups: Multiplicative Unitaries Approach.'' \emph{Communications in Mathematical Physics} \textbf{351} (1) (2017): 249--282.
\bibitem[PT73]{PT73} G.\,K.\,Pedersen and M.\,Takesaki. ``The Radon--Nikodym theorem for von Neumann algebras.'' \emph{Acta Mathematica} \textbf{130} (1973): 53--87.
\bibitem[PT99]{PT99} B.\,Ponsot and J.\,Teschner. ``Liouville bootstrap via harmonic analysis on a noncompact quantum group.'' \emph{arXiv preprint} arXiv:hep-th/9911110.
\bibitem[PT01]{PT01} B.\,Ponsot and J.\,Teschner. ``Clebsch--Gordan and Racah--Wigner coefficients for a continuous series of representations of {$U_q(\sl(2,\mathbb{R}))$}.'' \emph{Communications in Mathematical Physics} \textbf{224} (3) (2001): 613--655. 
\bibitem[RS81]{RS81} M.\,Reed and B.\,Simon. ``Methods of Modern Mathematical Physics I: Functional Analysis.'' Academic Press (1981).
\bibitem[Rie93]{Rie93} M.\,A.\,Rieffel. ``Deformation Quantization for Actions of $\R^d$.'' \emph{
Memoirs of the American Mathematical Society} \textbf{106} (506) (1993).
\bibitem[Rui97]{Rui97} S.\,N.\,M.\,Ruijsenaars. ``First order analytic difference equations and integrable quantum systems.''
\emph{Journal of Mathematical Physics} \textbf{38} (2) (1997): 1069--1146.
\bibitem[Rui05]{Rui05} S.\,N.\,M.\,Ruijsenaars. ``A unitary joint eigenfunction transform for the A$\Delta$Os $\exp(ia_{\pm}d/dz)$ $+$ $\exp(2\pi z/a_{\mp})$.'' \emph{Journal of Nonlinear Mathematical Physics} \textbf{12} (2) (2005): 253--294.
\bibitem[Sch94]{Sch94} K.\,Schm\"{u}dgen. ``Integrable Operator Representations of $\R_q^2$, $X_{q,\gamma}$ and $SL_q(2,\R)$.'' \emph{Communications in mathematical physics} \textbf{159} (2) (1994): 217--237.
\bibitem[Sch12]{Sch12} K.\,Schm\"{u}dgen. ``Unbounded self-adjoint operators on Hilbert Space.'' \emph{Graduate Texts in Mathematics} \textbf{265} (1) (2012), Springer.
\bibitem[SS17]{SS17} G.\,Schrader and A.\,Shapiro. ``Continuous tensor categories from quantum groups I: algebraic aspects.'' \emph{arXiv preprint} arXiv:1708.08107.
\bibitem[SW01]{SW01} P.\,M.\,So\l tan and S.\,L.\,Woronowicz. ``A remark on manageable multiplicative unitaries.'' \emph{Letters in Mathematical Physics} \textbf{57} (2001): 239--252.
\bibitem[SW07]{SW07} P.\,M.\,So\l tan and S.\,L.\,Woronowicz. ``From multiplicative unitaries to quantum groups II.'' \emph{Journal of Functional Analysis} \textbf{252} (1) (2007): 42--67.
\bibitem[Tak02]{Tak02} M.\,Takesaki. ``Theory of operator algebras I.'' \emph{Encyclopaedia of Mathematical Sciences} \textbf{124}, Springer--Verlag, Berlin (2002). 
\bibitem[Tak03]{Tak03} M.\,Takesaki. ``Theory of operator algebras II.'' \emph{Encyclopaedia of Mathematical Sciences} \textbf{125}, Springer--Verlag, Berlin (2003).
\bibitem[Vae01]{Vae01} S.\,Vaes. ``A Radon--Nikodym theorem for von Neumann algebras.'' \emph{Journal of Operator Theory} \textbf{46} (3) (2001): 477--489.
\bibitem[VVD03]{VVD03} S.\,Vaes and A.\,Van Daele. ``The Heisenberg commutation relations, commuting squares and the Haar measure on locally compact quantum groups.'' In: \emph{Operator algebras and mathematical physics: conference Proceedings, Constanta (Romania), July 2–7, 2001. Theta Foundation, Bucarest} (2003): 379--400.
\bibitem[VD98]{VD98} A.\,Van Daele. ``An Algebraic Framework for Group Duality.'' \emph{Advances in Mathematics} \textbf{140} (2) (1998): 323--366.
\bibitem[VD01]{VD01} A.\,Van Daele. ``The Haar measure on some locally compact quantum groups.'' \emph{arXiv preprint} arXiv:math/0109004.
\bibitem[Wor91]{Wor91} S.\,L.\,Woronowicz. ``Unbounded elements affiliated with C$^*$-algebras and noncompact quantum groups.'' \emph{Communications in mathematical physics} \textbf{136} (2) (1991): 399--432.
\bibitem[Wor95]{Wor95} S.\,L.\,Woronowicz. ``C$^*$-algebras generated by unbounded elements.'' \emph{Reviews in Mathematical Physics} \textbf{7} (3) (1995): 481--521.
\bibitem[Wor96]{Wor96} S.\,L.\,Woronowicz. ``From multiplicative unitaries to quantum groups.'' \emph{International Journal of Mathematics} \textbf{7} (1) (1996): 127--149. 
\bibitem[Wor00]{Wor00} S.\,L.\,Woronowicz. ``Quantum exponential function.'' \emph{Reviews in Mathematical Physics} \textbf{12} (6) (2000): 873--920.
\bibitem[Wor03]{Wor03} S.\,L.\,Woronowicz. ``Haar weight on some quantum groups.'' In: \emph{Group 24: Physical and Mathematical Aspects of Symmetries, Proceedings of the 24th
International Colloquium on Group Theoretical Methods in Physics, Paris, 15–20 July 2002. Institute of Physics Conference Series} \textbf{173} (2003): 763--772.
\bibitem[WN92]{WN92} S.\,L.\,Woronowicz and K.\,Napi\'{o}rkowski. ``Operator theory in the C$^*$-algebra framework.'' \emph{Reports on mathematical physics} \textbf{31} (3) (1992):  353--371.
\bibitem[WZ02]{WZ02} S.\,L.\,Woronowicz and S.\,Zakrzewski. ``Quantum `$ax+b$' group.'' 
\emph{Reviews in Mathematical Physics} \textbf{14} (2002): 797--828.
\end{thebibliography}
\end{document}